\documentclass[12pt,a4paper]{amsart}

\usepackage[top = 2.5cm, bottom = 2cm, right = 2cm, left = 2cm]{geometry}

\usepackage[utf8]{inputenc}
\usepackage{amsmath}
\usepackage{amsfonts}
\usepackage{amssymb}
\usepackage{amsthm}
\usepackage[all]{xy}
\usepackage{color}
\usepackage{hyperref}
\usepackage{xypic}
\usepackage{mathrsfs}
\usepackage{centernot}
\usepackage{tikz-cd}
\usepackage{enumitem}

\theoremstyle{plain}
\newtheorem{theorem}{Theorem}[section]

\newtheorem{conjecture}[theorem]{Conjecture}

\newtheorem{corollary}[theorem]{Corollary}
\newtheorem{definition}[theorem]{Definition}

\newtheorem{lemma}[theorem]{Lemma}

\newtheorem{proposition}[theorem]{Proposition}

\newtheorem{remark}[theorem]{Remark}

\DeclareMathOperator{\codim}{codim}

\DeclareMathOperator{\Div}{div}

\DeclareMathOperator{\Gal}{Gal}

\DeclareMathOperator{\height}{ht}

\DeclareMathOperator{\Pic}{Pic}

\DeclareMathOperator{\Stab}{Stab}

\newcommand{\del}{\partial}

\newcommand{\IA}{\mathbb{A}}

\newcommand{\IZ}{\mathbb{Z}}

\newcommand{\IQ}{\mathbb{Q}}

\newcommand{\IK}{\mathbb{K}}

\newcommand{\IKbar}{\overline{\mathbb{K}}}
\newcommand{\IQbar}{\overline{\mathbb{Q}}}
\newcommand{\IR}{\mathbb{R}}
\newcommand{\IC}{\mathbb{C}}

\newcommand{\Gm}{\mathbb{G}_m}
\newcommand{\IP}{\mathbb{P}}

\newcommand{\IF}{\mathbb{F}}

\newcommand{\vol}{\mathrm{vol}}

\newcommand{\dprime}{{\prime\prime}}

\newcommand{\pr}{\mathrm{pr}}

\newcommand{\caO}{\mathcal{O}}

\newcommand{\id}{\mathrm{id}}

\newcommand{\Ltil}{\widetilde{L}}

\newcommand{\Mtil}{\widetilde{M}}

\newcommand{\Ntil}{\widetilde{N}}
\newcommand{\Otil}{\widetilde{\mathcal{O}}}

\newcommand{\Lo}{\overline{L}}
\newcommand{\Mo}{\overline{M}}

\usepackage[T2D,T1]{fontenc}

\newcounter{cnstcnt}
\newcommand{\newconstant}{%
\refstepcounter{cnstcnt}
\ensuremath{c_{\thecnstcnt}}}
\newcommand{\oldconstant}[1]{\ensuremath{   c_{\ref*{#1}}}}

\AtBeginDocument{%
   \def\MR#1{}
}

\usepackage{bm}

\makeatletter
\newcommand\footnoteref[1]{\protected@xdef\@thefnmark{\ref{#1}}\@footnotemark}
\makeatother

\newcommand{\volh}{\widehat{\mathrm{vol}}}

\date{\today}
\begin{document}

    \author[F. Barroero]{Fabrizio Barroero}
\address[F. Barroero]{Universit\`a degli studi Roma Tre, Dipartimento di Matematica e Fisica, Largo San Murialdo 1, 00146 Roma, Italy}
\email{fabrizio.barroero@uniroma3.it}

\author[G. A. Dill]{Gabriel A. Dill}
\address[G. A. Dill]{Institut de Math\'ematiques, Universit\'e de Neuch\^atel, Rue Emile-Argand 11, 2000 Neuch\^atel, Switzerland}
\email{gabriel.dill@unine.ch}

\author[L. Kuehne]{Lars Kuehne}
\address[L. Kuehne]{UCD School of Mathematics and Statistics, University College Dublin, Belfield, Dublin 4, Ireland}
\email{lars.kuehne@ucd.ie}
	

	\title[Intersecting a curve with multiples of another curve]{Intersecting a curve in an abelian variety with multiples of another curve}
    
	\begin{abstract}
    Levin asked what can be said about the locus of points lying on a given curve in $\mathbb{G}_m^n$ that have a non-zero integer power on another given curve in $\mathbb{G}_m^n$ for $n \geq 3$. We give a definite answer to the abelian analog of Levin's question, proving what is predicted by the Zilber-Pink conjecture in this case. An important ingredient in our proof is a strengthening of a height inequality by Vojta and R\'emond.
    \end{abstract}
	
	\maketitle

    \tableofcontents

    \section{Introduction}

    Since the pioneering work of Bombieri, Masser, and Zannier \cite{Bombieri1999}, unlikely intersections in both semiabelian varieties and Shimura varieties have been a central topic in arithmetic geometry. For semiabelian varieties, the main expectations are summarized by conjectures proposed independently by Zilber \cite{Zilber2002} and Pink \cite{Pink2005a}, refining ideas already present in \cite{Bombieri1999}.
    
    For illustration, let us state a version of their conjectures in the setting of semiabelian varieties. We write $G^{[r]}$ for the union of all algebraic subgroups having codimension $\geq r$ in a given semiabelian variety $G$. Note that $G^{[r]}$ is in general a countable union of algebraic subvarieties, which is not necessarily Zariski-closed. However, the following conjecture stemming from the work of Zilber and Pink is widely accepted. Throughout this article, we let $K$ denote an algebraically closed field of characteristic $0$.

    \begin{conjecture}
    \label{conjecture::pinkzilber}
	   Let $G$ be a semiabelian variety defined over $K$ and let $X \subseteq G$ be an irreducible subvariety. If $X$ is not contained in any algebraic subgroup $H \subsetneq G$, then the intersection $X \cap G^{[\dim(X)+1]}$ is not Zariski dense in $X$.
    \end{conjecture}

    This seemingly innocuous conjecture has the renowned Mordell-Lang conjecture as a rather immediate consequence \cite[Theorem 5.3]{Pink2005a}. Despite the work of many authors \cite{Bays2015,Bombieri1999,Carrizosa2009,Maurin2008,Remond2005a,Remond2003a,Remond2007,Remond2009}, it remains widely open --- even if $G$ is the multiplicative torus $\Gm^g$ and $\dim(X) \geq 2$. However, the case where $K=\IQbar$ and $\dim(X)=1$ has been settled by Maurin \cite{Maurin2008} as well as Bombieri, Habegger, Masser, and Zannier \cite{Bombieri2010} if $G$ is a multiplicative torus, by Habegger and Pila \cite{Habegger2016} if $G$ is an abelian variety, and by Barroero, Kuehne, and Schmidt \cite{BarroeroKuehneSchmidt2023} in general. In this article, we prove a related consequence of Conjecture \ref{conjecture::pinkzilber}, which has been suggested by Levin in the case of multiplicative tori and can be found in Zannier's monograph on unlikely intersections \cite[Notes to Chapter 1]{Zannier2012}. 

    Recently, OpenAI has released the preprint \cite{OpenAI} claiming a complete proof of Conjecture \ref{conjecture::pinkzilber} for abelian varieties over $\IQbar$. By \cite{BarroeroDill2022}, this would also imply the Zilber-Pink conjecture for abelian varieties over $K$. This release prompted us to publish a first version of this article the day after on arXiv. The work contained here was already presented at an ICM Satellite Conference in August 2026 (Workshop on Arakelov and Diophantine Geometry) as well as in the Manchester Number Theory seminar before the release of  \cite{OpenAI}.

    The following is the main theorem of this article.

    \begin{theorem}
        \label{theorem::main}
		Let	$A$ be an abelian variety of dimension $\geq 3$ defined over $K$ and let $C_1, C_2 \subseteq A$ be irreducible algebraic curves such that there is no proper algebraic subgroup of $A$ containing both of them. Then, every irreducible component of the Zariski closure of the set
        \begin{equation*}
        \Xi =\{(x,y) \in (C_1 \times C_2)(K) \mid  [n](x) = y \text{ for some } n \in \mathbb{Z}
        \}
        \end{equation*}
        has to be one of the following:
        \begin{enumerate}
            \item[(i)] a point,
            \item[(ii)] a graph $(\id_A,[n])(C_1)$ for some $n \in \mathbb{Z}$,
            \item[(iii)] a horizontal curve $C_1 \times \{y_0\}$ for some $y_0 \in C_2(K) \setminus \{0_A \}$, or
            \item[(iv)] a vertical curve $\{x_0\} \times C_2$ for some $x_0 \in C_1(K)$.
        \end{enumerate}
        Furthermore, there are at most three components of type (ii) and, if there are three, then one must be $C_1 \times \{0_A\}$. Finally, components of type (iii) or (iv) can appear only if $C_1$ or $C_2$ respectively is a component of an algebraic subgroup.
    \end{theorem}

    There can indeed be three components of type (ii) as one can see by taking $C_1$ equal to a hyperelliptic curve of genus at least 3, embedded into its Jacobian $J$ by use of a Weierstrass point as a base point so that $[-1](C_1) = C_1$ and $0_J \in C_1(K)$, and $C_2 = [n](C_1)$ for some arbitrary $n \in \mathbb{Z} \backslash \{0\}$.
    
    The assertion of this theorem is related to Conjecture \ref{conjecture::pinkzilber} for the surface $X=C_1 \times C_2$, which puts it outside of the reach of the previously known results on unlikely intersections. Nevertheless, Bays and Habegger \cite{Bays2015} proved the analogous statement in the setting of multiplicative tori in the case where $C_1 = C_2$, but under a further technical condition on the subcurve $C = C_1 = C_2 \subseteq \Gm^g$. They needed to assume that $K = \IQbar$ and that, for some homomorphism $\sigma: K \hookrightarrow \IC$, the curve $C \times_{K, \sigma} \IC$ intersects the maximal compact subgroup of $\Gm^g(\IC)$ in at most finitely many points. This condition presents a substantial obstruction to any generalization of their approach to general semiabelian varieties. In fact, the maximal compact subgroup of an abelian variety is the abelian variety itself. Thus, there is no technical condition for curves in abelian varieties similar to the one in \cite{Bays2015}. Although we think that the method presented here will most likely work in the setting of \cite{Bays2015} and general semiabelian varieties as well, we restrict ourselves here to the case of abelian varieties.
    
    Other work on Levin's question has been conducted by Boxall \cite{Boxall}, Smith \cite{Smith}, Baldi and Urbanik \cite{Baldi_Urbanik}, and Dogra \cite{Dogra}. See also \cite[Sections 4.2 and 4.3]{Ballini} for some results related to an analog in powers of $Y(1)$. While completing this article, the authors also became aware of independent work of Smith \cite{Smith26} on related questions.

    We already mentioned \cite{Habegger2016}, where Habegger and Pila establish Conjecture \ref{conjecture::pinkzilber} in the case of curves in abelian varieties over $\IQbar$. In the same article, they also prove a result for subvarieties of
arbitrary dimension \cite[Theorem~1.3(ii)]{Habegger2016}: the set $X \cap A^{[\dim(X)+1]}$ is not Zariski dense in $X$ if the image of $X$ under every quotient map $A \to A/B$, with $B$ an abelian subvariety, is equal to $A/B$ or has the same dimension as $X$. As already mentioned in \cite{Habegger2016}, squares (and products) of curves in the square of an abelian variety of dimension $\geq 2$ do not satisfy this hypothesis and therefore this result does not cover the setting of our Theorem \ref{theorem::main}.

    We now give a sketch of the proof of Theorem \ref{theorem::main}. In Lemma \ref{lemma::main_reduction}, we reduce Theorem \ref{theorem::main} to the statement that the set $\Xi$ is not Zariski dense in $C_1 \times C_2$ (Theorem \ref{theorem::main_reduction}). If we additionally assume that $C_1$ and $C_2$ are not both contained in translates of the same proper algebraic subgroup of $A$ (which is a harmless assumption given that the Mordell-Lang conjecture for abelian varieties is known by work of Vojta \cite{Vojta1991}, Faltings \cite{Faltings1994}, and Hindry \cite{Hindry1988}), then, for every $n\ne 0$, the restriction to $C_1 \times C_2$ of the map
\begin{equation*}
	   \varphi_n : A \times A \rightarrow A, \ (x, y) \longmapsto [n] (x) - y
    \end{equation*}
is generically finite onto its image. Together with the theorem of the cube, this gives a geometric degree estimate (Lemma \ref{lemma::degreebound}), which is later matched by the corresponding arithmetic analog.

 Theorem \ref{theorem::mainQbar} is the main result of the article and establishes Theorem \ref{theorem::main_reduction} in the case 
$K=\overline{\mathbb Q}$. The proof is reduced to a height inequality (Proposition \ref{proposition::21st_century_vojta}) of Vojta type: for every $(x,y) \in (C_1 \times C_2)(\overline{\mathbb Q})$, and for $|n|$ sufficiently large, one has
\begin{equation}
    \label{equation::vojta}
        h_{\widetilde{L}}([n](x)-y) \gg n^2 h_{\widetilde{L}}(x)+h_{\widetilde{L}}(y),
\end{equation}
where $L$ is a symmetric, very ample line bundle on $A$ and $C_1$ and $C_2$ are now assumed to have finite stabilizer as well.

This immediately implies the desired non-Zariski density in the arithmetic case (under the mentioned harmless additional assumptions on $C_1$ and $C_2$). Indeed, if $y=[n](x)$, then the left-hand side is zero; hence, the inequality forces $x$ and $y$ to have canonical height zero, and Manin-Mumford gives only finitely many such points on $C_1$ and $C_2$. For the case $C_1=C_2$, a bound similar to \eqref{equation::vojta} appears in the work of R\'emond (\cite[Th\'eor\`eme 1.1]{Remond2000}) as a consequence of his generalized Vojta inequality. However, a large height assumption appears throughout the results in \cite{Remond2000}. The essential improvement here is to remove this assumption. At the same time, we use in our proof the generalized Vojta-R\'emond inequality from \cite{D20} to deal with the large height case.

Our proof of \eqref{equation::vojta} revisits ideas from the original works of Faltings \cite{Faltings1991} and Vojta \cite{Vojta1991} and is based on the controlled construction of an auxiliary section using Arakelov theory. We supply the necessary background in Section \ref{section::arithmetic}. A novelty is an equidistribution argument (Subsection \ref{subsection::equidistribution}) to prove that the height of $C_1\times C_2$ with respect to $\Ltil_n= \varphi_n^\ast \Ltil$ grows at least like a positive multiple of $n^2$ (Proposition \ref{proposition::heightbound}). In the case $C_1 = C_2$, such a lower bound can be proved by a simple argument using properties of arithmetic intersection numbers (Proposition \ref{proposition::heightbound_symmetric}). However, no such argument seems feasible in the general case.

An adelic form of Minkowski's Theorem produces, for each large $n$, a non-zero section $\mathbf{s}_n$ of a high tensor power of $L_n=\varphi_n^\ast L$ which is very small at one chosen place and bounded elsewhere, thanks to the above-mentioned estimate for the height of $C_1 \times C_2$. Since any section can be used to calculate the height of a point, provided it does not vanish there, if $p=(x,y)$ does not belong to the divisor of $\mathbf{s}_n$, these bounds for $\mathbf{s}_n$ give the desired inequality directly. If $p$ lies on the divisor, one considers the index of $\mathbf{s}_n$ at $p$. Points of large index are controlled by the Product Theorem (Proposition \ref{prop:productthm}), which forces them to have bounded degree so that only finitely many points of large index can violate the height inequality (here, we use the above-mentioned generalized Vojta-R\'emond inequality as well as a sufficiently uniform version of the Bogomolov conjecture from \cite{David2002}). Points of small index are treated by differentiating $\mathbf{s}_n$ to its first non-vanishing leading coefficient at $p$ thereby obtaining a rational section which does not vanish at $p$ and can be used in the height formula. The uniform Schwarz estimates of Section \ref{section::schwarz} control the effect of these differentiations at all places. For their proof, we use (finitely many) local affine charts and eventually reduce the proof to the case of open balls in affine space. This requires the quantitatively controlled extension of polynomial functions. For this control, we use models at non-archimedean places and Oka's extension theorem for holomorphic functions at archimedean places.

Finally, Section \ref{section::proof_general} deduces Theorem \ref{theorem::main_reduction} over an arbitrary algebraically closed field of characteristic $0$ from the arithmetic case using ideas from \cite{BarroeroDill2022}.
One may assume that the ground field $K$ has finite transcendence degree over $\overline{\mathbb{Q}}$ and we let $K_0$ be the ``minimal field of definition'' of $A$. We argue by induction on $\mathrm{trdeg}(K/K_0)$ and note that if $K_0= \overline{\mathbb{Q}}$ the base case is nothing but Theorem \ref{theorem::mainQbar}. If $K_0\ne \overline{\mathbb{Q}}$ we rely on \cite[Theorem 3.1]{BarroeroDill2022}, which follows from \cite{Gao2018}, a form of ``structure theorem'' for families of abelian varieties, and the Zilber-Pink conjecture for curves in abelian varieties (\cite[Corollary 1.6]{BarroeroDill2022}) for arbitrary algebraically closed fields of characteristic $0$.
The inductive step relies on a more traditional ``structure theorem'' for abelian varieties due to R\'emond \cite{Remond2009} and again \cite[Corollary 1.6]{BarroeroDill2022}.

\subsection*{Notations}

    If $F$ is any field of characteristic $0$, then an $F$-variety is a separated, reduced scheme of finite type over $F$, but we allow varieties to be reducible. If $F$ is our fixed algebraically closed base field $K$, then we will also write variety instead of $K$-variety.  Varieties of pure dimension $1$ are called curves. For a projective variety $X \subseteq \IP^N$ of pure dimension $d$ defined over $K$, we consider its degree 
    \begin{equation*}
    \deg(X):=\deg(c_1(\caO(1))^d \cap [X])  
    \end{equation*}
    as defined in \cite{Fulton1998}. 
    
    We use some notions from algebraic intersection theory as follows. For line bundles $L_1, L_2, \dots, L_d$ on a proper irreducible algebraic variety $X$ of dimension $d$ over $K$, we write
    \begin{equation*}
        L_1 \cdot L_2 \cdots L_d \in \IZ
    \end{equation*}
    for the algebraic intersection number as defined by Kleiman \cite{Kleiman1966} and Snapper \cite{Snapper1960} (see \cite[Section VI.2]{Kollar1996} for a good introduction). These coincide with the numbers
\begin{equation*}
\deg(c_1(L_1) \cap c_1(L_2) \cap \cdots \cap c_1(L_d) \cap [X]) \in \IZ
\end{equation*}
in the terminology of \cite{Fulton1998}. If $\{ M_1, M_2, \dots, M_r \} = \{ L_1, L_2, \dots, L_d \}$ and each $M_i$ occurs precisely $n_i$ times among $L_1,L_2,\dots,L_d$ with the $M_i$ pairwise distinct, we set
\begin{equation*}
M_1^{n_1} \cdot M_2^{n_2} \cdots M_r^{n_r} := L_1 \cdot L_2 \cdots L_d.
\end{equation*}
A similar notation is used for the arithmetic intersection numbers defined in Subsection \ref{subsection::arithmetic_intersection_numbers}. Furthermore, we write $\deg_L(X)$ for $L^{d}$.

\section{Reductions}\label{section::reduction}

Recall that $K$ is an algebraically closed field of characteristic $0$. The goal of this section is to show that Theorem \ref{theorem::main} for some fixed $K$ follows from the following seemingly weaker theorem. This section also contains some other reductions and consequences of these that will be useful later.

\begin{theorem}\label{theorem::main_reduction}
Let	$A$ be an abelian variety of dimension $\geq 3$ defined over $K$ and let $C_1, C_2 \subseteq A$ be irreducible algebraic curves such that there is no proper algebraic subgroup of $A$ containing both of them. Then the set
\begin{equation*}
        \Xi =\{(x,y) \in (C_1 \times C_2)(K) \mid  [n](x) = y \text{ for some } n \in \mathbb{Z}
        \}
        \end{equation*}
is not Zariski dense in $C_1 \times C_2$.
\end{theorem}

We recall the following terminology that will be used throughout this article: a torsion coset of an abelian variety is a translate of an abelian subvariety by a torsion point. A coset is a translate of an abelian subvariety by an arbitrary point.

\subsection{A special case of Theorem \ref{theorem::main}}
\label{subsection::specialcase}

We first resolve some special cases, in which we can just prove Theorem \ref{theorem::main} directly.

\begin{lemma} \label{lemma::nocoset}
    Assume that there exists a proper abelian subvariety $B \subsetneq A$ and points $x_1,x_2 \in A(K)$ such that $C_i \subseteq x_i + B$ for $i \in \{1,2\}$, but that there is no proper algebraic subgroup of $A$ containing both $C_1$ and $C_2$.
    
    Then, $[n](C_1) \cap C_2 \neq \emptyset$ for at most one $n \in \mathbb{Z}$. Hence, the Zariski closure of the set
       \[\{(x,y) \in (C_1 \times C_2)(K) \mid [n](x) = y \text{ for some } n \in \mathbb{Z}\}\]
       is equal to either a finite set or a graph $(\id_A,[n_0])(C_1)$ for some $n_0 \in \mathbb{Z}$.
\end{lemma}

\begin{proof}
We first show that $[n](C_1) \cap C_2 \neq \emptyset$ for at most one $n \in \mathbb{Z}$. We therefore assume that $[n_i](C_1) \cap C_2 \neq \emptyset$ ($i = 1,2$) with $n_1, n_2 \in \mathbb{Z}$ distinct. 

Writing $\pi: A \to A/B$ for the quotient homomorphism, we note that $[n_i]\pi(x_1)=\pi([n_i](C_1))=\pi(C_2) = \pi(x_2)$. This means $[n_i](x_1) - x_2 \in B(K)$ ($i = 1,2 $). We deduce that $x_1 \in [n_1-n_2]^{-1}(B)(K)$ and, as $[n_1-n_2]\pi(x_2)=[n_1-n_2][n_1]\pi(x_1)=[n_1][n_1-n_2]\pi(x_1) =0_{A/B}$, the same holds for $x_2$.

But this means that $C_1$ and $C_2$ are both contained in the proper algebraic subgroup $[n_1-n_2]^{-1}(B) \subsetneq A$, which contradicts our assumption. So $[n](C_1) \cap C_2$ can be non-empty for at most one $n$, which we call $n_0$ (if it exists).

The claim about the Zariski closure follows since $[n_0](C_1) \cap C_2$ is either finite or equal to $C_2$.
\end{proof}

\begin{proposition}
\label{proposition::MordellLang}
    Let	$A$ be an abelian variety over an algebraically closed field $K$ of characteristic $0$. Let $C \subseteq A$ be an irreducible algebraic curve that is not a coset.
    
    Then, the intersection $C \cap \{ z \in A(K)\mid [n](z)\in \langle x \rangle \text{ for some $n \ne 0$}  \}$ is finite for every point $x \in A(K)$.
    
    In particular, the sets
    \begin{align*}
        \{ x\in C(K) \mid [n](x)=y_0 \text{ for some $n \neq 0$}\}
    \end{align*}
    and
    \begin{align*}
    \{ y\in C(K) \mid y = [n](x_0) \text{ for some $n \in \mathbb{Z}$}\}&
    \end{align*}
    are finite for all given points $x_0, y_0\in A(K)$.
\end{proposition} 

\begin{proof}
    This is a straightforward consequence of the Mordell-Lang conjecture for finite rank subgroups, which is proven in \cite{McQuillan1995}.
\end{proof}

\begin{lemma}\label{lemma::nocoset2}
Assume that either $C_1$ or $C_2$ is contained in a coset of codimension $\geq 2$, but that there is no proper algebraic subgroup of $A$ containing both $C_1$ and $C_2$.

Then, every irreducible component of the Zariski closure of the set
   \[\{(x,y) \in (C_1 \times C_2)(K) \mid [n](x) = y \text{ for some } n \in \mathbb{Z}\}\]
       is either a point, a horizontal curve $C_1 \times \{y_0\}$ for some $y_0 \in C_2(K) \backslash \{0_A\}$, a vertical curve $\{x_0\} \times C_2$ for some $x_0 \in C_1(K)$, or a graph $(\id_A,[n_0])(C_1)$ for some $n_0 \in \mathbb{Z}$.
       
       In particular, this is the case if one of the stabilizers $\mathrm{Stab}_A(C_1)$ or $\mathrm{Stab}_A(C_2)$ is infinite.
\end{lemma}

A statement similar to Lemma \ref{lemma::nocoset2} has been proved for powers of elliptic curves in \cite[Proposition 4.0.7]{Smith}.

\begin{proof}
Let $B\subsetneq A$ be an abelian subvariety  of codimension $\geq 2$ such that some translate of $B$ contains $C_1$ or $C_2$. Let $\pi: A \to A/B$ denote the quotient map. 

Because of Lemma \ref{lemma::nocoset}, we may and will assume that $C_1$ and $C_2$ are not both contained in a translate of the \textit{same} proper algebraic subgroup of $A$. Either $\pi(C_1)$ or $\pi(C_2)$ is a singleton, but not both since otherwise $C_1$ and $C_2$ would both be contained in a translate of $B$.

We first treat the case where $\pi(C_1) = \{x\}$ is a singleton. If $\pi(C_2) \cap \langle x \rangle$ is infinite, then, by Proposition \ref{proposition::MordellLang}, $\pi(C_2)$ is equal to a coset in $A/B$. But then $C_1$ and $C_2$ are both contained in a translate of the same proper algebraic subgroup of $A$, whence a contradiction. So $\pi(C_2) \cap \langle x \rangle$ is finite, which implies that the Zariski closure of the set in the lemma is a finite union of singletons and curves of the form $C_1 \times \{y_0\}$.

We next treat the case where $\pi(C_2) = \{y\}$ is a singleton. We know that $\pi(C_1)$ is not a coset, so, by Proposition \ref{proposition::MordellLang}, the set
\[ \{x \in \pi(C_1)(K) \mid [n](x) = y \text{ for some } n \neq 0\}\]
is finite. It follows that the Zariski closure of the set in the lemma is a finite union of singletons, curves of the form $\{x_0\} \times C_2$, and $C_1 \times \{0_A\} = (\id_A,[0])(C_1)$.
\end{proof}

\subsection{Generic finiteness}

By Lemma \ref{lemma::nocoset}, the proof of Theorem \ref{theorem::main} reduces to the case where $C_1$ and $C_2$ are not both contained in a translate of the same proper algebraic subgroup of $A$. In this subsection, we note some simple consequences of this assumption. 

The main application of the next lemma is for curves $C_1, C_2 \subseteq A$ as in Theorem \ref{theorem::main}. However, a single invocation in the proof of Lemma \ref{lemma::stabilizer_of_difference} below makes it necessary to formulate it for curves in an abelian variety $B$ of dimension $\geq 2$ (and not just $\geq 3$). We define
    \begin{equation} \label{eq::def_phi_n}
	   \varphi_{B,n} : B \times B \rightarrow B, \ (x, y) \longmapsto [n] (x) - y,
    \end{equation}
for every integer $n$. Below, we write $\varphi_{n}$ instead of $\varphi_{A,n}$.

\begin{lemma}\label{lem:genfinite}
Let $C_1$ and $C_2$ be irreducible algebraic curves in an abelian variety $B$ of dimension $\geq 2$. Assume additionally that $C_1$ and $C_2$ are not both contained in translates of the same proper algebraic subgroup of $B$. Then, the morphism $(\varphi_{B,n})|_{C_1 \times C_2}$ is generically finite for all $n \in \mathbb{Z} \backslash \{0\}$ and quasi-finite for all but at most two $n \in \mathbb{Z} \backslash \{0\}$. In contrast, $\dim \varphi_{B,0}(C_1 \times C_2) < \dim (C_1 \times C_2)$. Moreover, for any $n \in \mathbb{Z} \backslash \{0\}$, the image $\varphi_{B,n}(C_1 \times C_2)$ is not an irreducible component of a proper algebraic subgroup of $B$.
\end{lemma}

\begin{proof}
Clearly, 
    \begin{equation*}
    \dim \varphi_{B,0}(C_1 \times C_2) = \dim [-1](C_2) = \dim C_2 < \dim (C_1 \times C_2).    
    \end{equation*}

The fiber of $(\varphi_{B,n})|_{C_1 \times C_2}$ ($n \in \mathbb{Z}\backslash \{0\}$) over a point $z \in B(K)$ is isomorphic to $C_1 \cap [n]^{-1}(C_2 + z)$. This intersection has dimension $1$ if and only if $C_1 \subseteq [n]^{-1}(C_2 + z)$, which implies $[n](C_1) \subseteq (C_2 + z)$. By irreducibility, this is only the case if $[n](C_1) = C_2 + z$. For any two points $z_1, z_2 \in B(K)$ such that both fibers $\varphi_{B,n}^{-1}(z_i) \cap (C_1 \times C_2)$, $i \in \{1, 2\}$, have dimension $1$, their difference $z_1 - z_2$ is contained in 
\begin{equation*}
\Gamma_n = \mathrm{Stab}_B([n](C_1)) \cap \mathrm{Stab}_B(C_2).    
\end{equation*}
In particular, $\Gamma_n$ is infinite if $(\varphi_{B,n})|_{C_1 \times C_2}$ has infinitely many $1$-dimensional fibers. But $\Gamma_n$ can be infinite only if both $[n](C_1)$ and $C_2$ are translates of the same elliptic curve and therefore also $C_1$ and $C_2$ are translates of the same elliptic curve, contrary to our assumption. Consequently, all but finitely many fibers of $(\varphi_{B,n})|_{C_1 \times C_2}$ are finite and thus the morphism $(\varphi_{B,n})|_{C_1 \times C_2}$ is generically finite since its image is certainly positive-dimensional.

If $(\varphi_{B,n})|_{C_1 \times C_2}$ were not quasi-finite for more than two $n \in \mathbb{Z} \backslash \{0\}$, we could find $n_1, n_2 \in \mathbb{Z} \backslash \{0\}$ and $z_1, z_2 \in B(K)$ such that $n_1 \neq \pm n_2$, $[n_1](C_1) = C_2 + z_1$, and $[n_2](C_1) = C_2 + z_2$.
In particular, we have
\[[n_1](C_1) = C_2 + z_1 = (C_2 + z_2) + (z_1 - z_2) = [n_2](C_1) + (z_1-z_2).\]
We deduce that $[n_1^k](C_1)$ is a translate of $[n_2^k](C_1)$ for every integer $k \geq 1$. Picking an arbitrary ample symmetric line bundle $L$ on $B$, we infer that
\[ \deg_L([n_1^k](C_1)) = \deg_L([n_2^k](C_1)) \]
for every integer $k \geq 1$. By the projection formula \cite[Proposition 2.5 (c)]{Fulton1998}, this implies that
\[ \frac{n_1^{2k}}{|\mathrm{Stab}_B(C_1) \cap \ker [n_1^k]|} = \frac{n_2^{2k}}{|\mathrm{Stab}_B(C_1) \cap \ker [n_2^k]|}\]
for all $k \in \mathbb{N}$. Since $n_1 \neq \pm n_2$, this forces $\mathrm{Stab}_B(C_1)$ to be positive-dimensional. Similarly, one can use that
\[ [n_1](C_2) = [n_1^2](C_1)-[n_1](z_1) = [n_1n_2](C_1) - [n_1](z_2) = [n_2](C_2) + [n_2](z_1) - [n_1](z_2)\]
implies that $[n_1^k](C_2)$ and $[n_2^k](C_2)$ are translates of each other for every integer $k \geq 1$. In this way, one obtains that $\mathrm{Stab}_B(C_2)$ must be positive-dimensional as well. By hypothesis, the stabilizers $\mathrm{Stab}_B(C_1)$ and $\mathrm{Stab}_B(C_2)$ are distinct elliptic subcurves of $B$. But this means that $[n](C_1)$ is not contained in any translate of $C_2$ for all $n \neq 0$. By the previous paragraph, this implies that $(\varphi_{B,n})|_{C_1 \times C_2}$ is quasi-finite for all $n \in \mathbb{Z} \backslash \{0\}$, which contradicts the assumption made at the beginning of this paragraph.

It remains to show that the image $\varphi_{B,n}(C_1 \times C_2)$, $n \in \mathbb{Z} \backslash \{0\}$, is not an irreducible component of a proper algebraic subgroup. In order to derive a contradiction, let us assume this were the case. If $\dim B = 2$, then, by what we have already proved, we would have
\[2 = \dim(C_1 \times C_2) = \dim \varphi_{B,n}(C_1 \times C_2) \le 1,\]
a contradiction. We can then assume $\dim B > 2$. Since 
    \begin{equation*}
    \dim \varphi_{B,n}(C_1 \times C_2) = \dim (C_1 \times C_2) < \dim(B),
    \end{equation*}
    the image $\varphi_{B,n}(C_1 \times C_2)$ would be contained in a proper algebraic subgroup $H \subsetneq B$. But this implies that $C_2$ is contained in a translate of $[-1](H) = H$ and that $C_1$ is contained in a translate of $[n]^{-1}(H)$. The connected components of $H$ and $[n]^{-1}(H)$ are all translates of the same abelian subvariety and so $C_1$ and $C_2$ would both be contained in translates of the same proper algebraic subgroup of $B$, a contradiction.
\end{proof}

\begin{lemma}\label{lem:inters_one_n}
    Let $C_1$ and $C_2$ be as in Theorem \ref{theorem::main} and assume additionally that $C_1$ and $C_2$ are not both contained in translates of the same proper algebraic subgroup of $A$. Then, for all but at most two integers $n$, the intersection
    \begin{equation*}
        [n](C_1 + x_1) \cap (C_2 + x_2) 
    \end{equation*}
    is finite for all points $x_1,x_2 \in A(K)$.

\end{lemma}

\begin{proof} 
If $n = 0$, the intersection is trivially finite. If $n \neq 0$, the lemma is a direct consequence of Lemma \ref{lem:genfinite}.
\end{proof}

\subsection{Theorem \ref{theorem::main_reduction} implies Theorem \ref{theorem::main}}

\begin{lemma}\label{lemma::main_reduction}
     Theorem \ref{theorem::main_reduction} for a given algebraically closed field $K$ of characteristic zero implies Theorem \ref{theorem::main} for the same field $K$.
\end{lemma}

\begin{proof}
    We first determine the curves $C \subseteq C_1 \times C_2$ dominating both factors that can appear as components of the Zariski closure of $\Xi$.

We fix one such $C$ and denote by $H$ the smallest algebraic subgroup of $A^2$ containing it. It follows from the Zilber-Pink conjecture for curves in abelian varieties \cite[Theorem 1.1]{BarroeroDill2022} that $C$ contains at most finitely many points $(x,[n](x))$ such that
\[ \dim(H \cap (\id_A,[n])(A)) \leq \dim H - 2.\]
So if $C$ is a component of the Zariski closure of $\Xi$, it must contain a Zariski dense set of points $(x,[n](x))$ with $\dim(H \cap (\id_A,[n])(A)) \geq \dim H - 1$.

If this dimension inequality holds for at most one $n_0 \in \mathbb{Z} \backslash \{0\}$, we deduce that $C = (\id_A,[n_0])(C_1)$.

If the inequality holds for $n_1, n_2 \in \mathbb{Z} \backslash \{0\}$ with $n_1 \neq n_2$, then it follows that
\[ 0 = \dim(H \cap (\id_A,[n_1])(A) \cap (\id_A,[n_2])(A)) \geq \dim H - 2,\]
so $\dim H \leq 2$. We denote by $\pi_i: A^2 \to A$ the projection to the $i$-th factor ($i = 1,2$). Since there is no proper algebraic subgroup of $A$ containing both $C_1$ and $C_2$ and $C$ dominates both factors, we deduce that
\[ \pi_1(H) + \pi_2(H) = A\]
and so
\[ \dim(\pi_1(H) \cap \pi_2(H)) = \dim \pi_1(H) + \dim \pi_2(H) - \dim A \leq 4 - \dim A \leq 1\]
since $\dim A \geq 3$.

Finally, if $(x,[n](x)) \in C(K)$, then $[n](x) \in (\pi_1(H) \cap \pi_2(H))(K)$, so if $C$ contains a Zariski dense set of such points, then $C_2 \subseteq \pi_1(H) \cap \pi_2(H)$ and so $\pi_2(H) \subseteq \pi_1(H)$ by the minimality of $H$. But then $C_1$ and $C_2$ are both contained in $\pi_1(H) \subsetneq A$ ($\dim \pi_1(H) \leq 2 < \dim A$), a contradiction.

Hence, every irreducible component of the Zariski closure of $\Xi$ is of the form described in Theorem \ref{theorem::main}. If a component $C_1 \times \{y_0\}$ of type (iii) appears with $y_0 \neq 0_A$, then, by Proposition \ref{proposition::MordellLang}, $C_1$ is a translate of an algebraic subgroup. But since $\Xi \cap (C_1 \times \{y_0\})$ is Zariski dense in $C_1 \times \{y_0\}$, there are $n_1, n_2 \in \IZ$ such that $n_1 \neq n_2$ and $y_0 \in [n_1](C_1) \cap [n_2](C_1)$. In particular, $[n_1](C_1) \cap [n_2](C_1) \neq \emptyset$ and so $C_1$ has to be a torsion coset and therefore a component of an algebraic subgroup. If a component of type (iv) appears, then $C_2$ is a component of an algebraic subgroup by a similar argument.

Finally, in addition to $C_1 \times \{0_A\}$, which appears in case $C_2$ contains the origin, Lemmas \ref{lemma::nocoset} and \ref{lem:inters_one_n} imply that there can be at most two components of type (ii) with $n\ne 0$.
\end{proof}

	\section{Geometry}
    \label{section::geometry}
	
	This and the next section rely heavily on Faltings's previous work \cite{Faltings1991}, which was inspired by Vojta's proof of the Mordell conjecture \cite{Vojta1991}. A very detailed exposition of the techniques in \cite{Faltings1991} can be found in \cite{Edixhoven1993}. We claim no originality for most of it and have only added this material to fix notations and allow easy references. In addition, the special situation considered here allows us to simplify most of the general arguments in our presentation.

    \subsection{Meromorphic tangent fields}
    \label{subsection::tangent_fields}
    In this subsection, we construct a meromorphic tangent field $\mathbf{D}$ on an arbitrary irreducible algebraic curve $C$. We summarize this construction in the following lemma. 
    
    \begin{lemma}
        \label{lemma::tangent_field}
        Let $C$ be an irreducible algebraic curve and $\pi: C \rightarrow \IP^1$ a non-constant map. Then, every meromorphic tangent field $$\mathbf{D}_0 \in \mathrm{Hom}_{K(\IP^1)}(\Omega^1_{K(\IP^1)/K}, K(\IP^1))$$ on $\IP^1$ lifts along $\pi$ to a meromorphic tangent field $$\mathbf{D} \in \mathrm{Hom}_{K(C)}(\Omega^1_{K(C)/K}, K(C)),$$
        i.e. there exists $\mathbf{D}$ such that
        \[ \pi^\ast \circ \mathbf{D}_0 = \mathbf{D} \circ \pi^\ast.\]
    \end{lemma} 	

   The tangent bundle of $\IP^1$ is globally generated and there are plenty of global tangent fields on $\IP^1$ that can be generically lifted. However, the lifting $\mathbf{D}$ of a global tangent field $\mathbf{D}_0$ on $\IP^1$ is in general only a meromorphic tangent field on $C$.

    \begin{proof}[Proof of Lemma \ref{lemma::tangent_field}]

    The map $\pi: C \rightarrow \IP^1$ corresponds to a finite field extension
    \[\pi^\ast: K(\IP^1) \hookrightarrow K(C).\]
    As $K$ has characteristic $0$, this extension is separable. Therefore,
    \[\Omega^1_{K(C)/K} = \Omega^1_{K(\IP^1)/K} \otimes_{K(\IP^1)} K(C)\]
    by \cite[Theorem 25.3]{Matsumura1986}.

        We write $\pi^\ast: \Omega^1_{K(\IP^1)/K} \rightarrow \Omega^1_{K(C)/K}$ for the induced pullback map on differentials. A meromorphic tangent field $\mathbf{D}$ on $C$ lifts a meromorphic tangent field $\mathbf{D}_0$ on $\IP^1$ if and only if $\pi^\ast \circ \mathbf{D}_0 = \mathbf{D} \circ \pi^\ast$. For any given meromorphic tangent field $\mathbf{D}_0: \Omega^1_{K(\IP^1)/K} \rightarrow K(\IP^1)$, a lifting can be explicitly defined by
        \begin{equation*}
            \mathbf{D}: \Omega^1_{K(C)/K} = \Omega^1_{K(\IP^1)/K} \otimes_{K(\IP^1)} K(C) \longrightarrow K(C), \
            \omega \otimes f \longmapsto f \cdot (\pi^\ast \circ \mathbf{D}_0)(\omega). \qedhere
        \end{equation*}
    \end{proof}

  \subsection{Indices and leading coefficients}
    \label{subsection::indices}

    Let $X$ be an irreducible algebraic variety. A meromorphic tangent field $\mathbf{D}$ on $X$, i.e. an element of $\mathrm{Hom}_{K(X)}(\Omega^1_{K(X)/K}, K(X))$, always spreads out to a tangent field on some non-empty Zariski open subset of $X^{\mathrm{sm}}$ by \cite[Tags \href{https://stacks.math.columbia.edu/tag/01CP}{01CP} and \href{https://stacks.math.columbia.edu/tag/01V3}{01V3}]{stacksProjectAuthors2015}. Since $\Omega^1_{X^{\mathrm{sm}}/K}$ is locally free, there is a maximal such Zariski open subset $U$ of $X^{\mathrm{sm}}$ and $\mathbf{D}$ admits a unique extension to $U$. If $\mathbf{D}$ is non-zero, we define the divisor $\mathrm{Div}(\mathbf{D})$ of $\mathbf{D}$ as the union of the complement of $U$ with the set of $p \in U$ such that the extension of $\mathbf{D}$ maps the entire cotangent space at $p$ to $0$; in this way, $\mathrm{Div}(\mathbf{D})$ is a proper Zariski closed subset of $X$. The singular locus of $X$ is always contained in $\mathrm{Div}(\mathbf{D})$; nevertheless, we will often write $X^{\mathrm{sm}} \backslash \mathrm{Div}(\mathbf{D})$ for the complement of $\mathrm{Div}(\mathbf{D})$ below to avoid any ambiguity.
    
    Suppose that $X$ is endowed with commuting non-zero meromorphic tangent fields $\mathbf{D}_1$ and $\mathbf{D}_2$ whose divisors are denoted by $\mathrm{Div}(\mathbf{D}_1)$ and $\mathrm{Div}(\mathbf{D}_2)$. We write    
    \begin{equation*}
        \mathbf{D}^{i_1, i_2} = \frac{1}{i_1! i_2!} \mathbf{D}_1^{i_1} \circ \mathbf{D}_2^{i_2}
    \end{equation*}
    for the meromorphic differential operator of degree $i_1 + i_2$ on $X$ that is the concatenation of the $i_j$-fold iteration of $\mathbf{D}_j$ ($i_j \geq 0$) for $j = 1, 2$, multiplied by $\frac{1}{i_1! i_2!}$.

    Furthermore, we consider a line bundle $L$ on $X$, a non-zero rational section $\mathbf{s}$ of $L$, and a point $p \in X \backslash \bigcup_{i=1}^{2} \mathrm{Div}(\mathbf{D}_i)$ at which $\mathbf{s}$ is regular. Choose a trivialization $\psi: \caO_{U} \rightarrow L|_U$ on a Zariski-open neighborhood $U \subseteq X \backslash \bigcup_{i=1}^{2} \mathrm{Div}(\mathbf{D}_i)$ of $p$. In this way, the section $\mathbf{s}$ corresponds to a rational function $f = \psi^{-1}(\mathbf{s})$ on $U \subseteq X \backslash \bigcup_{i=1}^{2} \mathrm{Div}(\mathbf{D}_i)$, regular at $p$. For any given pair $\boldsymbol{d} = (d_1,d_2)$ of positive integers, we use the trivialization to define the index
	\begin{equation}
            \label{definition::index}
		\mathbf{I}_{\boldsymbol{d}}(\mathbf{s},p) := \min \left\{ \frac{i_1}{d_1} + \frac{i_2}{d_2} \ \middle\vert \ 
        \mathbf{D}^{i_1,i_2} f(p) \neq 0 \right\}.
	\end{equation}

	In this generality, it is possible that the index is $\infty$ if all derivatives under consideration vanish at $p$. As the notation suggests, the index does not depend on the choice of the trivialization $\psi$. Indeed, this is implied by the following lemma whose proof is almost trivial. However, the independence of leading coefficients from the chosen trivializations expressed in \eqref{equation::leading_coefficient} will prove to be crucial later.
    
    \begin{lemma}
    \label{lemma::leading_coefficient} Let $(d_1,d_2)$ be a pair of positive integers, let $\psi_i: \mathcal{O}_U \rightarrow L|_U$, $i \in \{1,2\}$, be two trivializations of $L$ in a Zariski-open neighborhood $U$ of $p$, and set $f_i = \psi_i^{-1}(\mathbf{s})$. Let $i_1, i_2 \in \mathbb{Z}_{\geq 0}$ and consider the following statements: 
    \begin{enumerate}
        \item[(i)\textsubscript{1}] $\mathbf{D}^{i_1,i_2} f_1 \notin  \mathfrak{m}_p$,
        \item[(ii)\textsubscript{1}] $\mathbf{D}^{j_1,j_2} f_1 \in \mathfrak{m}_p$ for all $(j_1,j_2)$ satisfying
        \[ \frac{j_1}{d_1} + \frac{j_2}{d_2} < \frac{i_1}{d_1} + \frac{i_2}{d_2},\]
        \item[(i)\textsubscript{2}]
         $\mathbf{D}^{i_1,i_2} f_2 \notin  \mathfrak{m}_p$, and
         \item[(ii)\textsubscript{2}] 
         $\mathbf{D}^{j_1,j_2} f_2 \in \mathfrak{m}_p$ for all $(j_1,j_2)$ satisfying
         \[ \frac{j_1}{d_1} + \frac{j_2}{d_2} < \frac{i_1}{d_1} + \frac{i_2}{d_2}.\]
    \end{enumerate}
    Then \emph{(ii)\textsubscript{1}} $\implies$ \emph{(ii)\textsubscript{2}} and \emph{(i)\textsubscript{1}} $\land$ \emph{(ii)\textsubscript{1}} $\implies$ \emph{(i)\textsubscript{2}} $\land$ \emph{(ii)\textsubscript{2}}.
    
    In addition, if \emph{(i)\textsubscript{1}} and \emph{(ii)\textsubscript{1}} hold, then
    \begin{equation}
    \label{equation::leading_coefficient}
        \psi_2(\mathbf{D}^{i_1,i_2}f_2)_p - \psi_1(\mathbf{D}^{i_1,i_2}f_1)_p \in  \mathfrak{m}_pL_{p}.
    \end{equation}
\end{lemma}

\begin{proof}
    There exists $g \in \mathcal{O}^\times(U)$ such that $\psi_2^{-1} = g 
    \cdot \psi_1^{-1}$. Consequently, $f_2 = gf_1$ with $g \notin \mathfrak{m}_{p}$. A repeated use of the product rule shows that (ii)\textsubscript{1} implies (ii)\textsubscript{2}. Using (ii)\textsubscript{2}, we obtain
    \begin{equation*}
        \mathbf{D}^{i_1,i_2}(f_2) - g \cdot \mathbf{D}^{i_1,i_2}(f_1) \in \mathfrak{m}_p
    \end{equation*}
    and so (i)\textsubscript{1} $\land$ (ii)\textsubscript{1} $\implies$ (i)\textsubscript{2} $\land$ (ii)\textsubscript{2}. Applying $\psi_{2,p}$ to both sides, we infer \eqref{equation::leading_coefficient}
    since $\psi_2 = g^{-1} \cdot \psi_1$.
\end{proof}

Thanks to Lemma \ref{lemma::leading_coefficient}, we can define, under the hypotheses there and using some trivialization $\psi$, the value 
\[(\mathbf{D}^{i_1,i_2}\mathbf{s})(p) =  \psi(\mathbf{D}^{i_1,i_2}\psi^{-1}(\mathbf{s}))_p \text{ mod } \mathfrak{m}_pL_p\]
inside $L_p/\mathfrak{m}_pL_p$, which is canonically isomorphic to $K$, and this value does not depend on the choice of $\psi$.

\begin{lemma}
    \label{lemma::leading_coefficient_product}
    Let $L$ and $M$ be two line bundles on $X$, let $\mathbf{d} = (d_1, d_2)\in \mathbb{Z}^2_{> 0}$, let $k_1, k_2 \in \IZ_{\geq 0}$, and let $\mathbf{s}$ and $\mathbf{t}$ be rational sections of $L$ and $M$ respectively. Let $p \in X^{\mathrm{sm}} \setminus (\mathrm{Div}(\mathbf{D}_1) \cup \mathrm{Div}(\mathbf{D}_2))$. Suppose that $\mathbf{s}$ and $\mathbf{t}$ are regular at $p$ and that $\mathbf{t}$ does not vanish at $p$.
    Then
\[ \mathbf{I}_{\mathbf{d}}(\mathbf{s} \otimes \mathbf{t},p) = \mathbf{I}_{\mathbf{d}}(\mathbf{s},p).\]
    If, moreover,
    \[ \mathbf{I}_{\mathbf{d}}(\mathbf{s},p) = \frac{k_1}{d_1} + \frac{k_2}{d_2},\]
    then
\begin{equation*}
    (\mathbf{D}^{k_1,k_2} (\mathbf{s} \otimes \mathbf{t}))(p) = (\mathbf{D}^{k_1,k_2} \mathbf{s})(p) \otimes \mathbf{t}(p) \in (L \otimes M)_p/\mathfrak{m}_p(L \otimes M)_p.
\end{equation*}
\end{lemma}

\begin{proof}
This follows from the proof of Lemma \ref{lemma::leading_coefficient} since one can choose a trivialization of $L \otimes M$ around $p$ under which $\mathbf{t}$ corresponds to the constant function $1$.
\end{proof}

\subsection{Norms and line bundles}
\label{subsection::norms}

We summarize here basic facts related to norm maps along a finite, surjective map $\pi: X \rightarrow Y$ between irreducible algebraic varieties. \textit{We assume additionally that $Y$ is normal.} Our exposition is based on \cite[\href{https://stacks.math.columbia.edu/tag/0BCX}{Section 0BCX}]{stacksProjectAuthors2015} and we claim no originality. However, it is important for us to define a norm map as in \eqref{equation::global_sections_norm} below and this map is not explicitly defined in \cite{stacksProjectAuthors2015}.

Write $K(X)$ and $K(Y)$ for the function fields of $X$ and $Y$, respectively, and let $\delta = [K(X):K(Y)]$ denote the degree of $\pi$. There exists a unique multiplicative morphism of sheaves
\begin{equation}
\label{equation::norm_map}
N_{X/Y}: \pi_\ast \caO_X \longrightarrow \caO_Y   
\end{equation}
inducing the field norm $N_{K(X)/K(Y)}: K(X) \rightarrow K(Y)$ over the generic point $\eta_Y$ of $Y$; we define the map \eqref{equation::norm_map} by describing it for sections over affine open subsets of $Y$ in a compatible way. For every such affine open $U=\mathrm{Spec}(A^\prime) \subseteq Y$, there exists an $A^\prime$-finite ring $A$ such that $\pi^{-1}(U)=\mathrm{Spec}(A)$. We describe multiplicative maps
\begin{equation}
    N_{X/Y}(U): A = \mathcal{O}_X(\pi^{-1}(U)) \rightarrow \mathcal{O}_Y(U) = A^\prime
\end{equation}
that induce the field norm on the quotient field $K(X)$ and $K(Y)$. As both $A$ and $A^\prime$ are integral domains, they inject into their quotient fields $K(X)$ and $K(Y)$. Therefore, there is at most one possible choice for $N_{X/Y}(U)$, namely the restriction of the field norm $N_{K(X)/K(Y)}$ to $A$. As uniqueness automatically implies compatibility, it remains to check that $N_{K(X)/K(Y)}(A) \subseteq A^\prime$. For this purpose, take an arbitrary element $\alpha \in A$. There exists a monic polynomial
\begin{equation*}
    P(T) = T^m + b_{m-1} T^{m-1} + \cdots + b_0 \in A^\prime[T]
\end{equation*}
such that $P(\alpha) = 0$. As an element of $K(Y)[T]$, the polynomial $P(T)$ is divisible by the minimal polynomial
\begin{equation*}
    Q(T) = T^n + l_{n-1} T^{n-1} + \cdots + l_0 \in K(Y)[T]
\end{equation*}
of $\alpha$ over $K(Y)$. By \cite[\href{https://stacks.math.columbia.edu/tag/00H6}{Lemma 00H6 (1)}]{stacksProjectAuthors2015}, the coefficients $l_{0},\dots,l_{n-1}$ are integral over the normal domain $A^\prime$. We conclude that $l_{0},\dots,l_{n-1} \in A^\prime$, which implies that
\begin{equation*}
    N_{K(X)/K(Y)}(\alpha) 
    = N_{K(Y)(\alpha)/K(Y)}(\alpha)^{\delta/n} = (-1)^{\delta}l_0^{\delta/n} \in A^\prime,
\end{equation*}
which completes the definition of $N_{X/Y}:\pi_\ast \caO_X \rightarrow \caO_Y$. From its construction, we also see that the composite 
\begin{equation*}
    \caO_Y \xrightarrow{\pi^\#} \pi_\ast \caO_X \xrightarrow{N_{X/Y}} \caO_Y
\end{equation*}
is just the $\delta$-th power map $\caO_Y \rightarrow \caO_Y$. Furthermore, multiplicativity implies that $N_{X/Y} (\pi_\ast \caO_X^\times) \subseteq \caO_Y^\times$.

The norm map $N_{X/Y}$ induces a homomorphism $\Pic(X) \rightarrow \Pic(Y)$. Indeed, let $L$ be a line bundle determined by a covering $\{U_i \}$ and a subordinate cocycle $\alpha_{ij} \in \mathcal{O}_X^\times(U_i \cap U_j)$. Using \cite[\href{https://stacks.math.columbia.edu/tag/0BUT}{Lemma 0BUT}]{stacksProjectAuthors2015}, we can ensure that there exists an open covering $\{V_i\}$ of $Y$ such that $\pi^{-1}(V_i) = U_i$. Then, the norms of $\pi_\ast \alpha_{ij} \in \pi_\ast\caO_X^\times(V_i \cap V_j)$ yield elements $\beta_{ij} = N_{X/Y}(\pi_\ast \alpha_{ij}) \in \mathcal{O}_Y^\times(V_i \cap V_j)$. As the $(\beta_{ij})$ satisfy the same cocycle relations as the $\alpha_{ij}$ by multiplicativity, we obtain a line bundle $N_{X/Y}L$ on $Y$. Finally, it is easy to see that we obtain the same line bundle for an arbitrary refinement of $\{V_i\}$. Hence, the norm $N_{X/Y}L$ of any line bundle $L$ on $X$ is well-defined. In addition, we obtain a homomorphism
\begin{equation}
\label{equation::global_sections_norm}
    H^0(X,L) \longrightarrow H^0(Y,N_{X/Y}L), \ \mathbf{s} \longmapsto N_{X/Y}(\mathbf{s})
\end{equation}
as follows. Any section $\mathbf{s} \in H^0(X,L)$ corresponds to a collection $f_i \in \caO_X(U_i)$, $i \in I$, such that $f_i|_{U_i \cap U_j} = \alpha_{ij} \cdot f_j|_{U_i \cap U_j}$. Taking norms, we obtain a collection $N_{X/Y}(f_i) \in \mathcal{O}_Y(V_i)$, $i \in I$, satisfying the relations $N_{X/Y}(f_i)|_{V_i \cap V_j} = N_{X/Y}(\alpha_{ij}) \cdot N_{X/Y}(f_j)|_{V_i \cap V_j}$, and hence defining a global section $N_{X/Y}(\mathbf{s})$ of the line bundle $N_{X/Y} L$.

We conclude this subsection with a lemma on the relation of leading coefficients to the norm map.

\begin{lemma}
\label{lemma::norm_derivatives}
Let $\pi: X \rightarrow Y$ be a finite, surjective map between irreducible algebraic varieties. Assume additionally that $Y$ is normal. Let further $\mathbf{D}_1,\mathbf{D}_2$ be commuting non-zero meromorphic tangent fields on $X$ lifting commuting non-zero meromorphic tangent fields $\mathbf{D}_1^\prime,\mathbf{D}_2^\prime$ on $Y$. Let $d_1, d_2 \in \mathbb{Z}_{>0}$ and let $\sigma \in (d_1d_2)^{-1} \mathbb{Z}_{\geq 0}$. Let $L$ be a line bundle on $X$ and let $\mathbf{s} \in H^0(X,L)$ be a regular global section of $L$. Furthermore, consider a point
\begin{equation*}
p \in X^{\mathrm{sm}} \setminus \left(\mathrm{Div}(\mathbf{D}_1) \cup \mathrm{Div}(\mathbf{D}_2) \right)    
\end{equation*}
with image
\begin{equation*}
    \pi(p) \in Y^{\mathrm{sm}} \setminus \left( \mathrm{Div}(\mathbf{D}_1^\prime) \cup \mathrm{Div}(\mathbf{D}_2^\prime)\right).
\end{equation*}
Then, 
\begin{equation*}
\left[\mathbf{D}_1^{k_1}\mathbf{D}_2^{k_2} \mathbf{s}\right]\! (p) = 0 \text{\ \ for all \ \ } \frac{k_1}{d_1}+ \frac{k_2}{d_2} \leq \sigma
\end{equation*}
implies that
\begin{equation*}
\left[(\mathbf{D}_1^\prime)^{k_1} (\mathbf{D}_2^\prime)^{k_2}N_{X/Y}(\mathbf{s})\right]\! (\pi(p)) = 0 \text{\ \ for all \ \ } \frac{k_1}{d_1} + \frac{k_2}{d_2} \leq \sigma.
\end{equation*}
\end{lemma}

The hypothesis 
\begin{equation*}
\left[\mathbf{D}_1^{k_1}\mathbf{D}_2^{k_2} \mathbf{s}\right]\! (p) = 0 \text{\ \ for all \ \ } \frac{k_1}{d_1}+ \frac{k_2}{d_2} \leq \sigma
\end{equation*}
is well-defined because of Lemma \ref{lemma::leading_coefficient} and the same is true for the conclusion
\begin{equation*}
\left[(\mathbf{D}_1^\prime)^{k_1} (\mathbf{D}_2^\prime)^{k_2}N_{X/Y}(\mathbf{s})\right]\! (\pi(p)) = 0 \text{\ \ for all \ \ } \frac{k_1}{d_1} + \frac{k_2}{d_2} \leq \sigma.
\end{equation*}

\begin{proof}
   Using \cite[\href{https://stacks.math.columbia.edu/tag/0BUT}{Lemma 0BUT}]{stacksProjectAuthors2015}, we find an affine open neighborhood $V=\mathrm{Spec}(A^\prime)$ of $\pi(p)$ such that $L|_{\pi^{-1}(V)}$ is trivial. Under any fixed trivialization, the section $\mathbf{s}$ corresponds to a regular function $f \in \caO(\pi^{-1}(V))$ and  
    \begin{equation*}
    \left[\mathbf{D}_1^{k_1}  \mathbf{D}_2^{k_2} \mathbf{s}\right]\! (p) = 0 \ \ \text{(resp. } \left[(\mathbf{D}_1^\prime)^{k_1} (\mathbf{D}_2^\prime)^{k_2} N_{X/Y}(\mathbf{s})\right]\! (\pi(p)) = 0) 
    \end{equation*}
    just means
    \begin{equation*}
    \left[\mathbf{D}_1^{k_1} \mathbf{D}_2^{k_2} f \right]\! (p) = 0 \ \
    \text{(resp. } \left[(\mathbf{D}_1^\prime)^{k_1} (\mathbf{D}_2^\prime)^{k_2} N_{K(X)/K(Y)} (f) \right]\! (\pi(p)) = 0) 
    \end{equation*}    
    for any integers $k_1,k_2$. If
    \begin{equation*}
        P(T) = T^n + l_{n-1} T^{n-1} + \cdots + l_0 \in K(Y)[T]
    \end{equation*}
    is the minimal polynomial of $f$ over $K(Y)$, then $N_{K(X)/K(Y)}(f)=(-1)^{[K(X):K(Y)]}l_0^{[K(X):K(Y)(f)]}$. Using \cite[\href{https://stacks.math.columbia.edu/tag/00H6}{Lemma 00H6 (1)}]{stacksProjectAuthors2015} as above, we find that $P(T) \in A^\prime[T]$.

    Assume now that
    \begin{equation*}
        \mathbf{D}_1^{k_1}  \mathbf{D}_2^{k_2} f(p) = 0
    \end{equation*}
    for all $\frac{k_1}{d_1} + \frac{k_2}{d_2} \leq \sigma$. Applying some differential operator $\mathbf{D}_1^{k^\prime_1} \mathbf{D}_2^{k^\prime_2}$, $\frac{k_1^\prime}{d_1} + \frac{k_2^\prime}{d_2} \leq \sigma$, to both sides of the equation
    \begin{equation*}
        f^n + l_{n-1} f^{n-1} + \cdots + l_1 f + l_0 = 0,
    \end{equation*}
    we obtain that 
    \begin{equation*}
    \left[(\mathbf{D}_1^\prime)^{k_1^\prime} (\mathbf{D}_2^\prime)^{k_2^\prime} (l_0) \right] (\pi(p))
    = \left[\mathbf{D}_1^{k_1^\prime} \mathbf{D}_2^{k_2^\prime} (l_0) \right](p) = 0
    \end{equation*}
    by a repeated application of Leibniz rule. Since this holds for all $(k_1^\prime,k_2^\prime)$ with $\frac{k_1^\prime}{d_1} + \frac{k_2^\prime}{d_2} \leq \sigma$, a further application of the Leibniz rule delivers $\left[(\mathbf{D}_1^\prime)^{k^\prime_1} (\mathbf{D}_2^\prime)^{k^\prime_2} N_{K(X)/K(Y)}(f)\right](\pi(p)) = 0$. This completes the proof.
\end{proof}
  
    \subsection{The product theorem}
    \label{subsection::product_theorem}

    We first recall an explicit version of Faltings's product theorem \cite{Faltings1991} established by Evertse \cite{Evertse1995}. We can and do restrict to the special case $\IP^1 \times \IP^1$. We let $\boldsymbol{d} = (d_1, d_2)$ denote a pair of positive integers.
    
    We consider $\IP^1$ as a curve endowed with the line bundle $L = \mathcal{O}(1)$ and the global tangent field $\mathbf{D}_0 = X_0 \cdot (\del/\del X_0)$. In Subsection \ref{subsection::indices}, we have defined the index $\mathbf{I}_{\boldsymbol{d}}(\mathbf{s},p)$ of a non-zero section $\mathbf{s} \in H^0(\IP^1 \times \IP^1, \mathcal{O}(d_1,d_2))$ with respect to the tangent fields $(\mathbf{D}_0,0)$ and $(0,\mathbf{D}_0)$ at any point $p \in (\IP^1 \backslash \mathrm{Div}(\mathbf{D}_0))^2$. For every $\varepsilon > 0$,
    \begin{equation*}
        \mathbf{Z}_{\boldsymbol{d}}(\mathbf{s}, \varepsilon) = \{ p \in (\IP^1 \backslash \mathrm{Div}(\mathbf{D}_0))^2 \mid \mathbf{I}_{\boldsymbol{d}}(\mathbf{s},p) \geq \varepsilon \}
    \end{equation*}
    is an algebraic subvariety of $(\IP^1 \backslash \mathrm{Div}(\mathbf{D}_0))^2$. Additionally, the subvariety $\mathbf{Z}_{\boldsymbol{d}}(\mathbf{s}, \varepsilon)$ is defined over a subfield $K_0 \subseteq K$ if $\mathbf{s}$ is so.
\begin{proposition}\label{prop:productthm}
Assume that
\begin{equation*}
    \frac{d_1}{d_2} \geq \left( \frac{12}{\varepsilon} \right)^2
\end{equation*}
for some $\varepsilon \in (0,1]$. In addition, let $\mathbf{s} \in H^0(\IP^1 \times \IP^1,\caO(d_1,d_2))$ denote a non-zero section that is defined over a subfield $K_0 \subseteq K$. 

Then, every irreducible component $Z$ of $\mathbf{Z}_{\boldsymbol{d}}(\mathbf{s},\varepsilon)$ is contained in a proper irreducible product subvariety
\begin{equation*}
    Z_1 \times Z_2 \subsetneq \IP^{1} \times  \IP^{1}
\end{equation*}
where both $Z_1$ and $Z_2$ are defined over a finite extension $K_1 \supseteq K_0$ and
\begin{equation*}
    [K_1: K_0] \deg_{\caO(1)}(Z_1) \deg_{\caO(1)}(Z_2) \leq \left( \frac{12}{\varepsilon}\right)^2.
\end{equation*}
\end{proposition} 	

\begin{proof}
	This follows directly from \cite[Corollary on p.\ 219sq.]{Evertse1995}.
\end{proof}

We actually need to introduce a slightly more general version of the product theorem. For this, we consider two projective irreducible algebraic curves $C_1$ and $C_2$ endowed with non-constant maps $\pi_1: C_1 \rightarrow \IP^1$ and $\pi_2: C_2 \rightarrow \IP^1$ as in Subsection \ref{subsection::tangent_fields}. Using Lemma \ref{lemma::tangent_field}, we construct generic liftings $\mathbf{D}_1$ and $\mathbf{D}_2$ of $\mathbf{D}_0 = X_0 \cdot (\del/\del X_0)$ along $\pi_1$ and $\pi_2$ respectively. As in Subsection \ref{subsection::indices}, we let $L$ denote a line bundle on $C_1 \times C_2$. For any non-zero section $\mathbf{s} \in H^0(C_1 \times C_2, L)$, we can then define an index $\mathbf{I}_{\boldsymbol{d}}(\mathbf{s},p)$ with respect to $(\mathbf{D}_1,0)$ and $(0,\mathbf{D}_2)$ for every point $p \in (C_1^{\mathrm{sm}} \backslash \mathrm{Div}(\mathbf{D}_1)) \times (C_2^{\mathrm{sm}} \backslash \mathrm{Div}(\mathbf{D}_2))$. As above, we also define an algebraic subvariety $\mathbf{Z}_{\boldsymbol{d}}(\mathbf{s},\varepsilon) \subseteq (C_1^{\mathrm{sm}} \backslash \mathrm{Div}(\mathbf{D}_1)) \times (C_2^{\mathrm{sm}} \backslash \mathrm{Div}(\mathbf{D}_2))$.

\begin{corollary}
    \label{corollary::product_theorem}
    Assume that
    \[ \frac{d_1}{d_2} \geq \left(\frac{12\deg(\pi_1)\deg(\pi_2)}{\varepsilon}\right)^2\]
    and that there exists a globally generated line bundle $L^\prime$ on $C_1 \times C_2$ such that 
    \[ L \otimes L^\prime \approx (\pi_1 \times \pi_2)^\ast\caO(d_1,d_2).\]
    Let $\mathbf{s} \in H^0(C_1 \times C_2, L)$ denote a non-zero section and suppose that $C_1$, $C_2$, $\pi_1$, $\pi_2$, $\mathbf{D}_1$, $\mathbf{D}_2$, $L$, $L^\prime$, and $\mathbf{s}$ are all defined over a subfield $K_0 \subseteq K$.

    Then, every irreducible component $Z$ of
    \[\mathbf{Z}_{\boldsymbol{d}}(\mathbf{s},\varepsilon)\subseteq (C_1^{\mathrm{sm}} \backslash \mathrm{Div}(\mathbf{D}_1)) \times (C_2^{\mathrm{sm}} \backslash \mathrm{Div}(\mathbf{D}_2))\]
    is contained in a proper irreducible product subvariety
\begin{equation*}
    Z_1 \times Z_2 \subsetneq C_1 \times C_2
\end{equation*}
where both $Z_1$ and $Z_2$ are defined over a finite extension $K_1 \supseteq K_0$ and
\begin{equation*}
    [K_1: K_0] \deg_{\pi_1^\ast \caO(1)}(Z_1) \deg_{\pi_2^\ast \caO(1)}(Z_2)
    \leq \left( \frac{12}{\varepsilon}\right)^2 \deg(\pi_1)^3\deg(\pi_2)^3.
\end{equation*}
\end{corollary}

\begin{proof}

Consider a point $p \in (C_1^{\mathrm{sm}} \backslash \mathrm{Div}(\mathbf{D}_1)) \times (C_2^{\mathrm{sm}} \backslash \mathrm{Div}(\mathbf{D}_2))$. Since $L^\prime$ is defined over $K_0$ and globally generated, we may choose a global section $\mathbf{t} \in H^0(C_1 \times C_2, L^\prime)$ that is defined over $K_0$ and does not vanish at $p$. It follows from Lemma \ref{lemma::leading_coefficient_product} that $\mathbf{I}_{\boldsymbol{d}}(\mathbf{s},p) = \mathbf{I}_{\boldsymbol{d}}(\mathbf{s} \otimes \mathbf{t},p)$. As in Subsection \ref{subsection::norms}, we can define
\begin{equation*}
\mathbf{s}^\prime = N_{(C_1\times C_2)/(\IP^1 \times \IP^1)}(\mathbf{s} \otimes \mathbf{t}),    
\end{equation*}
which is a global section of
\begin{equation*}
N_{(C_1 \times C_2) / (\IP^1 \times \IP^1)}((\pi_1 \times \pi_2)^\ast \caO(d_1,d_2)) \approx\caO(d_1,d_2)^{\otimes \deg(\pi_1)\deg(\pi_2)}
\end{equation*}
(compare \cite[\href{https://stacks.math.columbia.edu/tag/0BCY}{Tag 0BCY}]{stacksProjectAuthors2015} for this isomorphism).
Furthermore, Lemma \ref{lemma::norm_derivatives} implies that $\mathbf{I}_{\boldsymbol{d}}(\mathbf{s},p) = \mathbf{I}_{\boldsymbol{d}}(\mathbf{s} \otimes \mathbf{t},p)\leq \deg(\pi_1)\deg(\pi_2)\mathbf{I}_{\boldsymbol{d}^\prime}(\mathbf{s}^\prime,(\pi_1 \times \pi_2)(p))$, where
\[\mathbf{d}^\prime = (\deg(\pi_1)\deg(\pi_2) d_1,\deg(\pi_1)\deg(\pi_2) d_2).\] We apply Proposition \ref{prop:productthm} to $\mathbf{s}^\prime$ and $\mathbf{d}^\prime$ with $\varepsilon^\prime = \varepsilon \deg(\pi_1)^{-1}\deg(\pi_2)^{-1}$. It shows that every irreducible component of $Z_{\boldsymbol{d}^\prime}(\mathbf{s}^\prime,\varepsilon^\prime)
\supseteq (\pi_1 \times \pi_2)(Z_{\boldsymbol{d}}(\mathbf{s},\varepsilon))$ is contained in a proper irreducible product subvariety $Z_1' \times Z_2' \subsetneq \IP^1 \times \IP^1$, defined over a finite extension $K_1 \supseteq K_0$, with
\begin{equation*}
    [K_1: K_0] \deg_{\caO(1)}(Z_1') \deg_{\caO(1)}(Z_2') \leq \left( \frac{12}{\varepsilon^\prime}\right)^2.
\end{equation*}
We conclude that every irreducible component $Z$ of
 \[\mathbf{Z}_{\boldsymbol{d}}(\mathbf{s},\varepsilon)\subseteq (C_1^{\mathrm{sm}} \backslash \mathrm{Div}(\mathbf{D}_1)) \times (C_2^{\mathrm{sm}} \backslash \mathrm{Div}(\mathbf{D}_2))\]
 is contained in a proper product subvariety $\pi_1^{-1}(Z_1') \times \pi_2^{-1}(Z_2') \subsetneq C_1 \times C_2$.

Let $i \in \{1,2\}$. It follows from the projection formula \cite[Proposition 2.5 (c)]{Fulton1998} that
\[ \deg_{\pi_i^\ast \caO(1)}(\pi_i^{-1}(Z_i')) \leq \deg(\pi_i)\deg_{\caO(1)}(Z_i').\]
If $Z_i$ is any irreducible component of $\pi_i^{-1}(Z_i')$ and $K_{1,i} \supseteq K_1$ is a minimal field of definition of $Z_i$, then 
\[[K_{1,i}:K_1]\deg_{\pi_i^\ast \caO(1)}(Z_i) \leq \deg_{\pi_i^\ast \caO(1)}(\pi_i^{-1}(Z_i'))\]
since every Galois conjugate of $Z_i$ over $K_1$ occurs as an irreducible component of $\pi_i^{-1}(Z_i')$ and there are $[K_{1,i}:K_1]$ of them. We obtain that any product $Z_1 \times Z_2$ of two such irreducible components is defined over a finite field extension $K_2 \supseteq K_1$ such that $[K_2:K_0] \deg_{\pi_1^\ast \caO(1)}(Z_1) \deg_{\pi_2^\ast \caO(1)}(Z_2)$ is bounded from above by
\[[K_1:K_0]\deg(\pi_1)\deg(\pi_2)\deg_{\caO(1)}(Z_1')\deg_{\caO(1)}(Z_2') \leq \deg(\pi_1)\deg(\pi_2)\left( \frac{12}{\varepsilon^\prime}\right)^2\]
and the corollary follows.
\end{proof}

    \subsection{A degree bound}
    \label{subsection::degreebound}

    We conclude this section with asymptotic estimates of the degree. In Section \ref{section::equidistribution}, the bound is complemented with a similar result on arithmetic intersection numbers (Proposition \ref{proposition::heightbound}). 
    
    As in Theorem \ref{theorem::main}, $A$ denotes an abelian variety of dimension $\geq 3$. We let $L$ be a very ample symmetric line bundle on $A$. We let $C_1$ and $C_2$ denote two irreducible algebraic subcurves in $A$ as in Theorem \ref{theorem::main}. The results in this subsection are to be applied directly to the curves from our main theorem.

    We start with a well-known consequence of the theorem of the cube. Recall the definition of $\varphi_n$ in \eqref{eq::def_phi_n}. On $A \times A$, we define the line bundles
    \begin{equation*}
	   L_n := \varphi_n^\ast L
    \end{equation*}
    for every integer $n$ as well as the line bundle
    \begin{equation*}
        P := \varphi_1^\ast L \otimes \pr_1^\ast L^{\otimes -1} \otimes \pr_2^\ast L^{\otimes -1}.
    \end{equation*}
    
    \begin{lemma}
    \label{lemma::formula_Ln}
    For every integer $n$, the line bundles $L_n$ and $P^{\otimes n} \otimes \mathrm{pr}_1^{\ast} L^{\otimes n^2} \otimes \mathrm{pr}_2^{\ast} L$ are isomorphic. Consequently, the line bundles $L_n \otimes L_{-n}$ and $\pr_1^\ast L^{\otimes 2n^2} \otimes \pr_2^\ast L^{\otimes 2}$ are isomorphic.
	\end{lemma}
	
	\begin{proof}
		We prove this by both upward and downward induction on $n$. In fact, our base case is $n \in \{0,1\}$ and we show that the claim for $n$ and $n+1$ implies both the claim for $n+2$ as well as the claim for $n-1$. 
  
        For $n=0$, we note that
		\begin{equation*}
			\varphi_0^\ast L = \pr_2^\ast [-1]^\ast L \approx \pr_2^\ast L
		\end{equation*}
        as $L$ is symmetric. In the case $n=1$, we note that
		\begin{align*}
			\varphi_1^\ast L 
			&= (\varphi_1^\ast L \otimes \pr_1^\ast L^{\otimes -1} \otimes \pr_2^\ast L^{\otimes -1})
			\otimes
			\pr_1^\ast L \otimes \pr_2^\ast L \\
			&= P \otimes \pr_1^{\ast} L \otimes \pr_2^\ast L.
		\end{align*}
		For the induction step(s), we recall that
		\begin{equation*}
			\pr_{123}^\ast L \otimes \pr_{12}^\ast L^{\otimes -1} \otimes \pr_{13}^\ast L^{\otimes -1} \otimes \pr_{23}^\ast L^{\otimes -1} \otimes \pr_1^\ast L \otimes \pr_2^\ast L \otimes \pr_3^\ast L \approx \caO
		\end{equation*}
		on $A^3$ by the theorem of the cube \cite[Corollary 2 on p.\ 58sq.]{Mumford1970}, where $\pr_{123} = \pr_1 + \pr_2 + \pr_3$ etc. Pullback by
		\begin{equation*}
		(\varphi_n,\pr_1,-\pr_1): A \times A \longrightarrow A^3	
		\end{equation*}
		yields
		\begin{equation*}
		L_{n+1}
		\approx
		L_n^{\otimes 2} \otimes L^{\otimes(-1)}_{n-1}	\otimes \pr_1^\ast L^{\otimes 2}
		\end{equation*}
		for every integer $n$. For the induction step $n-1,n \rightarrow n+1$, we deduce that
		\begin{align*}
			L_{n+1}
			&\approx 
			L_n^{\otimes 2} \otimes L^{\otimes(-1)}_{n-1}	\otimes \pr_1^\ast L^{\otimes 2} \\
			&\approx
			(P^{\otimes n} \otimes \pr^\ast_1 L^{\otimes n^2} \otimes \pr^\ast_2 L)^{\otimes 2}
            \otimes
            (P^{\otimes (n-1)} \otimes \pr^\ast_1 L^{\otimes (n-1)^2} \otimes \pr^\ast_2 L)^{\otimes(-1)}
            \otimes
            \pr_1^\ast L^{\otimes 2}
			\\
			&\approx P^{\otimes(n+1)} \otimes \pr^\ast_1 L^{\otimes (n+1)^2} \otimes \pr^\ast_2 L.
		\end{align*}
  For the induction step $n+1,n \rightarrow n-1$, we show similarly that
	\begin{align*}
			L_{n-1}
			&\approx 
			L_n^{\otimes 2} \otimes L^{\otimes(-1)}_{n+1}	\otimes \pr_1^\ast L^{\otimes 2} \\
			&\approx
			(P^{\otimes n} \otimes \pr^\ast_1 L^{\otimes n^2} \otimes \pr^\ast_2 L)^{\otimes 2}
            \otimes
            (P^{\otimes (n+1)} \otimes \pr^\ast_1 L^{\otimes (n+1)^2} \otimes \pr^\ast_2 L)^{\otimes(-1)}
            \otimes
            \pr_1^\ast L^{\otimes 2}
			\\
			&\approx P^{\otimes(n-1)} \otimes \pr^\ast_1 L^{\otimes (n-1)^2} \otimes \pr^\ast_2 L. \qedhere
		\end{align*}    
    \end{proof}
    
	\begin{lemma}\label{lemma::degreebound}
    Assume that $C_1$ and $C_2$ are not both contained in translates of the same proper algebraic subgroup of $A$. Then, there exists a constant $\newconstant{\label{c_degbnd}} = \oldconstant{c_degbnd}(C_1, C_2) \geq 1$ such that
	\begin{equation*}
		\oldconstant{c_degbnd}^{-1} \cdot n^2 \leq L_n^2 \cdot [C_1 \times C_2] \leq \oldconstant{c_degbnd} \cdot n^2
    \end{equation*}
    for all integers $n$.
\end{lemma}

	\begin{proof}
		It follows from Lemma \ref{lemma::formula_Ln} that
			\[ L_n^2 \cdot [C_1 \times C_2]  = \sum_{i=0}^{4}{a_in^i}\]
		with coefficients $a_i \in \mathbb{Z}$ that are independent of $n$ ($i = 0, \hdots, 4$). We compute that
		\[ a_4 = \left(\mathrm{pr}_1^{\ast} L\right)^2 \cdot [C_1 \times C_2] = 0.\]
		
		 If $n \neq 0$, then, by Lemma \ref{lem:genfinite}, the restriction of the morphism $\varphi_n$ to $C_1 \times C_2$ is generically finite onto its image. It follows from the projection formula \cite[Proposition 2.5 (c)]{Fulton1998} and the ampleness of $L$ that $L_n^2 \cdot [C_1 \times C_2] > 0$ for $n \neq 0$. This implies that $a_3 = 0$ and $a_2 \geq 0$. 
         
         We prove that $a_2 \neq 0$ by contradiction. Assume that $a_2 = 0$, then the previous argument shows that $a_1 = 0$ and $a_0 > 0$. But by Lemma \ref{lem:genfinite}, the restriction of $\varphi_0$ to $C_1\times C_2$ is not generically finite onto its image. Hence the contradiction $a_0 = L_0^2 \cdot [C_1 \times C_2] = 0$. So $a_2 > 0$ and the lemma follows.
    \end{proof}

	\section{Arithmetic}
    \label{section::arithmetic}
    
    In this section, we introduce the basic notions from Arakelov geometry, including canonical metrics and heights.  
    We closely follow the presentation in \cite[Section 2]{Kuehne2022}, and we refer there for a detailed exposition of arithmetic intersection theory \cite{Gillet1990}. Essential references for this section are \cite{Chambert-Loir2000,Chambert-Loir2006,Chambert-Loir2011, Gubler1997, Gubler1998, Gubler2003, Gubler2007a, Maillot2000, Moriwaki2014,Yuan2012,Zhang1995a}, but we mostly refer to \cite{Kuehne2022} where the reader can find further details.
    
	We let $\mathbb{K} \subseteq \IQbar$ be a number field with ring of integers $\mathcal{O}_{\mathbb{K}}$. In addition, we write $\Sigma_f(\mathbb{K})$ (resp.\ $\Sigma_\infty(\mathbb{K})$) for the finite (resp.\ archimedean) places of $\mathbb{K}$ and set $\Sigma(\mathbb{K}) = \Sigma_f(\mathbb{K}) \cup \Sigma_\infty(\mathbb{K})$.

    We consider an irreducible, projective, $d$-dimensional ${\mathbb{K}}$-variety $X$. Furthermore, we work with flat, integral, projective schemes $\mathcal{X}$ of relative dimension $d$ over the ring of integers $\mathcal{O}_{\mathbb{K}}$ such that their generic fiber equals $X$. We call such $\mathcal{X}$ models of $X$. In addition, we consider a line bundle $L$ on $X$. For us, a model of $L$ is a line bundle $\mathcal{L}$ on a model $\mathcal{X}$ of $X$ together with a fixed isomorphism $L \approx \mathcal{L}_{\mathbb{K}}$.
  
    \subsection{Metrics}
    \label{subsection::metrics}
    
    For a place $\nu$ of ${\mathbb{K}}$, we denote by ${\mathbb{K}}_\nu$ the associated completion of ${\mathbb{K}}$, by $\overline{{\mathbb{K}}}_\nu$ the algebraic closure of ${\mathbb{K}}_\nu$, and by $\IC_\nu$ the metric completion of the algebraic closure $\overline{{\mathbb{K}}}_\nu$. If $\nu$ is non-archimedean, we endow $\IC_\nu$ with the standard absolute value $|\cdot|_\nu$ normalized such that $|p|_\nu=p^{-[\IK_\nu:\IQ_p]}$ if ${\mathbb{K}}_\nu$ has residue characteristic $p$. For an archimedean place $\nu$ of ${\mathbb{K}}$, we use the standard absolute value $|\cdot|_\nu$ on $\IC_\nu$. In some formulas, this choice leads to an additional factor
    \begin{equation*}
        \delta_\nu =
        \begin{cases}
            2 & \text{if $\nu$ is complex archimedean,} \\
            1 & \text{otherwise}.
        \end{cases}
    \end{equation*}
 
    By the valuative criterion of properness \cite[Proposition 7.3.3]{EGA2}, for $\nu \in \Sigma_f(\IK)$, any point $x \in X(\IC_\nu)$ extends to a unique $\mathcal{O}_{\IC_\nu}$-point of $\mathcal{X}$ and in this way induces a point $\mathrm{red}_\nu(x) \in \mathcal{X}(\overline{\IF}_\nu)$.
    
    A \textit{$\nu$-metric} on $L$ (for $\nu \in \Sigma(\IK)$) is a map $\Vert \cdot \Vert: L(\IC_\nu) \rightarrow \IR_{\geq 0}$ such that for each open subset $U \subseteq X$, each section $\mathbf{s}: U \rightarrow L$, and every $f \in \mathcal{O}_X(U)$, 
\begin{itemize}
\item[(i)] the function $\Vert \mathbf{s} \Vert = \Vert \cdot \Vert \circ \mathbf{s}_{\IC_\nu}: U(\IC_\nu) \rightarrow \IR_{\geq 0}$ is continuous,
\item[(ii)] if $\mathbf{s}$ vanishes nowhere on $U$, then $\Vert \mathbf{s} \Vert: U(\IC_\nu) \rightarrow \IR_{\geq 0}$ has no zeros,
\item[(iii)] for every $f \in \mathcal{O}_X(U)$, we have $\Vert f \mathbf{s} \Vert = |f| \cdot \Vert \mathbf{s} \Vert$ on $U(\IC_\nu)$, and
\item[(iv)] $\Vert \cdot \Vert$ is $\Gal(\IC_\nu/{\mathbb{K}}_\nu)$-invariant.
\end{itemize}

Given a model $\mathcal{L}$ of $L^{\otimes e}$, $e \geq 1$, on some model $\mathcal{X}$ of $X$, we can define a $\nu$-metric on $L$ for every place $\nu \in \Sigma_f({\mathbb{K}})$. Let us recall the construction of these metrics: We first choose a covering $\{ \mathcal{U}_i \}$ of $\mathcal{X}$ by open subschemes such that there are isomorphisms $\varphi_{i}: \mathcal{L}|_{\mathcal{U}_i} \rightarrow \IA^1_{\mathcal{U}_i}$.

For every place $\nu \in \Sigma_f({\mathbb{K}})$ and every $x \in X(\IC_\nu)$ such that $\mathrm{red}_\nu(x) \in \mathcal{U}_i(\overline{\IF}_\nu)$, the isomorphism $\varphi_i$ induces a $\IC_\nu$-linear homomorphism
\begin{equation*}
\varphi_{i,x}: (L^{\otimes e}|_{x})(\IC_\nu) \longrightarrow \IC_\nu
\end{equation*}
and we use this to set $$\Vert v \Vert_{\nu} =|\varphi_{i,x}(v)|_\nu$$ for all $v \in (L^{\otimes e}|_x)(\IC_\nu)$. As the notation suggests, $\Vert \cdot \Vert_\nu$ does not depend on the choice of $\mathcal{U}_i$. In fact, if $\mathrm{red}_\nu(x)$ is contained in an overlap $\mathcal{U}_{ij}=\mathcal{U}_{i} \cap \mathcal{U}_{j}$, the composite $\varphi_j|_{\mathcal{U}_{ij}} \circ \varphi_i^{-1}|_{\mathcal{U}_{ij}}$ corresponds to multiplication by a regular function $f_{ij} \in \caO(\mathcal{U}_{i} \cap \mathcal{U}_{j})^\times$. As $f_{ij}$ restricts to a unit in $\mathcal{O}_{\mathcal{X},\mathrm{red}_\nu(x)}$, we have $|f_{ij}(x)|_\nu =1$. 
In this way, we obtain a $\nu$-metric on $L^{\otimes e}$ and, after taking the $e$-th root, also a $\nu$-metric $\Vert \cdot \Vert_{\mathcal{L},\nu}$ on $L$. (We suppress the $e$ in the notation $\Vert \cdot \Vert_{\mathcal{L},\nu}$ as it is always clear from context.)

\subsection{Metrized line bundles}
 
A collection of $\nu$-metrics $\{ \Vert \cdot \Vert_\nu \}_{\nu \in \Sigma({\mathbb{K}})}$ on $L$ is called an \textit{adelic metric} if there exists a model $\mathcal{L}$ of $L^{\otimes e}$ for some $e \geq 1$ such that
\begin{equation}\label{eq::differencetomodel}
\log \left(\frac{\Vert \cdot \Vert_\nu}{\Vert \cdot \Vert_{\mathcal{L},\nu}}\right)    
\end{equation}
is continuous and bounded for all $\nu \in \Sigma_f({\mathbb{K}})$ and identically zero for all but finitely many $\nu \in \Sigma_f({\mathbb{K}})$.
The pair $\Ltil=(L, \{ \Vert \cdot \Vert_\nu \}_{\nu \in \Sigma({\mathbb{K}})})$ is then called a \textit{metrized line bundle}. We call $\Ltil$ \emph{algebraically metrized} if there exists a model $\mathcal{L}$ such that \eqref{eq::differencetomodel} is identically zero for all $\nu \in \Sigma_f({\mathbb{K}})$.
An algebraically metrized line bundle $(L, \{ \Vert \cdot \Vert_\nu \}_{\nu \in \Sigma({\mathbb{K}})})$ is called \emph{algebraically semipositive} if it is induced by a model $\mathcal{L}$ that is relatively nef and such that, for every $\nu \in \Sigma_\infty({\mathbb{K}})$, the Chern current $c_1(\Lo)$ of the hermitian line bundle $\Lo = (L, \Vert \cdot \Vert_{\nu})$ is semipositive.

If $\widetilde{L}$ and $\widetilde{M}$ are metrized line bundles with underlying line bundles $L$ and $M$, there is an obvious way to endow $L \otimes M$ and $L^{-1}$ with the structure of metrized line bundles. We write $\widetilde{L} \otimes \widetilde{M}$ and $\widetilde{L}^{-1}$, respectively, for these metrized line bundles. If $f: Y \rightarrow X$ is an algebraic map between irreducible, projective ${\mathbb{K}}$-varieties and $\widetilde{L}$ is a metrized line bundle, we can endow the pullback $f^\ast L$ with a canonical adelic metric, obtaining a metrized line bundle $f^\ast \widetilde{L}$. For a closed immersion $f: Y \hookrightarrow X$ of ${\mathbb{K}}$-varieties, we write $\Ltil|_Y$ instead of $f^\ast \Ltil$. 

We say that a sequence of adelic metrics $\{ \Vert \cdot \Vert_{i,\nu} \}_{\nu \in \Sigma({\mathbb{K}})}$ ($i \geq 1$) on $L$ converges to an adelic metric $\{ \Vert \cdot \Vert_{\nu} \}_{\nu \in \Sigma({\mathbb{K}})}$ on $L$ if the functions
\[ \log \left( \frac{\Vert \cdot \Vert_{i,\nu}}{\Vert \cdot \Vert_\nu} \right)\]
converge to $0$ uniformly on $X(\IC_\nu)$ for every $\nu \in \Sigma({\mathbb{K}})$ and are constantly $0$ for all but finitely many $\nu \in \Sigma_f({\mathbb{K}})$ and every $i \geq 1$.

A metrized line bundle $\Ltil = (L,\{ \Vert \cdot \Vert_{\nu} \})$ is called \emph{semipositive} if there exists a sequence of algebraically semipositive metrized line bundles $(L,\{ \Vert \cdot \Vert_{i,\nu} \})$ such that their adelic metrics $\{ \Vert \cdot \Vert_{i,\nu} \}$ converge to the adelic metric $\{ \Vert \cdot \Vert_{\nu} \}$ of $\Ltil$.

There is a natural notion of \textit{isometry} between metrized line bundles and the isometry classes of metrized line bundles form a Picard group $\widehat{\Pic}(X)$. An \emph{integrable} metrized line bundle is one that is isometric to $\widetilde{L} \otimes \widetilde{M}^{-1}$ for two semipositive metrized line bundles $\widetilde{L}, \widetilde{M}$. Note that both semipositivity and integrability are preserved under tensor products and pullback.
 
\subsection{Arithmetic intersection numbers}
\label{subsection::arithmetic_intersection_numbers}

Given integrable 
\begin{equation*}
\Ltil_i = (L_i, \{\Vert \cdot \Vert_{i,\nu} \}_{\nu \in \Sigma({\mathbb{K}})}) \in \widehat{\Pic}(X), \ 1 \leq i \leq d+1,
\end{equation*}
one can define an \textit{arithmetic intersection number}
\begin{equation*}
\Ltil_1 \cdot \Ltil_2 \cdots \Ltil_{d+1} \in \IR.
\end{equation*}
This can be done as in \cite[(1.4) and (1.5)]{Zhang1995a} or by relying on Gubler's theory of local heights \cite{Gubler1997,Gubler1998,Gubler2003}. The arithmetic intersection numbers obtained are the same (see \cite[Remark 11.24]{Gubler2003}). For an irreducible subvariety $Y \subseteq X$ of dimension $d^\prime$, we also set
\begin{equation*}
    \Ltil_1 \cdot \Ltil_2 \cdots \Ltil_{d^\prime} \cdot [Y] = \Ltil_1|_{Y} \cdot \Ltil_2|_{Y}\cdots \Ltil_{d^\prime}|_Y.
\end{equation*}

Let us list some well-known properties of arithmetic intersection numbers.
We say that a metrized line bundle $\Ltil$ is \emph{algebraically nef} if it is algebraically semipositive and, for every irreducible subvariety $Y \subseteq X$, the arithmetic intersection number $\Ltil|_{Y}^{\dim(Y)+1}$ is non-negative. We call a metrized line bundle $\Ltil$ \emph{nef} if there exists a sequence of algebraically nef metrized line bundles $(L,\{ \Vert \cdot \Vert_{i,\nu} \})$ such that their adelic metrics $\{ \Vert \cdot \Vert_{i,\nu} \}$ converge to the adelic metric $\{ \Vert \cdot \Vert_{\nu} \}$ of $\Ltil$.

    \begin{lemma}
		\label{lemma::intersectionnumber}
		Let $\Ltil_1,\dots,\Ltil_{d+1},\Ltil_1^\prime \in \widehat{\Pic}(X)$ be integrable. 
		\begin{enumerate}[label=(\roman*)]
			\item (Multilinearity) We have
			\begin{equation*}
				(\Ltil_1 + \Ltil_1^\prime) \cdot \Ltil_2 \cdots \Ltil_{d+1} =
				\Ltil_1 \cdot \Ltil_2 \cdots \Ltil_{d+1} +\Ltil_1^\prime \cdot \Ltil_2 \cdots \Ltil_{d+1}.
			\end{equation*}
			\item (Commutativity) For any permutation $\sigma: \{ 1, \dots, d + 1 \} \rightarrow \{1, \dots, d + 1 \}$, we have 
			\begin{equation*}
				\Ltil_{\sigma(1)} \cdot \Ltil_{\sigma(2)} \cdots \Ltil_{\sigma(d+1)} = \Ltil_1 \cdot \Ltil_2 \cdots \Ltil_{d+1}.
			\end{equation*}
			\item (Projection Formula) Let $Y$ be an irreducible, projective ${\mathbb{K}}$-variety of dimension $d^\prime$, and let further ${f: Y \rightarrow X}$ be a proper map. Then, the intersection number
            \begin{equation*}
                f^\ast \Ltil_1 \cdot f^\ast \Ltil_2 \cdots f^\ast \Ltil_{d^\prime+1}   
            \end{equation*}
            equals
			\begin{equation*}
				[{\mathbb{K}}(Y):{\mathbb{K}}(f(Y))] \cdot (\Ltil_1|_{f(Y)} \cdot \Ltil_2|_{f(Y)} \cdots \Ltil_{d^\prime+1}|_{f(Y)})
			\end{equation*}
            if $\dim(f(Y)) = d^\prime$ and is equal to $0$ else.
			\item (Scaling invariance) If $\Ltil_1 = (L_1, \{ \Vert \cdot \Vert_\nu \}_{\nu \in \Sigma({\mathbb{K}})})$ and $\Ltil_1^\prime = (L_1, \{ c_\nu \cdot \Vert \cdot \Vert_\nu \}_{\nu \in \Sigma({\mathbb{K}})})$ for a collection of positive real constants $c_\nu$ ($\nu \in \Sigma(\mathbb{K})$), all but finitely many of which are $1$, satisfying $\prod_{\nu \in \Sigma(\IK)}{c_\nu^{\delta_\nu}} = 1$, we have
			\begin{equation*}
				\Ltil_1^\prime \cdot \Ltil_2 \cdots \Ltil_{d+1} = \Ltil_1 \cdot \Ltil_2 \cdots \Ltil_{d+1}.
			\end{equation*}
            In particular, this holds if $c_\nu = |c|_\nu$ for some $c \in {\mathbb{K}}^\times$.
			\item (Base change) If ${\mathbb{K}}^\prime/{\mathbb{K}}$ is a finite extension, then
			\begin{equation*}
				\Ltil_{1,{\mathbb{K}}^\prime} \cdot \Ltil_{2,{\mathbb{K}}^\prime} \cdots \Ltil_{d+1,{\mathbb{K}}^\prime} = [{\mathbb{K}}^\prime:{\mathbb{K}}] \cdot (\Ltil_1 \cdot \Ltil_2 \cdots \Ltil_{d+1}).
			\end{equation*}
		\end{enumerate}
	\end{lemma}
	
	\begin{proof} Each of the first three statements follows from the respective property of Gubler's local heights (\cite[Theorem 10.6]{Gubler2003}), which extends to continuous archimedean metrics by \cite[Appendix A]{Kuehne2022}. Statement (iv) is a consequence of \cite[Propositions 3.7 and 9.4]{Gubler2003} and the product formula. Statement (v) is a consequence of our normalizations, as can be seen from the induction formula (\cite[Proposition 3.5 or Remark 9.5]{Gubler2003}) by using that $[{\mathbb{K}}^\prime:{\mathbb{K}}] = \sum_{\nu^\prime | \nu} [{\mathbb{K}}^\prime_{\nu^\prime}: {\mathbb{K}}_\nu]$.
	\end{proof}

    Because of Lemma \ref{lemma::intersectionnumber} (i), we can extend the definition of the arithmetic intersection numbers to elements of $\widehat{\Pic}(X)_{\IQ} = \widehat{\Pic}(X) \otimes_{\IZ} \IQ$ that have some positive integer multiple in $\widehat{\Pic}(X)$ that is integrable. We call these \emph{integrable} and we call an element of $\widehat{\Pic}(X)_{\IQ}$ \emph{nef} if it has some positive integer multiple in $\widehat{\Pic}(X)$ that is nef. All assertions in Lemma \ref{lemma::intersectionnumber} evidently carry over to the setting of $\IQ$-line bundles.    
    \begin{lemma}\label{lemma::intersectionnumber_nef}
        Let $\Ltil_1,\dots,\Ltil_{d+1} \in \widehat{\Pic}(X)_{\IQ}$ be integrable and nef, then
        \begin{equation*}
            \Ltil_1 \cdot \Ltil_2 \cdots \Ltil_{d+1} \geq 0.
        \end{equation*}
        Consequently, we have
        \begin{equation*}
            \Mtil_1 \cdot \Mtil_2 \cdots \Mtil_{d+1} \geq \Ltil_1 \cdot \Ltil_2 \cdots \Ltil_{d+1}
        \end{equation*}
        for all integrable $\Mtil_1,\dots,\Mtil_{d+1} \in \widehat{\Pic}(X)_{\IQ}$ such that all differences $\Mtil_i - \Ltil_i$, $1 \leq i \leq d+1$, are nef.
    \end{lemma}

    \begin{proof}
        After multiplying by some common denominator, we may and will assume that the $\Ltil_i$ and the $\Mtil_i$ are in $\widehat{\Pic}(X)$. The first assertion is \cite[Proposition 5.6]{Yuan2017}.
        For the second assertion, we use Lemma \ref{lemma::intersectionnumber} (i) to expand 
        \[ (\Ltil_1 + (\Mtil_1 - \Ltil_1)) \cdot (\Ltil_2 + (\Mtil_2 - \Ltil_2)) \cdots (\Ltil_{d+1} + (\Mtil_{d+1} - \Ltil_{d+1}))\]
        and obtain a sum consisting only of non-negative terms.
    \end{proof}

The following lemma will be useful to conclude that certain intersection numbers are zero for dimension reasons.

	\begin{lemma}\label{lemma::projformulaposdim}
		Let $f: X \rightarrow Y$ be a morphism of projective irreducible varieties over $\IK$. For every $d \in \IZ$ with $\dim (Y) < d \leq \dim(X)$ and for all $\Ltil_1,\hdots, \Ltil_{\dim(X)-d} \in \widehat{\Pic}(X)_{\IQ}$ and $\Mtil_1, \hdots, \Mtil_{d+1} \in \widehat{\Pic}(Y)_{\IQ}$, we have
		\begin{equation*}
			\Ltil_1 \cdots \Ltil_{\dim(X)-d} \cdot f^{\ast}\Mtil_1 \cdots f^{\ast}\Mtil_{d+1} = 0.
	    \end{equation*}
	\end{lemma}
    
    \begin{proof}
    After multiplying by some common denominator, we may and will again assume that the $\Ltil_i$ are in $\widehat{\Pic}(X)$ and that the $\Mtil_i$ are in $\widehat{\Pic}(Y)$. The lemma follows from \cite[Lemma 2.4]{Kuehne2022} with $\Ntil_i = \Mtil_i$ ($1 \leq i \leq \dim(Y) +1$) and $\Ltil_{\dim(X)-d+i} = f^\ast \Mtil_{\dim(Y)+1+i}$ ($i = 1, \hdots, d-\dim(Y)$) since, for $\eta_Y$ the generic point of $Y$, the line bundle $(f^\ast M_{d+1})_{\eta_Y}$ is trivial, being the pullback of a line bundle on a point.
	\end{proof}

Let $\mu \in \Sigma_\infty(\IK)$ be a complex archimedean place of $\IK$ and let $f \in \mathcal{C}^\infty(X(\IK_\mu))$ be a function, where a function on a potentially singular reduced complex analytic space is called smooth if it is locally equal to the restriction of a smooth function on some ambient manifold. Then, we write $\Otil_X(f)$ for the metrized line bundle $(\caO_X,\{\Vert \cdot \Vert_\nu \})$ where
\begin{equation*}
    \Vert 1 \Vert_\nu(x) = 
    \begin{cases}
    e^{-f(x)} &\text{if $\nu = \mu$,} \\
    1 &\text{otherwise.}
    \end{cases}
\end{equation*}
A simple calculation using the recursive formula for arithmetic intersection numbers (compare \cite[Eq.\ (2.3)]{Kuehne2022}) yields that
\begin{equation}
\label{equation::intersection_integral}
\widetilde{L}_1 \cdot \widetilde{L}_2 \cdots \widetilde{L}_{d+1-i} \cdot \widetilde{\caO}_X(f)^i
= \delta_\mu \int_{X(\IC_\mu)} f c_1(\Lo_{1}) \wedge c_1(\Lo_{2}) \wedge \dots \wedge c_1(\Lo_{d+1-i}) \wedge c_1(\overline{\caO}_{X}(f))^{i-1},
\end{equation}
where $i \geq 1$ and, for a metrized line bundle $\Mtil$, we denote the associated hermitian line bundle at the place $\mu$ by $\overline{M}$.
If we consider $\kappa \in \IR$ as a constant function on $X({\IC_\mu})$, we hence have
\begin{equation*}
(\Ltil+ \widetilde{\caO}_{X}(\kappa))^{d+1} = \Ltil^{d+1} + \delta_\mu \kappa (d+1) L^d.
\end{equation*}
If $L$ is nef, this means
\begin{equation}
\label{equation::intersection_Lc}
(\Ltil+ \widetilde{\caO}_{X}(\kappa))^{d+1} = \Ltil^{d+1} + \delta_\mu \kappa (d+1) \vol(L)
\end{equation}
by the algebraic Riemann-Roch Theorem (\cite[Corollary 1.4.41]{Lazarsfeld2004}); as usual, we define the volume $\vol(L)$ of a line bundle $L$ as the limes superior
\begin{equation*}
\limsup_{N \rightarrow \infty} \frac{\dim H^0(X,L^{\otimes N})}{N^d/d!}.
\end{equation*}

    \subsection{Heights}
        
	Using the intersection numbers defined above, we can define the Arakelov height $h_{\Ltil}(Y)$ of an irreducible subvariety $Y \subseteq X$ with respect to an integrable metrized line bundle $\Ltil$ whose underlying line bundle $L$ is big and nef when restricted to $Y$. In fact, we set
	\begin{equation}
		\label{equation::definition_height}
		h_{\Ltil}(Y)= \frac{(\Ltil|_Y)^{d^\prime+1}}{[{\mathbb{K}}:\IQ](d^\prime+1)(L|_{Y})^{d^\prime}}
	\end{equation}
	for an irreducible subvariety $Y \subseteq X$ of dimension $d^\prime$ such that $L|_Y$ is big and nef.
	
	We can make the above definition more explicit if $Y$ is a closed point $x \in X$. As $(L|_x)^{0} = [{\mathbb{K}}(x):{\mathbb{K}}]$ (compare e.g.\ \cite[Proposition VI.2.7]{Kollar1996}), the definition in \eqref{equation::definition_height} simplifies to 
    \begin{equation*}
        h_{\Ltil}(x) = \frac{\Ltil|_x}{[{\mathbb{K}}(x):\IQ]}.
    \end{equation*}
    We write $\mathbf{O}_\nu(x)$ for the $\IC_\nu$-valued points of the base change $x \otimes_{\mathbb{K}} \IC_\nu$. Using the induction formula \cite[Proposition 3.5 and Remark 9.5]{Gubler2003}, we obtain
	\begin{equation}
		\label{equation::height_section}
		h_{\Ltil}(x) = - \frac{1}{[{\mathbb{K}}(x):\IQ]} \sum_{\nu \in \Sigma({\mathbb{K}})} \sum_{y \in \mathbf{O}_\nu(x)}
		\delta_\nu \log \Vert \mathbf{s}(y) \Vert_{\nu}
	\end{equation}
	for any non-zero rational section $\mathbf{s}$ of $L$ such that $x \notin\left\vert \Div(\mathbf{s}) \right\vert$. For an algebraic point $x \in X(\overline{{\mathbb{K}}})$ of $X$, we write $h_{\Ltil}(x)$ for the height of its associated closed point.

The essential minimum of the Arakelov height associated with a metrized line bundle $\Ltil$ on $X$ is defined as
\begin{equation*}
    \mu_{\mathrm{ess}}(X,\Ltil) = 
    \sup_{Z \subsetneq X} \inf_{x \in X \setminus Z} h_{\Ltil}(x)
\end{equation*}
where the supremum ranges over all Zariski-closed subsets $Z \subsetneq X$ and the infimum ranges over all closed points $x \in X \setminus Z$. In other words, the essential minimum $\mu_{\mathrm{ess}}(X,\Ltil)$ is the smallest real number $\vartheta$ such that the set of closed points $x \in X$ with height $h_{\Ltil}(x)\leq \vartheta + \varepsilon$ is Zariski-dense for every $\varepsilon > 0$.

\subsection{Minkowski's theorem}

The following lemma is standard in applications of Arakelov theory to diophantine geometry and is a straightforward consequence of an adelic version of Minkowski's second theorem \cite[Theorem C.2.11]{Bombieri2006}. We define the \textit{arithmetic volume} $\volh_\chi(\Ltil)$ of a metrized line bundle $\Ltil\in \widehat{\Pic}(X)$ as in \cite[Section 2.9]{Kuehne2022}. For a section $\mathbf{s} \in H^0(X, L)$ of a metrized line bundle $\Ltil = (L, \{\Vert \cdot \Vert_\nu\})$, we consider its supremum norm
\begin{equation*}
    \Vert \mathbf{s} \Vert_\nu = \sup_{x \in X(\IC_\nu)} \Vert \mathbf{s}(x) \Vert_\nu
\end{equation*}
for every place $\nu \in \Sigma({\mathbb{K}})$.

\begin{lemma} 
\label{lemma::minkowski}
Let $\mu \in \Sigma({\mathbb{K}})$ and a real $\varepsilon>0$ be given. For each $\Ltil = (L, \{ \Vert \cdot \Vert_\nu \}) \in \widehat{\Pic}(X)$ whose underlying line bundle $L$ is big and nef, there exists an integer $M \geq 1$ and a non-zero section $\mathbf{s} \in H^0(X, L^{\otimes M})$ such that
\begin{equation*}
\delta_\mu \log \Vert \mathbf{s} \Vert_{\mu} \leq \left( - \frac{\volh_\chi(\Ltil)}{(d+1)L^d \cdot [X]} + \varepsilon \right) M
\end{equation*}
and
\begin{equation*}
\log \Vert \mathbf{s} \Vert_\nu  \leq 0
\end{equation*}
for all other places $\nu \neq \mu$ of ${\mathbb{K}}$.
\end{lemma}

\begin{proof} 
This is \cite[Lemma 2.7]{Kuehne2022}.
\end{proof}

In our application of this lemma, we need to control the arithmetic volume. For this reason, we conclude this subsection with two further lemmas.

\begin{lemma} 
\label{lemma::arithmeticvolumes_basics}
Let $\widetilde{L} \in \widehat{\Pic}(X)$ be a metrized line bundle. For every complex archimedean place $\mu \in \Sigma_\infty(\IK)$ and $\kappa \in \IR$, which is considered as a real-valued constant function on $X(\IC_\mu)$, we have
\begin{equation*}
\volh_\chi(\Ltil + \Otil_X(\kappa)) = \volh_\chi(\Ltil) + \delta_\mu \kappa (d+1) \vol(L).
\end{equation*}
\end{lemma}

\begin{proof}
This is \cite[Lemma 2.5 (a)]{Kuehne2022}.
\end{proof}

The next lemma is a version of Yuan's bigness theorem \cite{Yuan2008}.

\begin{lemma} 
\label{lemma::siu}
Let $\Ltil, \Mtil \in \widehat{\Pic}(X)$ be semipositive. Then,
\begin{equation*}
\volh_\chi(\Ltil - \Mtil) \geq \Ltil^{d+1} - (d+1) \Ltil^d \cdot \Mtil.
\end{equation*}
\end{lemma}

\begin{proof}
See \cite[Lemma 2.6]{Kuehne2022}.
\end{proof}

\subsection{Canonical metrics on abelian varieties}
    \label{subsection::canonical}

The references for this subsection are \cite{Boucksom2021, Kuehne2022, Zhang1995a}. Our first goal is to endow $L$ with a collection of canonical metrics $\Vert \cdot \Vert_\nu$, $\nu \in \Sigma({\mathbb{K}})$, and to prove some facts about the metrized line bundle $\Ltil$ obtained in this way.

The existence and the basic properties of canonical metrics can actually be established in a dynamical setting and we do so with the following lemma.

\begin{lemma} 
\label{lemma::canonicalmetrics}
Let $X$ be an irreducible projective variety, $L$ a line bundle on $X$, $f: X \rightarrow X$ a surjective algebraic map, $r>1$ an integer, and $\phi: L^{\otimes r} \rightarrow f^\ast L$ an isomorphism of line bundles over $X$.

Then, there exists a unique adelic metric $\{ \Vert \cdot \Vert_\nu \}_{\nu \in \Sigma(\IK)}$ on $L$ such that the isomorphism $\phi$ is an isometry between the metrized line bundles $(L^{\otimes r}, \{ \Vert \cdot \Vert_\nu^{\otimes r} \}_{\nu \in \Sigma(\mathbb{K})})$ and $(f^\ast L, \{ f^\ast \Vert \cdot \Vert_\nu\}_{\nu \in \Sigma(\mathbb{K})} )$. Moreover, the following assertions are true:

\begin{enumerate}
\item[(i)]
On replacing $\phi$ with $c \phi$ for some scalar $c \in \mathbb{K}^\times$, the metric $\Vert \cdot \Vert_\nu$, $\nu \in \Sigma(\mathbb{K})$, changes to $|c|_\nu^{1/(r-1)} \Vert \cdot \Vert_\nu$.

\item[(ii)] If there exists an algebraically semipositive (resp.\ nef) metric on $L$, then the metrized line bundle $\Ltil = (L, \{ \Vert \cdot \Vert_\nu \}_{\nu \in \Sigma({\mathbb{K}})})$ is semipositive (resp.\ nef). If there exists an algebraically nef metric $\{ \Vert \cdot \Vert_{0,\nu} \}_{\nu \in \Sigma({\mathbb{K}})}$ on $L$ such that, for all archimedean places $\nu \in \Sigma_\infty(\IK)$, we have 
\[
c_1(L, \lVert \cdot \rVert_{0,\nu}) \geq \omega_\nu
\]
for some strictly positive $(1,1)$-form $\omega_\nu$ on $X(\IC_\nu)$ such that $f_\nu^\ast \omega_\nu = r\omega_\nu$, then the adelic metric $\{ \Vert \cdot \Vert_\nu \}_{\nu \in \Sigma(\IK)}$ on $L$ is the limit of a sequence of adelic metrics $\{ \Vert \cdot \Vert_{l,\nu} \}_{\nu \in \Sigma(\IK)}$ ($l \in \mathbb{N}$) on $L$ such that each $(L, \{\Vert \cdot \Vert_{l,\nu} \}_{\nu \in \Sigma(\IK)})$ is an algebraically nef metrized line bundle for which 
\[
c_1(L,\lVert \cdot \rVert_{l,\nu}) \geq \omega_\nu
\] 
for all $\nu \in \Sigma_\infty(\IK)$.

\item[(iii)] Let $\mathcal{X}$ and $\mathcal{L}$ be models of $X$ and $L^{\otimes e}$, respectively. For every place $\nu \in \Sigma_f(\IK)$, let $\Vert \cdot \Vert_{\mathcal{L},\nu}: L(\IC_\nu) \rightarrow \IR_{\geq 0}$ be the $\nu$-metric associated with the model $\mathcal{L}$ as in Section \ref{subsection::metrics}. Then, there exist constants $\newconstant{\label{cnu}}_{,\nu} = \oldconstant{cnu}{}_{,\nu}(\mathcal{L}) \geq 0$, $\nu \in \Sigma_f(\IK)$, all but finitely many of which are $0$ such that
\begin{equation*}
    \log \Vert \cdot \Vert_{\mathcal{L},\nu} - \oldconstant{cnu}{}_{,\nu}
    \leq \log \Vert \cdot \Vert_\nu 
    \leq \log \Vert \cdot \Vert_{\mathcal{L},\nu} + \oldconstant{cnu}{}_{,\nu}.
\end{equation*}

\item[(iv)] For every $\varepsilon > 0$ and every model $(\mathcal{X},\mathcal{L})$ of $(X,L^{\otimes e})$, there exists some integer $N \geq 1$ and a model $(\mathcal{X}^\prime,\mathcal{L}^\prime)$ of $(X, L^{\otimes e N})$ together with a surjective map $\pi: \mathcal{X}^\prime \rightarrow \mathcal{X}$ such that
\begin{equation*}
    \sum_{\nu \in \Sigma_f(\IK)} \sup_{L(\IC_\nu)} \left\vert \log \Vert \cdot \Vert_{\mathcal{L}^\prime,\nu} - \log \Vert \cdot \Vert_\nu \right\vert
    \leq \varepsilon.
\end{equation*}

\end{enumerate}

\end{lemma}

This lemma is well-known (compare \cite[Lemma 3.1]{Kuehne2022}) and goes back to \cite[Theorem 2.2]{Zhang1995a}. Of special importance for us are its parts (iii) and (iv), which are less well recorded in the literature, though they are well-known to experts as well. It is for these observations that we indicate some details of the proof here.

\begin{proof}
This is \cite[Lemma 3.1]{Kuehne2022} except for the nefness and positivity part of (ii) and the assertions in (iii) and (iv). However, these assertions follow immediately from the said proof; note that the construction of a model as in said proof is not necessary in the present setting since we already start with a model $(\mathcal{X},\mathcal{L})$ of $(X,L^{\otimes e})$. The model $(\mathcal{X}^\prime,\mathcal{L}^\prime)$ whose existence is claimed in (iv) is just the model $(\overline{\Gamma(f^{\circ n})}, \mathcal{L}_n)$ of $(X,L^{\otimes r^n e})$ for a sufficiently large integer $n$. The proof of the claim concerning nefness in (ii) uses the fact that $L$ with the adelic metric induced by such a model is algebraically nef because of the arithmetic version of the projection formula (Lemma \ref{lemma::intersectionnumber} (iii)). The inequality between Chern currents in (ii) follows from the construction in the proof of \cite[Lemma 3.1]{Kuehne2022} and the fact that $r^{-1} f_\nu^\ast \omega_\nu = \omega_\nu$.
\end{proof}

Similarly as in Subsections \ref{subsection::specialcase} to \ref{subsection::degreebound}, we let $A$ denote an abelian variety over a number field ${\mathbb{K}}$ and $L$ an ample symmetric line bundle on $A$ until the end of this section. In addition, we fix an arbitrary isomorphism $\phi: L^{\otimes 4}\rightarrow [2]^\ast L$. Applying Lemma \ref{lemma::canonicalmetrics} to $L$ and $\phi$, we obtain a metrized line bundle $\Ltil = (L, \{\Vert\cdot \Vert_\nu \}_\nu)$. Recall that we write
\begin{equation*}
	\varphi_n : A \times A \rightarrow A, \ (x, y) \longmapsto [n](x) - y,
\end{equation*}
and note that the metrized line bundle $\Ltil_n := \varphi_n^\ast \Ltil$ has underlying line bundle $L_n$.

For all $\nu \in \Sigma_\infty(\IK)$, we can fix an identification $A(\IC_\nu) = \IC^{g}/\Lambda_\nu$ for some lattice $\Lambda_\nu$ and obtain in this way a strictly positive $(1,1)$-form 
\begin{equation*}
\omega_\nu = i\sum_{j=1}^{g} dz_j \wedge d\overline{z}_j    
\end{equation*}
on $A(\IC_\nu)$ such that $[2]^\ast \omega_\nu = 4\omega_\nu$. Since $L$ is ample, we can use the pullback of the Fubini-Study metrics on $\caO(1)$ with respect to a projective embedding associated with some positive tensor power of $L$ in order to construct an algebraically nef adelic metric $\{\lVert \cdot \rVert_{0,\nu}\}_{\nu \in \Sigma(\IK)}$ on $L$ such that, after potential rescaling of the coordinates $z_j$, we can ascertain that $c_1(L, \lVert \cdot \rVert_{0,\nu}) \geq \omega_\nu$ for all $\nu \in \Sigma_\infty(\IK)$. We can therefore apply the second part of Lemma \ref{lemma::canonicalmetrics} (ii) and obtain a sequence of adelic metrics $\{\Vert\cdot \Vert_{l,\nu} \}_{\nu}$ ($l \in \mathbb{N}$) on $L$ converging to $\{\Vert\cdot \Vert_\nu \}_{\nu}$ and such that each $(L, \{\Vert \cdot \Vert_{l,\nu} \}_{\nu \in \Sigma(\IK)})$ is an algebraically nef metrized line bundle for which 
\[
c_1(L,\lVert \cdot \rVert_{l,\nu}) \geq \omega_\nu
\] 
for all $\nu \in \Sigma_\infty(\IK)$.

We also need that, in the case of abelian varieties, the archimedean metrics constructed in Lemma \ref{lemma::canonicalmetrics} are actually smooth.

\begin{lemma}
    \label{lemma::metrics_smooth}
    For every archimedean place $\nu \in \Sigma_\infty(\IK)$, the canonical metric $\Vert \cdot \Vert_\nu$ on $L_\nu$ constructed above is smooth, i.e.\ for each holomorphic section $\mathbf{s}$ of $L_\nu$ over an open $U \subseteq A(\IC_\nu)$, the function $\lVert \mathbf{s} \rVert^2_\nu$ is smooth on $U$.
\end{lemma}

\begin{proof}
By \cite[Proposition 9.18 (1)]{Moriwaki2014}, there exists a smooth metric $\lVert \cdot \rVert^\prime_\nu$ on $L_\nu$ with translation-invariant Chern form. The two smooth metrics $[2]^\ast \lVert \cdot \rVert^\prime_\nu$ and $(\lVert \cdot \rVert^\prime_\nu)^{\otimes 4}$ on $L_\nu^{\otimes 4} \approx [2]^\ast L_\nu$ with translation-invariant Chern forms differ by a positive constant because of \cite[Proposition 9.18 (2)]{Moriwaki2014}. It follows that some positive multiple of $\Vert \cdot \Vert^\prime_\nu$ satisfies the characteristic property of $\lVert \cdot \rVert_\nu$ stated in Lemma \ref{lemma::canonicalmetrics}. Thus, the metrics $\Vert \cdot \Vert_\nu$ and $\Vert \cdot \Vert_\nu^\prime$ are proportional (by a positive constant) and $\lVert \cdot \rVert_\nu$ is smooth as well.
\end{proof}

By Lemma \ref{lemma::canonicalmetrics} (i), the metrized line bundle $\Ltil$ depends on the choice of the isomorphism $\phi$ that becomes an isometry with respect to the chosen canonical metric. However, this does not play any role for arithmetic intersection numbers or heights by Lemma \ref{lemma::intersectionnumber} (iv). In general, we want to consider two adelic metrics that differ only by multiplication with a scalar $\alpha \in \mathbb{K}^\times$ as equivalent.

\begin{definition}
    Two metrized line bundles $\Ltil_1$ and $\Ltil_2$ are called homothetic if there exist an element $c \in \IK^\times$ and a positive integer $s$ such that $\Ltil_1 \otimes \Ltil_2^{\otimes -1}$ is isometric to $(\caO_A, \{|c|^{1/s}_\nu | \cdot |_\nu\}_\nu)$. In this case, we write $\Ltil_1 \sim \Ltil_2$.
\end{definition}

Note that homothety is an equivalence relation and that $(\caO_A, \{|c|^{1/s}_\nu | \cdot |_\nu\}_\nu)$ is an algebraically nef line bundle: because of Lemma \ref{lemma::intersectionnumber} (iv), the only thing we have to verify is that the metric is induced by a model. For this it suffices to take $e = s$ and, for an arbitrary model $\mathcal{X}$ of $A$, the model $\caO_{\mathcal{X}}$ of $\caO_A = \caO_A^{\otimes e}$, but twist the identification of $(\caO_{\mathcal{X}})_{\IK}$ with $\caO_A$ by multiplying by $c^{-1}$. In particular, being semipositive or being nef is preserved under homothety.

\begin{remark}\label{remark::unique_metric_trivial_bundle}
Note that for the trivial line bundle $\mathcal{O}_A$, a choice of a nowhere vanishing global section $\mathbf{s}_0$ determines an isomorphism $\phi: \caO_{A}^{\otimes 4} \rightarrow [2]^\ast \caO_{A}$ identifying $\mathbf{s}_0^{\otimes 4}$ with $[2]^\ast \mathbf{s}_0$. In the case where $\mathbf{s}_0$ is the constant function $1$, we note that $\phi$ becomes an isometry if we endow $\mathcal{O}_A$ with the collection of trivial $\nu$-metrics $|\cdot|_\nu$. By Lemma \ref{lemma::canonicalmetrics} (i), any other line bundle $\widetilde{M}$ metrized as in Lemma \ref{lemma::canonicalmetrics} with $f = [2]$ and $r = 4$ whose underlying line bundle is trivial is homothetic to the metrized line bundle $(\mathcal{O}_A, \{|\cdot|_\nu\}_{\nu \in \Sigma(\IK)})$.
\end{remark}

The following lemma is well-known, but demands proof.

\begin{lemma}\label{lemma::metric_symmetric}
        The metrized line bundles $\Ltil$ and $[-1]^\ast \Ltil$ are homothetic.
   \end{lemma}

  \begin{proof}
        The underlying line bundle of the metrized line bundle $\Ltil \otimes [-1]^\ast \Ltil^{\otimes(-1)}$ is trivial. The lemma follows from Remark \ref{remark::unique_metric_trivial_bundle}.
  \end{proof}

   We set
\begin{equation*}
    \widetilde{P} = \varphi_1^{\ast} \Ltil \otimes \mathrm{pr}_1^{\ast} \Ltil^{\otimes (-1)} \otimes \mathrm{pr}_2^{\ast} \Ltil^{\otimes (-1)}.
\end{equation*}
       
	\begin{lemma}\label{lem:formulaforln}
        The metrized line bundles $\Ltil_n$ and $ \widetilde{P}^{\otimes n} \otimes \mathrm{pr}_1^{\ast} \Ltil^{\otimes n^2} \otimes \mathrm{pr}_2^{\ast} \Ltil$ are homothetic. Consequently, the metrized line bundles $\Ltil_n \otimes \Ltil_{-n}$ and $\pr_1^\ast \Ltil^{\otimes 2n^2} \otimes \pr_2^\ast \Ltil^{\otimes 2}$ are homothetic.
	\end{lemma}
	
	\begin{proof}
		We prove this by (upward and downward) induction on $n$, showing that the claim for $n$ and $n+1$ implies both the claim for $n+2$ as well as the claim for $n-1$ (for all integers $n \in \mathbb{Z}$). The base case is then $n \in \{0,1\}$. For $n=0$, we note that
		\begin{equation*}
			\varphi_0^\ast \Ltil = \pr_2^\ast [-1]^\ast \Ltil \sim \pr_2^\ast \Ltil
		\end{equation*}
		by Lemma \ref{lemma::metric_symmetric} and, for $n=1$, we note that
		\begin{align*}
			\varphi_1^\ast \Ltil 
			&= (\varphi_1^\ast \Ltil \otimes \pr_1^\ast \Ltil^{\otimes -1} \otimes \pr_2^\ast \Ltil^{\otimes -1})
			\otimes
			\pr_1^\ast \Ltil \otimes \pr_2^\ast \Ltil \\
			&= \widetilde{P} \otimes \pr_1^{\ast} \Ltil \otimes \pr_2^\ast \Ltil.
		\end{align*}
		For the induction step(s), we recall that the underlying line bundle of the metrized line bundle
		\begin{equation*}
			\pr_{123}^\ast  \Ltil \otimes \pr_{12}^\ast  \Ltil^{\otimes -1} \otimes \pr_{13}^\ast  \Ltil^{\otimes -1} \otimes \pr_{23}^\ast  \Ltil^{\otimes -1} \otimes \pr_1^\ast  \Ltil \otimes \pr_2^\ast  \Ltil \otimes \pr_3^\ast  \Ltil
		\end{equation*}	
		on $A^3$ is trivial by the theorem of the cube, where $\pr_{123} = \pr_1 + \pr_2 + \pr_3$ etc. Moreover, this metrized line bundle satisfies the hypothesis in Remark \ref{remark::unique_metric_trivial_bundle} (with $A^3$ in place of $A$). Hence, it follows from Remark \ref{remark::unique_metric_trivial_bundle} that
		\begin{equation*}
			\pr_{123}^\ast \Ltil \otimes \pr_{12}^\ast \Ltil^{\otimes -1} \otimes \pr_{13}^\ast \Ltil^{\otimes -1} \otimes \pr_{23}^\ast \Ltil^{\otimes -1} \otimes \pr_1^\ast \Ltil \otimes \pr_2^\ast \Ltil \otimes \pr_3^\ast \Ltil
		\end{equation*}
		is homothetic to the trivial line bundle $\caO_{A^3}$, endowed with the trivial metrics. Pullback by
		\begin{equation*}
		(\varphi_n,\pr_1,-\pr_1): A \times A \longrightarrow A^3	
		\end{equation*}
		yields
		\begin{equation*}
		\Ltil_{n+1}
		\sim
		\Ltil_n^{\otimes 2} \otimes \Ltil^{\otimes(-1)}_{n-1}	\otimes \pr_1^\ast \Ltil^{\otimes 2}
		\end{equation*}
		for every integer $n$. For the induction step $n-1,n \rightarrow n+1$, we note that
		\begin{align*}
			\Ltil_{n+1}
			&\sim 
			\Ltil_n^{\otimes 2} \otimes \Ltil^{\otimes(-1)}_{n-1}	\otimes \pr_1^\ast \Ltil^{\otimes 2} \\
			&\sim
			(\widetilde{P}^{\otimes n} \otimes \pr^\ast_1 \Ltil^{\otimes n^2} \otimes \pr^\ast_2 \Ltil)^{\otimes 2}
            \otimes
            (\widetilde{P}^{\otimes (n-1)} \otimes \pr^\ast_1 \Ltil^{\otimes (n-1)^2} \otimes \pr^\ast_2 \Ltil)^{\otimes(-1)}
            \otimes
            \pr_1^\ast \Ltil^{\otimes 2}
			\\
			&\sim \widetilde{P}^{\otimes(n+1)} \otimes \pr^\ast_1 \Ltil^{\otimes (n+1)^2} \otimes \pr^\ast_2 \Ltil.
		\end{align*}
  For the induction step $n+1,n \rightarrow n-1$, we note that similarly
	\begin{align*}
			\Ltil_{n-1}
			&\sim 
			\Ltil_n^{\otimes 2} \otimes \Ltil^{\otimes(-1)}_{n+1}	\otimes \pr_1^\ast \Ltil^{\otimes 2} \\
			&\sim
			(\widetilde{P}^{\otimes n} \otimes \pr^\ast_1 \Ltil^{\otimes n^2} \otimes \pr^\ast_2 \Ltil)^{\otimes 2}
            \otimes
            (\widetilde{P}^{\otimes (n+1)} \otimes \pr^\ast_1 \Ltil^{\otimes (n+1)^2} \otimes \pr^\ast_2 \Ltil)^{\otimes(-1)}
            \otimes
            \pr_1^\ast \Ltil^{\otimes 2}
			\\
			&\sim \widetilde{P}^{\otimes(n-1)} \otimes \pr^\ast_1 \Ltil^{\otimes (n-1)^2} \otimes \pr^\ast_2 \Ltil.\qedhere
		\end{align*}
 \end{proof}

\section{Lower height bounds}\label{section::equidistribution}

In this section, we establish the arithmetic analog of the geometric degree bound given in Lemma \ref{lemma::degreebound}. Our main tools are Yuan's bigness theorem \cite{Yuan2008} and equidistribution techniques from Arakelov theory \cite{Szpiro1997}. We work in the setting of Theorem \ref{theorem::main}, albeit over a fixed number field $\IK \subseteq \IQbar$. To be precise, we let $A$ be an abelian variety over $\IK$, $L$ be an ample, symmetric line bundle on $A$, and $C_1$ and $C_2$ two geometrically irreducible algebraic subcurves in $A$. \textit{All constants in this section depend only on $\IK$, $A$, $L$, $C_1$, and $C_2$ unless explicitly stated otherwise.}
    
\subsection{Stabilizers and bigness}
\label{subsection::stabilizers}

In this subsection, we collect a few further geometric lemmas needed in the equidistribution argument given in Subsection \ref{subsection::equidistribution} below. We recall that we write
    \begin{equation*}
	   \varphi_{B,n} : B \times B \rightarrow B, \ (x, y) \longmapsto [n] (x) - y,
    \end{equation*}
for every abelian variety $B$ and every integer $n$.

\begin{lemma}
\label{lemma::stabilizer_of_difference}
Let $C_1$ and $C_2$ be as in Theorem \ref{theorem::main}. Assume additionally that $C_1$ and $C_2$ are not both contained in translates of the same proper algebraic subgroup of $A$ and that the stabilizers $\mathrm{Stab}_A(C_1)$ and $\mathrm{Stab}_A(C_2)$ are both finite. Then, the stabilizer $\mathrm{Stab}_A(\varphi_n(C_1 \times C_2))$ is finite for every integer $n \in \IZ$. 
\end{lemma}

\begin{proof}
The case $n=0$ is trivial as $\varphi_0(C_1 \times C_2) = - C_2$ has stabilizer $\mathrm{Stab}_A(C_2)$ in this case. In the following, we assume $n \in \IZ \setminus \{0 \}$. For readability, we write $X_n = \varphi_n(C_1 \times C_2)$. From Lemma \ref{lem:genfinite}, we know that $\dim X_n = 2$.

We first exclude the case that $\dim(\mathrm{Stab}_A(X_n)) = 1$. 
For this purpose, let us assume that $\dim(\mathrm{Stab}_A(X_n)) = 1$ and consider the associated quotient map $\pi: A \to A/\mathrm{Stab}_A(X_n) =: B$. By hypothesis, neither $\pi(C_1)$ nor $\pi(C_2)$ is a point, so they are both curves. Moreover,
\begin{equation*}
\pi(X_n) = \varphi_{B,n}(\pi(C_1) \times \pi(C_2)) 
\end{equation*}
would be a curve. Furthermore, the curves $\pi(C_1)$ and $\pi(C_2)$ are not both contained in translates of the same proper algebraic subgroup of $B$; for else $C_1$ and $C_2$ would both be contained in translates of the same proper algebraic subgroup of $A$. Since moreover $\dim B = \dim A - 1 \geq 2$, we can apply Lemma \ref{lem:genfinite} to the curves $\pi(C_1), \pi(C_2) \subseteq B$ and infer that the morphism $(\varphi_{B,n})|_{\pi(C_1) \times \pi(C_2)}$ is generically finite for all $n \in \mathbb{Z} \backslash \{0\}$. But this implies the contradiction $\dim(\pi(X_n))=2$.

Let us finally exclude the case $\dim(\mathrm{Stab}_A(X_n)) = 2$. Assume this were the case and write again $\pi: A \rightarrow A/\mathrm{Stab}_A(X_n)$. Then $\pi(X_n)$ would be a point and so would be $\pi(C_1)$ and $\pi(C_2)$. But this would imply that $C_1$ and $C_2$ are both contained in a translate of $\mathrm{Stab}_A(X_n) \subsetneq A$, which contradicts our assumption.
\end{proof}

We state a further lemma that will be used to reduce from a finite to a trivial stabilizer in the proof of Proposition \ref{proposition::heightbound} below.

\begin{lemma}
    \label{lemma::stabilizers_quotient}
    Let $X \subseteq A$ be a geometrically irreducible algebraic subvariety. For any finite subgroup $G \subseteq A$ with associated quotient homomorphism $\pi: A \rightarrow A/G$, we have
    \begin{equation*}
        \pi(\mathrm{Stab}_A(X)) = \mathrm{Stab}_{A/G}(\pi(X)).
    \end{equation*}
\end{lemma}

\begin{proof}
    The inclusion ``$\subseteq$'' is clear. Conversely, consider $p \in \mathrm{Stab}_{A/G}(\pi(X))$ and choose a lifting $\widetilde{p} \in A(\IQbar)$ of $p$. By definition, we have
    \begin{equation*}
        \pi(X) + p = \pi(X),
    \end{equation*}
    which implies
    \begin{equation*}
        X + \widetilde{p} = X + \gamma
    \end{equation*}
    for some $\gamma \in G(\IQbar)$. (Indeed, the geometrically irreducible components of the preimage $\pi^{-1}(\pi(X))$ are of the form $X + \gamma$ for some $\gamma \in G(\IQbar)$.) But this implies $\widetilde{p} - \gamma \in \mathrm{Stab}_A(X)$ and hence $p \in \pi(\mathrm{Stab}_A(X))$.
\end{proof}

The following lemma is standard for equidistribution arguments. We briefly introduce the notation we use here and further below. For any abelian variety $B$, we consider the Faltings-Zhang map
\begin{align*}
    \Delta_{B,k}: B^k &\longrightarrow B^{k-1}, \\
    (z_1,z_2,\dots,z_k) &\longmapsto (z_1-z_2,\dots,z_{k-1}-z_k).
\end{align*}

\begin{lemma}
\label{lemma::faltings_zhang_map}
Let $B$ be an abelian variety over $\IK$ and $Y\subseteq B$ a geometrically irreducible subvariety whose stabilizer $\mathrm{Stab}_B(Y)$ is trivial. Then, the restriction $\Delta_{B,k}|_{Y^k}$ is birational onto its image for all sufficiently large integers $k$.
\end{lemma}

\begin{proof}
This follows from \cite[Lemma 4.1]{Abbes1997}.
\end{proof}

Furthermore, we define the line bundle
\begin{equation*}
    M_{n,k} = \pi_1^\ast L_n \otimes \cdots \otimes \pi_k^\ast L_n
\end{equation*}
on $(A \times A)^k$ where $\pi_i: (A \times A)^k \rightarrow A \times A$ is the projection on the $i$-th factor. 

\begin{lemma}
\label{lemma::bigness_Mprime}
Let $C_1$ and $C_2$ be as in Theorem \ref{theorem::main} and let $n \in \IZ \setminus \{ 0 \}$. Assume additionally that $C_1$ and $C_2$ are not both contained in translates of the same proper algebraic subgroup of $A$ and that both stabilizers $\mathrm{Stab}_A(C_1)$ and $\mathrm{Stab}_A(C_2)$ are trivial. Write $B = A \times A$ and $S = C_1 \times C_2$. For all sufficiently large integers $k$, the restriction of the line bundle $M_{n,k-1}$ to $\Delta_{B,k}(S^k) \subseteq B^{k-1}$ is big.
\end{lemma}

\begin{proof}
By Lemma \ref{lem:genfinite}, the restriction $\varphi_{n}|_S$ is generically finite onto its image $\varphi_{n}(S)$. Consequently, the map $\varphi_{n}^k|_{S^k}=(\varphi_{n} \times \cdots \times \varphi_{n})|_{S^k}$ is generically finite onto its image $\varphi_{n}^k(S^k) = \varphi_n(S)^k$. By Lemma \ref{lemma::stabilizer_of_difference}, the image $\varphi_n(S)$ has finite stabilizer $G$. Let $\pi_G: A \to A/G$ denote the projection and let $\pi_G^k: A^k \to (A/G)^k$ denote its $k$-fold power. By Lemma \ref{lemma::stabilizers_quotient}, the stabilizer of $\pi_G(\varphi_n(S))$ is trivial.

By Lemma \ref{lemma::faltings_zhang_map}, the restriction of $\Delta_{A/G,k}$ to $\pi_G(\varphi_n(S))^k$ is birational onto its image for all sufficiently large $k$. Because of the equality $\Delta_{A/G,k} \circ \pi_G^k = \pi_G^{k-1} \circ \Delta_{A,k}$, this implies that $\Delta_{A,k}|_{\varphi^k_n(S^k)}$ is generically finite onto its image
\begin{equation*}
(\Delta_{A,k} \circ \varphi_{n}^k)(S^k)= (\varphi_{n}^{k-1} \circ \Delta_{B,k})(S^k).
\end{equation*}
Thus, also the restriction of $\varphi_{n}^{k-1}:B^{k-1}\rightarrow A^{k-1}$ to $\Delta_{B,k}(S^k)$ is generically finite. As $M_{n,k-1}$ is the pullback of the ample line bundle 
\begin{equation*}
    \pi^\ast_1 L \otimes \cdots \otimes \pi_{k-1}^\ast L
\end{equation*}
along the morphism $\varphi_{n}^{k-1}: B^{k-1} \rightarrow A^{k-1}$, this implies that its restriction to $\Delta_{B,k}(S^k)$ is big.
\end{proof}

\subsection{A lower height bound (special case $C_1 = C_2$)} In this subsection, we restrict ourselves to the special case where $C_1 = C_2$ and give a simple proof of a height bound that is essential for the main proof. Although there seems to be no way to generalize the proof we state in this special case, it seems still worthwhile to provide it separately as the assumption $C_1 = C_2$ was also present in \cite{Bays2015}. The arguments are furthermore simple and close to the ones for the geometric analog (Lemma \ref{lemma::degreebound}). The general version of the height bound is proven in Subsection \ref{subsection::equidistribution} below.

\begin{proposition}\label{proposition::heightbound_symmetric}
    Assume that $C = C_1 = C_2$ is not contained in a translate of a proper algebraic subgroup of $A$. Then, there exists a constant $\newconstant\label{c3.2} \geq 1$, depending only on $C$, such that
	\begin{equation*}
		\oldconstant{c3.2}^{-1} \cdot n^2 < h_{\Ltil_n}(C \times C) < \oldconstant{c3.2} \cdot n^2
	\end{equation*}
	for all $n \neq 0$.
\end{proposition}

\begin{proof}
    By Lemma \ref{lemma::degreebound} and the definition of the height, it suffices to prove that
	\begin{equation*}
		\newconstant\label{c3.3}^{-1} \cdot n^4 \leq (\Ltil_n|_{C \times C})^3 \leq \oldconstant{c3.3}  \cdot n^4
	\end{equation*}
    for some constant $\oldconstant{c3.3} \geq 1$, depending only on $C$.
    
    By Lemma \ref{lem:formulaforln}, we can write
	\[ (\Ltil_n|_{C \times C})^3 = \sum_{i=0}^{6}{b_i n^i}\]
	for certain coefficients $b_i \in \mathbb{R}$ that are independent of $n$ ($i = 0, \hdots, 6$). It follows from Lemma \ref{lemma::intersectionnumber} (iii) that
	\[ b_6 = ((\mathrm{pr}_1^{\ast} \Ltil)|_{C \times C})^3 = 0.\]
	Furthermore, Lemma \ref{lem:genfinite} and the projection formula (Lemma \ref{lemma::intersectionnumber} (iii)) imply that 
	\[ (\Ltil_n|_{C \times C})^3 = \deg(\varphi_n|_{C \times C}) \cdot (\Ltil|_{\varphi_n(C \times C)})^3 > 0\]
	if $n \neq 0$ by \cite{Ullmo1998, Zhang1998} (and Zhang's inequalities \cite[Th\'eor\`eme 1.5]{Chambert-Loir2000}), and $(\Ltil_n|_{C \times C})^3 = 0$ if $n=0$. We infer that $b_0=0$. As $b_6 = 0$, we deduce that $b_5 = 0$ and $b_4 \geq 0$.
	
	Consider the morphism 
	\begin{equation*}
		\sigma: A \times A \longrightarrow A \times A, \ (x,y) \longmapsto (y,x)
	\end{equation*}
	switching the two factors. Note that $\sigma$ induces an automorphism of $C \times C$ and that
	\begin{align*}
		\sigma^\ast \widetilde{P} 
		&= 
		\sigma^\ast (\varphi_1^{\ast} \Ltil \otimes \mathrm{pr}_1^{\ast} \Ltil^{\otimes (-1)} \otimes \mathrm{pr}_2^{\ast} \Ltil^{\otimes (-1)}) \\
		&=
		\varphi^{\ast}_1 [-1]^\ast \Ltil \otimes \pr_2^\ast \Ltil^{ \otimes (-1)} \otimes \pr_1^\ast \Ltil^{\otimes (-1)} \\
		&\sim \widetilde{P}
	\end{align*}
	by Lemma \ref{lemma::metric_symmetric}. We can therefore deduce from the projection formula that
	\begin{align*}
	\frac{b_4}{3} 
	&= ((\mathrm{pr}_1^{\ast} \Ltil)|_{C \times C})^2(\mathrm{pr}_2^{\ast} \Ltil)|_{C \times C} + ((\mathrm{pr}_1^{\ast} \Ltil)|_{C \times C}) (\widetilde{P}|_{C \times C})^2 \\
	&= (((\mathrm{pr}_1 \circ \sigma)^{\ast} \Ltil)|_{C \times C})^2((\mathrm{pr}_2 \circ \sigma)^{\ast} \Ltil)|_{C \times C} + (((\mathrm{pr}_1 \circ \sigma)^{\ast} \Ltil)|_{C\times C}) (\sigma^{\ast}\widetilde{P})|_{C \times C}^2 \\
	&= ((\mathrm{pr}_2^{\ast} \Ltil)|_{C \times C})^2(\mathrm{pr}_1^{\ast} \Ltil)|_{C \times C} + ((\mathrm{pr}_2^{\ast} \Ltil)|_{C \times C}) (\widetilde{P}|_{C \times C})^2\\
	&= \frac{b_2}{3}.
	\end{align*}
	In summary, if $b_4 = 0$ then also $b_2 = 0$. But then it follows that also $b_3 = b_1 = 0$ because of the positivity above and we get a contradiction with said positivity. Hence, $b_4 \neq 0$ and so $b_4 > 0$ and the claim from the proposition follows.
	\end{proof}
    
	\subsection{A lower height bound (general case)}
    \label{subsection::equidistribution}

    As usual for equidistribution arguments, we pick an arbitrary archimedean place $\mu \in \Sigma(\IK)$. \textit{Enlarging $\IK$ if necessary, we may assume that $\mu$ is complex and we simply write $\IC$ for the completion $\IK_\mu$. For any metrized line bundle $\Ntil$ with underlying $\nu$-metric $\Vert \cdot \Vert_\nu$ ($\nu \in \Sigma_\infty(\IK)$), we write $\overline{N}_\nu$ for the hermitian line bundle $(N, \Vert \cdot \Vert_{\nu})$. If $\nu = \mu$, we also write $\overline{N}$ instead of $\overline{N}_\mu$ for simplicity.}
    \begin{proposition}\label{proposition::heightbound}
    Assume that
    \begin{enumerate}
    \item[(i)] there exists no proper algebraic subgroup $H \subsetneq A$ such that both $C_1$ and $C_2$ are contained in translates of $H$ and that
    \item[(ii)] the stabilizers $\Stab_A(C_1)$ and $\Stab_A(C_2)$ are finite.
    \end{enumerate}
    Then, there exist constants $\newconstant\label{3.4}, \newconstant\label{3.5} \geq 1$ such that
	\begin{equation}
        \label{equation::lower_height_bound}
		\oldconstant{3.4}^{-1} \cdot n^2 \leq h_{\Ltil_n}(C_1 \times C_2) \leq \oldconstant{3.4} \cdot n^2
	\end{equation}
    for all integers $|n| \geq \oldconstant{3.5}$.
    \end{proposition}
\begin{lemma}
    \label{lemma::trivial_stabilizer_reduction}
    It suffices to prove Proposition \ref{proposition::heightbound} where (ii) is replaced with the stronger assumption
    \begin{enumerate}
        \item[(ii')] The stabilizers $\Stab_A(C_1)$ and $\Stab_A(C_2)$ are trivial.
    \end{enumerate}
\end{lemma}

\begin{proof}
Write $G$ for the (finite) group generated by $\mathrm{Stab}_A(C_1)$ and $\mathrm{Stab}_A(C_2)$ and let $\pi: A \rightarrow A/G$ be the associated quotient homomorphism. By Lemma \ref{lemma::stabilizers_quotient}, the curves $C_1^\prime = \pi(C_1)$ and $C_2^\prime = \pi(C_2)$ have trivial stabilizer.

We fix an ample symmetric line bundle $L_G$ on $A/G$ and let $\Ltil_G$ denote $L_G$ with the adelic metric given by Lemma \ref{lemma::canonicalmetrics}. There exists some $m \in \mathbb{Z}_{>0}$ such that the line bundles $(\pi^\ast L_G)^{\otimes m} \otimes L^{\otimes(-1)}$ and $L^{\otimes m} \otimes (\pi^\ast L_G)^{\otimes(-1)}$ are both very ample. Any projective embedding associated to a very ample line bundle yields an algebraically nef metric on it; this follows from \cite[Proposition 3.2.4]{Bost1994} as $\caO(1)$ on $\IP^N$ endowed with the Fubini-Study metric is generated globally by sections $X_0,X_1,\dots,X_N$ of supremum norm $\leq 1$. Therefore, by Lemma \ref{lemma::canonicalmetrics} (ii), the canonically metrized line bundles $(\pi^\ast \Ltil_G)^{\otimes m} \otimes \Ltil^{\otimes(-1)}$ and $\Ltil^{\otimes m} \otimes (\pi^\ast \Ltil_G)^{\otimes(-1)}$ are both nef.

Since $L$, $L_G$, and
\[
\varphi_n^\ast (\pi^\ast L_G^{\otimes m}\otimes L^{\otimes -1}) 
= (\pi \circ \varphi_n)^\ast L_G^{\otimes m} \otimes (\varphi_n^\ast L)^{\otimes -1}
= (\varphi_{A/G,n} \circ \pi)^\ast L_G^{\otimes m} \otimes L_n^{\otimes -1}
\]
are nef, we estimate that
\[ L_n^2 \cdot [C_1 \times C_2] 
\leq m^2 ((\varphi_{A/G,n} \circ \pi)^\ast L_G)^2 \cdot [C_1 \times C_2] 
\leq m^2|G|^2 (\varphi_{A/G,n}^\ast L_G)^2 \cdot [C_1^\prime \times C_2^\prime],\]
using the projection formula for the second inequality. Similarly, we obtain that
\begin{equation*}
    (\varphi^\ast_{A/G,n}L_G)^2 \cdot [C_1^\prime \times C_2^\prime]
    \leq
    ((\varphi_{A/G,n} \circ \pi)^\ast L_G)^2 \cdot [C_1 \times C_2] 
    \leq m^2 L_n^2 \cdot [C_1 \times C_2].
\end{equation*}
To obtain analogous bounds for the arithmetic intersection numbers, we start from the related fact that the metrized line bundles $\Ltil$, $\Ltil_G$, and
\begin{equation*}
    \varphi_n^\ast (\pi^\ast \Ltil_G^{\otimes m}\otimes \Ltil^{\otimes -1}) 
    = (\pi \circ \varphi_n)^\ast \Ltil_G^{\otimes m} \otimes (\varphi_n^\ast \Ltil)^{\otimes -1}
    =
    (\varphi_{A/G,n} \circ \pi)^\ast \Ltil_G^{\otimes m} \otimes \Ltil_n^{\otimes -1}
\end{equation*}
are nef; we remark that the pullback of a nef metrized line bundle can be easily seen to be nef---compare \cite[Subsection 2.5.5]{Yuan2026}; our nef corresponds to strongly nef in \cite{Yuan2026}. Using Lemma \ref{lemma::intersectionnumber} (iii) and Lemma \ref{lemma::intersectionnumber_nef}, we can estimate
\[ 
\Ltil_n^3 \cdot [C_1 \times C_2] \leq m^3 ((\varphi_{A/G,n} \circ \pi)^\ast \Ltil_G)^3 \cdot [C_1 \times C_2] \leq m^3|G|^2 (\varphi_{A/G,n}^\ast \Ltil_G)^3 \cdot [C_1^\prime \times C_2^\prime].
\]
Similarly, we obtain
\begin{equation*}
    (\varphi^\ast_{A/G,n}\Ltil_G)^3 \cdot [C_1^\prime \times C_2^\prime]
    \leq
    ((\varphi_{A/G,n} \circ \pi)^\ast \Ltil_G)^3 \cdot [C_1 \times C_2] \leq m^3 \Ltil_n^3 \cdot [C_1 \times C_2].
\end{equation*}
As
\begin{equation*}
    h_{\Ltil_n}(C_1 \times C_2) = \frac{\Ltil_n^3 \cdot [C_1 \times C_2]}{3[\IK:\IQ] \cdot (L_n^2 \cdot [C_1\times C_2])}
    \end{equation*}
 and
    \begin{equation*}
    h_{\varphi_{A/G,n}^\ast \Ltil_G}(C_1^\prime \times C_2^\prime) = \frac{(\varphi_{A/G,n}^\ast \Ltil_G)^3 \cdot [C_1^\prime \times C_2^\prime]}{3[\IK:\IQ] \cdot ((\varphi_{A/G,n}^\ast L_G)^2 \cdot [C_1^\prime \times C_2^\prime])},
\end{equation*}
we can use the above bounds to infer \eqref{equation::lower_height_bound} from the analogous height bound for the geometrically irreducible curves $C_1^\prime$ and $C_2^\prime$ in $A/G$.
\end{proof}

In order to prove Proposition \ref{proposition::heightbound} via equidistribution, we need to construct an auxiliary function. For this purpose, we recall some standard facts on fiberwise integration in the following lemma. For us, a smooth differential form on a (reduced) complex analytic space $X$ is a differential form on its regular locus that, locally, is the restriction of a smooth differential form on an ambient complex manifold. Furthermore, we understand an integral of such a differential form over a (reduced) complex analytic space to be the integral over its smooth locus.

\begin{lemma}
\label{lemma::fiberwise}
Consider geometrically irreducible algebraic subvarieties $X_1$ and $X_2$ of $A$ of dimension $d_1$ and $d_2$, respectively. Write $\pi:X_1\times X_2\to X_1$ for the projection to the first factor. For any smooth semipositive \((d_1+d_2,d_1+d_2)\)-form $\omega$ on $(A \times A)(\IC)$, there exists a unique smooth semipositive \((d_1,d_1)\)-form \(\pi_*\omega\) on $X_1(\IC)$ such that
\begin{equation}
\label{equation::integral_fiberwise}
  \int_{X_1(\IC)} g\,\pi_{*}\omega
  =
  \int_{(X_1 \times X_2)(\IC)}(g\circ\pi)\,\omega
\end{equation}
for every \(g \in \mathcal{C}^0(X_1(\IC))\).
\end{lemma}

For the proof of the lemma, we introduce a few further notations. A multi-index
$$
\underline p=(p_1,\ldots,p_r), \
1\leq p_1<\cdots<p_r\leq g,
$$
is always assumed to be strictly increasing, and we write $|\underline p|=r$ for its length. We denote by $\underline p^c=(p_1^c,\ldots,p_{g-r}^c)$ the multi-index complementary to $\underline p$; to wit,
\begin{equation*}
    \underline{p}^c = (p^c_1,\dots,p_{g-r}^c), \ 1 \leq p_{1}^c < \cdots < p_{g-r}^c \leq g,
\end{equation*}
and 
\begin{equation*}
\{p_1,\dots,p_r\} \cap \{p_1^c,\dots,p_{g-r}^c\} = \emptyset.
\end{equation*}
In particular, $|\underline p^c|=g-|\underline p|$.

Choose a basis $z_1,\dots,z_g$ of standard coordinates on a uniformization of $A(\IC)$. We use multi-indices to define the forms
\begin{equation*}
dz_{\underline p}
=
dz_{p_1}\wedge\cdots\wedge dz_{p_r}
\ \ \text{and} \ \
d\overline z_{\underline p}
=
d\overline z_{p_1}\wedge\cdots\wedge d\overline z_{p_r}
\end{equation*}
on $A(\IC)$.

For two multi-indices $\underline p,\underline q$ of the same length, we define $\gamma_{\underline p,\underline q}\in\{-1,1\}$ by
\begin{equation*}
dz_{\underline p}\wedge d\overline z_{\underline q} \wedge dz_{\underline p^c}\wedge d\overline z_{\underline q^c} = \gamma_{\underline p,\underline q} \cdot 
dz_1\wedge\cdots\wedge dz_g \wedge d\overline z_1\wedge\cdots\wedge d\overline z_g.
\end{equation*}

\begin{proof}
By Lelong's Theorem \cite{Lelong1957}, integration over the smooth locus $X_i^{\mathrm{sm}}$, $i \in \{1,2\}$, defines a positive current $[X_i(\IC)]$ of bidegree $(g-d_i,g-d_i)$ on $A(\IC)$. Using \cite[Proposition I.2.9]{Demailly2012}, we can write
\begin{equation*}
    [X_1(\IC)] = \sum_{\underline{j},\underline{k}}u_{\underline{j},\underline{k}} \cdot (dz_{\underline{j}} \wedge d\overline{z}_{\underline{k}}) \ \ \text{and} \ \ [X_2(\IC)] = \sum_{\underline{j},\underline{k}}v_{\underline{j},\underline{k}} \cdot (dw_{\underline{j}} \wedge d\overline{w}_{\underline{k}}) 
\end{equation*}
where $u_{\underline{j},\underline{k}}$, $|\underline{j}| = |\underline{k}| = g - d_1$, and $v_{\underline{j},\underline{k}}$, $|\underline{j}| = |\underline{k}| = g - d_2$, are distributions on $A(\IC)$, considered as currents of bidegree $(0,0)$ via identifying $dz_1 \wedge \cdots \wedge dz_g \wedge d\overline{z}_1 \wedge \cdots \wedge d\overline{z}_g$ with the constant function $1$. By Fubini's Theorem, the integration current on $(A \times A)(\IC)$ that is associated with the product $(X_1 \times X_2)(\IC)$ equals
\begin{equation*}
    [(X_1 \times X_2)(\IC)] = \sum_{\underline{j},\underline{j}^\prime,\underline{k},\underline{k}^\prime} (u_{\underline{j},\underline{k}} \otimes v_{\underline{j}^\prime,\underline{k}^\prime}) \cdot (dz_{\underline{j}}\wedge d\overline{z}_{\underline{k}} \wedge d{w}_{\underline{j}^\prime} \wedge d\overline{w}_{\underline{k}^\prime})
\end{equation*}
where we define the tensor product of distributions as in \cite[Chapter V]{Hormander2003}. Without loss of generality, we may assume that
\begin{equation*}
    \omega
    =
    \sum_{\substack{
    |\underline j|=|\underline k|=d_1\\
    |\underline j'|=|\underline k'|=d_2}}
    \alpha_{\underline j,
    \underline k,\underline j',\underline k'}
    \cdot
    (dz_{\underline j}\wedge d\overline z_{\underline k}
    \wedge dw_{\underline j'}\wedge d\overline w_{\underline k'}),
\end{equation*}
where each $\alpha_{\underline j, \underline k,\underline j', \underline k'}$ is a smooth function on $(A\times A)(\IC)$. For fixed multi-indices $\underline j,\underline k$ and an arbitrary point $z\in A(\IC)$, set
\begin{equation*}
\omega_{\underline j,\underline k,z}
=
\sum_{\substack{
|\underline j'|=|\underline k'|=d_2}}
\alpha_{\underline j,
\underline k,\underline j',\underline k'}(z,\cdot)
\cdot
(dw_{\underline j'}\wedge d\overline w_{\underline k'}).
\end{equation*}
Pairing this $(d_2,d_2)$-form with the integration current $[X_2]$
gives
\begin{align*}
    \beta_{\underline j,\underline k}(z)
= \int_{X_2(\IC)} \omega_{\underline j,\underline k,z}
    &=
   \sum_{\underline{j}', \underline{k}'} \gamma_{\underline j',\underline k'}\, v_{\underline{j}',\underline{k}'}(\alpha_{\underline j,
\underline k,(\underline j')^c,(\underline k')^c}(z,\cdot)).
\end{align*}
By \cite[Theorem 2.1.3]{Hormander2003}, the function $z \mapsto \beta_{\underline{j},\underline{k}}(z)$ is smooth on $A(\IC)$. We can hence define the smooth $(d_1,d_1)$-form
\begin{equation*}
    \pi_\ast \omega = \sum_{\underline{j}, \underline{k}} \beta_{\underline{j},\underline{k}}(z) dz_{\underline{j}} \wedge d\overline{z}_{\underline{k}}
\end{equation*}
on $X_1(\IC)$.

It suffices to show that this form satisfies the assertions of the lemma. For this purpose, consider an arbitrary function $g \in \mathcal{C}^0(X_1(\IC))$. Unraveling definitions, we obtain
\begin{align*}
\int_{X_1(\IC)} g \pi_\ast \omega
&=
\sum_{\underline j,\underline k}
\gamma_{\underline j,\underline k}\,
u_{\underline j,\underline k}
(g\cdot\beta_{\underline j^c,\underline k^c}) \\
&=
\sum_{\underline j,\underline k}
\sum_{\underline j^\prime,\underline k^\prime} \gamma_{\underline{j},\underline{k}} \gamma_{\underline{j}^\prime,\underline{k}^\prime} u_{\underline{j},\underline{k}}(z \mapsto g(z) \cdot v_{\underline j',\underline k'} (\alpha_{\underline j^c,\underline k^c,(\underline j')^c, (\underline k')^c}(z,\cdot))) \\
&=\sum_{\underline j,\underline k}
\sum_{\underline j^\prime,\underline k^\prime} \gamma_{\underline{j},\underline{k}} \gamma_{\underline{j}^\prime,\underline{k}^\prime} u_{\underline{j},\underline{k}}(z \mapsto v_{\underline j',\underline k'} (g(z) \cdot \alpha_{\underline j^c,\underline k^c,(\underline j')^c, (\underline k')^c}(z,\cdot))) \\
&=\sum_{\underline j,\underline k}
\sum_{\underline j^\prime,\underline k^\prime} \gamma_{\underline{j},\underline{k}} \gamma_{\underline{j}^\prime,\underline{k}^\prime} (u_{\underline{j},\underline{k}} \otimes v_{\underline{j}^\prime, \underline{k}^\prime})((g \circ \pi) \cdot \alpha_{\underline j^c,\underline k^c,(\underline j')^c, (\underline k')^c}) \\
&= \int_{(X_1\times X_2)(\IC)} (g \circ \pi) \omega,
\end{align*}
whence \eqref{equation::integral_fiberwise}.

Testing with continuous functions, it is easy to see that the $(d_1,d_1)$-form $\pi_\ast \omega$ is uniquely characterized by \eqref{equation::integral_fiberwise}. Similarly, testing with non-negative functions we can infer that $\pi_\ast \omega$ is semipositive.
\end{proof}

In the remainder of this subsection, we set $B = A \times A$ and $S = C_1 \times C_2$. The following lemmas are in preparation of the proof of Proposition \ref{proposition::heightbound}. \textit{By Lemma \ref{lemma::trivial_stabilizer_reduction}, we may and will assume that the stabilizers of $C_1$ and $C_2$ are trivial. Moreover, we fix a positive integer $k$ so that $\Delta_{A,k}: C_1^k \to \Delta_{A,k}(C_1^k)$ is birational, which is possible by Lemma \ref{lemma::faltings_zhang_map}.}

Write $\pr_1: B \rightarrow A$ for the projection to the first coordinate and set $\pr_1^k = \pr_1 \times \cdots \times \pr_1: B^k \rightarrow A^k$ for its $k$-fold product. In addition, we define the metrized line bundle
\begin{equation*}
    \widetilde{M}_{n,k} = \pi_1^\ast \widetilde{L}_n + \cdots + \pi_k^\ast \widetilde{L}_n
\end{equation*}
on $B^k$ and the metrized line bundle
\begin{equation*}
    \Ntil_{k} = \pi_1^\ast \widetilde{L} + \cdots + \pi_k^\ast \widetilde{L}    
\end{equation*}
on $A^k$. (Here, we abuse notation slightly to let $\pi_i$ denote both the $i$-th projection $B^k \rightarrow B$ as well as the $i$-th projection $A^k \rightarrow A$.)

Recall that, by Lemmas \ref{lemma::canonicalmetrics} (ii) and \ref{lemma::metrics_smooth}, the $(1,1)$-form $c_1(\Lo)$ is smooth and semipositive. Hence, also the forms $c_1(\Lo_n)^{\wedge 2}$ and $c_1(\Mo_{n,k-1})^{\wedge 2k}$ are smooth and semipositive. We apply Lemma \ref{lemma::fiberwise} twice. For $X_1 = C_1(\IC)$, $X_2 = C_2(\IC)$, and the $(2,2)$-form $c_1(\Lo_n)^{\wedge 2}$, we find a smooth semipositive $(1,1)$-form $\omega_n$ on $C_1(\IC)$ such that
\begin{equation*}
    \int_{C_1(\IC)} f \omega_n = \int_{S(\IC)} (f \circ \pr_1) c_1(\Lo_{n})^{\wedge 2} 
\end{equation*}
for all $f \in \mathcal{C}^0(C_1(\IC))$. We note that this implies
\begin{equation}\label{equation::boxtimes}
    \int_{C_1^k(\IC)}{f \omega_n^{\boxtimes k}} = \frac{2^{k}}{(2k)!}\int_{S^k(\IC)}{(f \circ \mathrm{pr}_1^k)c_1(\overline{M}_{n,k})^{\wedge 2k}}
\end{equation}
for all $f \in \mathcal{C}^0(C_1^k(\IC))$, where
\begin{equation*}
    \omega_n^{\boxtimes k} = (\pi_1|_{C_1^k})^\ast \omega_n \wedge \cdots \wedge (\pi_k|_{C_1^k})^\ast \omega_n
\end{equation*}
is a $(k,k)$-form on $C_1^k(\IC)$. Applying the lemma for $X_1 = \Delta_{A,k}(C_1^k)(\IC)$, $X_2 = \Delta_{A,k}(C_2^k)(\IC)$, and the $(2k,2k)$-form $c_1(\Mo_{n,k-1})^{\wedge 2k}$, we obtain a smooth semipositive $(k,k)$-form $\eta_{n,k}$ on $\Delta_{A,k}(C_1^k)(\IC)$ such that
\begin{equation}\label{equation::eta_delta}
    \int_{\Delta_{A,k}(C_1^k)(\IC)} g \eta_{n,k} = \int_{\Delta_{B,k}(S^k)(\IC)} (g \circ \pr_1^{k-1}) c_1(\Mo_{n,k-1})^{\wedge 2k}
\end{equation}
for all $g \in \mathcal{C}^0(\Delta_{A,k}(C_1^k)(\IC))$.

We note the homogeneity governing $\omega_n$ and $\eta_{n,k}$.

\begin{lemma}
\label{lemma::homogeneity}
For every integer $n$, we have $\omega_n = n^2 \omega_1$ and $\eta_{n,k} = n^{2k} \eta_{1,k}$. Consequently, 
\begin{equation*}
    L_n^2 \cdot [S] = n^2 \cdot (L_1^2 \cdot [S])
\end{equation*}    
and
\begin{equation*}    
    M_{n,k-1}^{2k} \cdot [\Delta_{B,k}(S^k)] = n^{2k} \cdot (M^{2k}_{1,k-1} \cdot [\Delta_{B,k}(S^k)])
\end{equation*}
for every integer $n$.
\end{lemma}

\begin{proof}
    At every smooth point $p = (x,y) \in S(\IC)$ (resp.\ $p = (x,y) \in \Delta_{B,k}(S^k)(\IC)$), we have
    \begin{equation*}
        T_{\IR,p}^\ast(S(\IC))^{\wedge 4} = T_{\IR,x}^\ast(C_1(\IC))^{\wedge 2} \wedge T_{\IR,y}^\ast(C_2(\IC))^{\wedge 2}
    \end{equation*}
    (resp.\
    \begin{equation*}
        T_{\IR,p}^\ast(\Delta_{B,k}(S^k)(\IC))^{\wedge 4k} = T_{\IR,x}^\ast(\Delta_{A,k}(C_1^{k})(\IC))^{\wedge 2k} \wedge T_{\IR,y}^\ast(\Delta_{A,k}(C_2^{k})(\IC))^{\wedge 2k}),
    \end{equation*}
    where we write $T_{\IR,\cdot}^\ast(\cdot)$ for the real cotangent space. Recall from Lemma \ref{lem:formulaforln} that
    \begin{equation*}
        c_1(L_n) = n \cdot c_1(P) + n^2 \cdot  c_1(\pr_1^\ast L) + c_1(\pr_2^\ast L),    \end{equation*}
    where the Chern currents are computed with respect to the metrics induced by the canonical metric of $L$ at $\mu$ and we have dropped the metrics from the notation for simplicity. Moreover, it is clear that $c_1(\pr_1^\ast L)  \in T_{\IR,x}^\ast(C_1(\IC))^{\wedge 2}$ and $c_1(\pr_2^\ast L) \in T_{\IR,y}^\ast(C_2(\IC))^{\wedge 2}$ while an explicit calculation shows that $c_1(P) \in T_{\IR,x}^\ast(C_1(\IC))^{\wedge 1} \wedge T_{\IR,y}^\ast(C_2(\IC))^{\wedge 1}$ (see \cite[Proof of Lemma IX.2]{Edixhoven1993}). By cancellation, we hence get
    \begin{equation*}
        c_1(L_n)^{\wedge 2} = n^2 \cdot \left( 2c_1( \pr_1^\ast L) \wedge c_1(\pr_2^\ast L) + c_1(P)^{\wedge 2}\right),
    \end{equation*}
    from which $c_1(L_n)^{\wedge 2} = n^2 \cdot c_1(L_1)^{\wedge 2}$ and hence $\omega_n = n^2 \cdot \omega_1$ follows. A similar argument works for $c_1(M_{n,k-1})$. We infer the second assertion by observing that
\begin{equation*}
\int_{C_1(\IC)} \omega_n =
\int_{S(\IC)} c_1(\Lo_{n})^{\wedge 2} = L_n^2 \cdot [S] 
\end{equation*}
and
\begin{equation*}
    \int_{\Delta_{A,k}(C_1^k)(\IC)} \eta_{n,k} 
    = \int_{\Delta_{B,k}(S^k)(\IC)} c_1(\Mo_{n,k-1})^{\wedge 2k}
    = M_{n,k-1}^{2k} \cdot [\Delta_{B,k}(S^k)]
\end{equation*}
by Wirtinger's theorem.
\end{proof}

\begin{lemma}\label{lemma::homogeneity2}
    Let $X_1, X_2 \subseteq A^k$ be two geometrically irreducible algebraic subvarieties of dimension $d_1$ and $d_2$, respectively. Setting $d= d_1 + d_2$, we have
    \begin{equation*}
        M_{n,k}^d \cdot [X_1 \times X_2] = n^{2d_1}M_{1,k}^d \cdot [X_1 \times X_2]
    \end{equation*}
    for every integer $n$.
\end{lemma}

\begin{proof}
This is \cite[Lemma IX.2]{Edixhoven1993} for the ambient abelian variety $A^k$, $\mathcal{L} = N_k$, $m = 2$, $s_1 = n$, $s_2 = 1$, $Y_i = X_i$ ($i = 1,2$), and $\epsilon = 0$ (while $\epsilon$ is assumed to be a positive rational in \cite{Edixhoven1993}, the same proof works for $\epsilon = 0$).
\end{proof}

\begin{lemma}\label{lemma::test_function}
    There exist a constant $\newconstant\label{c3.6} > 0$ and a function $g_0 \in \mathcal{C}^\infty(\Delta_{A,k}(C_1^k)(\IC))$ such that
    \begin{equation*}
        \left\vert
        \frac{1}{(L_n^2\cdot [S])^k} \int_{C_1^k(\IC)} (\Delta_{A,k}^\ast g_0) \omega_n^{\boxtimes k} 
        -
        \frac{1}{M_{n,k-1}^{2k} \cdot [\Delta_{B,k}(S^k)]} \int_{\Delta_{A,k}(C_1^k)(\IC)} g_0 \eta_{n,k}
        \right\vert > \oldconstant{c3.6}
    \end{equation*}
    for every integer $n \neq 0$.
\end{lemma}

\begin{proof}
Let $W \subsetneq C_1^k$ be Zariski closed such that $\Delta_{A,k}$ restricts to an isomorphism $C_1^k \backslash W \to \Delta_{A,k}(C_1^k \backslash W)$.
By Lemma \ref{lemma::homogeneity}, it suffices to consider the case $n=1$. As both $\omega_1^{\boxtimes k}$ and $\eta_{1,k}$ are smooth differential forms and $W$ is nowhere dense in $C_1^k(\IC)$, it suffices to show that 
\begin{equation}
\label{equation::differential_forms_neq}
    \frac{ \omega_1^{\boxtimes k}}{(L_1^2 \cdot [S])^k}(p) \neq \frac{(\Delta_{A,k})^\ast \eta_{1,k}}{M_{1,k-1}^{2k} \cdot [\Delta_{B,k}(S^k)]}(p)
\end{equation}
at some point $p \in C_1^k(\IC)$. Since $L_1^2 \cdot [S] \geq 1$, we can pick a point $(x,y) \in S(\IC)$ such that $c_1(\Lo_1)^{\wedge 2}$ is $>0$ at $(x,y) \in (C_1 \times C_2)(\IC)$. From its definition, we find that $\omega_1$ is $>0$ at $x \in C_1(\IC)$. So is $\omega_1^{\boxtimes k}$ at $(x,x,\dots,x) \in C_1^k(\IC)$. However, $\Delta_{A,k}^\ast \eta_{1,k} = 0$ at this point since the differential $d\Delta_{A,k}$ is zero along diagonal directions. This establishes \eqref{equation::differential_forms_neq} at $(x,x,\dots,x)$.
\end{proof}

Set
\begin{equation*}
    g := g_0 \circ \pr_1^{k-1} \in \mathcal{C}^\infty(\Delta_{B,k}(S^k)(\IC)).
\end{equation*}
We obtain the following inequality that is needed later,
\begin{equation}
\label{equation::two_integrals_difference}
    \left\vert \frac{1}{M_{n,k}^{2k} \cdot [S^k]}\int_{S^k(\IC)} \Delta_{B,k}^\ast g \cdot c_1(\Mo_{n,k})^{\wedge 2k}
    - \frac{1}{M_{n,k-1}^{2k} \cdot [\Delta_{B,k}(S^k)]}\int_{\Delta_{B,k}(S^k)(\IC)} g \cdot c_1(\Mo_{n,k-1})^{\wedge 2k}\right\vert > \oldconstant{c3.6},
\end{equation}
where we have used \eqref{equation::boxtimes}, \eqref{equation::eta_delta}, and the fact that
\[ M_{n,k}^{2k} \cdot [S^k] = \frac{(2k)!}{2^k}(L_n^2 \cdot [S])^k.\]

\begin{lemma}
\label{lemma::smallness}
    Assume that there exist $\varepsilon^\prime > 0$ and $n \in \IZ \backslash \{0\}$ such that
    \begin{equation*}
        \mu_{\mathrm{ess}}(S,\Ltil_n) <\varepsilon^\prime \cdot n^2.
    \end{equation*}
    Then, there exists a Zariski-dense subset $\mathcal{S} = \mathcal{S}(k) \subseteq S^k(\IQbar)$ such that
    \begin{equation*}
        h_{\Mtil_{n,k}}(\underline{z}) <2k \varepsilon^\prime \cdot n^2 
    \end{equation*}
    and
    \begin{equation*}
        h_{\Mtil_{n,k-1}}(\Delta_{B,k}(\underline{z})) < 8k \varepsilon^\prime \cdot n^2 
    \end{equation*}
    for all points $\underline{z} \in \mathcal{S}$.
\end{lemma}

\begin{proof}

By the definition of the essential minimum, we can pick a Zariski-dense subset $\mathcal{S}_0 \subseteq S$ of $\IQbar$-points such that $$h_{\Ltil_n}(z) < 2 \varepsilon^\prime \cdot n^2$$ for all $z \in \mathcal{S}_0$. Setting $\mathcal{S} = \mathcal{S}_0^k$, we have
\begin{equation*}
    h_{\Mtil_{n,k}}(\underline{z}) = h_{\Ltil_n}(z_1) + h_{\Ltil_n}(z_2) + \cdots + h_{\Ltil_n}(z_k) < 2k \varepsilon^\prime \cdot n^2
\end{equation*}
for all $\underline{z}=(z_1,\dots,z_k) \in \mathcal{S}$.

Similarly, we note that $\Delta_{B,k}(\mathcal{S})$ is Zariski-dense in $\Delta_{B,k}(S^k)$. For every point $\underline{z}=(z_1,\dots,z_k) \in \mathcal{S}$, we have
\begin{align*}
    h_{\Mtil_{n,k-1}}(\Delta_{B,k}(\underline{z})) 
    &= 
    h_{\Mtil_{n,k-1}}(z_1 - z_2, \dots, z_{k-1} - z_{k}) \\
    &=
    \sum_{i=1}^{k-1} h_{\Ltil_n}(z_i - z_{i+1}) \\
    &=
    \sum_{i=1}^{k-1} h_{\Ltil}\bigl([n](x_i-x_{i+1})-(y_i-y_{i+1})\bigr) \\
    &=
    \sum_{i=1}^{k-1} h_{\Ltil}\bigl(([n](x_i)-y_i)-([n](x_{i+1})-y_{i+1})\bigr),
\end{align*}
where we write $z_i = (x_i, y_i) \in (C_1 \times C_2)(\IQbar)$. By the quadraticity of the Néron-Tate height, this is at most
\begin{equation*}
    2 h_{\Ltil}([n](x_1)-y_1)
    +4\sum_{i=2}^{k-1}h_{\Ltil}([n](x_i)-y_i)
    +2 h_{\Ltil}([n](x_k)-y_k),
\end{equation*}
which equals
\begin{equation*}
    2 h_{\Ltil_n}(z_1)
    +4\sum_{i=2}^{k-1}h_{\Ltil_n}(z_i)
    +2 h_{\Ltil_n}(z_k).
\end{equation*}
It follows that
\begin{equation*}
   h_{\Mtil_{n,k-1}}(\Delta_{B,k}(\underline{z})) < 8k \varepsilon^\prime \cdot n^2. \qedhere
\end{equation*}
\end{proof}

We continue with establishing a technical lemma that will allow us to bound certain intersection numbers using Lemma \ref{lemma::intersectionnumber_nef}.

\begin{lemma}
	\label{lemma::nef_equidistribution} 
Consider a geometrically irreducible subvariety $X \subseteq B^k$ having dimension $d$ whose projection $X_1 = \pr_1^k(X) \subseteq A^k$ has dimension $d^\prime$. Furthermore, let $h: X_1(\IC) \rightarrow \IR$ be a non-negative smooth function. There exists $q = q(h) \in \mathbb{N}$ such that, for all $\lambda, \lambda^\prime \in \IQ_{> 0}$ with $\lambda^\prime \geq \lambda$ and all $n \in \IZ \backslash \{0\}$, the metrized $\IQ$-line bundle
\begin{equation}\label{equation::nef_line_bundle}
\Mtil_{n,k}|_{X} \otimes \Otil_X(\lambda n^2 \cdot (h \circ \pr_1^k)) \otimes (\pr_1^k|_X)^\ast \Ntil_{k}^{\otimes \lambda^\prime n^2q} \in \widehat{\Pic}(X)_{\IQ}
\end{equation}
is nef.
\end{lemma}

\begin{proof}
By Lemma \ref{lemma::canonicalmetrics} (ii) and the paragraphs following it, there is a sequence of adelic metrics $\{\Vert\cdot \Vert_{l,\nu} \}_{\nu}$ ($l \in \mathbb{N}$) on $L$ converging to the adelic metric of $\Ltil$ and such that each $(L, \{\Vert \cdot \Vert_{l,\nu} \}_{\nu \in \Sigma(\IK)})$ is an algebraically nef metrized line bundle satisfying
\[
c_1(L,\lVert \cdot \rVert_{l,\nu}) \geq \omega_\nu
\] 
for all $\nu \in \Sigma_\infty(\IK)$ and some strictly positive $(1,1)$-form $\omega_\nu$ on $A(\IC_\nu)$ such that $[2]^\ast \omega_\nu = 4\omega_\nu$.

Writing $\Ltil^{(l)} = (L, \{\Vert \cdot \Vert_{l,\nu} \}_\nu )$, we obtain algebraically nef line bundles
\begin{equation*}
    \Mtil^{(l)}_{n,k} = \pi_1^\ast \varphi_n^\ast \Ltil^{(l)} + \cdots + \pi_k^\ast \varphi_n^\ast \Ltil^{(l)}
\end{equation*}
and
\begin{equation*}
    \Ntil^{(l)}_{k} = \pi_1^\ast \Ltil^{(l)} + \cdots + \pi_k^\ast \Ltil^{(l)}
\end{equation*}
with underlying line bundles equal to $M_{n,k}$ and $N_{k}$, respectively. Moreover, we have that
\[c_1(\overline{N}^{(l)}_{k,\nu}) \geq \omega_{k,\nu}\]
for all $l \in \mathbb{N}$ and all $\nu \in \Sigma_\infty(\IK)$, where
\[\omega_{k,\nu} = \sum_{j=1}^{k} \pi_j^\ast \omega_\nu\]
is a strictly positive $(1,1)$-form on $A(\IC_\nu)^k$.

We choose $q \in \mathbb{N}$ such that 
\[
dd^c h + q \cdot (\omega_{k,\mu})|_{X_1}
\] 
is semipositive. After multiplying
\begin{equation*}
\Mtil_{n,k}^{(l)}|_{X} \otimes \Otil_X(\lambda n^2 \cdot (h \circ \pr_1^k)) \otimes (\pr_1^k|_X)^\ast (\Ntil_{k}^{(l)})^{\otimes \lambda^\prime n^2q}
\end{equation*}
by a common denominator $\Lambda$ of $\lambda$ and $\lambda'$, we obtain an algebraically semipositive line bundle because of our choice of $q$ and the fact that $\Mtil_{n,k}^{(l)}$ and $\Ntil^{(l)}_k$ are algebraically semipositive. It follows that, after multiplying \eqref{equation::nef_line_bundle} by $\Lambda$, we obtain a semipositive metrized line bundle.

To show the nefness of \eqref{equation::nef_line_bundle}, it remains to prove that, for all $l$, the arithmetic top self-intersection number of the restriction of
\[ \Mtil^{(l)}_{n,k}|_{X} \otimes \Otil_X(\lambda n^2 \cdot (h \circ \pr_1^k)) \otimes (\pr_1^k|_X)^\ast (\Ntil^{(l)}_{k})^{\otimes \lambda^\prime n^2q}\]
to any irreducible closed subvariety $Y \subseteq X$ is non-negative. For this purpose, assume $Y \subseteq X$ is an irreducible closed subvariety of dimension $d_Y$. Then,
\begin{equation*}
(\Mtil^{(l)}_{n,k}|_X \otimes \Otil_X(\lambda n^2 \cdot (h \circ \pr_1^k)) \otimes (\pr_1^k|_X)^\ast (\Ntil^{(l)}_{k})^{\otimes \lambda^\prime n^2q})^{d_Y+1} \cdot [Y]
\end{equation*}
equals
\begin{align*}
    &\sum_{i=0}^{d_Y +1} \binom{d_Y +1}{i} (\Mtil^{(l)}_{n,k}|_X)^i \cdot (\Otil_X(\lambda n^2 \cdot (h \circ \pr_1^k)) + (\pr_1^k|_X)^\ast (\Ntil^{(l)}_{k})^{\otimes \lambda^\prime n^2q})^{d_Y + 1 - i} \cdot [Y].
\end{align*}
Because of the algebraic identity
\[ (a+b)^{d_Y+1-i} = b^{d_Y+1-i} + \sum_{j=1}^{d_Y+1-i}{ab^{j-1}(a+b)^{d_Y+1-i-j}},\]
this intersection number is a sum of terms of the form
\[(\Mtil^{(l)}_{n,k}|_X)^i \cdot ((\pr^k_1|_X)^\ast (\Ntil^{(l)}_{k})^{\otimes \lambda^\prime n^2q})^{d_Y+1-i} \cdot [Y]\]
and terms of the form
\begin{multline*}
(\Mtil^{(l)}_{n,k}|_X)^i \cdot \Otil_X(\lambda n^2 \cdot (h \circ \pr_1^k)) \cdot ((\pr^k_1|_X)^\ast (\Ntil^{(l)}_{k})^{\otimes \lambda^\prime n^2q})^{j-1} \\ \cdot (\Otil_X(\lambda n^2 \cdot (h \circ \pr_1^k)) + (\pr_1^k|_X)^\ast (\Ntil^{(l)}_{k})^{\otimes \lambda^\prime n^2q})^{d_Y + 1 - i-j} \cdot [Y],
\end{multline*}
which equals $\delta_\mu\lambda n^{2(d_Y + 1 -i)} (\lambda^\prime q)^{j-1}$ times
\[\int_{Y(\IC)} (h \circ \pr_1^k)~c_1(\Mo^{(l)}_{n,k}|_X)^i \wedge c_1((\pr^k_1|_X)^\ast \overline{N}^{(l)}_{k})^{j-1} \wedge (\lambda dd^c(h \circ \pr_1^k) + \lambda^\prime q c_1((\pr_1^k|_X)^\ast \overline{N}^{(l)}_k))^{d_Y + 1 - i - j}\]
by \eqref{equation::intersection_integral}. All these terms are non-negative because of Lemma \ref{lemma::intersectionnumber_nef}, the fact that $\Mtil^{(l)}_{n,k}$ and $\Ntil^{(l)}_k$ are algebraically nef, that $dd^c h + q \cdot c_1(\overline{N}^{(l)}_{k}|_{X_1}) \geq dd^c h + q \cdot (\omega_{k,\mu})|_{X_1} \geq 0$ on $X_1(\IC)$, and that $h \geq 0$.
\end{proof}

The following lemma gives an essential lower volume bound for certain arithmetic line bundles. 

\begin{lemma}
	\label{lemma::volume_yuan} 
Consider a geometrically irreducible subvariety $X \subseteq B^k$ having dimension $d$ whose projection $X_1 = \pr_1^k(X) \subseteq A^k$ has dimension $d^\prime$. Furthermore, let $h: X_1(\IC) \rightarrow \IR$ be a smooth function. Then, for all $\varepsilon \in (0,1] \cap \mathbb{Q}$ and all $n \in \mathbb{Z} \backslash \{0\}$, we have
\begin{equation*}
    \volh_\chi(\Mtil_{n,k}(\varepsilon, h)) - \Mtil_{n,k}(\varepsilon, h)^{d+1} \geq - \newconstant{\label{c3.14}}(X,h) \cdot \varepsilon^2 n^{2(d^\prime+1)}
\end{equation*}
where
\begin{equation*}
    \Mtil_{n,k}(\varepsilon, h) = \Mtil_{n,k}|_{X} \otimes \Otil_X(\varepsilon n^2 \cdot (h \circ \pr_1^k)). 
\end{equation*}
\end{lemma}

The proof of this lemma uses Yuan’s bigness theorem \cite{Yuan2008}.

\begin{proof}
For arbitrary $\sigma \in \IR$, we note that 
\begin{equation*}
\widehat{\vol}_\chi(\Mtil_{n,k}(\varepsilon, h + \sigma)) - \Mtil_{n,k}(\varepsilon, h + \sigma)^{d+1} = \widehat{\vol}_\chi(\Mtil_{n,k}(\varepsilon,h)) - \Mtil_{n,k}(\varepsilon,h)^{d+1}\end{equation*}
by Lemma \ref{lemma::arithmeticvolumes_basics} and \eqref{equation::intersection_Lc}. We may therefore replace $h$ with $h + \sigma$ for a suitable choice of $\sigma$ so that we may assume $h \geq 0$. Lemma \ref{lemma::nef_equidistribution} yields some constant $q = q(h) \in \mathbb{N}$ such that the line bundle 
\[\Mtil_{n,k}(\varepsilon, h) + \varepsilon n^2q \cdot (\pr_1^k|_X)^\ast \Ntil_{k}\]
is nef.

As also
$(\pr_1^k|_X)^\ast \Ntil_{k}$ is nef, we can apply Lemma \ref{lemma::siu} to the decomposition
\begin{equation*}
    (\Mtil_{n,k}(\varepsilon, h) + \varepsilon n^2 q \cdot (\pr_1^k|_X)^\ast \Ntil_{k}) - \varepsilon n^2q \cdot (\pr_1^k|_X)^\ast \Ntil_{k}.
\end{equation*}
We obtain that 
\begin{equation*}
    \widehat{\vol}_\chi(\Mtil_{n,k}(\varepsilon, h) )
\end{equation*}
is bounded from below by
\begin{multline}
\label{equation::arithvol_estimate1}
    (\Mtil_{n,k}(\varepsilon, h) + \varepsilon n^2q \cdot (\pr_1^k|_X)^\ast \Ntil_{k})^{d+1} 
    \\
    - (d+1) \cdot (\Mtil_{n,k}(\varepsilon, h) + \varepsilon n^2q \cdot (\pr_1^k|_X)^\ast \Ntil_{k})^{d} \cdot (\varepsilon n^2 q \cdot (\pr_1^k|_X)^\ast \Ntil_{k}).
\end{multline}
Subtracting $\Mtil_{n,k}(\varepsilon, h) ^{d+1}$ from the above expression results in
\begin{equation}
\label{equation::difference}
- \sum_{i=2}^{d+1} (-1)^i\binom{d+1}{i} (\Mtil_{n,k}(\varepsilon, h) + \varepsilon n^2q \cdot (\pr_1^k|_X)^\ast \Ntil_{k})^{d+1-i} \cdot ( \varepsilon n^2q \cdot(\pr_1^k|_X)^\ast \Ntil_{k})^i;
\end{equation}
this is just the identity
\begin{equation*}
(a+b)^{d+1} - (d+1) (a + b)^{d} b - a^{d+1} =  
-\sum_{i=2}^{d+1} \binom{d+1}{i} (a+b)^{d+1-i} (-b)^i.    
\end{equation*}
To prove the lemma, it hence suffices to prove that, for each $i \in \{ 2,\dots,d+1\}$, the intersection number
\begin{equation}
\label{equation::to_be_majorized_1}
(\Mtil_{n,k}(\varepsilon, h) + \varepsilon n^2q \cdot (\pr_1^k|_X)^\ast \Ntil_{k})^{d+1-i} \cdot ( n^2q \cdot(\pr_1^k|_X)^\ast \Ntil_{k})^i
\end{equation}
is bounded from above in absolute value by $\newconstant{\label{c3.7}} n^{2(d^\prime+1)}$ for some constant $\oldconstant{c3.7} = \oldconstant{c3.7}(X,h,q) > 0$.

As both $\Mtil_{n,k}(\varepsilon, h) + \varepsilon n^2q \cdot (\pr_1^k|_X)^\ast \Ntil_{k}$ and
$(\pr_1^k|_X)^\ast \Ntil_{k}$ are nef, Lemma \ref{lemma::intersectionnumber_nef} implies that the intersection number \eqref{equation::to_be_majorized_1} is non-negative and bounded from above by 
\begin{equation}
\label{equation::to_be_majorized_2}
(\Mtil_{n,k}(\varepsilon, h) + 2 n^2q \cdot (\pr_1^k|_X)^\ast \Ntil_{k})^{d+1} = (\Mtil_{n,k} + 2n^2q \cdot (\pr_1^k|_X)^\ast \Ntil_{k} + \varepsilon n^2 \Otil_X(h \circ \pr_1^k))^{d+1}.
\end{equation}
By Lemmas \ref{lemma::canonicalmetrics} (ii) and \ref{lem:formulaforln}, the metrized line bundle 
\begin{equation*}
    (\pr_1^k)^\ast \Ntil_k^{\otimes 2n^2} \otimes (\pr_2^k)^\ast \Ntil_k^{\otimes 2} \otimes \Mtil_{n,k}^{\otimes -1} \sim \Mtil_{-n,k}
\end{equation*}
is nef (nefness is preserved under homothety). We can hence invoke Lemma \ref{lemma::intersectionnumber_nef} once more to bound the intersection number \eqref{equation::to_be_majorized_2} from above by
\begin{equation*}
   (n^2 \cdot (\pr_1^k|_X)^\ast (2(q+1)\Ntil_k|_{X_1} + \varepsilon \Otil_{X_1}(h)) + 2(\pr_2^k|_{X})^\ast \Ntil_k)^{d+1},
\end{equation*}
which is bounded from above by $\oldconstant{c3.7}n^{2(d^\prime+1)}$ for reasons of dimension. Indeed, it follows from Lemma \ref{lemma::projformulaposdim} that
\begin{equation*}
     (\pr_1^k|_{X})^\ast (2(q+1)\Ntil_k|_{X_1} +  \varepsilon \Otil_{X_1}(h))^{i} \cdot (\pr_2^k|_{X})^\ast (\Ntil_k)^{d+1-i} = 0
\end{equation*}
if $i > d^\prime + 1$.
\end{proof}

\begin{lemma}
	\label{lemma::expansion}
    Let $X$ be a geometrically irreducible subvariety $X \subseteq B^k$ having dimension $d$ whose projection $\pr_1^k(X) \subseteq A^k$ has dimension $d^\prime$. Furthermore, let $h: \pr_1^k(X)(\IC) \rightarrow \IR$ be a smooth function. Then, there exists a constant $\newconstant{\label{c3.9}} = \oldconstant{c3.9}(X,h)$ such that
    \begin{equation*}
         \Mtil_{n,k}(\varepsilon, h)^{d+1} \cdot [X]- \Mtil_{n,k}^{d+1} \cdot [X] - \delta_\mu \varepsilon n^2(d+1) \int_{X(\IC)} (h \circ \pr_1^k) c_1(\Mo_{n,k})^{\wedge d} > -\varepsilon^2 n^{2(d^\prime+1)}\cdot \oldconstant{c3.9} 
    \end{equation*}
    for all $\varepsilon \in (0,1] \cap \IQ$ and all $n \in \IZ \backslash \{0\}$, where $\Mtil_{n,k}(\varepsilon, h)$ is defined as in Lemma \ref{lemma::volume_yuan}.
\end{lemma}

\begin{proof}
    Lemma \ref{lemma::intersectionnumber} (i) and (ii) yield an expansion
    \begin{equation*}
        \Mtil_{n,k}(\varepsilon,h)^{d+1} \cdot [X] = \sum_{i=0}^{d+1} \binom{d+1}{i} \cdot \Mtil_{n,k}^{d+1-i} \cdot \Otil(\varepsilon n^2 \cdot (h \circ \pr_1^k))^{i} \cdot [X].
    \end{equation*}
    As
    \begin{equation*}
        \Mtil_{n,k}^{d} \cdot \Otil(\varepsilon n^2 \cdot (h \circ \pr_1^k)) \cdot [X] = \delta_\mu \varepsilon n^2 \int_{X(\IC)} (h \circ \pr_1^k) c_1(\Mo_{n,k})^{\wedge d}
    \end{equation*}
    by \eqref{equation::intersection_integral}, it suffices to show, again by \eqref{equation::intersection_integral}, that there exists a constant $\newconstant{\label{c3.10}} = \oldconstant{c3.10}(X,h) > 0$ such that
    \begin{equation*}
    \int_{X(\IC)} (h \circ \pr_1^k) (dd^c (h \circ \pr_1^k))^{\wedge (i-1)} \wedge c_1(\Mo_{n,k})^{\wedge (d+1-i)} > -\oldconstant{c3.10} \cdot n^{2(d^\prime - i + 1)}
    \end{equation*}
    for every integer $i \in \{ 2, \dots, d+1\}$. Writing $dd^c h$ as a difference of two semipositive $(1,1)$-forms, both of which are majorized by a semipositive $(1,1)$-form $\omega = \omega(h)$ on $\pr_1^k(X)$, it suffices to show
    \begin{equation*}
        \int_{X(\IC)} (|h| \circ \pr_1^k) (\pr_1^k)^\ast \omega^{\wedge (i - 1)} \wedge c_1(\Mo_{n,k})^{\wedge (d+1-i)}    < \newconstant{\label{c3.11}} \cdot n^{2(d^\prime - i + 1)}
    \end{equation*}
    for some constant $\oldconstant{c3.11} = \oldconstant{c3.11}(X,h,\omega) > 0$.
    
    Using Lemma \ref{lem:formulaforln} and the fact that $\Mtil_{n,k}$ and $\Mtil_{-n,k}$ are semipositive by Lemma \ref{lemma::canonicalmetrics}, we reduce to showing that
    \begin{equation*}
        \int_{X(\IC)} (|h| \circ \pr_1^k) (\pr_1^k)^\ast \omega^{\wedge (i - 1)} \wedge (n^2 \cdot (\pr_1^k)^\ast c_1(\overline{N}_k) + (\pr_2^k)^\ast c_1(\overline{N}_k))^{\wedge (d+1-i)} < \newconstant{\label{c3.12}} \cdot n^{2(d^\prime - i + 1)}
    \end{equation*}
    for some constant $\oldconstant{c3.12} = \oldconstant{c3.12}(X,h,\omega) > 0$. The latter inequality is evident by reasons of dimension.
\end{proof}

The next lemma encapsulates the main equidistribution argument.

\begin{lemma}
\label{lemma::equidistribution}
  Let $X_1, X_2 \subseteq A^k$ be two geometrically irreducible algebraic subvarieties of dimension $d_1$ and $d_2$, respectively. Set $d= d_1 + d_2$ and $X = X_1 \times X_2 \subseteq A^k \times A^k = B^k$ and suppose that the restriction $M_{n,k}|_X$ is big for some $n \in \mathbb{Z} \backslash \{0\}$. Let $h \in \mathcal{C}^\infty(X_1(\IC))$. Then, there exists a constant $\newconstant{\label{c3.13}} = \oldconstant{c3.13}(X_1,X_2,h) > 0$ and a Zariski-closed subset $Z = Z(h,n) \subsetneq X$ such that
\begin{equation*}
        \frac{1}{[\IK(\underline{z}):\IK]}\sum_{w \in \mathbf{O}_\mu(\underline{z})} (h \circ \pr_1^k)(w) - \frac{1}{M_{n,k}^{d} \cdot [X]}\int_{X(\IC)} (h \circ \pr_1^k)c_1(\Mo_{n,k})^{\wedge d} \geq -\frac{\oldconstant{c3.6}}{2}
\end{equation*}
for all $\underline{z} \in (X \setminus Z)(\IQbar)$ with $h_{\Mtil_{n,k}}(\underline{z}) < \oldconstant{c3.13} \cdot n^2$.
\end{lemma}

\begin{proof}
Let us first note that bigness implies that $M_{n,k}^d \cdot [X] > 0$. As in Lemma \ref{lemma::volume_yuan}, we use 
\begin{equation*}
    \Mtil_{n,k}(\varepsilon, h) = \Mtil_{n,k}|_{X} + \Otil_X(\varepsilon n^2 \cdot (h \circ \pr_1^k))
\end{equation*}
for some $\varepsilon \in (0,1] \cap \IQ$. Let $\varepsilon'' > 0$. By Lemma \ref{lemma::minkowski}, there exists some positive integer $m$ and a non-zero section $\mathbf{s} \in H^0(X,M_{n,k}^{\otimes m})$ such that
\begin{equation*}
\frac{\delta_\mu \log \Vert \mathbf{s}(x) \Vert_{\mu}}{m} \leq  - \frac{\volh_\chi(\Mtil_{n,k}(\varepsilon, h))}{(d+1)M_{n,k}^{d} \cdot [X]} + \varepsilon^\dprime 
\end{equation*}
for every point $x \in (X \setminus \Div(\mathbf{s}))(\IC)$ and 
\begin{equation*}
\log \Vert \mathbf{s}(x) \Vert_{\nu} \leq 0    
\end{equation*}
for every place $\nu \in \Sigma(\IK) \setminus \{ \mu \}$ and every point $x \in (X \setminus \Div(\mathbf{s}))(\IC_\nu)$.

By Lemma \ref{lemma::volume_yuan} applied to $\widetilde{M}_{n,k}(\varepsilon,h)$, we have
\begin{equation*}
    \volh_\chi(\Mtil_{n,k}(\varepsilon, h)) - \Mtil_{n,k}(\varepsilon, h)^{d+1} \cdot [X] \geq - \oldconstant{c3.14}(X,h) \cdot \varepsilon^2 n^{2(d_1+1)}
\end{equation*}
and thus
\begin{equation*}
\frac{\delta_\mu \log \Vert \mathbf{s}(x) \Vert_{\mu}}{m} \leq - \frac{ \Mtil_{n,k}(\varepsilon, h)^{d+1} \cdot [X]}{(d+1)M_{n,k}^{d} \cdot [X]} + \oldconstant{c3.14}(X,h)\frac{\varepsilon^2 n^{2(d_1+1)}}{(d+1)M^{d}_{n,k} \cdot [X]} + \varepsilon^\dprime
\end{equation*}
for every point $x \in (X \setminus \Div(\mathbf{s}))(\IC)$. Using Lemma \ref{lemma::expansion}, this inequality becomes
\begin{align*}
\frac{\delta_\mu \log \Vert \mathbf{s}(x) \Vert_{\mu}}{m} \leq &- \frac{ \Mtil_{n,k}^{d+1} \cdot [X]}{(d+1)M_{n,k}^{d} \cdot [X]} - \frac{\delta_\mu \varepsilon n^2}{M^{d}_{n,k} \cdot [X]} \int_{X(\IC)} (h \circ \pr_1^k) c_1(\Mo_{n,k})^{\wedge d}\\ &+ \oldconstant{c3.15}\frac{\varepsilon^2 n^{2(d_1+1)}}{(d+1)M^{d}_{n,k} \cdot [X]} + \varepsilon^\dprime
\end{align*}
for every point $x \in (X \setminus \Div(\mathbf{s}))(\IC)$ and some constant $\newconstant{\label{c3.15}} = \oldconstant{c3.15}(X,h)$. Since $\Mtil_{n,k}^{d+1} \cdot [X] \geq 0$, we can simplify this to
\begin{equation*}
\frac{\delta_\mu \log \Vert \mathbf{s}(x) \Vert_{\mu}}{m} \leq -\frac{\delta_\mu \varepsilon n^2}{M^{d}_{n,k} \cdot [X]} \int_{X(\IC)} (h \circ \pr_1^k) c_1(\Mo_{n,k})^{\wedge d}+ \oldconstant{c3.15}\frac{\varepsilon^2 n^{2(d_1+1)}}{(d+1)M^{d}_{n,k} \cdot [X]} + \varepsilon^\dprime.
\end{equation*}
Using Lemma \ref{lemma::homogeneity2}, we deduce that
\[ \frac{n^{2(d_1+1)}}{M^{d}_{n,k} \cdot [X]} = \frac{n^2}{M^d_{1,k} \cdot [X]}.\]

It follows that there exists a constant $\newconstant\label{c3.16} = \oldconstant{c3.16}(X,h)$ such that the height
\begin{equation*}
    h_{\Mtil_{n,k}(\varepsilon, h)}(\underline{z}) = h_{\Mtil_{n,k}}(\underline{z}) + \frac{\delta_\mu \varepsilon n^2}{[\IK(\underline{z}):\IQ]}\sum_{w \in \mathbf{O}_\mu(\underline{z})} (h \circ \pr_1^k)(w)
\end{equation*}
of any point $\underline{z} \in (X \setminus \Div(\mathbf{s}))(\IQbar)$ is bounded from below by
\begin{equation*}
     \frac{\delta_\mu \varepsilon n^2}{[\IK: \IQ](M^{d}_{n,k} \cdot [X])} \int_{X(\IC)} (h \circ \pr_1^k) c_1(\Mo_{n,k})^{\wedge d} - \oldconstant{c3.16} \cdot \varepsilon^2 n^2 - \varepsilon^\dprime.
\end{equation*}
In summary and choosing $\varepsilon^\dprime = \varepsilon^2$, we get
\begin{multline}
\nonumber
\frac{1}{[\IK(\underline{z}):\IK]}\sum_{w \in \mathbf{O}_\mu(\underline{z})} (h \circ \pr_1^k)(w) - \frac{1}{M^{d}_{n,k} \cdot [X]}\int_{X(\IC)} (h \circ \pr_1^k)c_1(\Mo_{n,k})^{\wedge d} \\
\geq
- \frac{[\IK : \IQ] \cdot h_{\Mtil_{n,k}}(\underline{z})}{\delta_\mu \varepsilon n^2}
- \oldconstant{c317} \cdot \varepsilon
\end{multline}
for some constant $\newconstant{\label{c317}} = \oldconstant{c317}(X,h) \geq \oldconstant{c3.6}/4$. Picking a rational number $\varepsilon \in (\oldconstant{c3.6}/8\oldconstant{c317},\oldconstant{c3.6}/4\oldconstant{c317})$ and setting $Z = \Div(\mathbf{s})$ as well as
\[ 
\oldconstant{c3.13} = \frac{\delta_\mu \oldconstant{c3.6}^2 }{32[\IK:\IQ]\oldconstant{c317}},
\]
we obtain that
\begin{multline*}
\frac{1}{[\IK(\underline{z}):\IK]}\sum_{w \in \mathbf{O}_\mu(\underline{z})} (h \circ \pr_1^k)(w) - \frac{1}{M^{d}_{n,k} \cdot [X]}\int_{X(\IC)} (h \circ \pr_1^k)c_1(\Mo_{n,k})^{\wedge d} \\
\geq
- \frac{[\IK : \IQ] \cdot h_{\Mtil_{n,k}}(\underline{z})}{\delta_\mu \varepsilon n^2}
- \frac{\oldconstant{c3.6}}{4}
\geq -\frac{\oldconstant{c3.6}}{2}
\end{multline*}
for all $\underline{z} \in (X \setminus Z)(\IQbar)$ with $h_{\Mtil_{n,k}}(\underline{z}) < \oldconstant{c3.13} \cdot n^2$.
\end{proof}

With these preparations completed, we can finally prove the main result of this section.

\begin{proof}[Proof of Proposition \ref{proposition::heightbound}]
Invoking Lemma \ref{lemma::trivial_stabilizer_reduction}, we may and will assume that $C_1$ and $C_2$ have trivial stabilizer. The upper bound follows as in the proof of Proposition \ref{proposition::heightbound_symmetric}. For the lower bound, assume to the contrary that 
\begin{equation*}
    \liminf_{|n| \rightarrow \infty} \left(\frac{h_{\Ltil_n}(S)}{n^2} \right) = 0.
\end{equation*}
By Zhang's inequalities \cite[Th\'eor\`eme 1.5]{Chambert-Loir2000}, this implies that for every $\varepsilon^\prime > 0$ and every $N > 0$, there exists an integer $n$ such that $|n| > N$ and
\begin{equation*}
    \mu_{\mathrm{ess}}(S, \Ltil_n) < \varepsilon^\prime \cdot n^2.
\end{equation*}
The goal of our proof is to establish a contradiction if $\varepsilon^\prime$ is chosen sufficiently small.

Let $g_0$ be the function supplied by Lemma \ref{lemma::test_function}. From Lemma \ref{lemma::smallness}, we obtain a Zariski-dense set $\mathcal{S} \subseteq S^k(\IQbar)$ such that
\begin{equation*}
    h_{\Mtil_{n,k}}(\underline{z}) < 2k \varepsilon^\prime \cdot n^2
\end{equation*}
and
\begin{equation*}
    h_{\Mtil_{n,k-1}}(\Delta_{B,k}(\underline{z})) < 8k \varepsilon^\prime \cdot n^2
\end{equation*}
for all points $\underline{z} \in \mathcal{S}$.

 By Lemma \ref{lemma::degreebound}, the line bundle $L_n$ is big on $S$ and hence $M_{n,k}$ is big on $S^k$. Furthermore, the line bundle $M_{n,k-1}$ is big on $\Delta_{B,k}(S^k)$ by Lemma \ref{lemma::bigness_Mprime}. We can hence apply Lemma \ref{lemma::equidistribution} for the varieties $X_1 = C_1^k$ and $X_2 = C_2^k$ and the functions $f_0 = \Delta_{A,k}^\ast g_0$ as well as $-f_0$. In this way, we obtain a proper Zariski closed subset $Z_1 = Z_1(n) \subsetneq S^k$ such that
\begin{equation*}
        \left\vert \frac{1}{[\IK(\underline{z}):\IK]}\sum_{w \in \mathbf{O}_\mu(\underline{z})} (f_0 \circ \pr_1^k)(w) - \frac{1}{M_{n,k}^{2k} \cdot [S^k]}\int_{S^k(\IC)} (f_0 \circ \pr_1^k)c_1(\Mo_{n,k})^{\wedge 2k} \right\vert \leq \frac{\oldconstant{c3.6}}{2}
\end{equation*}
for all $\underline{z} \in (S^{k} \setminus Z_1)(\IQbar)$ with 
\begin{equation*}
    h_{\Mtil_{n,k}}(\underline{z}) < \min\{\oldconstant{c3.13}(f_0),\oldconstant{c3.13}(-f_0)\} \cdot n^2.
\end{equation*}
Similarly, we can apply Lemma \ref{lemma::equidistribution} for the varieties $X_1 = \Delta_{A,k}(C_1^k)$ and $X_2 = \Delta_{A,k}(C_2^k)$ and the functions $g_0$ and $-g_0$ on $\Delta_{A,k}(C_1^k)$. This yields a proper Zariski closed subset $Z_2 = Z_2(n) \subsetneq \Delta_{B,k}(S^k)$ such that
\begin{equation*}
        \left\vert \frac{1}{[\IK(\underline{z}):\IK]}\sum_{w \in \mathbf{O}_\mu(\underline{z})} (g_0 \circ \pr_1^{k-1})(w) - \frac{1}{M_{n,k-1}^{2k} \cdot [\Delta_{B,k}(S^k)]}\int_{\Delta_{B,k}(S^k)(\IC)} (g_0 \circ \pr_1^{k-1})c_1(\Mo_{n,k-1})^{\wedge 2k} \right\vert \leq \frac{\oldconstant{c3.6}}{2}
\end{equation*}
for all $\underline{z} \in (\Delta_{B,k}(S^k) \setminus Z_2)(\IQbar)$ with 
\begin{equation*}
h_{\Mtil_{n,k-1}}(\underline{z}) < \min\{\oldconstant{c3.13}(g_0),\oldconstant{c3.13}(-g_0)\} \cdot n^2.    
\end{equation*}

Recall that $g = g_0 \circ \pr_1^{k-1}$ and note that $\Delta_{B,k}^\ast g = f_0 \circ \pr_1^{k}$. Choosing
\[ \varepsilon^\prime = \frac{\min\{\oldconstant{c3.13}(f_0),\oldconstant{c3.13}(-f_0),\oldconstant{c3.13}(g_0),\oldconstant{c3.13}(-g_0)\}}{8k},\] 
we get
\begin{equation*}
    \left\vert \frac{1}{M_{n,k}^{2k} \cdot [S^k]}\int_{S^k(\IC)} \Delta_{B,k}^\ast g \cdot c_1(\Mo_{n,k})^{\wedge 2k} - \frac{1}{M_{n,k-1}^{2k} \cdot [\Delta_{B,k}(S^k)]}\int_{\Delta_{B,k}(S^k)(\IC)} gc_1(\Mo_{n,k-1})^{\wedge 2k} \right\vert \leq \oldconstant{c3.6}
\end{equation*}
since $\mathcal{S} \not\subseteq Z_1 \cup \Delta_{B,k}^{-1}(Z_2)$. This contradicts \eqref{equation::two_integrals_difference}.
\end{proof}
\section{Proof of Theorem \ref{theorem::main_reduction} for $K=\IQbar$}
\label{section::proof_main_Qbar}

The following theorem is the special arithmetic case of Theorem \ref{theorem::main_reduction}. For its resolution we can use the height-theoretic tools developed in the preceding sections. In Section \ref{section::proof_general}, the general case is deduced from this special case.

\begin{theorem}
    \label{theorem::mainQbar}
		Let	$A$ be an abelian variety of dimension $\geq 3$ defined over $\IQbar$ and let $C_1, C_2 \subseteq A$ be irreducible algebraic curves such that there is no proper algebraic subgroup of $A$ containing both of them. Then the set
\begin{equation*}
        \Xi =\{(x,y) \in (C_1 \times C_2)(\IQbar) \mid  [n](x) = y \text{ for some } n \in \mathbb{Z}
        \}
        \end{equation*}
is not Zariski dense in $C_1 \times C_2$.
\end{theorem}

The remainder of this section comprises the proof of this theorem, which is reduced to a height inequality (Proposition \ref{proposition::21st_century_vojta}) after some reductions.

\subsection{Setup}

As in Theorem \ref{theorem::mainQbar}, let $A$ be an abelian variety of dimension at least $3$ defined over a fixed number field $\mathbb{K} \subseteq \IQbar$. Furthermore, let $C_1, C_2 \subseteq A$ be geometrically irreducible subcurves such that there is no proper algebraic subgroup of $A$ containing both of them. As in Section \ref{section::geometry}, we let $L$ denote a symmetric, very ample, normally generated line bundle on $A$. (We will use that $L$ is normally generated only in the proof of Proposition \ref{proposition::known_bounds}; this can always be ensured by replacing $L$ by $L^{\otimes 4}$, see \cite[Theorem 9]{MumfordCIME}.) We fix an arbitrary isomorphism $\phi: L^{\otimes 4} \rightarrow [2]^\ast L$ and denote by $\widetilde{L}$ the associated metrized line bundle produced by Lemma \ref{lemma::canonicalmetrics}. 

Let $i \in \{1, 2\}$. Since $L$ is very ample, we can fix a projective embedding associated to $L$, choose a hyperplane that does not contain the image of $C_i$, and choose another hyperplane that contains none of the finitely many points of intersection between $C_i$ and the first hyperplane. This yields two global sections of $L$ that have no common zeroes on $C_i$ and hence a non-constant map $\pi_i: C_i \rightarrow \IP^1$ such that $\pi_i^\ast \mathcal{O}(1) \approx L|_{C_i}$. We can and will assume that $L$, $C_1$, $C_2$, $\pi_1$, and $\pi_2$ are all defined over $\IK$.

Applying Lemma \ref{lemma::tangent_field} to $\pi_1$ and $\pi_2$ respectively and the global tangent field $\mathbf{D}_0 = X_0  (\del/\del X_0)$ on $\IP^1$, we obtain non-zero meromorphic tangent fields $\mathbf{D}_1 \in \mathrm{Hom}_{\IQbar(C_1)}(\Omega^1_{\IQbar(C_1)/\IQbar}, \IQbar(C_1))$ and $\mathbf{D}_2 \in \mathrm{Hom}_{\IQbar(C_2)}(\Omega^1_{\IQbar(C_2)/\IQbar}, \IQbar(C_2))$ lifting $\mathbf{D}_0$ along $\pi_1$ and $\pi_2$ respectively. We can and will assume that $\mathbf{D}_1$ and $\mathbf{D}_2$ are base changes of meromorphic tangent fields defined over $\IK$. Note that this coincides with the setting of Subsection \ref{subsection::product_theorem} so that we may apply the product theorem in Subsection \ref{subsection::large_index} below.

\textit{Unless explicitly mentioned otherwise, all implicit and explicit constants in this section depend only on the data introduced in this subsection.}

\subsection{A height inequality}

The following proposition can be understood as a strengthening of Vojta's inequality \cite{Vojta1991} in a special case, explicitly stated as \cite[Inégalité 1]{Remond2005}.

\begin{proposition}
    \label{proposition::21st_century_vojta}
    In the above setting, assume additionally that $C_1$ and $C_2$ are not both contained in a translate of the same proper algebraic subgroup of $A$ and that both $C_1$ and $C_2$ have finite stabilizer. Then,
    there exist constants $\newconstant\label{c1}, \newconstant\label{c2} > 0$ such that
    \begin{equation}
    \label{equation::height_bound_21st_century_vojta}
        h_{\widetilde{L}}([n](x)-y) \geq \oldconstant{c1}  (n^2 h_{\widetilde{L}}(x)+h_{\widetilde{L}}(y))
    \end{equation}
    for all points $(x,y) \in (C_1\times C_2) (\IQbar)$ and all integers $n \in \IZ$ with $|n| \geq \oldconstant{c2}$.
\end{proposition}

\begin{proof}[Proof of Theorem \ref{theorem::mainQbar} using Proposition \ref{proposition::21st_century_vojta}]

By Lemma \ref{lemma::nocoset}, we may and will assume that $C_1$ and $C_2$ are not both contained in a translate of the same proper algebraic subgroup of $A$.

If $C_1$ or $C_2$ does not have finite stabilizer, it follows from Lemma \ref{lemma::nocoset2} that $\Xi$ is not Zariski dense in $C_1 \times C_2$. From now on, we assume that $C_1$ and $C_2$ both have finite stabilizer. We have that $[n](C_1) \not\subseteq C_2$ for all but at most two non-zero integers $n$ by Lemma \ref{lem:inters_one_n} and so for all but at most three integers in total. It therefore suffices to show that $\bigcup_{[n](C_1) \not\subseteq C_2} ([n](C_1) \cap C_2)(\IQbar)$ is finite.

    Suppose that
    \begin{equation}
        \label{equation::vojta_height}
        h_{\widetilde{L}}([n](x) - y) \geq \oldconstant{c1}  (n^2 h_{\widetilde{L}}(x)+h_{\widetilde{L}}(y)) 
    \end{equation}
    for every $(x,y) \in (C_1 \times C_2)(\IQbar)$ and every integer $|n| \geq \oldconstant{c2}$. 

    Let $y \in \bigcup_{[n](C_1) \not\subseteq C_2} ([n](C_1) \cap C_2)(\IQbar)$. The intersection $[n](C_1) \cap C_2$ is finite for every $n \in \IZ$ such that $[n](C_1) \not\subseteq C_2$. Therefore, we may and do assume that $|n| \geq \oldconstant{c2}$ as this excludes at most finitely many points $y \in C_2(\IQbar)$. Let $x \in C_1(\IQbar)$ be such that $[n](x) =y$ for some $|n| \geq \oldconstant{c2}$. Then, the inequality \eqref{equation::vojta_height} implies that $h_{\widetilde{L}}(x) = h_{\widetilde{L}}(y) = 0$. The Manin-Mumford conjecture, which is proven in \cite{Raynaud1983}, implies that there are at most finitely many candidates $y \in C_2(\IQbar)$ since $C_2$ has finite stabilizer. This completes the deduction of Theorem \ref{theorem::mainQbar} from Proposition \ref{proposition::21st_century_vojta}.
\end{proof}

The remainder of this section is devoted to a proof of Proposition 
\ref{proposition::21st_century_vojta}.

\subsection{First observations}\label{subsection::first_observations}

First of all, we observe that it suffices to prove that the inequality  
\eqref{equation::height_bound_21st_century_vojta} holds for all $(x,y)$ outside of a fixed finite subset $Z$ of $(C_1 \times C_2)(\IQbar)$ and all $n \in \IZ$ with $|n| \geq \oldconstant{c2}$. Indeed, if $(x_0,y_0) \in (C_1 \times C_2)(\IQbar)$ is fixed, then either $x_0$ is torsion and the left-hand side of \eqref{equation::height_bound_21st_century_vojta} is equal to $h_{\Ltil}(y_0)$, or $x_0$ is not torsion and the left-hand side of \eqref{equation::height_bound_21st_century_vojta} is a quadratic polynomial in $n$ with positive leading coefficient. Therefore, in the following, we will always allow ourselves to exclude such a finite exceptional subset $Z$.

We next state a proposition that summarizes straightforward consequences of well-known results.

\begin{proposition}\label{proposition::known_bounds}
There exist constants $\newconstant\label{c3}, \newconstant\label{c4}, \newconstant\label{c5}, \newconstant\label{c6} > 0$ such that the following assertion holds true: For any $(x,y) \in (C_1 \times C_2)(\IQbar)$ and $n \in \mathbb{Z}$ with 
\begin{equation} \label{hypothesis::5.4}
    h_{\widetilde{L}}([n](x)-y) < \oldconstant{c3} (n^2 h_{\widetilde{L}}(x)+h_{\widetilde{L}}(y)) \text{ and } |n| \geq \oldconstant{c4},
\end{equation}
 we have
\begin{enumerate}
    \item[(i)] $h_{\Ltil}(x) \leq \oldconstant{c5}$,
    \item[(ii)] $h_{\Ltil}(y) \leq \oldconstant{c5} n^2$, and
    \item[(iii)] $[\mathbb{K}(x,y):\mathbb{K}(y)] \leq \oldconstant{c6}$.
\end{enumerate}
\end{proposition}

\begin{proof}

We start by proving the following claim: if there exists an $\varepsilon \in (0,1/4)$ such that
\begin{equation*}
    h_{\widetilde{L}}([n](x)-y) < \varepsilon  (n^2 h_{\widetilde{L}}(x)+h_{\widetilde{L}}(y)),
\end{equation*}
then
\begin{equation}\label{equation::sandwich}
    \frac{1}{5} n^2h_{\widetilde{L}}(x) \leq h_{\widetilde{L}}(y) \leq 5n^2h_{\widetilde{L}}(x).
\end{equation}
Indeed, we have
\begin{align*}
  n^2 h_{\widetilde{L}}(x) &= h_{\widetilde{L}}([n](x) -y + y)  \\ &\le 2 h_{\widetilde{L}}([n](x) -y ) + 2  h_{\widetilde{L}}(y) \le 2 \varepsilon n^2 h_{\widetilde{L}}(x) + (2 + 2\varepsilon) h_{\widetilde{L}}(y),
\end{align*}
which implies 
$$
\frac{1}{5}n^2 h_{\widetilde{L}}(x) \leq \frac{1-2\varepsilon}{2+2\varepsilon}n^2 h_{\widetilde{L}}(x) \le h_{\widetilde{L}}(y).
$$
Analogously, one obtains 
$$
 h_{\widetilde{L}}(y) \leq \frac{2+2\varepsilon}{1-2\varepsilon} n^2h_{\widetilde{L}}(x) \leq 5 n^2h_{\widetilde{L}}(x),
$$
and we get the claim.

We next prove by contradiction that there exist constants $\oldconstant{c3}$, $\oldconstant{c4}$, and $\oldconstant{c5}$ such that, for all $(x,y) \in (C_1 \times C_2)(\IQbar)$, \eqref{hypothesis::5.4} implies (i). Indeed, if such constants do not exist, then for all real numbers $\oldconstant{c3}$, $\oldconstant{c4}$, and $\oldconstant{c5}$, there exists a point $(x,y) \in (C_1 \times C_2)(\IQbar)$ satisfying \eqref{hypothesis::5.4} and
\[ h_{\widetilde{L}}(x) > \oldconstant{c5}.\]
If we choose $\oldconstant{c3}$ small enough and $\oldconstant{c4}$ large enough, we can furthermore deduce from \eqref{equation::sandwich} that also
\[ h_{\widetilde{L}}(y) > \oldconstant{c5}.\]
We show that such a point violates the generalized Vojta-R\'emond inequality \cite[Theorem 1.1]{D20} if the constants are suitably chosen.

In the notation of \cite{D20}, we choose $m = 2$, $X_1 = C_1$, $X_2 = C_2$, and $\mathcal{X} = X_1 \times X_2$. Moreover, we set $\mathcal{L}_i = L|_{C_i}$ for $i = 1,2$. We further set $a_1 = n^2$, $a_2 = 1$, $\mathcal{P} = \mathcal{L}_1^{\otimes 2n^2} \otimes \mathcal{L}_2^{\otimes 2}$, $t_1 = 2$, $t_2 = 3$, and $\mathcal{M} = (\varphi_{n}^\ast L)|_{C_1 \times C_2}$. We recall from \cite{D20} that $\mathcal{N}_a = \mathcal{L}_1^{\otimes a_1} \otimes \mathcal{L}_2^{\otimes a_2}$ and fix a system $W= (W_0, \dots, W_{N_1})$ of $M = N_1 + 1 = N_2 +1$ homogeneous
coordinates (i.e., pullbacks of the standard coordinates on projective space with respect to some closed embedding) for $L$ which is a basis of $H^0(A,L)$ and, by restriction, yields systems of homogeneous coordinates $W^{(i)} = (W^{(i)}_0, \hdots, W^{(i)}_{N_i})$ for $\mathcal{L}_i$ ($i = 1,2$). Using Segre and Veronese embeddings, we also obtain a system of homogeneous coordinates for $\mathcal{P}$ whose elements are monomials of bidegree $t_1(a_1,a_2)=(2n^2,2)$ in the $W^{(i)}$.

Pulling back the homogeneous coordinates for $L$ along $\varphi_{-n}$, we obtain a family $Z$ of $M$ global sections of $(\varphi_{-n}^\ast L)|_{X_1 \times X_2} \approx \mathcal{P} \otimes \mathcal{M}^{\otimes(-1)}$ generating $\mathcal{P} \otimes \mathcal{M}^{\otimes(-1)}$. The isomorphism
\[ (\mathcal{P} \otimes \mathcal{M}^{\otimes(-1)}) \otimes \mathcal{M} \otimes \mathcal{N}_a \approx \mathcal{N}_a^{\otimes t_2}\]
leads to an injection
\[ (\mathcal{P} \otimes \mathcal{M}^{\otimes(-1)}) \hookrightarrow \mathcal{N}_a^{\otimes t_2}\]
induced by tensoring with $\varphi_{n}^\ast W_{i_0} \otimes (W^{(1)}_{i_1})^{\otimes a_1} \otimes W^{(2)}_{i_2}$ of $\mathcal{M} \otimes \mathcal{N}_a$ for some $i_0, i_1, i_2 \in \{0, \hdots, N_1\}$. We can choose $i_0, i_1, i_2$ so that the section $\varphi_{n}^\ast W_{i_0} \otimes (W^{(1)}_{i_1})^{\otimes a_1} \otimes W^{(2)}_{i_2}$ does not vanish at the point $(x,y)$ and we let $U^0 \subseteq C_1 \times C_2$ denote the complement of the zero locus of that section.

With $\omega = 0$, the necessary lower bound on the intersection number in \cite[Theorem 1.1]{D20} (for a suitable choice of $\theta$) follows from Lemma \ref{lemma::degreebound} if $Y_i = X_i$ ($i = 1,2$). If $Y_1$ or $Y_2$ is zero-dimensional, then the lower bound for the intersection number trivially follows from $[n]^\ast L \approx L^{\otimes n^2}$. The necessary upper bound on the height of the polynomials representing the elements of $Z$ follows from \cite[Proposition 5.2]{Remond2000}.

We note that, for $h_{\mathcal{N}_a}$ and $h_{\mathcal{M}}$ as defined in \cite{D20}, we have that
\[ |h_{\mathcal{N}_a}(x,y) - n^2h_{\Ltil}(x) - h_{\Ltil}(y)| \leq \newconstant\label{cRem}(n^2+1)\]
and
\[ |h_{\mathcal{M}}(x,y) - h_{\Ltil}([n](x)-y)| \leq \oldconstant{cRem}(n^2+1)\]
for some suitable constant $\oldconstant{cRem}$. Therefore, \eqref{hypothesis::5.4} together with
\[h_{\Ltil}(x) > \oldconstant{c5}\]
directly leads to a contradiction with \cite[Theorem 1.1]{D20} if $\oldconstant{c3}$ is chosen small enough and $\oldconstant{c4}$ and $\oldconstant{c5}$ are chosen sufficiently large.

Now, we can deduce (ii) from (i) and \eqref{equation::sandwich} after potentially enlarging $\oldconstant{c5}$.

For (iii), we note that (i), (ii), and \eqref{hypothesis::5.4} already imply
\begin{equation*}
    h_{\widetilde{L}}([n](x)-y) < 2\oldconstant{c3} \oldconstant{c5} n^2.
\end{equation*}
For every $\sigma \in \mathrm{Gal}(\IQbar/\mathbb{K}(y))$, we hence have
\begin{equation*}
    h_{\widetilde{L}}([n](\sigma(x))-y) < 2\oldconstant{c3} \oldconstant{c5} n^2
\end{equation*}
by the Galois invariance of the height. It follows that
\begin{equation*}
    n^2  h_{\widetilde{L}}(\sigma(x)-x)= h_{\widetilde{L}}([n](\sigma(x))-[n](x)) \leq 8\oldconstant{c3} \oldconstant{c5} n^2.
\end{equation*}
Each of the points $\sigma(x)-x$, $\sigma \in \mathrm{Gal}(\IQbar/\mathbb{K}(y))$, is contained in the curve $C_1 - x$ and has height $\leq 8 \oldconstant{c3} \oldconstant{c5}$. By a sufficiently uniform version of the Bogomolov conjecture \cite[Th\'eor\`eme 1.2]{David2002}, there exist constants $\newconstant\label{c8}$ and $\newconstant\label{c9}$, both independent of $x$, such that $C_1-x$ contains at most $\oldconstant{c8}$ points of height $\leq \oldconstant{c9}$. Thus, (iii) follows if we impose $\oldconstant{c3} \leq (8\oldconstant{c5})^{-1} \oldconstant{c9}$, which we may freely do.
\end{proof}

In the next proposition we show that Proposition \ref{proposition::21st_century_vojta} is true if we fix one of the two factors. This will allow us to exclude finitely many points of $C_1$ or $C_2$ as the $x$ or $y$ coordinate respectively.

Note that, in the proposition below, the exceptional sets $Z_2(y_0)$ and $Z_1(x_0)$ are allowed to depend on $y_0 \in C_2(\IQbar)$ and $x_0 \in C_1(\IQbar)$ respectively.

\begin{proposition}
\label{proposition::MordellLangplusepsilon}
    Consider two arbitrary points $x_0\in C_1(\IQbar), y_0 \in C_2(\IQbar)$. Then, there exist positive constants $\newconstant\label{c10}=\oldconstant{c10}(y_0)$, $\newconstant\label{c11}=\oldconstant{c11}(y_0) $, and a finite set $Z_2(y_0) \subseteq C_1(\IQbar)$ such that
    \begin{equation*}
        h_{\widetilde{L}}([n](x)-y_0) \geq \oldconstant{c10} (n^2 h_{\widetilde{L}}(x)+h_{\widetilde{L}}(y_0))
    \end{equation*}
    for all points $x \in C_1(\IQbar) \setminus Z_2(y_0)$ and all integers $n$ satisfying $|n| \geq \oldconstant{c11}(y_0)$. 
    
    Similarly, there exist positive constants $\newconstant\label{c12}=\oldconstant{c12}(x_0)$, $\newconstant\label{c13}=\oldconstant{c13}(x_0)$, and a finite set $Z_1(x_0) \subseteq C_2(\IQbar)$ such that
    \begin{equation*}
        h_{\widetilde{L}}([n](x_0)-y) \geq \oldconstant{c12} (n^2 h_{\widetilde{L}}(x_0)+h_{\widetilde{L}}(y))
    \end{equation*} 
    for all points $y \in C_2(\IQbar) \setminus Z_1(x_0)$ and all integers $n$ satisfying $|n| \geq \oldconstant{c13}(x_0)$. 
\end{proposition}
    
\begin{proof}
    For the first statement, we may assume without loss of generality that $n \neq 0$. Note that
    \begin{equation*}
        h_{\widetilde{L}}([n](x)-y_0) < \oldconstant{c10} (n^2  h_{\widetilde{L}}(x)+h_{\widetilde{L}}(y_0))
    \end{equation*}
    implies that 
    \begin{equation*}
        n^2   h_{\widetilde{L}}(x) =  h_{\widetilde{L}}([n](x)) \leq 2 \oldconstant{c10}  (n^2 h_{\widetilde{L}}(x)+h_{\widetilde{L}}(y_0)) + 2  h_{\widetilde{L}}(y_0),
    \end{equation*}
    which may be rewritten as
    \begin{equation*}
        n^2  (1-2\oldconstant{c10})  h_{\widetilde{L}}(x) \leq 2 (\oldconstant{c10} + 1)  h_{\widetilde{L}}(y_0).
    \end{equation*}
    By the Bogomolov conjecture \cite{Ullmo1998, Zhang1998}, there exist at most finitely many such points $x \in C_1(\IQbar)$ if $\oldconstant{c10}(y_0)$ is sufficiently small and $\oldconstant{c11}(y_0)$ is sufficiently large.
    
    For the second statement, we distinguish two cases according to whether $x_0$ is torsion or not. \textit{Assume first that $x_0$ is torsion.} Then, 
    \begin{equation*}
         h_{\widetilde{L}}([n](x_0)-y) =  h_{\widetilde{L}}(y) = n^2  h_{\widetilde{L}}(x_0)+ h_{\widetilde{L}}(y).
    \end{equation*}
    
    \textit{Assume next that $x_0$ is non-torsion.} Then,
    \begin{equation}
    \label{equation::wanna_cite}
        h_{\widetilde{L}}([n](x_0)-y) < \oldconstant{c12}  (n^2  h_{\widetilde{L}}(x_0)+ h_{\widetilde{L}}(y))
    \end{equation}
    would imply that 
    \begin{align*}
        h_{\widetilde{L}}(y) 
        &\geq
        \frac{1}{2}  h_{\widetilde{L}}([n](x_0)) - h_{\widetilde{L}}(y - [n](x_0)) \\
        &> \left(\frac{1}{2} - \oldconstant{c12}\right) n^2  h_{\widetilde{L}}(x_0) -  \oldconstant{c12}  h_{\widetilde{L}}(y).
    \end{align*}
    Assuming that $\oldconstant{c12}<1/4$, we obtain a lower bound
    \begin{equation*}
        h_{\widetilde{L}}(y) \geq \frac{1-2\oldconstant{c12}}{2+2\oldconstant{c12}}  n^2  h_{\widetilde{L}}(x_0) \geq \frac{n^2}{5}  h_{\widetilde{L}}(x_0).
    \end{equation*}
    It follows that \eqref{equation::wanna_cite} implies that
    \begin{equation*}
        h_{\widetilde{L}}([n](x_0) - y) < 6 \oldconstant{c12}  h_{\widetilde{L}}(y).
    \end{equation*}
    In contrast, \cite[Theorem 1.3]{Zhang2000} states that there exist positive constants $\newconstant\label{c14}=\oldconstant{c14}(x_0)$ and $\newconstant\label{c15}=\oldconstant{c15}(x_0) $ such that 
    \begin{equation*}
        h_{\widetilde{L}}([n](x_0)-y) \leq \oldconstant{c14}  h_{\widetilde{L}}(y) + \oldconstant{c15}
    \end{equation*}
    for at most finitely many $y \in C_2(\IQbar)$. Therefore, it suffices to ensure that $\oldconstant{c12}< \oldconstant{c14}/6$ for finiteness.
\end{proof}

In the following, we write $C_{i,0}$ for the open curve $C_i^{\mathrm{sm}} \setminus |\Div(\mathbf{D}_i)|$ ($i = 1,2$). Note that this means we also remove points where the tangent field $\mathbf{D}_i$ is zero. Proposition \ref{proposition::MordellLangplusepsilon} allows us in particular to restrict ourselves to smooth points $y \in C_{2,0}(\IQbar)$ in the proof of Proposition \ref{proposition::21st_century_vojta}. Invoking Proposition \ref{proposition::MordellLangplusepsilon} once more, we can additionally restrict ourselves to smooth points $x \in C_{1,0}(\IQbar)$. To do these restrictions, it suffices to choose
\[ \oldconstant{c1} \leq \min\{\oldconstant{c10}(y_0),\oldconstant{c12}(x_0)\} \quad \text{and} \quad \oldconstant{c2} \geq \max\{\oldconstant{c11}(y_0),\oldconstant{c13}(x_0)\}\]
and enlarge the finite exceptional set $Z$ by $(Z_2(y_0) \times \{y_0\}) \cup (\{x_0\} \times Z_1(x_0))$ for all $x_0 \in (C_1 \backslash C_{1,0})(\IQbar)$ and all $y_0 \in (C_2 \backslash C_{2,0})(\IQbar)$.

\subsection{Construction of a small section}

We use Lemma \ref{lemma::minkowski} to construct a section with a small supremum norm at a fixed place $\mu \in \Sigma(\mathbb{K})$. Note that $\mu$ can be both archimedean or non-archimedean.

	\begin{lemma}\label{lem::existence_s}
		There exist constants $\newconstant\label{c16}, \newconstant\label{c17} > 0$ such that the following assertion holds true. For every integer $|n| \geq \oldconstant{c17}$, there exist an integer $M = M(n) > 0$ and a section $\mathbf{s}_n \in H^0(C_1 \times C_2, (L_n^{\otimes M})|_{C_1 \times C_2})$ such that
		\begin{equation}
            \label{equation::height_distinguished_place}
			\sup_{p \in C_1(\IC_\mu) \times C_2(\IC_\mu)} \log \Vert \mathbf{s}_n(p) \Vert_{\mu}^{1/M} < - \oldconstant{c16} n^2
		\end{equation}
		and
		\begin{equation}
            \label{equation::height_other_places}
			\sup_{p \in C_1(\IC_\nu) \times C_2(\IC_\nu)} \log \Vert \mathbf{s}_n(p) \Vert_\nu^{1/M} \leq 0
		\end{equation}
		for all places $\nu \in \Sigma(\IK) \setminus \{ \mu \}$. 
	\end{lemma}

	\begin{proof}
		This follows from Lemmas \ref{lemma::minkowski} and \ref{lemma::siu} and Proposition \ref{proposition::heightbound}. Note that the line bundle $(L_n)|_{C_1 \times C_2}$ is big for all $n \in \mathbb{Z} \backslash \{0\}$ because of Lemma \ref{lem:genfinite}.
	\end{proof}

    Using \eqref{equation::height_section}, we infer that
	\begin{equation}
        \label{equation::height_lower_bound}
		h_{\Ltil}([n](x)-y)=h_{\Ltil_n}(p) > \frac{\delta_\mu \oldconstant{c16}}{[\IK:\IQ]}  n^2
	\end{equation}
	for all $p=(x,y)\not\in |\Div(\mathbf{s}_n)|$. 
    
    Now, if \eqref{equation::height_bound_21st_century_vojta} were false, we could apply Proposition \ref{proposition::known_bounds} and get the two bounds (i) and (ii) there. These, in combination with \eqref{equation::height_lower_bound}, give \eqref{equation::height_bound_21st_century_vojta} (for a suitable choice of $\oldconstant{c1}$), which must therefore hold for all $p=(x,y)\not\in |\Div(\mathbf{s}_n)|$ and $n$ large enough.
       
    However, the section $\mathbf{s}_n$ and hence its support $|\Div(\mathbf{s}_n)|$ depend on $n$, so there remain infinitely many points $p=(x,y)$ to be dealt with.

	\subsection{The index}

    To conclude our proof, we are going to make a case distinction for the remaining points, depending on their ``order of vanishing''.
    
    To ease notation,  we write
    \begin{equation*}
        \mathbf{D}^{i_1,i_2}= \frac{1}{i_1! i_2!} (\mathbf{D}_1^{i_1},0) \circ (0,\mathbf{D}_2^{i_2})
    \end{equation*}
    for all integers $i_1,i_2 \geq 0$. For each section $\mathbf{s}_n$ constructed in the previous subsection, we define its index
    \begin{equation*}
		\mathbf{I}(\mathbf{s}_n,p)=\min \left\{ \frac{i_1}{M(n)n^2} + \frac{i_2}{M(n)} \ \middle\vert \ 
        \mathbf{D}^{i_1,i_2}\mathbf{s}_{n}(p) \neq 0 \right\}
	\end{equation*}
    at a point $p=(x,y) \in (C_{1,0}\times C_{2,0})(\IQbar)$ as in Subsection \ref{subsection::indices} and with parameters $(d_1,d_2)=(M(n)n^2,M(n))$. Recall that this definition requires a choice of trivialization, but that the index $\mathbf{I}(\mathbf{s}_n,p)$ itself does not depend on any specific such choice so that it is a well-defined quantity. 

    The following lemma relates lower bounds on the index to more geometric statements. Abusing notation, we write $\mathfrak{m}_{C_1,x}$ and $\mathfrak{m}_{C_2,y}$ also for the extension of the ideals $\mathfrak{m}_{C_1,x}$ and $\mathfrak{m}_{C_2,y}$ to $\mathcal{O}_{C_1\times C_2,p}$ via the natural ring extensions $\mathcal{O}_{C_1,x}\subseteq \mathcal{O}_{C_1\times C_2,p}$ and $\mathcal{O}_{C_2,y} \subseteq \mathcal{O}_{C_1\times C_2,p}$.
    \begin{lemma}
        \label{lemma::index}
        Let $(i_1,i_2) \in \IZ_{\geq 0}^2$ and suppose that the index $\mathbf{I}(\mathbf{s}_n,p)$ is
        \[\geq \frac{i_1}{M(n)n^2} + \frac{i_2}{M(n)}.\] Then the section $\mathbf{s}_{n}$ is contained in \[(\mathfrak{m}_{C_1,x}^{i_1} \cdot \mathfrak{m}_{C_2,y}^{i_2} + \mathfrak{m}_{C_1,x}^{i_1+1} + \mathfrak{m}_{C_2,y}^{i_2+1})  L_{n,p}^{\otimes M(n)}.\]
    \end{lemma}

    \begin{proof}
    We choose a trivialization $\psi: \mathcal{O}_U \rightarrow L_n^{\otimes M(n)}|_U$ on a neighborhood $U \subseteq C_1 \times C_2$ of $p$. Via this trivialization, we can identify the stalk $L_{n,p}^{\otimes M(n)}$ with the local ring $\mathcal{O}_{C_1 \times C_2,p}$. As $p$ is a smooth point of $C_1 \times C_2$, the completion $\widehat{\mathcal{O}}_{C_1 \times C_2,p}$ is a ring of formal power series $\IQbar[[t_1,t_2]]$, where $t_1 \in \mathcal{O}_{C_1,x}$ is a local uniformizer at $x$ and $t_2 \in \mathcal{O}_{C_2,y}$ is a local uniformizer at $y$ (see \cite[\href{https://stacks.math.columbia.edu/tag/0C0S}{Tag 0C0S} and its proof]{stacksProjectAuthors2015}). Moreover, the completions $\widehat{\mathcal{O}}_{C_1,x}$ and $\widehat{\mathcal{O}}_{C_2,y}$ equal the subrings $\IQbar[[t_1]]$ and $\IQbar[[t_2]]$, respectively. 

        The local rings $\mathcal{O}_{C_1,x}$, $\mathcal{O}_{C_2,y}$, and $\mathcal{O}_{C_1\times C_2,p}$ inject into their respective completions. In this way, we can consider the function $f_n = \psi^{-1}(\mathbf{s}_n)$ as a power series $\sum_{j_1,j_2} \gamma_{j_1,j_2} t^{j_1}_1 t^{j_2}_2 \in \IQbar[[t_1,t_2]]$. The tangent fields $\mathbf{D}^{1,0}$ and $\mathbf{D}^{0,1}$ induce derivations on $\widehat{\mathcal{O}}_{C_1 \times C_2,p}$, which we denote with the same symbols. As $x \notin |\Div(\mathbf{D}_1)|$ and $y \notin |\Div(\mathbf{D}_2)|$, there exist functions $g_1 \in \mathcal{O}_{C_1,x}^\times$ and $g_2 \in \mathcal{O}_{C_2,y}^\times$ such that $\mathbf{D}^{1,0} = g_1 \cdot \del/\del t_1$ and $\mathbf{D}^{0,1} = g_2 \cdot \del/\del t_2$ for the induced derivations.

        We prove by induction on $j_1 + j_2$ that the coefficient $\gamma_{j_1,j_2}$ is zero whenever $j_1 \leq i_1$, $j_2 \leq i_2$, and $(j_1,j_2) \neq (i_1,i_2)$. The base case $j_1 + j_2 = 0$ is clear since $\gamma_{0,0} = \mathbf{D}^{0,0} f_n(p) = f_n(p)$. For the inductive step, it follows from the inductive hypothesis and the Leibniz rule that
        \[ \gamma_{j_1,j_2} = g_1(x)^{-j_1}g_2(y)^{-j_2}\mathbf{D}^{j_1,j_2}f_n(p).\]
        This equals zero as
        \[\mathbf{I}(\mathbf{s}_n,p) 
        \geq \frac{i_1}{M(n)n^2} + \frac{i_2}{M(n)}
        > \frac{j_1}{M(n)n^2} + \frac{j_2}{M(n)}.\qedhere\]
    \end{proof}
    
    The remainder of the proof of Proposition \ref{proposition::21st_century_vojta} consists of a case distinction according to the size of the index $\mathbf{I}(\mathbf{s}_n,p)$. The already completed case $p \notin |\Div(\mathbf{s}_n)|$ corresponds to $\mathbf{I}(\mathbf{s}_n,p)=0$. 
    
    Let $\varepsilon > 0$ be a sufficiently small constant to be specified later. In Subsection \ref{subsection::small_index}, we consider points $p \in (C_{1,0} \times C_{2,0})(\IQbar)$ with index $\mathbf{I}(\mathbf{s}_n,p)< \varepsilon$ (``small index''). In general, we obtain \eqref{equation::height_bound_21st_century_vojta} by using a rational section $\mathbf{s}_n^{\prime}$ of $L_n$ that does not vanish at $p$. The section $\mathbf{s}^\prime_n$ is obtained by deriving $\mathbf{s}_n$ in some trivialization of $L_n$. We fix $\varepsilon$ at the end of Subsection \ref{subsection::small_index}.
    
    In Subsection \ref{subsection::large_index}, we consider the situation $\mathbf{I}(\mathbf{s}_n,p)\geq \varepsilon$ (``large index'') and show that there exists a finite set $Z \subseteq (C_{1,0} \times C_{2,0})(\IQbar)$ containing all points $p$ satisfying $\mathbf{I}(\mathbf{s}_n,p)\geq \varepsilon$ for any integer $n$ that is sufficiently large in terms of $\varepsilon$. This will complete the proof of Proposition \ref{proposition::21st_century_vojta}.

	\subsection{Models, coverings, and trivializations}
    \label{subsection::small_index_prelim}
    
    In order to study points $p = (x,y)$ where the section $\mathbf{s}_n$ has small index, we introduce models of $A \times A$ over $\mathcal{O}_\IK$. We start with a proper model $(\mathcal{A},\mathcal{L})$ of $(A,L^{\otimes e})$ over $\caO_\IK$. Recall that any such model induces model metrics $\Vert \cdot \Vert_{\mathcal{L},\nu}$, $\nu \in \Sigma_{f}(\IK)$, on $L$. By Lemma \ref{lemma::canonicalmetrics} (iv), we can choose a specific model $(\mathcal{A},\mathcal{L})$ of $(A,L^{\otimes e})$ for some positive integer $e$ such that
    \begin{equation}\label{equation::model_error_small_nu}
        \left\vert \log \Vert \cdot \Vert_{\Ltil,\nu} - \log \Vert \cdot \Vert_{\mathcal{L},\nu} \right\vert
        \leq \oldconstant{c20}{}_{,\nu}
    \end{equation}    
    for a collection $\newconstant{}_{,\nu}\label{c20}$, $\nu \in \Sigma_f(\IK)$, of constants such that \begin{equation}\label{equation::model_error_small}
        \sum_{\nu \in \Sigma_f(\IK)} \oldconstant{c20}{}_{,\nu} \leq \frac{\oldconstant{c16}}{32[\IK:\IQ]}
    \end{equation}
    and $\oldconstant{c20}{}_{,\nu} = 0$ for all but finitely many places $\nu \in \Sigma_f(\IK)$. Up to replacing $(\mathbf{s}_n,M(n))$ by $(\mathbf{s}_n^{\otimes e},eM(n))$, we can and will assume that $e$ divides $M(n)$.
    
    Our main interest below is in the line bundle
    \begin{equation*}
        L_n \otimes L_{-n} \approx \pr_{1}^{\ast} L^{\otimes 2n^2} \otimes \pr_2^\ast L^{\otimes 2} =: M_n
    \end{equation*}
    on $A \times A$ and its model
    \begin{equation*}
        \pr_{1}^{\ast} \mathcal{L}^{\otimes 2n^2} \otimes \pr_2^\ast \mathcal{L}^{\otimes 2} =: \mathcal{M}_n
    \end{equation*}
    over $\mathcal{A} \times \mathcal{A}$. We write $\lVert \cdot \rVert_{\Mtil_n,\nu}$ for the $\nu$-metric of the metrized line bundle $\pr_1^\ast\Ltil^{\otimes 2n^2} \otimes \pr_2^\ast \Ltil^{\otimes 2}$. It follows from \eqref{equation::model_error_small_nu} that
    \begin{equation*}
        \left\vert \log \Vert \cdot \Vert_{\mathcal{M}_n,\nu} - \log \Vert \cdot \Vert_{\widetilde{M}_n,\nu} \right\vert \leq \newconstant{}_{,\nu}\label{c22}n^2
    \end{equation*}
    with $\oldconstant{c22}{}_{,\nu} = 4\oldconstant{c20}{}_{,\nu}$.

    Note that there is also another way to metrize $M_n$, namely as $\Ltil_n \otimes \Ltil_{-n}$. The two metrized line bundles one obtains in this way are homothetic by Lemma \ref{lem:formulaforln}. Writing $\lVert \cdot \rVert_\nu$ and $\lVert \cdot \rVert_{\Ltil_{-n},\nu}$ for the $\nu$-metrics of $\Ltil_n$ and $\Ltil_{-n}$ respectively, we find that there exists $\gamma_n \in \IK^\times$ and $r_n \in \IZ_{>0}$ such that
    \[ \log \Vert \cdot \Vert_\nu + \log \Vert \cdot \Vert_{\Ltil_{-n},\nu} = \frac{1}{r_n}\log |\gamma_n|_\nu + \log \Vert \cdot \Vert_{\Mtil_{n},\nu}\]
    for all $\nu \in \Sigma(\IK)$. Below, we will also use $\Vert \cdot \Vert_\nu$ to denote the $\nu$-metric of positive tensor powers of $\Ltil_n$.
    
    We choose a covering of $\mathcal{A} \times \mathcal{A}$ by finitely many affine opens $\mathcal{U}_\iota$ indexed by $\iota \in I$. We can furthermore assume that each $\mathcal{U}_\iota$ is of the form $\mathcal{V}_{\iota,1} \times \mathcal{V}_{\iota,2}$ for affine opens $\mathcal{V}_{\iota,1}, \mathcal{V}_{\iota,2} \subseteq \mathcal{A}$ such that $\mathcal{L}$ trivializes on both $\mathcal{V}_{\iota,1}$ and $\mathcal{V}_{\iota,2}$. In this way, we obtain trivializations $\psi_{n,\iota} : \caO_{\mathcal{U}_\iota} \rightarrow \mathcal{M}_{n}|_{\mathcal{U}_\iota}$. Furthermore, we fix closed embeddings $\vartheta_{\iota,\kappa}: \mathcal{V}_{\iota,\kappa} \hookrightarrow \IA_{\mathcal{O}_\IK}^{n_{\iota,\kappa}}$ ($\kappa \in \{1,2\}$) and set    
    \begin{equation*}
        \vartheta_\iota = \vartheta_{\iota,1} \times \vartheta_{\iota,2}: \mathcal{V}_{\iota,1} \times \mathcal{V}_{\iota,2} \hookrightarrow \IA_{\mathcal{O}_\IK}^{n_{\iota,1}} \times \IA_{\mathcal{O}_\IK}^{n_{\iota,2}}
    \end{equation*}
    for every $\iota \in I$. We write $U_\iota$ for the generic fiber of $\mathcal{U}_\iota$ and $V_{\iota,\kappa}$ for the generic fiber of $\mathcal{V}_{\iota,\kappa}$ ($\kappa \in \{1,2\}$). By the definition of the model metrics, we have 
    \begin{equation}
        \label{equation::metric_comparison}
        \left\vert \log \vert \psi_{n,\iota}^{-1}(\mathbf{t})(q) \vert_\nu - \log \Vert \mathbf{t}(q) \Vert_{\widetilde{M}_n^{\otimes e},\nu} \right\vert \leq \oldconstant{c22}{}_{,\nu}en^2
    \end{equation}
    for every non-vanishing section $\mathbf{t}$ of $M_n^{\otimes e}$ over $U_\iota$ and every $q \in U_\iota(\IKbar_\nu)$ whose reduction modulo $\nu$ belongs to $\mathcal{U}_\iota$.

    We need an analog of this estimate for an arbitrary archimedean place $\nu \in \Sigma_\infty(\IK)$ as well. This is delegated to the next lemma. For $n \in \IZ_{> 0}$, $w = (w_1, \hdots, w_n) \in \overline{\IK}_\nu^n$, and $r > 0$, we consider the open polydisc
    \begin{equation*}
        D_{n}(r,w) = \left\{ (z_1,\dots,z_{n})\in \IA^{n}(\overline{\IK}_\nu) \mid \max\{|z_1-w_1|,\dots,|z_{n}-w_n|\} < r \right\}
        \end{equation*}
    as well as the closed polydisc
    \begin{equation*}
        \overline{D}_{n}(r,w) = \left\{ (z_1,\dots,z_{n})\in \IA^{n}(\overline{\IK}_\nu) \mid \max\{|z_1-w_1|,\dots,|z_{n}-w_n|\} \leq r \right\}.
    \end{equation*}
    If $w = (0,0,\dots,0) \in \overline{\IK}_\nu^n$, we also write $D_n(r)$ and $\overline{D}_n(r)$ instead of $D_n(r,0)$ and $\overline{D}_n(r,0)$ respectively.

    \begin{lemma}\label{lemma::trivialization_error}
        For every $\rho > 0$, there exist constants $\newconstant\label{c23}_{,\nu}(\rho)> 0$ and $\newconstant\label{ca}_{,\nu}(\rho)> 0$ ($\nu \in \Sigma_\infty(\IK)$) such that 
        \[ \sum_{\nu \in \Sigma_\infty(\IK)}{\delta_\nu\oldconstant{c23}{}_{,\nu}(\rho)} < \frac{\oldconstant{c16}}{4[\IK:\IQ]}\]
        and, for every index $\iota \in I$ and every $n \in \IZ \backslash \{0\}$, we have
        \begin{equation*}
            \left\vert \log \lambda_{n,\iota,\nu}(q^\prime) - \log \lambda_{n,\iota,\nu}(q) \right\vert \leq \oldconstant{c23}{}_{,\nu}(\rho) n^2
        \end{equation*}
        for all points
        \begin{equation*}
            q = (u,v) \in \vartheta_{\iota}^{-1}(D_{n_1+n_2}(\rho))
        \end{equation*}
        and all points
        \begin{equation*}
            q^\prime \in \vartheta_{\iota}^{-1}(\overline{D}_{n_1}(\oldconstant{ca}{}_{,\nu}(\rho),u) \times \overline{D}_{n_2}(\oldconstant{ca}{}_{,\nu}(\rho),v)),
        \end{equation*}
        where $\lambda_{n,\iota,\nu}$ is locally defined by
        \[\frac{\vert \psi^{-1}_{n,\iota}(\mathbf{t}^{\otimes e}) \vert^{1/e}_{\nu}}{\Vert \mathbf{t} \Vert_{\Mtil_n,\nu}} \]
        for any non-vanishing local section $\mathbf{t}$ of $M_n$ (and $\lambda_{n,\iota,\nu}$ is independent of the choice of $\mathbf{t}$).
    \end{lemma}

    \begin{proof}
        This follows immediately from the continuity of both the metric $\Vert \cdot \Vert_{\widetilde{L},\nu}$ and the trivialization.
    \end{proof}
    
For every $\iota \in I$ and every non-archimedean place $\nu \in \Sigma_f(\IK)$, we set 
\begin{equation*}
U^\circ_{\iota,\nu} = \{q \in U_\iota(\overline{\IK}_\nu) \mid \forall f \in \mathcal{O}_{\mathcal{U}_\iota}(\mathcal{U}_\iota) : |f(q)|_\nu \leq 1\}.
\end{equation*}
Additionally, for an archimedean place $\nu \in \Sigma_\infty(\IK)$, we define
\begin{equation*}
    U^\circ_{\iota,\nu}(\rho) = U_\iota(\overline{\IK}_\nu) \cap \vartheta_{\iota}^{-1}(D_{n_1+n_2}(\rho))
\end{equation*}
for every $\rho > 0$.

\begin{lemma}
\label{lemma::exhaustion_archimedean}
For every place $\nu \in \Sigma_\infty(\IK)$, there exists a constant $\newconstant\label{cc}_{,\nu}$ such that
        \begin{equation*}
            (A \times A)(\overline{\IK}_\nu) = \bigcup_{\iota \in I} U^\circ_{\iota,\nu}(\rho)
        \end{equation*}
        for all $\rho \geq \oldconstant{cc}{}_{,\nu}$.
\end{lemma}

   \begin{proof}
 By construction, we have
        \begin{equation*}
            (A \times A)(\overline{\IK}_\nu) = \bigcup_{\iota \in I} U_{\iota}(\overline{\IK}_\nu) = \bigcup_{\iota \in I} \bigcup_{n = 1}^\infty U^\circ_{\iota,\nu}(n).
        \end{equation*}
        The claim hence follows from the compactness of $(A \times A)(\overline{\IK}_\nu)$.
    \end{proof}

\begin{lemma}
\label{lemma::exhaustion_nonarchimedean}
 For every place $\nu \in \Sigma_f(\IK)$, 
        \begin{equation*}
            (A \times A)(\overline{\IK}_\nu) =\bigcup_{\iota \in I} U^\circ_{\iota,\nu}.
                 \end{equation*}
\end{lemma}

   \begin{proof}
This follows from the valuative criterion of properness as $\mathcal{U}_\iota$, $\iota \in I$, is a covering of $\mathcal{A} \times \mathcal{A}$. 

    \end{proof}
    
    \subsection{Points of small index}
    \label{subsection::small_index}

    Consider a point $p = (x,y) \in (C_{1,0} \times C_{2,0})(\IQbar)$ such that $\mathbf{I}(\mathbf{s}_n,p) < \varepsilon$. By definition, there exists a pair $(k_1,k_2)$ of non-negative integers such that
    \begin{equation*}
        \mathbf{I}(\mathbf{s}_n,p) = \frac{k_1}{M(n)n^2} + \frac{k_2}{M(n)}
    \end{equation*}
    and
    \begin{equation*}
        \mathbf{D}^{k_1,k_2} \mathbf{s}_n (p) \neq 0.
    \end{equation*}
    For reference below, let us remark that
    \begin{equation}
        \label{equation::bound_k1_k2}
        k_1 \leq \varepsilon M(n)n^2 \ \ \text{and} \ \ k_2  \leq \varepsilon M(n).
    \end{equation}
    By Lemma \ref{lemma::index}, we have
    \begin{equation*}
        \mathbf{s}_{n} \in (\mathfrak{m}_{C_1,x}^{k_1}\cdot \mathfrak{m}_{C_2,y}^{k_2} + \mathfrak{m}_{C_1,x}^{k_1+1} + \mathfrak{m}_{C_2,y}^{k_2+1}) \cdot L^{\otimes M(n)}_{n,p}.
    \end{equation*}
   We need to choose a trivialization to obtain a concrete rational section $\mathbf{D}^{k_1,k_2}\mathbf{s}_n$, but the quantities 
    \begin{equation*}
        \Vert \mathbf{D}^{k_1,k_2}\mathbf{s}_n(q) \Vert_{\nu}, \ \nu \in \Sigma(\IK), \ q \in \mathbf{O}_\nu(p),
    \end{equation*}
    do \textit{not} depend on the chosen trivialization by Lemma \ref{lemma::leading_coefficient}. This independence from the chosen trivialization is in fact used below as we allow ourselves to use different trivializations in order to obtain estimates at different places. \textit{Abusing notation, we write $\Vert \mathbf{D}^{k_1,k_2} \mathbf{s}_n(q) \Vert_\nu$ for these quantities.} Using the rational section $\mathbf{D}^{k_1,k_2}\mathbf{s}_n$ to compute the height, we obtain
	\begin{equation}
        \label{equation::heightLn}
		h_{\Ltil_n}(p) = \frac{1}{M(n)} h_{\Ltil_n^{\otimes M(n)}}(p) 
        = - \frac{1}{[\mathbb{K}(p):\IQ]}\sum_{\nu \in \Sigma(\mathbb{K})} \delta_\nu \sum_{q \in \mathbf{O}_\nu(p)} \log \Vert \mathbf{D}^{k_1,k_2} \mathbf{s}_n(q) \Vert_{\nu}^{1/M(n)}.
    \end{equation}

    Recall that the metrized line bundles $\Ltil_n \otimes \Ltil_{-n}$ and $\widetilde{M}_n$ are homothetic by Lemma \ref{lem:formulaforln}. We choose, once and for all, non-zero global sections $\mathbf{t}_1, \mathbf{t}_2,\dots,\mathbf{t}_r$ of $L$ such that $\bigcap_{i=1}^r\Div(\mathbf{t}_i) = \emptyset$. For each $i \in \{1,\dots,r\}$, we obtain global sections 
    \[
    \mathbf{s}_{n,i}^\prime = \mathbf{s}_n \otimes ((\varphi_{-n})|_{C_1 \times C_2})^\ast\mathbf{t}_i^{\otimes M(n)} \in H^0(C_1 \times C_2, \left(M_n^{\otimes M(n)}\right)|_{C_1 \times C_2}).
    \]

    \begin{lemma}\label{lemma::tensor_with_section}
    There exist constants $\newconstant\label{cd}{}_{,\nu}$ ($\nu \in \Sigma(\IK)$), all but finitely many of which are zero, such that the following assertion is true: for every point $q \in \mathbf{O}_\nu(p)$, there exists some $i = i(q) \in \{1,\dots,r\}$ such that 
           \[ \log \lVert \mathbf{D}^{k_1,k_2}\mathbf{s}_n(q) \rVert_{\nu} \leq \log \left\Vert \mathbf{D}^{k_1,k_2} \mathbf{s}_{n,i}^\prime(q) \right\Vert_{\Mtil_n^{\otimes M(n)},\nu} + \frac{M(n)}{r_n}\log | \gamma_n|_\nu + \oldconstant{cd}{}_{,\nu} M(n).\]
    \end{lemma}

    Again, one needs a trivialization to define $\mathbf{D}^{k_1,k_2} \mathbf{s}_{n,i}^\prime$, but $\Vert \mathbf{D}^{k_1,k_2} \mathbf{s}_{n,i}^\prime(q) \Vert_{\Mtil_n^{\otimes M(n)},\nu}$ is independent of this choice by Lemma \ref{lemma::leading_coefficient}.

    \begin{proof}
    By \cite[Corollary 10.3.3]{Lang1983}, there exist constants $\oldconstant{cd}{}_{,\nu} \geq 0$ ($\nu \in \Sigma(\IK)$), of which all but finitely many are zero, such that 
    \begin{equation}\label{equation::estimate_lang}
        -\oldconstant{cd}{}_{,\nu} \leq \max_{1\leq i \leq r} \left\{ \log \Vert \mathbf{t}_i(u) \Vert_{\Ltil,\nu} \right\} \leq \oldconstant{cd}{}_{,\nu}
    \end{equation}
    for all $u \in A(\overline{\IK}_\nu)$.

    For any given point $q = (u,v) \in \mathbf{O}_\nu(p) \subseteq (C_1 \times C_2)(\overline{\IK}_\nu)$, we can therefore choose $i = i(q) \in \{1,\dots,r\}$ such that
    \[ \log \lVert \mathbf{t}_i([-n](u) - v) \rVert_{\Ltil,\nu} \geq -\oldconstant{cd}{}_{,\nu}.\]
    By Lemma \ref{lemma::leading_coefficient_product}, this implies that
    \[ \log \lVert \mathbf{D}^{k_1,k_2}\mathbf{s}_n(q) \rVert_{\nu} \leq \log \left\Vert \mathbf{D}^{k_1,k_2}\mathbf{s}_{n,i}^\prime(q)\right\Vert_{\Mtil_n^{\otimes M(n)},\nu} + \frac{M(n)}{r_n}\log |\gamma_n|_\nu + \oldconstant{cd}{}_{,\nu} M(n)\]
    as desired.
    \end{proof}

    We set
    \[ \newconstant\label{c29}{}_{,\nu} = \begin{cases}
        2\oldconstant{c22}{}_{,\nu} & \text{ if } \nu \in \Sigma_f(\IK),\\
        \oldconstant{c23}{}_{,\nu}(\oldconstant{cc}{}_{,\nu}) &\text{ if } \nu \in \Sigma_\infty(\IK).
    \end{cases}\]
    
    \begin{lemma}
        \label{lemma::estimate_main_proof}
        For every $\iota \in I$, there exist non-zero regular functions $g_{\iota,\kappa}, h_{\iota,\kappa}$ on $V_{\iota,\kappa} \cap C_\kappa \subseteq C_\kappa$ ($\kappa = 1,2$) and constants $\newconstant\label{c24}{}_{,\iota,\nu}$ and $\newconstant\label{c25}{}_{,\iota,\nu}$, of which all but finitely many are zero, such that the following assertion is true for all $\nu \in \Sigma(\IK)$: for every $i \in \{1,\dots,r\}$ and $q = (u,v) \in \mathbf{O}_\nu(p)$ such that $\mathbf{D}^{k_1,k_2}\mathbf{s}^\prime_{n,i}(q) \neq 0$, there exists $\iota = \iota(q,\nu) \in I$ such that $q \in U_\iota(\IKbar_\nu)$ and if
        \[g_{\iota,1}(u)g_{\iota,2}(v)h_{\iota,1}(u)h_{\iota,2}(v) \neq 0,\]
        then the quantity
        \begin{equation*}
            \log \Vert \mathbf{D}^{k_1,k_2} \mathbf{s}_{n,i}^\prime(q) \Vert_{\Mtil_n^{\otimes M(n)},\nu}
        \end{equation*}
        is bounded from above by
        \begin{align*}
            &\oldconstant{c29}{}_{,\nu} M(n) n^2 + \log \Vert \mathbf{s}_n \Vert_{\nu} \\
            &+ k_1 \log |g_{\iota,1}(u)|^{-1}_\nu \nonumber
            + k_2 \log |g_{\iota,2}(v)|^{-1}_\nu \\
            &+ 2k_1 \log^+ |h_{\iota,1}(u)|_\nu^{-1} \nonumber
            + 2k_2 \log^+ |h_{\iota,2}(v)|_\nu^{-1}\\
            &+ \oldconstant{c24}{}_{,\iota,\nu} M(n) + \oldconstant{c25}{}_{,\iota,\nu} (k_1+k_2) - \frac{M(n)}{r_n} \log |\gamma_n|_\nu. \nonumber
    \end{align*} 
    \end{lemma}
    
    \begin{proof}
    Assume first that $\nu \in \Sigma_f(\IK)$. Then, there exists some $\iota(q,\nu) \in I$ such that $q \in U^\circ_{\iota,\nu}$ by Lemma \ref{lemma::exhaustion_nonarchimedean}.  We set $f = \left(\psi_{n,\iota}^{\otimes (M(n)/e)}\right)^{-1}(\mathbf{s}_{n,i}^\prime)$ and we recall from \eqref{equation::metric_comparison} that
    \begin{equation}\label{equation::section_trivialization_comparison}
        \left\vert \log \vert \psi_{n,\iota}^{-1}(\mathbf{t})(q^\prime) \vert_\nu - \log \Vert \mathbf{t}(q^\prime) \Vert_{\widetilde{M}_n^{\otimes e},\nu} \right\vert \leq  \oldconstant{c22}{}_{,\nu}e n^2
    \end{equation}
    for all $q^\prime \in U_{\iota,\nu}^\circ$ and for every non-vanishing section $\mathbf{t}$ of $M_n^{\otimes e}$ over $U_\iota$.

    We obtain that $\log \,\lVert \mathbf{D}^{k_1,k_2}\mathbf{s}^\prime_{n,i}(q)\rVert_{\Mtil_n^{\otimes M(n)},\nu}$ is bounded from above by
    \[ \log |\mathbf{D}^{k_1,k_2} f(q)|_\nu + \oldconstant{c22}{}_{,\nu}M(n) n^2.\]

    In the following, we write
    \begin{equation*}
        \Vert f \Vert_\nu = \sup_{q^\prime \in U^\circ_{\iota,\nu} \cap (C_1 \times C_2)}{|f(q^\prime)|_\nu}.
    \end{equation*}
    It follows from Lemma \ref{lemma::index} that, for the function $f \in \mathcal{O}_{C_1 \times C_2}(U_\iota \cap (C_1 \times C_2))$, we have that
    \begin{equation*}
    f \in \mathfrak{m}_{C_1,x}^{k_1} \cdot \mathfrak{m}_{C_2,y}^{k_2} + \mathfrak{m}_{C_1,x}^{k_1+1} + \mathfrak{m}_{C_2,y}^{k_2+1}
    \end{equation*}
    and we obtain by Galois conjugation over $\IK$ that also
    \begin{equation}\label{equation::function_in_maximal_ideal}
    f \in \mathfrak{m}_{C_1,u}^{k_1} \cdot \mathfrak{m}_{C_2,v}^{k_2} + \mathfrak{m}_{C_1,u}^{k_1+1} + \mathfrak{m}_{C_2,v}^{k_2+1}.
    \end{equation}
    Therefore, we can apply Lemma \ref{lemma::schwarz_non_archimedean} to $f$ to deduce that there exists $\newconstant\label{c30}{}_{,\iota,\nu}$ such that
    \[\log \,\lVert \mathbf{D}^{k_1,k_2}\mathbf{s}^\prime_{n,i}(q)\rVert_{\Mtil_n^{\otimes M(n)},\nu}\]
    is bounded from above by
    \begin{align*}
        &\oldconstant{c22}{}_{,\nu} M(n) n^2 + \log \lVert f \rVert_\nu \\
        &- k_1 \log |g_{\iota,1}(u)|_\nu \nonumber
        - k_2 \log |g_{\iota,2}(v)|_\nu \\
        &- 2k_1 \log |h_{\iota,1}(u)|_\nu \nonumber
        - 2k_2 \log |h_{\iota,2}(v)|_\nu + \oldconstant{c30}{}_{,\iota,\nu}
    \end{align*}
    and moreover $\oldconstant{c30}{}_{,\iota,\nu} = 0$ for all but finitely many $\nu \in \Sigma_f(\IK)$, where $g_{\iota,1}, g_{\iota,2}, h_{\iota,1}, h_{\iota,2}$ are as in Lemma \ref{lemma::introduction_g_and_h} applied to $U_\iota \cap (C_1 \times C_2)$.
    Using \eqref{equation::section_trivialization_comparison} again, we can bound
    \[ \log \lVert f \rVert_\nu \leq \oldconstant{c22}{}_{,\nu} M(n) n^2 + \log \Vert \mathbf{s}^\prime_{n,i} \Vert_{\Mtil_n^{\otimes M(n)},\nu}.\]
    Together with the upper bound
    \begin{align*} \log \lVert \mathbf{s}^\prime_{n,i} \rVert_{\Mtil_n^{\otimes M(n)},\nu} &\leq \log \lVert \mathbf{s}_{n} \rVert_\nu + \log \left\lVert ((\varphi_{-n})|_{C_1 \times C_2})^\ast \mathbf{t}_{i}^{\otimes M(n)} \right\rVert_{\Ltil_{-n}^{\otimes M(n)},\nu} - \frac{M(n)}{r_n} \log |\gamma_n|_\nu\\
    &\leq \log \lVert \mathbf{s}_{n} \rVert_\nu + \oldconstant{cd}{}_{,\nu} M(n) - \frac{M(n)}{r_n} \log |\gamma_n|_\nu,
    \end{align*}
    this proves the lemma for $\nu$ non-archimedean, where $\oldconstant{cd}{}_{,\nu} $ is the constant from \eqref{equation::estimate_lang}, which is $0$ for all but finitely many $\nu \in \Sigma_f(\IK)$.
    
    If $\nu \in \Sigma_\infty(\IK)$, we set $\rho = \oldconstant{cc}{}_{,\nu}$ so that there exists some $\iota(q,\nu) \in I$ such that $q \in U^\circ_{\iota,\nu}(\rho)$ by Lemma \ref{lemma::exhaustion_archimedean}. We consider again $f = \left(\psi_{n,\iota}^{\otimes (M(n)/e)}\right)^{-1}(\mathbf{s}_{n,i}^\prime) \in \mathcal{O}_{C_1\times C_2}(U_\iota \cap (C_1 \times C_2))$. From Lemma \ref{lemma::trivialization_error}, we know that
    \begin{equation}
        \label{equation::archimedean_trivialization}
        M(n)\left\vert \log \lambda_{n,\iota,\nu}(q^\prime) - \log \lambda_{n,\iota,\nu}(q) \right\vert \leq \oldconstant{c23}{}_{,\nu}(\rho) M(n) n^2
    \end{equation}
        for all points
        \begin{equation*}
            q^\prime \in \vartheta_{\iota}^{-1}(\overline{D}_{n_1}(\oldconstant{ca}{}_{,\nu}(\rho),u) \times \overline{D}_{n_2}(\oldconstant{ca}{}_{,\nu}(\rho),v)), \ q = (u,v),
    \end{equation*}
    where $\lambda_{n,\iota,\nu}$ is defined as in Lemma \ref{lemma::trivialization_error}.
    We set
      \begin{equation*}
        U^\prime = (C_1\times C_2)(\overline{\IK}_\nu) \cap \vartheta_{\iota}^{-1}(\overline{D}_{n_1}(\oldconstant{ca}{}_{,\nu}(\rho),u) \times \overline{D}_{n_2}(\oldconstant{ca}{}_{,\nu}(\rho),v)).
    \end{equation*}
    Because of \eqref{equation::function_in_maximal_ideal}, we can apply Lemma \ref{lemma::schwarz_archimedean} to $f$ (with the same $\rho$ as here and $\rho^\prime = \min\{\rho/2,\oldconstant{ca}{}_{,\nu}(\rho)\}$) and obtain that there is some constant $\newconstant\label{c31}{}_{,\iota,\nu}$ such that
    \begin{align*}
     \log |\mathbf{D}^{k_1,k_2}f(q)|_\nu \leq &\sup_{q^\prime \in U^\prime}{\log |f(q^\prime)|_\nu} 
            - k_1 \log |g_{\iota,1}(u)|_\nu
            - k_2 \log |g_{\iota,2}(v)|_\nu \\
            &+ 2k_1 \log^+ |h_{\iota,1}(u)|^{-1}_\nu
            + 2k_2\log^+|h_{\iota,2}(v)|^{-1}_\nu \\
            &+ (k_1+k_2+1) \oldconstant{c31}{}_{,\iota,\nu}.
    \end{align*}
    Using \eqref{equation::archimedean_trivialization}, we obtain 
    \begin{align*}
        \sup_{q^\prime \in U^\prime}{\log |f(q^\prime)|_\nu} &\leq \sup_{q^\prime \in U^\prime} \log \Vert \mathbf{s}^\prime_{n,i}(q^\prime) \Vert_{\Mtil_n^{\otimes M(n)},\nu}  + \sup_{q^\prime \in U^\prime} M(n)\log \lambda_{n,\iota,\nu}(q^\prime) \\
        &\leq \log \lVert \mathbf{s}^\prime_{n,i} \rVert_{\Mtil_n^{\otimes M(n)},\nu} +\oldconstant{c23}{}_{,\nu}(\rho) M(n) n^2 + M(n)\log \lambda_{n,\iota,\nu}(q)\\
        &= \log \lVert \mathbf{s}^\prime_{n,i} \rVert_{\Mtil_n^{\otimes M(n)},\nu} +\oldconstant{c23}{}_{,\nu}(\rho) M(n) n^2 + \log \frac{| \mathbf{D}^{k_1,k_2}f(q)|_\nu}{\Vert \mathbf{D}^{k_1,k_2}\mathbf{s}^\prime_{n,i}(q) \Vert_{\Mtil_n^{\otimes M(n)},\nu}}.
     \end{align*}
    Finally, we can bound again
    \begin{align*} \log \lVert \mathbf{s}^\prime_{n,i} \rVert_{\Mtil_n^{\otimes M(n)},\nu} &\leq \log \lVert \mathbf{s}_{n} \rVert_\nu + \log \lVert ((\varphi_{-n})|_{C_1 \times C_2})^\ast \mathbf{t}_{i}^{\otimes M(n)} \rVert_{\Ltil_{-n}^{\otimes M(n)},\nu} - \frac{M(n)}{r_n} \log |\gamma_n|_\nu\\
    &\leq \log \lVert \mathbf{s}_{n} \rVert_\nu + \newconstant\label{c31.5}{}_{,\nu} M(n) - \frac{M(n)}{r_n} \log |\gamma_n|_\nu.
    \end{align*}
    Combining these inequalities, the inequality of the lemma follows in the archimedean case as well.    \end{proof}

    If $g_{\iota,1}(u)g_{\iota,2}(v)h_{\iota,1}(u)h_{\iota,2}(v) = 0$ for some $q = (u,v) \in \mathbf{O}_\nu(p) \cap U_\iota(\IKbar_\nu)$, then $p$ belongs to a finite set that is independent of $n$ and we are done. Therefore, we assume henceforth that this is not the case.
    
    Combining \eqref{equation::heightLn} with Lemmas \ref{lemma::tensor_with_section} and \ref{lemma::estimate_main_proof}, we obtain that the height $ h_{\widetilde{L}_n}(p)$ is bounded from below by
    \begin{align*}
        &-\frac{1}{[\IK:\IQ]} \left(\sum_{\nu \in \Sigma(\IK)} \delta_\nu\oldconstant{c29}{}_{,\nu} n^2  + \frac{1}{M(n)}\sum_{\nu \in \Sigma(\IK)}\delta_\nu \log \Vert \mathbf{s}_n \Vert_\nu\right) \\
        &- \frac{1}{[\IK(p):\IQ]}\sum_{\nu \in \Sigma(\IK)}\delta_\nu\sum_{q = (u,v) \in \mathbf{O}_\nu(p)}\left(\frac{k_1}{M(n)}\log |g_{\iota(q,\nu),1}(u)|^{-1}_\nu + \frac{k_2}{M(n)} \log |g_{\iota(q,\nu),2}(v)|^{-1}_\nu\right) \\
        &- \frac{1}{[\IK(p):\IQ]} \sum_{\nu \in \Sigma(\IK)} \delta_\nu \sum_{q = (u,v) \in \mathbf{O}_\nu(p)}\left(\frac{2k_1}{M(n)}\log^+ |h_{\iota(q,\nu),1}(u)|_\nu^{-1}
            + \frac{2k_2}{M(n)} \log^+ |h_{\iota(q,\nu),2}(v)|_\nu^{-1}\right)\\
        & - \sum_{\nu \in \Sigma(\IK)} \oldconstant{c26}{}_{,\nu}\left(\frac{k_1+k_2}{M(n)}+1\right)
    \end{align*}
    for suitable constants $\newconstant\label{c26}{}_{,\nu}$, all but finitely many of which are zero.
    
    Using Lemma \ref{lemma::aux_function_control}, the finiteness of $I$, and the upper bounds for the norms of $\mathbf{s}_n$ from Lemma \ref{lem::existence_s}, this can be further bounded from below by 
    \begin{equation*}
        \frac{1}{[\IK:\IQ]} \left(\oldconstant{c16}-\sum_{\nu \in \Sigma(\IK)} \delta_\nu \oldconstant{c29}{}_{,\nu}\right) n^2 - \oldconstant{c27} \left (\frac{k_1}{M(n)} h_{\Ltil}(x) + \frac{k_2}{M(n)}  h_{\Ltil}(y) + \frac{k_1+k_2}{M(n)}+1\right)
    \end{equation*}
    for some constant $\newconstant\label{c27}$.
    It follows from \eqref{equation::model_error_small}, Lemma \ref{lemma::trivialization_error}, and the definition of $\oldconstant{c29}{}_{,\nu}$ that
    \begin{equation*}
        \sum_{\nu \in \Sigma(\IK)} \delta_\nu\oldconstant{c29}{}_{,\nu} \leq \frac{\oldconstant{c16}}{2}
    \end{equation*}
    and hence
    \begin{equation*}
    h_{\widetilde{L}_n}(p) \geq \frac{\oldconstant{c16}}{2[\IK:\IQ]}n^2-\oldconstant{c27} 
            \left(\frac{k_1}{M(n)} h_{\Ltil}(x) + \frac{k_2}{M(n)} h_{\Ltil}(y) + \frac{k_1+k_2}{M(n)}+1
            \right).
    \end{equation*}
    
    Using the upper bounds for $k_1$ and $k_2$ from \eqref{equation::bound_k1_k2}, we deduce that
    \begin{equation*}
        h_{\widetilde{L}_n}(p) \geq \frac{\oldconstant{c16}}{2[\mathbb{K}:\mathbb{Q}]}n^2-\oldconstant{c28}
            \left(\varepsilon(n^2h_{\Ltil}(x) + h_{\Ltil}(y) + n^2) + 1  
            \right)
    \end{equation*}
    for a suitable constant $\newconstant\label{c28} > 0$.
    
    Proposition \ref{proposition::known_bounds} (i) and (ii) imply
    \begin{equation*}
        h_{\widetilde{L}_n}(p) \geq \min\left\{\frac{\oldconstant{c16}}{4[\mathbb{K}:\mathbb{Q}]}n^2,\oldconstant{c3}(n^2h_{\Ltil}(x)+h_{\Ltil}(y))\right\}
    \end{equation*}
    if (as we assume) $n$ is sufficiently large and
    \begin{equation*}
        \varepsilon < \frac{\oldconstant{c16}}{8(\oldconstant{c5}+1)\oldconstant{c28}[\IK:\IQ]}.
    \end{equation*}
    Choosing
    \begin{equation}\label{equation::choice_epsilon}
        \varepsilon = \frac{\oldconstant{c16}}{16(\oldconstant{c5}+1)\oldconstant{c28}[\IK:\IQ]}
    \end{equation}
    and using Proposition \ref{proposition::known_bounds} (i) and (ii) once more, we are done with the proof of Proposition \ref{proposition::21st_century_vojta} in the case of small index.

    \subsection{Points of large index}
    \label{subsection::large_index}	

    The following lemma settles the case of large index (with $\varepsilon$ as chosen in \eqref{equation::choice_epsilon}) and completes the proof of Proposition \ref{proposition::21st_century_vojta} provided that $\oldconstant{c1}$ is chosen sufficiently small and $\oldconstant{c2}$ is chosen sufficiently large (see the first paragraph of Subsection \ref{subsection::first_observations}).

    \begin{lemma}
        For every $\varepsilon > 0$, there exist constants $\newconstant\label{c45} > 0$ and $\newconstant\label{c2eps}(\varepsilon) > 0$ such that the set $Z_\varepsilon$ of points $p = (x,y) \in (C_{1,0} \times C_{2,0})(\IQbar)$ such that 
        \begin{equation*}
            h_{\widetilde{L}}([n](x)-y) < \oldconstant{c45} (n^2 h_{\widetilde{L}}(x)+h_{\widetilde{L}}(y))
        \end{equation*}
        and
        \begin{equation*}
            \mathbf{I}(\mathbf{s}_n,p) \geq \varepsilon
        \end{equation*}
        for some $n \in \IZ$ with $|n| \geq \oldconstant{c2eps}(\varepsilon)$ is finite. 
    \end{lemma}

    \begin{proof}
        We recall that $\pi_i$ has been constructed such that $\pi_i^\ast \caO(1) \approx L|_{C_i}$ ($i = 1,2$). Writing $\mathrm{pr}_i$ for the canonical projection $C_1 \times C_2 \to C_i$ ($i = 1,2$), we moreover note that
        \[ (L_n^{\otimes M} \otimes L_{-n}^{\otimes M})|_{C_1 \times C_2} \approx \mathrm{pr}_1^\ast (L|_{C_1})^{\otimes 2Mn^2} \otimes \mathrm{pr}_2^\ast (L|_{C_2})^{\otimes 2M}.\]
        We choose $\mathbf{t}_i$ such that $\varphi_{-n}^\ast \mathbf{t}_i$ does not vanish at $p$ and obtain that
        \[\mathbf{I}_{(2Mn^2,2M)}(\mathbf{s}^\prime_{n,i},p) = \frac{1}{2}\mathbf{I}_{(Mn^2,M)}(\mathbf{s}^\prime_{n,i},p)  = \frac{1}{2}\mathbf{I}(\mathbf{s}_n,p)\]
        by Lemma \ref{lemma::leading_coefficient_product}.
    
        We invoke Corollary \ref{corollary::product_theorem} with $d_1 = 2Mn^2$ and $d_2 = 2M$ to obtain that $[\IK(x):\IK] \leq \newconstant\label{c18}(\varepsilon)$ or $[\IK(y):\IK] \leq \oldconstant{c18}(\varepsilon)$. By Proposition \ref{proposition::known_bounds} (iii), we can ensure that $[\IK(x):\IK] \leq \newconstant\label{c19}(\varepsilon)$ in either case as long as $\oldconstant{c45}$ is chosen small enough and $\oldconstant{c2eps}(\varepsilon)$ is chosen large enough. As Proposition \ref{proposition::known_bounds} (i) states that $h_{\Ltil}(x) \leq \oldconstant{c5}$, this leaves only finitely many choices for $x$ by Northcott. Applying Proposition \ref{proposition::MordellLangplusepsilon}, we find that there are only finitely many choices $y \in C_{2,0}(\IQbar)$ for any given $x \in C_{1,0}(\IQbar)$, again provided that $\oldconstant{c45}$ is chosen small enough and $\oldconstant{c2eps}(\varepsilon)$ is chosen large enough.
    \end{proof}

\section{Uniform Schwarz estimates}
\label{section::schwarz}

In this section, we derive in detail the uniform Schwarz estimates used in Subsection \ref{subsection::small_index}. We use several notations regarding tuples $$\underline{u} = (u_1,u_2,\dots,u_n) \in \IZ^{n}_{\geq 0}.$$ Namely, we set
    \begin{equation*}
        |\underline{u}| = u_1 + u_2 + \dots + u_n
    \end{equation*}
    and
    \begin{equation*}
        \underline{u}! = u_1! u_2! \cdots u_n!.
    \end{equation*}
    We also write $\underline{0} = (0,0,\dots,0) \in \IZ^n$ where the length $n$ of the tuple is always clear from context. Furthermore, we use the shorthand notation
    \begin{equation*}
        \frac{\del^{|\underline{u}|+|\underline{v}|}f}{\del \underline{X}^{\underline{u}}\del \underline{Y}^{\underline{v}}} \ \ \text{instead of} \ \ \frac{\del^{|\underline{u}|+|\underline{v}|}f}{\del X_1^{u_1}X_2^{u_2}\cdots X_{m}^{u_{m}}Y_1^{v_1}Y_2^{v_2}\cdots Y_{n}^{v_{n}}}
    \end{equation*}
    for all $\underline{u} = (u_1,\dots,u_{m}) \in \IZ_{\geq 0}^m$ and $\underline{v} = (v_1,\dots,v_n) \in \IZ_{\geq 0}^{n}$.

\subsection{Setup}\label{subsection::setup}

Throughout this section, we work with an affine open $U$ inside the product $C_1 \times C_2$ of two projective geometrically irreducible curves $C_1$ and $C_2$ over a number field $\IK \subseteq \IQbar$. Moreover, we assume that $U = V_1 \times V_2$ for affine opens $V_\kappa \subseteq C_\kappa$ ($\kappa = 1,2$). For each $\kappa \in \{1,2\}$, we choose a flat affine $\mathcal{O}_\IK$-scheme $\mathcal{V}_\kappa = \mathrm{Spec}(\mathcal{B}_{\kappa})$ such that $V_\kappa = \mathcal{V}_\kappa \otimes_{\caO_\IK} \IK$ together with a closed embedding $\vartheta_\kappa: \mathcal{V}_\kappa \hookrightarrow \IA^{n_\kappa}_{\caO_{\IK}}$ such that $n_\kappa \geq 2$. We set $\mathcal{U} = \mathcal{V}_1 \times_{\caO_{\IK}} \mathcal{V}_2$ and $\mathcal{A} = \mathcal{B}_{1} \otimes_{\mathcal{O}_{\IK}} \mathcal{B}_{2}$. Note that $\mathcal{U} = \mathrm{Spec} (\mathcal{A})$. Finally, we set $A = \mathcal{A} \otimes_{\caO_\IK} \IK$  and $B_{\kappa} = \mathcal{B}_{\kappa} \otimes_{\caO_\IK} \IK$. As both $\mathcal{A}$ and $\mathcal{B}_{\kappa}$ are flat over $\caO_\IK$, there are natural inclusions $\mathcal{A} \subseteq A$ and $\mathcal{B}_{\kappa} \subseteq B_{\kappa}$. 

To simplify the notation, we let $\kappa \in \{1,2\}$ in this subsection without further mention. The closed embedding
\[ \vartheta = \vartheta_1 \times \vartheta_2: \mathcal{V}_1 \times \mathcal{V}_2 \hookrightarrow \IA_{\mathcal{O}_\IK}^{n_1} \times \IA_{\mathcal{O}_\IK}^{n_2}\]
induces surjective homomorphisms
\begin{equation*}
\pi_{\kappa}: \mathcal{R}_{\kappa} = \mathcal{O}_\IK[X_1,\dots,X_{n_\kappa}] \longrightarrow \mathcal{B}_{\kappa}    
\end{equation*}
with kernels $\mathcal{I}_{\kappa}$. The induced morphism
\begin{equation*}
    \pi=\pi_{1} \otimes \pi_{2} : \mathcal{R} = \caO_\IK[X_1,\dots, X_{n_1},Y_1,\dots,Y_{n_2}] \longrightarrow \mathcal{A}
\end{equation*}
is also surjective. We can regard $V_{\kappa}$ as a subvariety of $\IA^{n_\kappa}_\IK$. 

We also suppose that we are given non-zero meromorphic tangent fields $\mathbf{D}_\kappa$ on $C_\kappa$. On the variety $V_{\kappa}$, the tangent field $\mathbf{D}_\kappa$ induces a derivation
\[d_{\kappa}: \mathcal{B}_{\kappa} \rightarrow \mathrm{Frac}(B_{\kappa}).\]

    \begin{lemma}\label{lemma::introduction_g_and_h}
        In the setting introduced above,
        \begin{enumerate}
            \item[(i)] there exist functions $g_\kappa \in \mathcal{B}_\kappa$, $\kappa \in \{1,2\}$, such that $g_\kappa d_\kappa(\mathcal{B}_\kappa) \subseteq \mathcal{B}_\kappa$, and,
            \item[(ii)] after relabeling coordinates if necessary, there exist functions 
            \begin{equation*}
            g_{\kappa,1}, \dots, g_{\kappa,n_\kappa-1} \in \mathcal{I}_{\kappa}
            \end{equation*}
            such that
        \begin{equation*}
             \mathbf{A}_{\kappa} =
             \begin{pmatrix}
                \frac{\del g_{\kappa,1}}{\del X_1} & \cdots & \frac{\del g_{\kappa,n_\kappa-1}}{\del X_1} \\
                \vdots & & \vdots \\
                \frac{\del g_{\kappa,1}}{\del X_{n_\kappa-1}} & \cdots & \frac{\del g_{\kappa, n_\kappa-1}}{\del X_{n_\kappa-1}}
            \end{pmatrix}
            \in \caO_\IK[X_1,\dots,X_{n_\kappa}]^{(n_\kappa-1) \times (n_\kappa-1)}
        \end{equation*}
        has determinant $$h_{\kappa} = \det(\mathbf{A}_{\kappa}) \neq 0.$$
        \end{enumerate}
    \end{lemma}

    \begin{proof}
    (i) An iterative application of the Leibniz rule yields
    \begin{equation*}
        d_{\kappa} (\mathcal{B}_{\kappa}) \subseteq \mathcal{B}_{\kappa} \cdot d_{\kappa}(\pi_{\kappa}(X_1)) + \cdots + \mathcal{B}_{\kappa} \cdot d_{\kappa}(\pi_{\kappa}(X_{n_\kappa})).
    \end{equation*}
Hence, there exists some non-zero $g_{\kappa} \in \mathcal{B}_{\kappa}$ such that $g_{\kappa} \cdot d_{\kappa}(\mathcal{B}_{\kappa}) \subseteq \mathcal{B}_{\kappa}$.

    (ii) This is a straightforward consequence of generic smoothness \cite[Proposition 17.7.8(ii)]{EGA4IV}.
    \end{proof}

    The homomorphism $\pi_{\kappa}$ induces a $\pi_{\kappa}$-linear map $\psi_{\kappa}: \Omega^1_{\mathcal{R}_{\kappa}/\mathcal{O}_\IK} \rightarrow \Omega^1_{\mathcal{B}_{\kappa}/\mathcal{O}_\IK}$ sending $df$ to $d\pi_{\kappa}(f)$. The $\caO_\IK$-linear derivation $g_{\kappa} \cdot d_{\kappa}: \mathcal{B}_{\kappa} \rightarrow \mathcal{B}_{\kappa}$ corresponds to an $\mathcal{B}_{\kappa}$-linear map $\varphi_{\kappa}: \Omega^1_{\mathcal{B}_{\kappa}/\mathcal{O}_\IK} \rightarrow \mathcal{B}_{\kappa}$. As $\Omega^1_{\mathcal{R}_{\kappa}/\mathcal{O}_\IK}$ is a free $\mathcal{R}_{\kappa}$-module with basis $dX_i$ ($i=1,\dots,n_\kappa$), there exists an $\mathcal{R}_{\kappa}$-linear map $\widetilde{\varphi}_{\kappa}: \Omega^1_{\mathcal{R}_{\kappa}/\mathcal{O}_\IK} \rightarrow \mathcal{R}_{\kappa}$ such that $\pi_{\kappa}(\widetilde{\varphi}_{\kappa}(dX_i))=\varphi_{\kappa}(\psi_{\kappa}(dX_i))$ for all $i \in \{1,\dots,n_\kappa\}$ and hence $\pi_{\kappa} \circ \widetilde{\varphi}_{\kappa} = \varphi_{\kappa} \circ \psi_{\kappa}$. Writing $\delta_{\kappa}: \mathcal{R}_{\kappa} \rightarrow \mathcal{R}_{\kappa}$ for the derivation associated with $\widetilde{\varphi}_{\kappa}$, this implies that $\delta_{\kappa}(\ker(\pi_{\kappa})) \subseteq \ker(\pi_{\kappa})$. In the sequel, we also write
        \begin{equation*}
            \delta_{\kappa} = \sum_{i=1}^{n_\kappa} \delta_{\kappa,i} \cdot \frac{\del}{\del X_i}, \ \delta_{\kappa,i}(X_1,\dots,X_{n_\kappa}) \in \mathcal{R}_{\kappa}.
        \end{equation*}

        \textit{All numbered constants in Section \ref{section::schwarz} are allowed to depend on the data introduced so far without further explicit mention.}

For every $\nu \in \Sigma(\IK)$ and every $\kappa \in \{1,2\}$, we further define $U_{\nu} = U \otimes_{\IK} \IK_\nu$, $V_{\kappa,\nu} = V_{\kappa} \otimes_{\IK} \IK_\nu$, $A_\nu = A \otimes_{\IK} \IK_\nu$, and $B_{\kappa,\nu} = B_{\kappa} \otimes_\IK \IK_\nu$. If $\nu$ is a finite place, we define moreover $\mathcal{B}_{\kappa,\nu} = \mathcal{B}_{\kappa} \otimes_{\caO_\IK} \caO_{\nu}$, $\mathcal{R}_{\nu} = \mathcal{R} \otimes_{\caO_\IK} \caO_{\nu}$, $\mathcal{A}_{\nu} = \mathcal{A} \otimes_{\caO_\IK} \caO_{\nu}$, and
\[\pi_\nu = \pi \otimes_{\caO_\IK}  \IKbar_\nu: \mathcal{R} \otimes_{\caO_\IK} \IKbar_\nu \longrightarrow \mathcal{A} \otimes_{\caO_\IK} \IKbar_\nu,\]
where $\caO_\nu$ denotes the valuation ring of $\IKbar_\nu$. As both $\mathcal{A}_\nu$ and $\mathcal{R}_\nu$ are flat over $\caO_\nu$, there are natural inclusions $\mathcal{A}_\nu \subseteq \mathcal{A} \otimes_{\caO_\IK} \IKbar_\nu$ and $\mathcal{R}_\nu \subseteq \mathcal{R} \otimes_{\caO_\IK} \IKbar_\nu$, which we tacitly use in the following. 
    
\subsection{Extending functions (non-archimedean case)}
    
    We state an elementary proposition on lifting functions along $\pi$ with controlled growth. For $\nu \in \Sigma_f(\IK)$, we set
    \begin{equation*}
        U_{\nu}^\circ = \{ x \in U(\overline{\IK}_\nu) \mid \forall f \in \mathcal{A}: |f(x)|_\nu \leq 1 \}.
    \end{equation*}
    
    \begin{proposition}
        \label{proposition::extension}
        For $h_1$ and $h_2$ as in Lemma \ref{lemma::introduction_g_and_h}, the following assertion holds true: for every non-archimedean place $\nu \in \Sigma_f(\IK)$, every point $p = (x,y) \in U_{\nu}^\circ$ such that $h_{1}(x) \cdot h_{2}(y) \neq 0$, and every 
        \begin{equation*}
            f \in \mathcal{A}_{\nu} \cap (\mathfrak{m}_{C_1,x}^{k_1} \cdot \mathfrak{m}_{C_2,y}^{k_2} + \mathfrak{m}_{C_1,x}^{k_1+1} + \mathfrak{m}_{C_2,y}^{k_2+1}), \ k_1, k_2 \in \IZ_{\geq 0},
        \end{equation*}
        there exists some 
        \begin{equation*} 
            \widetilde{f} \in ( h_{1}(x)^{-2k_1} h_{2}(y)^{-2k_2} \cdot \mathcal{R}_{\nu} ) \cap (\mathfrak{m}_{\IA^{n_1},x}^{k_1}\cdot \mathfrak{m}_{\IA^{n_2},y}^{k_2} + \mathfrak{m}_{\IA^{n_1},x}^{k_1+1} + \mathfrak{m}_{\IA^{n_2},y}^{k_2+1})
        \end{equation*}
        satisfying $\pi_{\nu}(\widetilde{f}) = f$.
    \end{proposition}
    \begin{proof}
        The statement of the proposition holds trivially for $k_1 = k_2 = 0$ as $\pi_\nu$ is surjective. So we always assume that $k_1 + k_2 > 0$.

        Let $g_{\kappa,1},\dots,g_{\kappa,n_\kappa-1} \in \mathcal{I}_{\kappa}$ ($\kappa = 1,2$) be as in Lemma \ref{lemma::introduction_g_and_h}. 
        Recall that, after relabeling coordinates if necessary, $h_{\kappa} = \det(\mathbf{A}_{\kappa})$ where
        \begin{equation*}
             \mathbf{A}_{\kappa} =
             \begin{pmatrix}
                \frac{\del g_{\kappa,1}}{\del X_1} & \cdots & \frac{\del g_{\kappa,n_\kappa-1}}{\del X_1} \\
                \vdots & & \vdots \\
                \frac{\del g_{\kappa,1}}{\del X_{n_\kappa-1}} & \cdots & \frac{\del g_{\kappa, n_\kappa-1}}{\del X_{n_\kappa-1}}
            \end{pmatrix}
            \in \caO_\IK[X_1,\dots,X_{n_\kappa}]^{(n_\kappa-1) \times (n_\kappa-1)}
        \end{equation*}
        has (generically) full rank $n_\kappa-1$ on $C_\kappa$. Note that $h_{1}(x)h_{2}(y) \neq 0$ implies that $x$ is a smooth point of $V_{1,\nu}$ and that $y$ is a smooth point of $V_{2,\nu}$. Hence, the points $p=(x,y)$ under consideration in the proposition are smooth points of $U_{\nu}$. In the sequel, we write 
        \begin{equation*}
            \mathbf{B}_{\kappa} = h_{\kappa}^{-1}\cdot \mathrm{adj}(\mathbf{A_{\kappa}}) \quad (\kappa = 1,2).
        \end{equation*}   
        We set $n = n_1 + n_2$.    
 
        We define the linear polynomials
        \begin{align*}
            L_1 &= X_1 - x_1, & L_1^\prime &= Y_1 - y_1,\\
             &\cdots & &\cdots \\
            L_{n_1-1} &=  X_{n_1-1} - x_{n_1-1}, & L_{n_2-1}^\prime &= Y_{n_2-1} - y_{n_2-1}, \\
            L_{n_1} &= X_{n_1} - x_{n_1}, &L_{n_2}^\prime &= Y_{n_2} - y_{n_2}.
        \end{align*}
        These definitions are such that
        \begin{align}
            \label{equation::eval_L}
            &L_j(x) = 0,\  j \in \{1,\dots,n_1\}, \ \ \\
            &L_j^\prime(y) = 0,\ j \in \{1,\dots,n_2\}, 
        \end{align}
        and
        \begin{equation}
            \label{equation::differential_L}
            \frac{\del L_j}{\del X_i} =
            \begin{cases}
            1 & \text{if $i=j$}, \\
            0 & \text{else},
            \end{cases}
            \ \ \ \ \text{and} \ \ \ \ 
            \frac{\del L_j^\prime}{\del Y_i} =
            \begin{cases}
            1 & \text{if $i=j$}, \\
            0 & \text{else},
            \end{cases}
        \end{equation}
        for all $i,j \in \{ 1, \dots, n_1\}$ or all $i,j \in \{ 1, \dots, n_2\}$, respectively. 
        
        For every $j \in \{1,\dots,n_1-1\}$, we define
        \begin{equation*}
            h_{1,j} = \sum_{k=1}^{n_1-1} (\mathbf{B}_{1})_{k,j}(x) \cdot g_{1,k}(X_1,\dots,X_{n_1}) \in h_{1}(x)^{-1} \cdot \caO_\nu[X_1,\dots,X_{n_1}]
        \end{equation*}
        and, similarly, we set
        \begin{equation*}
            h_{2,j} = \sum_{k=1}^{n_2-1} (\mathbf{B}_{2})_{k,j}(y) \cdot g_{2,k}(Y_1,\dots,Y_{n_2}) \in h_{2}(y)^{-1} \cdot \caO_\nu[Y_1,\dots,Y_{n_2}]
        \end{equation*}        
        for every $j \in \{1,\dots,n_2-1\}$. Note that $h_{1,j} \in \mathcal{I}_{1} \otimes_{\caO_\IK} \IKbar_\nu$, $j \in \{1,\dots,n_1-1\}$, and $h_{2,j} \in \mathcal{I}_2 \otimes_{\caO_\IK} \IKbar_\nu$, $j \in \{1,\dots,n_2-1\}$. Furthermore, 
        \begin{equation}
            \label{equation::differential_delta}
            \frac{\del h_{1,j}}{\del X_i}(x) = 
            \begin{cases}
                1 & \text{if $i=j$,} \\
                0 & \text{else,}
            \end{cases}
            \ \ \ \  \text{and} \ \ \ \
            \frac{\del h_{2,j}}{\del Y_i}(y) = 
            \begin{cases}
                1 & \text{if $i=j$,} \\
                0 & \text{else,}
            \end{cases}
        \end{equation}
        for all $i,j \in \{ 1, \dots, n_1-1\}$ or $i,j \in \{ 1, \dots, n_2-1\}$, respectively. However, neither $\del h_{1,j}/\del X_{n_1}(x)$ nor $\del h_{2,j}/\del Y_{n_2}(y)$ is zero in general.     
        
        For every $\underline{s} = (s_1,\dots,s_{n_1}) \in \IZ_{\geq 0}^{n_1}$, we define $h_{1,\underline{s}} \in h_1(x)^{-1} \cdot \caO_\nu[X_1,\dots,X_{n_1}]$ by
        \begin{equation}
            \label{equation::h_1s}
            h_{1,\underline{s}} = \begin{cases} h_{1,i(\underline{s})} L_{i(\underline{s})}^{s_{i(\underline{s})}-1}
            L_{i(\underline{s})+1}^{s_{i(\underline{s})+1}} \cdots L_{n_1}^{s_{n_1}} & \text{if } s_{n_1}\neq|\underline{s}|, \\
            L_{n_1}^{s_{n_1}} & \text{otherwise,} \end{cases}
        \end{equation}
        where $$i(\underline{s}) = \min \{ i \mid s_i \neq 0 \} < n_1.$$ For every $\underline{s}^\prime = (s_1^\prime,\dots,s_{n_2}^\prime) \in \IZ_{\geq 0}^{n_2}$, we similarly define $h_{2,\underline{s}^\prime} \in h_2(y)^{-1} \cdot \caO_\nu[Y_1,\dots,Y_{n_2}]$ by
        \begin{equation}
            \label{equation::h_2s}
            h_{2,\underline{s}^\prime} = \begin{cases} h_{2,i(\underline{s}^\prime)} (L_{i(\underline{s}^\prime)}^\prime)^{s^\prime_{i(\underline{s}^\prime)}-1}(L_{i(\underline{s}^\prime)+1}^\prime)^{s^\prime_{i(\underline{s}^\prime)+1}} \cdots (L_{n_2}^\prime)^{s_{n_2}^\prime} & \text{if } s_{n_2}^\prime\neq|\underline{s}^\prime|,\\
            (L_{n_2}^\prime)^{s_{n_2}^\prime} & \text{otherwise,}
            \end{cases}
        \end{equation}
        where again $$i(\underline{s}^\prime) = \min \{ i \mid s_i^\prime \neq 0 \} < n_2.$$ Finally, we set
        \begin{equation*}
            h_{\underline{s},\underline{s}^\prime} = h_{1,\underline{s}} h_{2,\underline{s}^\prime} \in h_1(x)^{-1}h_2(y)^{-1} \cdot \caO_\nu[X_1,\dots,X_{n_1},Y_1,\dots,Y_{n_2}].
        \end{equation*}
        We note that, if $s_{n_1}\neq|\underline{s}|$ or $s_{n_2}^\prime\neq|\underline{s}^\prime|$, then $h_{\underline{s},\underline{s}'} \in \ker \pi_\nu$.

        We next note a technical result.
        
        \begin{lemma}
        \label{lemma::technical_result}
        Consider $\underline{m},\underline{s}\in \IZ_{\geq 0}^{n_1}$ and $\underline{m}^\prime,\underline{s}^\prime\in \IZ_{\geq 0}^{n_2}$ such that $s_{n_1} \neq |\underline{s}|$ or $s^\prime_{n_2} \neq |\underline{s}^\prime|$. If 
        \begin{equation*}
            \frac{\del^{|\underline{m}|+|\underline{m}^\prime|}h_{\underline{s},\underline{s}^\prime}}{\del \underline{X}^{\underline{m}} \del \underline{Y}^{\underline{m}^\prime}}(p) \neq 0,
        \end{equation*}
        then
        \begin{equation*}
            \underline{s} = \underline{m}, \ (|\underline{s}| = |\underline{m}| \ \text{and} \ 
                 s_{n_1} < m_{n_1}), \ \text{or} \ 
            (|\underline{s}| < |\underline{m}| \ \text{and} \ 
                 s_{n_1} \leq m_{n_1}),
        \end{equation*}
        and, similarly,
        \begin{equation*}
            \underline{s}^\prime = \underline{m}^\prime, \ (|\underline{s}^\prime| = |\underline{m}^\prime| \ \text{and} \ 
                 s_{n_2}^\prime < m_{n_2}^\prime), \ \text{or} \ 
            (|\underline{s}^\prime| < |\underline{m}^\prime| \ \text{and} \ 
                 s_{n_2}^\prime \leq m_{n_2}^\prime).
        \end{equation*}
        In particular, if moreover $(\underline{s},\underline{s}^\prime) \neq (\underline{m},\underline{m}^\prime)$, then
        \begin{equation*}
        |\underline{s}| + s_{n_1} + |\underline{s}^\prime| + s^\prime_{n_2} <  |\underline{m}| + m_{n_1} + |\underline{m}^\prime| + m^\prime_{n_2}.
        \end{equation*}
        \end{lemma}

        \begin{proof}
            As
            \begin{equation*}
                \frac{\del^{|\underline{m}|+|\underline{m}^\prime|}h_{\underline{s},\underline{s}^\prime}}{\del \underline{X}^{\underline{m}} \del \underline{Y}^{\underline{m}^\prime}}(p)
                =
                \frac{\del^{|\underline{m}|}h_{1,\underline{s}}}{\del \underline{X}^{\underline{m}}}(x) \cdot\frac{\del^{|\underline{m}^\prime|}h_{2,\underline{s}^\prime}}{\del \underline{Y}^{\underline{m}^\prime}}(y), 
            \end{equation*}
            it suffices by symmetry to prove that
            \begin{equation*}
                \frac{\del^{|\underline{m}|}h_{1,\underline{s}}}{\del \underline{X}^{\underline{m}}}(x) \neq 0
            \end{equation*}
            implies
        \begin{equation*}
            \underline{s} = \underline{m}, \ \ (|\underline{s}| = |\underline{m}| \ \ \text{and} \ \ 
                 s_{n_1} < m_{n_1}), \ \ \text{or} \ \ (|\underline{s}|<|\underline{m}| \ \ \text{and} \ \ s_{n_1} \leq m_{n_1}).
        \end{equation*}

        If $|\underline{s}| = s_{n_1}$, then $h_{1,\underline{s}} = L_{n_1}^{s_{n_1}}$ and the implication holds. From now on, we assume that $|\underline{s}| \neq s_{n_1}$.
        
        Using \eqref{equation::eval_L} and \eqref{equation::differential_L}, we derive that the non-vanishing of
        \begin{equation}
            \label{equation::non_vanishing_diff}
            \frac{\del^{|\underline{m}|} h_{1,\underline{s}}}{\del \underline{X}^{\underline{m}}}(x) =
            \frac{\del^{|\underline{m}|}}{\del \underline{X}^{\underline{m}}}\left( h_{1,i(\underline{s})} L_{i(\underline{s})}^{s_{i(\underline{s})}-1}
            L_{i(\underline{s})+1}^{s_{i(\underline{s})+1}} \cdots L_{n_1}^{s_{n_1}}\right)(x)
        \end{equation}
        implies
        \begin{equation*}
            s_{n_1} \leq m_{n_1}, \ s_{n_1-1} \leq m_{n_1-1}, \ \dots, \ s_{i(\underline{s})} -1 \leq m_{i(\underline{s})}.
        \end{equation*}
        As $h_{1,i(\underline{s})}(x) = 0$, the non-vanishing of \eqref{equation::non_vanishing_diff} implies by \eqref{equation::differential_delta} furthermore 
        \begin{equation*}
            s_{i(\underline{s})} \leq m_{i(\underline{s})} \ \ \text{or} \ \ s_{i} + 1 \leq m_{i} \ \ \text{for some} \ \ i \neq i(\underline{s}).
        \end{equation*}
        We deduce that $|\underline{s}| \leq |\underline{m}|$. If $|\underline{s}| < |\underline{m}|$, we are done, so we assume that $|\underline{s}| = |\underline{m}|$ in the following. If $s_{i(\underline{s})} \leq m_{i(\underline{s})}$, then $\underline{s} = \underline{m}$ and we are done. Similarly, if $s_{n_1} + 1 \leq m_{n_1}$, then $s_{n_1} < m_{n_1}$ and we are also done.

        It remains to treat the case where $s_i+1 \leq m_i$ for some $i \notin \{ i(\underline{s}), n_1 \}$. But then we must have $m_{i(\underline{s})} = s_{i(\underline{s})} -1$, $m_i = s_i + 1$, and $m_{j} = s_{j}$ for all $j \notin \{ i, i(\underline{s})\}$. As $h_{1,i(\underline{s})}(x) = 0$ and
       \begin{equation*}
           \frac{\del h_{1,i(\underline{s})}}{\del X_i}(x) = 0
        \end{equation*}
     by \eqref{equation::differential_delta},
        we deduce that
          \begin{equation*}
                \frac{\del^{|\underline{m}|}}{\del \underline{X}^{\underline{m}}}\left( h_{1,i(\underline{s})}  L_{i(\underline{s})}^{s_{i(\underline{s})}-1}
            L_{i(\underline{s})+1}^{s_{i(\underline{s})+1}} \cdots L_{n_1}^{s_{n_1}}\right)(x) = 0,
            \end{equation*}
        a contradiction. So this case cannot occur and we are done.
        \end{proof}

        Let us write
        \begin{equation*}
            \mathcal{D} = \{ (\underline{m},\underline{m}^\prime) \in \IZ^{n_1}_{\geq 0} \times \IZ^{n_2}_{\geq 0} \mid \text{$m_{n_1} <|\underline{m}|$ or $m_{n_2}^\prime < |\underline{m}^\prime|$}\}
        \end{equation*}
        and
        \[ \mathcal{D}_{k_1,k_2} = \{ (\underline{s},\underline{s}^\prime)\in \mathcal{D} \mid |\underline{s}| \leq k_1 \text{ and } |\underline{s}^\prime| \leq k_2\}.\]
        Set $l_0 = |\mathcal{D}_{k_1,k_2}|-1$, which is $\geq 0$ since $k_1 + k_2 > 0$, $n_1 \geq 2$, and $n_2 \geq 2$, and choose any bijection
        \begin{equation*}
            \varphi: \mathcal{D}_{k_1,k_2} \longrightarrow \{0,\hdots,l_0\}
        \end{equation*}
        such that 
        \begin{equation*}
        \varphi(\underline{s},\underline{s}^\prime) \leq \varphi(\underline{m},\underline{m}^\prime)    
        \end{equation*}
        implies 
        \begin{equation*}
        |\underline{s}| + s_{n_1} + |\underline{s}^\prime| + s^\prime_{n_2} \leq  |\underline{m}| + m_{n_1} + |\underline{m}^\prime| + m^\prime_{n_2}.
        \end{equation*}
       We write $\psi$ for the inverse of $\varphi$.

        Let $f \in \mathcal{A}_{\nu}$ be an arbitrary function satisfying
        \[ f \in
        \mathfrak{m}_{C_1,x}^{k_1} \cdot \mathfrak{m}_{C_2,y}^{k_2} + \mathfrak{m}_{C_1,x}^{k_1+1} + \mathfrak{m}_{C_2,y}^{k_2+1}.\]
        We choose an arbitrary lifting $\widetilde{f}_0 \in
        \mathcal{R}_\nu$ along $\pi_{\nu}$. By a recursion, we construct elements
        \begin{equation*}
            \gamma_{l} \in h_1(x)^{-|\underline{s}|-s_{n_1}}h_2(y)^{-|\underline{s}^\prime|-s^\prime_{n_2}} \cdot \caO_\nu, \ (\underline{s},\underline{s}^\prime) = \psi(l), \ l \in \{0,\hdots,l_0\},
        \end{equation*}
        such that
        \begin{equation}
            \label{equation::recursive_goal} \frac{\del^{|\underline{m}|+|\underline{m}^\prime|}\widetilde{f}_0}{\del \underline{X}^{\underline{m}}\del \underline{Y}^{\underline{m}^\prime}}(p) = \frac{\del^{|\underline{m}|+|\underline{m}^\prime|}}{\del \underline{X}^{\underline{m}}\del \underline{Y}^{\underline{m}^\prime}} 
            \left( \sum_{l=0}^k \gamma_l \cdot h_{\psi(l)} \right) (p)
        \end{equation}
        for every $k \in \{0,\hdots,l_0\}$ and all $(\underline{m},\underline{m}^\prime) \in \mathcal{D}_{k_1,k_2}$ such that $\varphi(\underline{m},\underline{m}^\prime)\leq k$.

        Let $k \in \{0,\hdots,l_0\}$ and suppose that we have already constructed $\gamma_l$, $0 \leq l < k$, such that
        \begin{equation}
                \label{equation::inductive_hypo} \frac{\del^{|\underline{m}|+|\underline{m}^\prime|}\widetilde{f}_0}{\del \underline{X}^{\underline{m}}\del \underline{Y}^{\underline{m}^\prime}}(p) = \frac{\del^{|\underline{m}|+|\underline{m}^\prime|}}{\del \underline{X}^{\underline{m}}\del \underline{Y}^{\underline{m}^\prime}} 
            \left( \sum_{l=0}^{k-1} \gamma_l \cdot h_{\psi(l)} \right)(p)
        \end{equation}
        for all $(\underline{m},\underline{m}^\prime) \in \mathcal{D}_{k_1,k_2}$ such that $\varphi(\underline{m},\underline{m}^\prime)< k$. Write $\psi(k)=(\underline{m_0},\underline{m}^\prime_0)$ and set
        \begin{equation}
                \label{equation::definition_gamma}
                \gamma_{k} = \frac{1}{\underline{m}_0!\cdot\underline{m}_0^\prime!}\frac{\del^{|\underline{m}_0|+|\underline{m}_0^\prime|}\widetilde{f}_0}{\del \underline{X}^{\underline{m}_0}\del \underline{Y}^{\underline{m}^\prime_0}}(p) - \sum_{l=0}^{k-1} \gamma_l \cdot \frac{1}{\underline{m}_0!\cdot\underline{m}_0^\prime!}  \frac{\del^{|\underline{m}_0|+|\underline{m}_0^\prime|} h_{\psi(l)}}{\del \underline{X}^{\underline{m}_0}\del \underline{Y}^{\underline{m}^\prime_0}}(p).
        \end{equation}
        We have to prove that 
        \begin{equation*}
            h_1(x)^{|\underline{m}_0|+m_{0,n_1}} h_2(y)^{|\underline{m}_0^\prime|+m^\prime_{0,n_2}} \cdot \gamma_k \in \caO_\nu.
        \end{equation*}
        The first term on the right-hand side of \eqref{equation::definition_gamma} is in $\caO_\nu$ as $\widetilde{f}_0 \in \mathcal{R}_{\nu}$ and $p \in U_{\nu}^\circ$. (This already settles the case $k=0$; note that $(\underline{m}_0,\underline{m}_0^\prime) \neq (\underline{0},\underline{0})$ as $(\underline{0},\underline{0}) \notin \mathcal{D}$.) Every other term on the right-hand side is of the form
        \begin{equation}
            \label{equation::term_rhs}
            \gamma_{\varphi(\underline{s},\underline{s}^\prime)} \cdot \frac{1}{\underline{m}_0!\cdot\underline{m}_0^\prime!}\frac{\del^{|\underline{m}_0|+|\underline{m}^\prime_0|}h_{\underline{s},\underline{s}^\prime}}{\del \underline{X}^{\underline{m}_0} \del \underline{Y}^{\underline{m}^\prime_0}}(p) = \gamma_{\varphi(\underline{s},\underline{s}^\prime)} \cdot
            \left(\frac{1}{\underline{m}_0!}\frac{\del^{|\underline{m}_0|}h_{1,\underline{s}}}{\del \underline{X}^{\underline{m}_0}}(x)\right)\cdot\left(\frac{1}{\underline{m}_0^\prime!}\frac{\del^{|\underline{m}^\prime_0|}h_{2,\underline{s}^\prime}}{\del \underline{Y}^{\underline{m}_0^\prime}}(y)\right) 
        \end{equation}
        for some $(\underline{s},\underline{s}^\prime) \in \mathcal{D}_{k_1,k_2}$ with $\varphi(\underline{s},\underline{s}^\prime) < \varphi(\underline{m}_0,\underline{m}_0^\prime)$. By Lemma \ref{lemma::technical_result}, such a term can be non-zero only in one of the following three cases:
        \begin{enumerate}
            \item[(i)] $|\underline{s}| + s_{n_1} < |\underline{m}_0| + m_{0,n_1}$ and $|\underline{s}^\prime| + s_{n_2}^\prime < |\underline{m}^\prime_0| + m_{0,n_2}^\prime$;
            \item[(ii)] $\underline{s} = \underline{m}_0$ and $|\underline{s}^\prime| + s_{n_2}^\prime < |\underline{m}_0^\prime| + m_{0,n_2}^\prime$;
            \item[(iii)] $|\underline{s}| + s_{n_1} < |\underline{m}_0| + m_{0,n_1}$ and $\underline{s}^\prime = \underline{m}_0^\prime$.
        \end{enumerate}
        Using that $h_{1,\underline{s}} \in h_1(x)^{-1} \cdot \caO_\nu[X_1,\dots,X_{n_1}]$ and  $h_{2,\underline{s}^\prime} \in h_2(y)^{-1} \cdot \caO_\nu[Y_1,\dots,Y_{n_2}]$ 
        in combination with the fact that
        \begin{equation*}
            \frac{1}{\underline{m}!}\frac{\del^{|\underline{m}|}h_{1,\underline{m}}}{\del \underline{X}^{\underline{m}}}(x)
            =
            \frac{1}{\underline{m}^\prime!}\frac{\del^{|\underline{m}^\prime|}h_{2,\underline{m}^\prime}}{\del \underline{Y}^{\underline{m}^\prime}}(y)
            =1
        \end{equation*}
        for every $\underline{m} \in \IZ^{n_1}$ and $\underline{m}^\prime \in \IZ^{n_2}$, one can see that
        \begin{equation*}
            \gamma_{\varphi(\underline{s},\underline{s}^\prime)} \in h_1(x)^{-|\underline{s}|-s_{n_1}}h_2(y)^{-|\underline{s}^\prime|-s^\prime_{n_2}} \cdot \caO_\nu, \ (\underline{s},\underline{s}^\prime) = \psi(l),
        \end{equation*}
        implies that the term in \eqref{equation::term_rhs} is contained in 
        \begin{equation*}
             h_1(x)^{-|\underline{m}_0|-m_{0,n_1}}h_2(y)^{-|\underline{m}^\prime_0|-m^\prime_{0,n_2}} \cdot \caO_\nu
        \end{equation*}
        and so is $\gamma_k$. By its construction, we already know that \eqref{equation::recursive_goal} holds for $(\underline{m},\underline{m}^\prime) = \psi(k)$. The last assertion of Lemma \ref{lemma::technical_result} implies that
        \begin{equation*}
            \frac{\del^{|\underline{m}|+|\underline{m}^\prime|}h_{\underline{m}_0,\underline{m}^\prime_0}}{\del \underline{X}^{\underline{m}} \del \underline{Y}^{\underline{m}^\prime}}(p) = 0
        \end{equation*}
        for all $(\underline{m},\underline{m}^\prime) = \psi(l)$, $l < k$. Therefore, we can infer \eqref{equation::recursive_goal} for all $(\underline{m},\underline{m}^\prime) = \psi(l)$, $l \leq k$, from \eqref{equation::inductive_hypo}. This completes the recursive construction of the sequence $(\gamma_l)$, $l \in \{0,\hdots,l_0\}$.
        
        We note that $h_1(x), h_2(y) \in \mathcal{O}_\nu$ as $(x,y) \in U^\circ_\nu$ and $h_\kappa \in \mathcal{R}_\kappa$ ($\kappa = 1,2$) so that
         \begin{equation*}
            \widetilde{f}= \widetilde{f}_0 -\sum_{l=0}^{l_0} \gamma_l \cdot h_{\psi(l)}
            \in h_1(x)^{-2k_1} h_2(y)^{-2k_2} \cdot \mathcal{R}_{\nu}.
        \end{equation*}
        We claim that $\widetilde{f}$ is the desired lifting of $f$. By construction,
        \begin{equation}\label{equation::already_known_vanishing}
            \frac{\del^{|\underline{m}|+|\underline{m}^\prime|} \widetilde{f}}{\del \underline{X}^{\underline{m}}\del\underline{Y}^{\underline{m}^\prime}}(p) = 0
        \end{equation}
        for all $\underline{m}=(m_1,\dots,m_{n_1}) \in \IZ^{n_1}_{\geq 0}$ and $\underline{m}^\prime = (m_1^\prime,\dots,m_{n_2}^\prime) \in \IZ^{n_2}_{\geq 0}$ such that $|\underline{m}| \leq k_1$, $|\underline{m}^\prime| \leq k_2$, and ($|\underline{m}|\neq m_{n_1}$ or $|\underline{m}^\prime| \neq m_{n_2}^\prime$).
        It hence suffices to prove that
        \begin{equation}
            \label{equation::induction_goal}
            \frac{\del^{m+m^\prime} \widetilde{f}}{\del X_{n_1}^m \del Y_{n_2}^{m^\prime}}(p) 
            = 0
        \end{equation}
        for all $m \in \{0,\dots,k_1\}$ and $m^\prime \in \{0,\dots,k_2\}$ such that $m + m^\prime < k_1 + k_2$. For this purpose, we let $$(\alpha_1,\dots,\alpha_{n_1-1}) \in \IK[X_1,\dots,X_{n_1}]^{1 \times ({n_1}-1)}$$ be the unique solution of 
        \begin{equation*}
             (\alpha_1,\dots,\alpha_{n_1-1}) \mathbf{A}_1 = -
              h_1\left(\frac{\del g_{1,1}}{\del X_{n_1}},
                \dots, \frac{\del g_{1,n_1-1}}{\del X_{n_1}}
              \right).
        \end{equation*}
        Then
        \begin{equation*}
            \tau = 
            \alpha_1 \frac{\del}{\del X_1} + \cdots + \alpha_{n_1-1} \frac{\del}{\del X_{n_1-1}} + h_1 \frac{\del}{\del X_{n_1}}
        \end{equation*}
        is a regular vector field tangent to $V_{1}\subseteq \IA^{n_1}_\IK$. In addition, we remark that
        \begin{equation*}
            (\alpha_1,\dots,\alpha_{n_1-1})
            =
            - 
            \left(\frac{\del g_{1,1}}{\del X_{n_1}},
                \dots, \frac{\del g_{1,n_1-1}}{\del X_{n_1}}
              \right)
            \mathrm{adj}(\mathbf{A}_{1}) \in \IK[X_1,\dots,X_{n_1}]^{1 \times ({n_1}-1)}.
        \end{equation*}        
        Similarly,
        \begin{equation*}
            \tau^\prime = 
            \alpha_1^\prime \frac{\del}{\del Y_1} + \cdots + \alpha_{n_2-1}^\prime \frac{\del}{\del Y_{n_2-1}} + h_2 \frac{\del}{\del Y_{n_2}},
        \end{equation*} 
        where
        \begin{equation*}
            (\alpha_1^\prime,\dots,\alpha_{n_2-1}^\prime)
            =-
            \left(\frac{\del g_{2,1}}{\del Y_{n_2}},
                \dots, \frac{\del g_{2,n_2-1}}{\del Y_{n_2}}
            \right)
            \mathrm{adj}(\mathbf{A}_{2}) \in \IK[Y_1,\dots,Y_{n_2}]^{1 \times (n_2-1)},
        \end{equation*}   
        is a regular vector field tangent to $V_{2} \subseteq \IA^{n_2}_\IK$.

        By construction, we have $\tau(g_{1,k}) = \tau^\prime(g_{2,k}) = 0$ for $k \in \{ 1, \hdots, n_1-1 \}$ or $k \in \{ 1, \hdots, n_2-1 \}$ respectively. As $h_1(x) \cdot h_2(y) \neq 0$, we have that $p = (x,y) \in (V_1 \times V_2)^{\mathrm{sm}}(\IKbar_\nu)$ so that the kernel of the restriction
        \begin{equation*}
        \IKbar_\nu[X_1,\dots,X_{n_1},Y_1,\dots,Y_{n_2}]_p \longrightarrow \mathcal{O}_{C_1 \times C_2,p}    
        \end{equation*}
        equals $(g_{1,1},\dots,g_{1,n_1-1},g_{2,1},\dots,g_{2,n_2-1})$.
        Since $f \in \mathfrak{m}_{C_1,x}^{k_1} \cdot \mathfrak{m}_{C_2,y}^{k_2} + \mathfrak{m}_{C_1,x}^{k_1+1} + \mathfrak{m}_{C_2,y}^{k_2+1}$ and $\widetilde{f}$ is a lift of $f$, we have 
        \begin{equation*}
            \tau^m(\tau^\prime)^{m^\prime} \widetilde{f}(p) = 0
        \end{equation*}
        for all $0 \leq m \leq k_1$ and $0 \leq m^\prime \leq k_2$ such that $m + m^\prime < k_1 + k_2$. Using this fact, we can prove \eqref{equation::induction_goal} by double induction. Indeed, this just follows from expanding $\tau^{m}(\tau^\prime)^{m^\prime}\widetilde{f}$ in terms of the symmetric powers of $\del/\del X_1$, $\del/\del X_2$, $\dots$, $\del/\del X_{n_1},\del/\del Y_1,\dots,\del/\del Y_{n_2}$. In this expansion, the term
        \begin{equation*}
            h_1^m h_2^{m^\prime} \frac{\del^{m+m^\prime} \widetilde{f}}{\del X_{n_1}^m \del Y_{n_2}^{m^{\prime}}}
        \end{equation*}
        always appears and all other terms can be seen to evaluate to zero at $p$ either by \eqref{equation::already_known_vanishing} or by the inductive hypothesis. Since $h_1(x) \cdot h_2(y) \neq 0$, the equality \eqref{equation::induction_goal} follows.
        \end{proof}

\subsection{Extending functions (archimedean case)}
\label{subsection::extending_function_archimedean}

From now on, we fix a place $\nu \in \Sigma_\infty(\IK)$ and an identification of $\IKbar_\nu$ with $\IC$. We then write $|\cdot|$ instead of $|\cdot|_\nu$. \textit{We suppose that some radii $\rho > \rho^\prime > 0$ have been fixed. The constants in this subsection are allowed to depend on $\nu$, $\rho$, and $\rho^\prime$ without further explicit mention.}

    Let $\kappa \in \{1,2\}$. Recall the lifting
    \[\delta_{\kappa}: \caO_\IK[X_1,\dots,X_{n_\kappa}] \longrightarrow \caO_\IK[X_1,\dots,X_{n_\kappa}]\]
    of the derivation $g_{\kappa} \cdot d_{\kappa}$ from Section \ref{subsection::setup}. Furthermore, there exist
    \begin{equation*}
        \delta_{\kappa,i}(X_1,\dots,X_{n_\kappa}) \in \caO_\IK[X_1,\dots,X_{n_\kappa}], \ 1 \leq i \leq n_\kappa,
    \end{equation*}
    such that
    \begin{equation*}
        \delta_{\kappa} = \sum_{i=1}^{n_\kappa} \delta_{\kappa,i} \frac{\del}{\del X_i}.    
    \end{equation*}
    By abuse of notation, we use $d_\kappa$, $\delta_\kappa$, $\delta_{\kappa,i}$, $\vartheta$ also for the respective base changes to $\IC$ via $\nu$. Let $\newconstant{}_{,\nu}\label{C}> 0$ be a constant such that
    \begin{equation}\label{equation::bound_delta_kappa_i}
        | \delta_{\kappa,i}(q)|_\nu \leq \oldconstant{C}{}_{,\nu}   
    \end{equation}
    for all $q \in \overline{D}_{n_\kappa}(2\rho)$ and all $i \in \{1,\dots,n_\kappa\}$. For $r > 0$ and $q \in U_\nu(\IC)$, we set
    \begin{equation*}
        \mathcal{A}_{\nu}(r,q) = \{ f \in A_{\nu} \mid \forall q^\prime \in U_{\nu}(\IC) \cap \vartheta^{-1}(D_{n_1+n_2}(r,q)):|f(q^\prime)|\leq 1\}
    \end{equation*}
    and
    \[U_\nu^\circ(r) = \vartheta^{-1}(D_{n_1+n_2}(r)).\]

    We next establish an archimedean replacement for Proposition \ref{proposition::extension}. In our notation, we do not distinguish between a point $q \in U_\nu(\IC)$ and its image under $\vartheta$ for readability.

   \begin{proposition}
    \label{proposition::archimedean_extension}
        Let $h_1$ and $h_2$ be as in Lemma \ref{lemma::introduction_g_and_h}. There exists a constant $\newconstant{}\label{c37}_{,\nu} > 0$ such that the following assertion holds true: for every point $p = (x,y) \in U_\nu^\circ(\rho)$ such that $h_{1}(x) \cdot h_{2}(y) \neq 0$ and for every 
        \begin{equation*}
            f \in \mathcal{A}_{\nu}(2\rho^\prime,p) \cap (\mathfrak{m}_{C_1,x}^{k_1} \cdot \mathfrak{m}_{C_2,y}^{k_2} + \mathfrak{m}_{C_1,x}^{k_1+1} + \mathfrak{m}_{C_2,y}^{k_2+1}), \ k_1, k_2 \in \IZ_{\geq 0},
        \end{equation*}
        there exists a holomorphic function 
        \begin{equation*}
        \widetilde{f}: D_{n_1+n_2}(\rho^\prime,p) \longrightarrow \IC    
        \end{equation*}
        such that
        \begin{enumerate}
            \item[(i)] $\widetilde{f}$ restricts to $f$ on $\vartheta(U_{\nu}(\IC)) \cap D_{n_1+n_2}(\rho^\prime,p)$,
            \item[(ii)] for all $\underline{m}=(m_1,\dots,m_{n_1}) \in \IZ^{n_1}_{\geq 0}$ and $\underline{m}^\prime = (m_1^\prime,\dots,m_{n_2}^\prime) \in \IZ^{n_2}_{\geq 0}$ such that $|\underline{m}| \leq k_1$, $|\underline{m}^\prime| \leq k_2$, and $|\underline{m}| + |\underline{m}^\prime| < k_1 + k_2$,
            \begin{equation*}
                \frac{\del^{|\underline{m}|+|\underline{m}^\prime|} \widetilde{f}}{\del \underline{X}^{\underline{m}}\del\underline{Y}^{\underline{m}^\prime}}(p) = 0, \ \text{and},
            \end{equation*}
        
            \item[(iii)] for all $q^\prime = (u^\prime,v^\prime) \in D_{n_1+n_2}(\rho^\prime,p)$,
            \begin{equation}
                \label{equation::upper_bound_archimedean_schwarz}
                \log |\widetilde{f}(q^\prime)| \leq 2k_1\log^+| h_{1}(x)|^{-1} + 2k_2\log^+|h_{2}(y)|^{-1} +  \oldconstant{c37}{}_{,\nu}(k_1+k_2+1) .
            \end{equation}
        \end{enumerate}
    \end{proposition}

    While the proof of this proposition follows mostly the one of Proposition \ref{proposition::extension}, with integral structures replaced by absolute values, our use of Oka's extension theorem \cite[Theorem 8.1.18 (ii)]{Noguchi2016} forces us to allow the extension $\widetilde{f}$ to be a holomorphic function in general. It would be interesting to establish a stronger version of the proposition with $\widetilde{f}$ a polynomial.

    \begin{proof}
        We follow mostly the argument from Proposition \ref{proposition::extension} and skip some definitions and arguments where repetition does not seem to be of benefit to the reader. For every $r>0$ and every $q \in \IC^{n_1+n_2}$, we  write $D(r,q)$ (resp.\ $\overline{D}(r,q)$) instead of $D_{n_1+n_2}(r,q)$ (resp.\ $\overline{D}_{n_1+n_2}(r,q)$) and $\Vert \cdot \Vert_{D(r,q)}$ (resp.\ $\Vert \cdot \Vert_{\overline{D}(r,q)}$) for the supremum norm on $D(r,q)$ (resp.\ $\overline{D}(r,q)$). If $q$ is the origin $0 \in \IC^{n_1+n_2}$, we will also drop $q$ from the notation.
        
        We define
        \begin{equation*}
            L_1,\dots,L_{n_1},h_{1,j} \in \IC[X_1,\dots,X_{n_1}]
        \end{equation*}
        and
        \begin{equation*}
            L_1^\prime, \dots,L_{n_2}^\prime, h_{2,j} \in \IC[Y_1,\dots,Y_{n_2}]
        \end{equation*}
        verbatim as in the proof of Proposition \ref{proposition::extension}. For every $\underline{s} = (s_1,\dots,s_{n_1}) \in \IZ_{\geq 0}^{n_1}$, we define $h_{1,\underline{s}} \in \IC[X_1,\dots,X_{n_1}]$ as in \eqref{equation::h_1s}. Similarly, we define $h_{2,\underline{s}^\prime} \in \IC[Y_1,\dots,Y_{n_2}]$ for every $\underline{s}^\prime = (s_1^\prime,\dots,s_{n_2}^\prime) \in \IZ_{\geq 0}^{n_2}$ as in \eqref{equation::h_2s} and set $h_{\underline{s},\underline{s}^\prime} = h_{1,\underline{s}} \cdot h_{2,\underline{s}^\prime}$. Lemma \ref{lemma::technical_result} still applies literally. 
        
        There exists a constant $\newconstant{}_{,\nu}\label{c38} \geq 1$ such that
        \begin{equation}\label{equation::bound_h_1,j}
            \max_{1 \leq j \leq n_1 - 1}
            \Vert h_{1,j} \Vert_{\overline{D}_{n_1}(4\rho+2,0)} \leq \oldconstant{c38}{}_{,\nu} |h_1(x)|^{-1}
        \end{equation}
        and
        \begin{equation*}
            \max_{1 \leq j \leq n_2 - 1}
            \Vert h_{2,j} \Vert_{\overline{D}_{n_2}(4\rho+2,0)} \leq \oldconstant{c38}{}_{,\nu} |h_2(y)|^{-1}.
        \end{equation*}
        We consequently obtain that there is a constant $\newconstant{}_{,\nu}\label{c38.5} \geq 1$ such that
        \begin{equation}\label{equation::bound_h_s,s'}
            \Vert h_{\underline{s},\underline{s}^\prime} \Vert_{\overline{D}_{n_1+n_2}(4\rho+2,0)} < \oldconstant{c38.5}{}_{,\nu}^{|\underline{s}|+|\underline{s}^\prime|+1} |h_1(x)|^{-\eta_1(\underline{s})} |h_2(y)|^{-\eta_2(\underline{s}^\prime)}
        \end{equation}
        for all $\underline{s} = (s_1,\dots,s_{n_1}) \in \IZ_{\geq 0}^{n_1}$ and $\underline{s}^\prime = (s_1^\prime,\dots,s_{n_2}^\prime) \in \IZ_{\geq 0}^{n_2}$ with $|\underline{s}| \leq k_1$ and $|\underline{s}^\prime| \leq k_2$, where \begin{equation*}
        \eta_1(\underline{s}) =
        \begin{cases}
        0 & \text{if $s_{n_1} = |\underline{s}|$,} \\
        1 & \text{otherwise,}
        \end{cases}
        \end{equation*}
        while
        \begin{equation*}
        \eta_2(\underline{s}^\prime) =
        \begin{cases}
        0 & \text{if $s^\prime_{n_2} = |\underline{s}^\prime|$,} \\
        1 & \text{otherwise.}
        \end{cases}
        \end{equation*}
        
        We also use again the subset $\mathcal{D}_{k_1,k_2}$ of $\IZ^{n_1}_{\geq 0} \times \IZ^{n_2}_{\geq 0}$ as well as the bijection $\varphi: \mathcal{D}_{k_1,k_2} \rightarrow \{0,\hdots,l_0\}$ and its inverse $\psi: \{0,\hdots,l_0\} \rightarrow \mathcal{D}_{k_1,k_2}$. Recall that the bijection $\varphi$ is chosen such that $\varphi(\underline{s},\underline{s}^\prime) \leq \varphi(\underline{m},\underline{m}^\prime)$ implies 
        \begin{equation*}
        |\underline{s}| + s_{n_1} + |\underline{s}^\prime| + s^\prime_{n_2} \leq  |\underline{m}| + m_{n_1} + |\underline{m}^\prime| + m^\prime_{n_2}.
        \end{equation*}
        We use this to find a crude upper bound on $l_0$. There are at most
        \begin{equation*}
            \binom{k_1+k_2+n_1+n_2}{n_1+n_2} < (k_1 + k_2 +n_1+n_2)^{n_1+n_2}
        \end{equation*}
        pairs $(\underline{s},\underline{s}^\prime) \in \IZ^{n_1}_{\geq 0} \times \IZ^{n_2}_{\geq 0}$ such that $|\underline{s}| + |\underline{s}^\prime| \leq  k_1 + k_2$. Since every pair $(\underline{s},\underline{s}^\prime) \in \mathcal{D}_{k_1,k_2}$ satisfies this inequality, we deduce that there exists a positive constant $\newconstant{}_{,\nu}\label{c39}$ such that
        \begin{equation}
            \label{equation::l0_bound}
            \log \max\{1,l_0\} \leq \oldconstant{c39}{}_{,\nu}\log (k_1 + k_2 + 2).
        \end{equation}

        By compactness, there exists a finite set $\mathcal{Z} \subseteq \IC^{n_1+n_2}$ such that
        \[\overline{D}(\rho) \subseteq \bigcup_{z \in \mathcal{Z}} D(\rho^\prime/3,z).\]

        Let $f \in \mathcal{A}_{\nu}(2\rho^\prime,p)$ be an arbitrary function satisfying
        \[f \in
       \mathfrak{m}_{C_1,x}^{k_1} \cdot \mathfrak{m}_{C_2,y}^{k_2} +  \mathfrak{m}_{C_1,x}^{k_1+1} + \mathfrak{m}_{C_2,y}^{k_2+1}.\]
       There exists some $z \in \mathcal{Z}$ such that $p \in D(\rho^\prime/3,z)$. We note that
       \[ D(5\rho^\prime/3,z) \subseteq D(2\rho^\prime,p)\]
       and so
       \[ \mathcal{A}_\nu(2\rho^\prime,p) \subseteq \mathcal{A}_\nu(5\rho^\prime/3,z).\]
       Simultaneously,
       \[\overline{D}(\rho^\prime,p) \subseteq  D(3\rho^\prime/2,z).\]

       By \cite[Theorem 6.6.8]{Noguchi2016}, the complex analytic space $U_\nu(\mathbb{C}) \cap D_{n_1+n_2}(3\rho^\prime/2,z)$ has at most finitely many irreducible components. Therefore, by Oka's extension theorem \cite[Theorem 8.1.18 (ii)]{Noguchi2016}, applied to
       \[ U_\nu(\mathbb{C}) \cap D(3\rho^\prime/2,z) \subseteq D(3\rho^\prime/2,z),\]
there exists a constant $\newconstant{}_{,\nu}\label{c40} \geq 1$ such that $f$ has a holomorphic extension $\widetilde{f}_0: D(3\rho^\prime/2,z) \rightarrow \IC$ with
        \begin{equation}\label{equation::oka_bound}
            \Vert \widetilde{f}_0 \Vert_{\overline{D}(\rho^\prime,p)} \leq \oldconstant{c40}{}_{,\nu}.
        \end{equation}
        Here, the constant depends on $\mathcal{Z}$, which in turn depends only on $\rho$ and $\rho^\prime$.
        
        Using the same recursion as in the proof of Proposition \ref{proposition::extension}, we construct an element $\gamma_l \in \IC$ for each $l \in \{0,\hdots,l_0\}$ and set
        \begin{equation*}
            \widetilde{f}= \widetilde{f}_0 - \sum_{l=0}^{l_0} \gamma_l h_{\psi(l)}.
        \end{equation*}
        We claim that $\widetilde{f}$ is again the desired extension of $f$. Indeed, the property (i) is immediate and (ii) follows from the recursive construction as in the integral case. It remains to replace  integrality assertions with archimedean estimates to prove also (iii).
        
        Writing $\psi(k)= (\underline{m}_0,\underline{m}_0^\prime)$, we claim that there exists a constant $\newconstant{}_{,\nu}\label{c41} > 1$ such that
        \begin{align}
            \label{equation::size_gamma_k}
            |\gamma_{k}| \leq \oldconstant{c41}{}_{,\nu}(\oldconstant{c41}{}_{,\nu}\max\{1,|h_1(x)|^{-1}\})^{|\underline{m}_0|+m_{0,n_1}}(\oldconstant{c41}{}_{,\nu}\max \{1,|h_2(y)|^{-1}\})^{|\underline{m}_0^\prime|+m_{0,n_2}^\prime}
        \end{align}
        for all $k \in \{0,\hdots,l_0\}$. For this purpose, let us recall that
        \begin{equation}
                \label{equation::definition_gamma2}
                \gamma_{k} = \frac{1}{\underline{m}_0!\underline{m}_0^\prime!}\frac{\del^{|\underline{m}_0|+|\underline{m}_0^\prime|}\widetilde{f}_0}{\del \underline{X}^{\underline{m}_0}\del \underline{Y}^{\underline{m}^\prime_0}}(p) - \sum_{l=0}^{k-1} \gamma_l \frac{1}{\underline{m}_0!\underline{m}_0^\prime!}  \frac{\del^{|\underline{m}_0|+|\underline{m}_0^\prime|} h_{\psi(l)}}{\del \underline{X}^{\underline{m}_0}\del \underline{Y}^{\underline{m}^\prime_0}}(p).
        \end{equation}

        Since $\overline{D}(\rho^\prime,p) \subseteq D(3\rho^\prime/2,z)$, we can apply Cauchy's integral formula \cite[Chapter A]{Gunning1990} to bound the first term on the right-hand side from above by
        \begin{equation}\label{equation::bound_gamma_0}
            ({\rho^\prime})^{-|\underline{m}_0|-|\underline{m}_0^\prime|}\Vert \widetilde{f}_0 \Vert_{\overline{D}(\rho^\prime,p)} \leq ({\rho^\prime})^{-|\underline{m}_0|-|\underline{m}_0^\prime|}\oldconstant{c40}{}_{,\nu}.
        \end{equation}
        To estimate the sum in this expression with respect to the archimedean absolute value, we need the additional information provided by the following lemma, which builds upon Lemma \ref{lemma::technical_result} and its proof.

        \begin{lemma}
            \label{lemma::even_more_technical_result}
            Consider arbitrary $\underline{m},\underline{s}\in \IZ_{\geq 0}^{n_1}$ and $\underline{m}^\prime,\underline{s}^\prime\in \IZ_{\geq 0}^{n_2}$. Then there exists a constant $\newconstant{}_{,\nu}\label{c42} \geq 1$ such that
            \begin{equation}
                \label{equation::absolute_value_derivative}
                \left\vert \frac{1}{\underline{m}! \underline{m}^\prime!}\frac{\del^{|\underline{m}|+|\underline{m}^\prime|}h_{\underline{s},\underline{s}^\prime}}{\del \underline{X}^{\underline{m}} \del \underline{Y}^{\underline{m}^\prime}}(p)\right\vert 
            \end{equation}
            is bounded from above by
            \begin{align*}
                2^{-|\underline{m}|-|\underline{m}^\prime|+|\underline{s}|+|\underline{s}^\prime|} \oldconstant{c42}^2{}_{,\nu} & \max\{1,|h_1(x)|^{-1}\}^{|\underline{m}|-|\underline{s}|+m_{n_1}-s_{n_1}} \\ &\cdot \max\{1,|h_2(y)|^{-1}\}^{|\underline{m}^\prime|-|\underline{s}^\prime|+m^\prime_{n_2}-s^\prime_{n_2}}.
            \end{align*}
            Furthermore, \eqref{equation::absolute_value_derivative} is zero unless
            \begin{equation*}
                |\underline{s}| \leq |\underline{m}| , \  
                |\underline{s}^\prime| \leq |\underline{m}^\prime|, \ |\underline{s}|+s_{n_1} \leq |\underline{m}|+m_{n_1},
                \ \text{and} \ 
                |\underline{s}^\prime|+s_{n_2}^\prime \leq |\underline{m}^\prime|+m_{n_2}^\prime.
            \end{equation*}
        \end{lemma}

        \begin{proof}
            Since
            \begin{equation*}    \frac{\del^{|\underline{m}|+|\underline{m}^\prime|}h_{\underline{s},\underline{s}^\prime}}{\del \underline{X}^{\underline{m}} \del \underline{Y}^{\underline{m}^\prime}}(p)
                =
                \frac{\del^{|\underline{m}|}h_{1,\underline{s}}}{\del \underline{X}^{\underline{m}}}(x) \cdot\frac{\del^{|\underline{m}^\prime|}h_{2,\underline{s}^\prime}}{\del \underline{Y}^{\underline{m}^\prime}}(y), 
            \end{equation*}
            it suffices by symmetry to show that
            \begin{equation*}
                \left\vert\frac{1}{\underline{m}!}\frac{\del^{|\underline{m}|}h_{1,\underline{s}}}{\del \underline{X}^{\underline{m}}}(x)\right\vert \leq 2^{-|\underline{m}|+|\underline{s}|}\oldconstant{c42}{}_{,\nu} \max\{1,|h_1(x)|^{-1}\}^{|\underline{m}|-|\underline{s}|+m_{n_1}-s_{n_1}}
            \end{equation*}
            to prove the bound.
            If $\underline{s} = \underline{m}$, the left-hand side is $1$ and there is nothing to prove. Moreover, if $|\underline{s}| = s_{n_1}$, then $h_{1,\underline{s}} = L_{n_1}^{s_{n_1}}$, the left-hand side is $0$ unless $\underline{s} = \underline{m}$, and we are done.
            
            We suppose from now on that $\underline{s} \neq \underline{m}$ and $|\underline{s}| \neq s_{n_1}$. In this case, the non-vanishing of $\del^{|\underline{m}|} h_{1,\underline{s}}/\del\underline{X}^{\underline{m}}(x)$ implies that 
            \begin{equation*}
                |\underline{s}| \leq |\underline{m}| \ \ \text{and} \ \ |\underline{s}|+s_{n_1} < |\underline{m}|+m_{n_1}  
            \end{equation*}
            by (the proof of) Lemma \ref{lemma::technical_result}.
            In fact, its proof shows that
            \begin{equation*}
                s_{i(\underline{s})}-1 \leq m_{i(\underline{s})}, \ s_{i(\underline{s})+1} \leq m_{i(\underline{s})+1}, \
                \dots, \
                s_{n_1} \leq m_{n_1}.
            \end{equation*}
            We find that
            \begin{align*}
             \frac{1}{\underline{m}!} \frac{\del^{|\underline{m}|}h_{1,\underline{s}}}{\del \underline{X}^{\underline{m}}}(x) 
            &=
            \frac{1}{\underline{m}!}
            \frac{\del^{|\underline{m}|}}{\del \underline{X}^{\underline{m}}}\left( h_{1,i(\underline{s})} L_{i(\underline{s})}^{s_{i(\underline{s})}-1}
            L_{i(\underline{s})+1}^{s_{i(\underline{s})+1}} \cdots L_{n_1}^{s_{n_1}}\right)(x) \\
                &= 
                \frac{1}{(\underline{m}-\underline{t})!}
                \frac{\del^{|\underline{m}|-|\underline{s}|+1}h_{1,i(\underline{s})}}{\del \underline{X}^{\underline{{m}}-\underline{t}}}(x),
            \end{align*}
            where $\underline{t}$ is equal to $\underline{s}$ except for the $i(\underline{s})$-th coordinate, which is equal to $s_{i(\underline{s})}-1$. By Cauchy's integral formula \cite[Chapter A]{Gunning1990}, applied on the polydisc $D_{n_1}(2,x)$, and \eqref{equation::bound_h_1,j}, it follows that
            \begin{align*}
                \left\vert \frac{1}{(\underline{m}-\underline{t})!} \frac{\del^{|\underline{m}|-|\underline{s}|+1}h_{1,i(\underline{s})}}{\del \underline{X}^{\underline{m}-\underline{t}}}(x)\right\vert 
                &\leq 2^{-|\underline{m}|+|\underline{s}|-1} \Vert h_{1,i(\underline{s})}\Vert_{\overline{D}_{n_1}(2,x)} \\
                &\leq 2^{-|\underline{m}|+|\underline{s}|-1} \Vert h_{1,i(\underline{s})}\Vert_{\overline{D}_{n_1}(4\rho+2,0)}\\
                &\leq 2^{-|\underline{m}|+|\underline{s}|-1}\oldconstant{c38}{}_{,\nu} |h_{1}(x)|^{-1}
            \end{align*}
            in both cases, which completes the proof.
        \end{proof}

        Having established Lemma \ref{lemma::even_more_technical_result}, we can now continue the proof of \eqref{equation::size_gamma_k}, arguing by induction on $k$. The case $k = 0$ follows from \eqref{equation::bound_gamma_0}.

        Suppose now that \eqref{equation::size_gamma_k} is already established for $\gamma_0, \hdots, \gamma_{k-1}$ ($1 \leq k \leq l_0$). By \eqref{equation::bound_gamma_0}, Lemmas \ref{lemma::technical_result} and \ref{lemma::even_more_technical_result}, and the inductive hypothesis, the absolute value of the sum in \eqref{equation::definition_gamma2} is bounded from above by
        \begin{multline*}
        {(\rho^\prime)}^{-|\underline{m}_0|-|\underline{m}_0^\prime|}\oldconstant{c40}{}_{,\nu} + \oldconstant{c42}^2{}_{,\nu}\max\{1,|h_1(x)|^{-1}\}^{|\underline{m}_0|+m_{0,n_1}}\max\{1, |h_2(y)|^{-1}\}^{|\underline{m}_0^\prime|+m^\prime_{0,n_2}}\\  \cdot \sum_{\underline{s}}\sum_{\underline{s}^\prime}\left(2^{-|\underline{m}_0|-|\underline{m}_0^\prime|+|\underline{s}|+|\underline{s}^\prime|}\oldconstant{c41}{}^{|s|+|s^\prime|+s_{n_1}+s_{n_2}^\prime+1}_{,\nu}\right),
        \end{multline*}
        where the double sum runs over all tuples $(\underline{s},\underline{s}^\prime) \in \psi(\{0,\hdots,k-1\})$ such that \eqref{equation::absolute_value_derivative} does not vanish. As recalled in the proof of Lemma \ref{lemma::even_more_technical_result}, all such $(\underline{s},\underline{s}^\prime)$ satisfy $s_i \leq m_{0,i} +1$ for all $i \in \{1,\hdots,n_1\}$ and $s^\prime_i \leq m^\prime_{0,i} + 1$ for all $i \in \{1,\hdots,n_2\}$ as well as $|\underline{s}|+s_{n_1} + |\underline{s}^\prime|+s^\prime_{n_2} < |\underline{m}_0| + m_{0,n_1} + |\underline{m}^\prime_0| + m^\prime_{0,n_2}$.
       The double sum can therefore be bounded from above by
        \[ \left(\sum_{i=-1}^{\infty}{2^{-i}}\right)^{n_1}\left(\sum_{i=-1}^{\infty}{2^{-i}}\right)^{n_2}\oldconstant{c41}{}_{,\nu}^{|\underline{m}_0|+|\underline{m}_0^\prime|+m_{0,n_1}+m_{0,n_2}^\prime} = 4^{n_1+n_2}\oldconstant{c41}{}_{,\nu}^{|\underline{m}_0|+|\underline{m}_0^\prime|+m_{0,n_1}+m_{0,n_2}^\prime}.\]
        It follows that
        \begin{multline*} |\gamma_k| \leq {(\rho^\prime)}^{-|\underline{m}_0|-|\underline{m}_0^\prime|}\oldconstant{c40}{}_{,\nu}\\
        + 4^{n_1+n_2}\oldconstant{c42}^2{}_{,\nu}(\oldconstant{c41}{}_{,\nu}\max\{1,|h_1(x)|^{-1}\})^{|\underline{m}_0|+m_{0,n_1}}(\oldconstant{c41}{}_{,\nu}\max\{1, |h_2(y)|^{-1}\})^{|\underline{m}_0^\prime|+m^\prime_{0,n_2}}.
        \end{multline*}
        Therefore, it suffices to choose
        \[ \oldconstant{c41}{}_{,\nu} \geq  \max\{1,(\rho^\prime)^{-1}\}\oldconstant{c40}{}_{,\nu} + 4^{n_1+n_2}\oldconstant{c42}^2{}_{,\nu} > 1\]
        to complete the induction step.
        
        This establishes our claim \eqref{equation::size_gamma_k}. From this bound, we infer further that, for $k \in \{0,\hdots,l_0\}$ and $\psi(k) = (\underline{s},\underline{s}^\prime)$, $\log |\gamma_k|$ is bounded from above by
        \begin{equation*}
            (2k_1-\eta_1(\underline{s}))\log^+|h_1(x)|^{-1}
            +
            (2k_2-\eta_2(\underline{s}^\prime))\log^+|h_2(y)|^{-1}
            +
            (k_1+k_2+1)\oldconstant{c39.5}{}_{,\nu}
        \end{equation*}
        for some constant $\newconstant{}_{,\nu}\label{c39.5}$, where we recall that
        \begin{equation*}
        \eta_1(\underline{s}) =
        \begin{cases}
        0 & \text{if $s_{n_1} = |\underline{s}|$,} \\
        1 & \text{otherwise,}
        \end{cases}
        \end{equation*}
        while
        \begin{equation*}
        \eta_2(\underline{s}^\prime) =
        \begin{cases}
        0 & \text{if $s^\prime_{n_2} = |\underline{s}^\prime|$,} \\
        1 & \text{otherwise.}
        \end{cases}
        \end{equation*}
        Combined with \eqref{equation::bound_h_s,s'}, this yields that
        \begin{equation}\label{equation::bound_gamma_times_h}
        |\gamma_k|\Vert h_{\psi(k)} \Vert_{\overline{D}(\rho^\prime,p)} \leq (\oldconstant{c38.5}{}_{,\nu}\exp(\oldconstant{c39.5}{}_{,\nu}))^{k_1+k_2+1}\max\{1,|h_1(x)|^{-1}\}^{2k_1}\max\{1,|h_2(y)|^{-1}\}^{2k_2}.
        \end{equation}
        
        This estimate allows us to prove (iii). In fact, for every $q^\prime \in D(\rho^\prime,p)$, we deduce that
        \begin{align*}
            |\widetilde{f}(q^\prime)|
            &\leq |\widetilde{f}_0(q^\prime)|+\sum_{l=0}^{l_0} |\gamma_l| |h_{\psi(l)}(q^\prime)| \\
            &\leq (l_0+2) \max \{ \Vert \widetilde{f}_0 \Vert_{\overline{D}(\rho^\prime,p)}, |\gamma_0| \Vert h_{\psi(0)}\Vert_{\overline{D}(\rho^\prime,p)},\dots, |\gamma_{l_0}| \Vert h_{\psi(l_0)} \Vert_{\overline{D}(\rho^\prime,p)} \}.
        \end{align*}
        Using \eqref{equation::l0_bound}, \eqref{equation::oka_bound}, and \eqref{equation::bound_gamma_times_h}, we obtain (iii).
    \end{proof}

\subsection{A uniform Schwarz estimate}

We continue with the setting from the previous subsections, including the functions $g_{\kappa} \in \mathcal{B}_\kappa \backslash \{0\}$, $h_\kappa \in \mathcal{R}_\kappa \backslash \mathcal{I}_\kappa$ ($\kappa \in \{1,2\}$). 

We start with the case $\nu \in \Sigma_f(\IK)$. The $\nu$-adic completion $\widehat{A}_{\nu}$ of $A_{\nu}$ is a strictly affinoid $\IK_\nu$-algebra (i.e., it is a quotient of a Tate algebra $\IK_\nu\langle T_1, \dots, T_g\rangle$, $g \geq 1$) by \cite[Corollary 1.1.9.6]{Bosch1984}. Recall that
    \begin{equation*}
        U_{\nu}^\circ = \{ x \in U(\overline{\IK}_\nu) \mid \forall f \in \mathcal{A}: |f(x)|_\nu \leq 1 \}.
    \end{equation*}
    We use this set to define a supremum norm on each $A_{\nu}$ for every place $\nu \in \Sigma_f(\IK)$ satisfying $U^\circ_{\nu} \neq \emptyset$. 
    Under this condition, we set 
    \begin{equation*}
        \Vert f \Vert_{\nu} = \max_{p \in U^\circ_{\nu}} \{ |f(p)|_\nu \}
    \end{equation*}
    for every $f \in A_{\nu}$. This is indeed a maximum by \cite[Proposition 6.2.1.4]{Bosch1984} applied to the image of $f$ in $\widehat{A}_{\nu}$.

    Let $\mathcal{M}(\widehat{A}_{\nu})$ denote the Berkovich spectrum of the $\IK_\nu$-Banach algebra $\widehat{A}_{\nu}$ as defined in \cite[Chapter 1]{Berkovich1990}. Recall that $\mathcal{M}(\widehat{A}_{\nu})$ is the set of bounded multiplicative seminorms on $\widehat{A}_{\nu}$ and that $U_{\nu}^\circ$ is a dense subset of $\mathcal{M}(\widehat{A}_{\nu})$ (see for example \cite[(3.6)]{Gubler2013}). This observation is used for applying \cite[Theorem 4.19 (v) $\Rightarrow$ (ii)]{Boucksom2021} in the proof of the following lemma.
 
    \begin{lemma}
        \label{lemma::schwarz_non_archimedean}
        There exists a family of non-negative constants $\newconstant{}_{,\nu}\label{c44}$, $\nu \in \Sigma_f(\IK)$, all but finitely many of which are zero such that the following assertion is true: for every non-archimedean place $\nu \in \Sigma_f(\IK)$, every point $p = (x,y) \in U_{\nu}^\circ$ such that $g_1(x) g_2(y) h_{1}(x) h_{2}(y) \neq 0$, and every 
        \begin{equation}
            \label{equation::lower_derivatives_vanish}
            f \in A_{\nu} \cap (\mathfrak{m}_{C_1,x}^{k_1}\cdot \mathfrak{m}_{C_2,y}^{k_2} + \mathfrak{m}_{C_1,x}^{k_1+1} + \mathfrak{m}_{C_2,y}^{k_2+1}), \ k_1, k_2 \in \IZ_{\geq 0},
        \end{equation}
        the quantity
        \begin{equation*}
            \log \left\vert \mathbf{D}^{k_1,k_2} f(p)\right\vert_{\nu}
        \end{equation*}
        is bounded from above by
        \begin{align*}
            \log \Vert f \Vert_{\nu} 
            &- k_1 \log |g_{1}(x)|_\nu
            - k_2 \log |g_{2}(y)|_\nu \\
            &- 2k_1 \log |h_{1}(x)|_\nu
            - 2k_2 \log |h_{2}(y)|_\nu
            + \oldconstant{c44}{}_{,\nu}.
        \end{align*}
        
    \end{lemma}
    
    \begin{proof}
        By \cite[Theorem 4.20]{Boucksom2021}, the integral closure $\widetilde{\mathcal{A}}_{\nu}$ of $\mathcal{A}_\nu$ in $A \otimes_{\IK} \IKbar_{\nu}$ is a finite $\mathcal{A}_{\nu}$-module. We can hence choose $\beta_{\nu} \in \caO_{\nu} \backslash \{0\}$ such that $\beta_{\nu} \cdot \widetilde{\mathcal{A}}_{\nu} \subseteq \mathcal{A}_{\nu}$. We may even assume that $\beta_{\nu}=1$ for all but finitely many $\nu \in \Sigma_f(\IK)$. In fact, since $C_1 \times C_2$ is geometrically reduced, all but finitely many fibers of $\mathcal{U} \to \mathrm{Spec}~\caO_\IK$ are geometrically reduced by \cite[Th\'eor\`eme 9.7.7 (iii)]{EGA4III}. For these places, we can take $\beta_{\nu} = 1$ by \cite[Theorem 4.19 (i) $\Rightarrow$ (v)]{Boucksom2021}.

        If $f = 0$, the lemma holds trivially. Otherwise, $\Vert f \Vert_\nu > 0$ and we can choose $\gamma \in \IKbar_\nu$ such that $|\gamma|_\nu = \Vert f \Vert_{\nu}^{-1}$ and hence $\Vert \gamma \cdot f\Vert_{\nu} = 1$. Using \cite[Lemma 4.17 and Theorem 4.19 (v) $\Rightarrow$ (ii)]{Boucksom2021}, it follows that $\gamma \cdot f \in \widetilde{\mathcal{A}}_{\nu}$ and thus $f_1 = \beta_{\nu} \gamma \cdot f \in \mathcal{A}_{\nu}$. By Proposition \ref{proposition::extension}, there exists an extension
        \begin{equation*} 
        \widetilde{f}_1 \in h_{1}(x)^{-2k_1} h_{2}(y)^{-2k_2} \cdot \mathcal{R}_{\nu} \cap (\mathfrak{m}_{\IA^{n_1},x}^{k_1} \cdot \mathfrak{m}_{\IA^{n_2},y}^{k_2} + \mathfrak{m}_{\IA^{n_1},x}^{k_1+1} + \mathfrak{m}_{\IA^{n_2},y}^{k_2+1})
        \end{equation*}
        of $f_1$.

        By definition, we have $k_1!k_2! \cdot \mathbf{D}^{k_1,k_2} f = (d_{1}^{k_1},d_{2}^{k_2})f$. We recall that we write $\delta_{1} = g_{1} \cdot d_{1}$ and $\delta_{2} = g_{2} \cdot d_{2}$.
        Using \eqref{equation::lower_derivatives_vanish} and the Leibniz rule, we infer that 
        \begin{equation*}
            g_{1} (x)^{k_1} g_{2} (y)^{k_2}k_1!k_2! \cdot \mathbf{D}^{k_1,k_2} f_1(p) 
            = ((g_{1}  d_{1})^{k_1},(g_{2} d_{2})^{k_2})f_1(p) 
            = (\delta_{1}^{k_1}, \delta_{2}^{k_2})\widetilde{f}_1(p).
        \end{equation*}        
        
        It hence suffices to prove that
        \begin{equation}
        \label{equation::goal}
            (\delta_{1}^{k_1}, \delta_{2}^{k_2})\widetilde{g}(p) \in k_1! k_2! \caO_{\nu}
        \end{equation}
        for every $\widetilde{g} \in \mathcal{R}_{\nu} \cap (\mathfrak{m}_{\IA^{n_1},x}^{k_1} \cdot \mathfrak{m}_{\IA^{n_2},y}^{k_2} + \mathfrak{m}_{\IA^{n_1},x}^{k_1+1} + \mathfrak{m}_{\IA^{n_2},y}^{k_2+1})$. We calculate that
        \begin{align*}
            (\delta_{1}^{k_1}, \delta_{2}^{k_2}) \widetilde{g}(p) 
            &= \left( \left( \sum_{i=1}^{n_1} \delta_{1,i} \cdot \frac{\del}{\del X_i} \right)^{k_1}, \left( \sum_{i=1}^{n_2} \delta_{2,i} \cdot \frac{\del}{\del Y_i} \right)^{k_2} \right) \widetilde{g}(p)\\
            &= \left( \left( \sum_{i=1}^{n_1} \delta_{1,i}(x) \cdot \frac{\del}{\del X_i} \right)^{k_1}, \left( \sum_{i=1}^{n_2} \delta_{2,i}(y) \cdot \frac{\del}{\del Y_i} \right)^{k_2} \right) \widetilde{g}(p).
        \end{align*}
        The assertion \eqref{equation::goal} follows upon expanding $\widetilde{g}$ as a (finite) Taylor series centered at $p = (x,y) \in U_{\nu}^\circ$ as all affine coordinates of $p$ are integral by the definition of $U_{\nu}^\circ$. Setting 
        \begin{equation*}
            \oldconstant{c44}{}_{,\nu} = \max \{ 0, - \log |\beta_{\nu}|_\nu \},
        \end{equation*}
        we obtain the bound stated in the lemma.
    \end{proof}

    The case $\nu \in \Sigma_\infty(\IK)$ is treated similarly. As in Subsection \ref{subsection::extending_function_archimedean}, we suppose that some radii $\rho > \rho^\prime > 0$ have been fixed. \textit{The constants in the remainder of this subsection are allowed to depend on $\nu$, $\rho$, and $\rho^\prime$ without further explicit mention.} Recall that, for every integer $n > 0$, every point $w = (w_1, \hdots, w_n) \in \overline{\IK}_\nu^n$, and every $r > 0$, we write
\begin{equation*}
    D_{n}(r,w) = \left\{ (z_1,\dots,z_{n})\in \IA^{n}(\overline{\IK}_\nu) \mid \max\{|z_1-w_1|,\dots,|z_{n}-w_n|\} < r \right\}
\end{equation*}
and
\begin{equation*}
    U^\circ_{\nu}(r) = U(\overline{\IK}_\nu) \cap \vartheta^{-1}(D_{n_1+n_2}(r)).
\end{equation*}

\begin{lemma}
        \label{lemma::schwarz_archimedean}
        There exists a non-negative constant $\newconstant{}_{,\nu}\label{c43}$ such that the following assertion is true: For every point $p = (x,y) \in U_{\nu}^\circ(\rho)$ such that $g_1(x) g_2(y) h_{1}(x) h_{2}(y) \neq 0$, and every 
        \begin{equation}
            \label{equation::lower_derivatives_vanish_3}
            f \in A_{\nu} \cap (\mathfrak{m}_{C_1,x}^{k_1} \cdot \mathfrak{m}_{C_2,y}^{k_2} + \mathfrak{m}_{C_1,x}^{k_1+1} + \mathfrak{m}_{C_2,y}^{k_2+1}), \ k_1, k_2 \in \IZ_{\geq 0},
        \end{equation}
        the quantity
        \begin{equation*}
            \log \left\vert \mathbf{D}^{k_1,k_2} f(p) \right\vert_{\nu}
        \end{equation*}
        is bounded from above by
        \begin{align*}
            \log \max_{q \in \vartheta^{-1}(\overline{D}(\rho^\prime,p))}{\left\{|f(q)|\right\}} 
            &- k_1 \cdot \log |g_{1}(x)|_\nu
            - k_2 \cdot \log |g_{2}(y)|_\nu \\
            &+ 2k_1 \cdot \log^+ |h_{1}(x)|^{-1}_\nu
            +2k_2 \cdot \log^+ |h_{2}(y)|^{-1}_\nu \\
            &+ \oldconstant{c43}{}_{,\nu}(k_1+k_2+1).
        \end{align*}
        
    \end{lemma}

    \begin{proof}
        We set
        \[ \lVert f \rVert_\nu = \max_{q \in \vartheta^{-1}(\overline{D}(\rho^\prime,p))}{\left\{|f(q)|\right\}}.\]
        If $\lVert f \rVert_\nu = 0$, then $f = 0$ and the lemma holds trivially. Otherwise, we set $f_1 = f/\lVert f \rVert_\nu$. Applying Proposition \ref{proposition::archimedean_extension} (with $\rho^\prime/2$ in place of $\rho^\prime$), we extend $f_1$ to a function $\widetilde{f}_1: D(\rho^\prime/2,p) \to \IC$. We set $\widetilde{f} = \lVert f \rVert_\nu \widetilde{f}_1$.
        
        By definition, we have $k_1!k_2!\mathbf{D}^{k_1,k_2} f = (d_{1}^{k_1},d_{2}^{k_2})f$. Using the fact that
        \[ f \in \mathfrak{m}_{C_1,x}^{k_1} \cdot \mathfrak{m}_{C_2,y}^{k_2} + \mathfrak{m}_{C_1,x}^{k_1+1} + \mathfrak{m}_{C_2,y}^{k_2+1} \]
        and the Leibniz rule, we infer that 
        \begin{equation*}
            g_{1} (x)^{k_1} g_{2} (y)^{k_2} k_1!k_2! \mathbf{D}^{k_1,k_2} f(p) = [((g_{1}  d_{1})^{k_1},(g_{2} d_{2})^{k_2})f](p).
        \end{equation*}
        By Proposition \ref{proposition::archimedean_extension} (i), this implies that
        \begin{equation*}
            g_{1} (x)^{k_1} g_{2} (y)^{k_2} k_1!k_2! \mathbf{D}^{k_1,k_2} f(p) = (\delta_1^{k_1},\delta_2^{k_2})\widetilde{f}(p).
        \end{equation*}
        Next, we deduce from Proposition \ref{proposition::archimedean_extension} (ii) that
        \begin{align*}
            (\delta_{1}^{k_1}, \delta_{2}^{k_2}) \widetilde{f}(p) 
            &= \left( \left( \sum_{i=1}^{n_{1}} \delta_{1,i} \cdot \frac{\del}{\del X_i} \right)^{k_1}, \left( \sum_{i=1}^{n_{2}} \delta_{2,i} \cdot \frac{\del}{\del Y_i} \right)^{k_2} \right) \widetilde{f}(p)\\
            &= \left( \left( \sum_{i=1}^{n_{1}} \delta_{1,i}(x) \cdot \frac{\del}{\del X_i} \right)^{k_1}, \left( \sum_{i=1}^{n_{2}} \delta_{2,i}(y) \cdot \frac{\del}{\del Y_i} \right)^{k_2} \right) \widetilde{f}(p).
        \end{align*}
        It follows from \eqref{equation::bound_delta_kappa_i} that
        \[ |(\delta_1^{k_1},\delta_2^{k_2})\widetilde{f}(p)| \leq \oldconstant{C}{}_{,\nu}^{k_1+k_2}n_1^{k_1}n_2^{k_2}\max_{|\underline{m}| = k_1, |\underline{m}^\prime| = k_2}{\left\vert \frac{\del^{k_1+k_2} \widetilde{f}}{\del \underline{X}^{\underline{m}} \del \underline{Y}^{\underline{m}^\prime}}(p)\right\vert}.\]
        By Cauchy's integral formula \cite[Chapter A]{Gunning1990}, we have
        \begin{equation*}
            \left\vert \frac{\del^{k_1+k_2} \widetilde{f}}{\del \underline{X}^{\underline{m}} \del \underline{Y}^{\underline{m}^\prime}}(p) \right\vert \leq (4/\rho^\prime)^{k_1+k_2} k_1! k_2!\sup_{q \in \overline{D}(\rho^\prime/4,p)} |\widetilde{f}(q)|
        \end{equation*}
        for all $\underline{m} \in \IZ_{\geq 0}^{n_1}$ and $\underline{m}^\prime \in \IZ_{\geq 0}^{n_2}$ such that $|\underline{m}| = k_1$ and $|\underline{m}^\prime| = k_2$.
        
        We deduce the desired bound from putting everything together and using the upper bound from Proposition \ref{proposition::archimedean_extension} (iii).
    \end{proof}

\subsection{Global control}

We conclude this section with an elementary lemma controlling the global size of the auxiliary functions $g_\kappa \in \mathcal{B}_\kappa$ and $h_\kappa \in \mathcal{R}_\kappa$ introduced in Lemma \ref{lemma::introduction_g_and_h}. For every $z \in V_\kappa(\IQbar)$, we write $\mathbf{O}_\nu(z)$ for the orbit of $z$ under the Galois group $\mathrm{Gal}(\IQbar/\IK)$ in $V_\kappa(\IKbar_\nu)$ ($\kappa = 1,2$).

    \begin{lemma}
        \label{lemma::aux_function_control}
        Suppose that $C_1, C_2$ are two geometrically irreducible curves inside an abelian variety $A$ and that $L$ is a symmetric ample line bundle on $A$. Let $g_1 \in \mathcal{B}_1$, $g_2 \in \mathcal{B}_2$, $h_1 \in \mathcal{R}_1$, and $h_2 \in \mathcal{R}_2$ be as in Lemma \ref{lemma::introduction_g_and_h}. There exists a constant $\newconstant\label{c47}$ such that
        \begin{equation*}
            \frac{1}{[{\mathbb{K}}(z):\IQ]}\sum_{\nu \in \Sigma({\mathbb{K}})} \delta_\nu \sum_{w \in \mathbf{O}_\nu(z)} \log^+ \vert g_\kappa(w)\vert_\nu^{-1} \leq \oldconstant{c47} \cdot h_{\Ltil}(z) + \oldconstant{c47}
        \end{equation*}
        for all $z \in V_\kappa(\IQbar)$ such that $g_\kappa(z) \neq 0$ and all $\kappa \in \{1,2\}$. 
        
        Similarly, there exists a constant $\newconstant\label{c48}$ such that
        \begin{equation*}
            \frac{1}{[{\mathbb{K}}(z):\IQ]}\sum_{\nu \in \Sigma({\mathbb{K}})} \delta_\nu \sum_{w \in \mathbf{O}_\nu(z)} \log^+ \vert h_\kappa(w)\vert_\nu^{-1} \leq \oldconstant{c48}\cdot h_{\Ltil}(z) + \oldconstant{c48}
        \end{equation*}
        for all $z \in V_\kappa(\IQbar)$ such that $h_\kappa(z) \neq 0$ and all $\kappa \in \{1,2\}$. 
    \end{lemma}

    \begin{proof}
        This follows from \cite[Theorem B.2.5]{Hindry2000}.
    \end{proof}

    \section{Proof of Theorem \ref{theorem::main_reduction}}

    \label{section::proof_general}

    In this section we use notation and terminology introduced in \cite{Habegger2012a} and \cite{BarroeroDill2022}: For a given irreducible subvariety $V$ of an abelian variety $A$, we write $\langle V \rangle$ for the smallest torsion coset containing $V$ and $\langle V \rangle_{\mathrm{ws}}$ for the smallest coset containing $V$. The defect $\delta(V)$ of $V$ is equal to $\dim \langle V \rangle - \dim V$ and its weak defect $\delta_{\mathrm{ws}}(V)$ is equal to $\dim \langle V \rangle_{\mathrm{ws}} - \dim V$. An irreducible subvariety $W \subseteq V$ is called optimal for $V$ in $A$ if $\delta(U) > \delta(W)$ for every irreducible subvariety $U$ such that $W \subsetneq U \subseteq V$ and it is called weakly optimal for $V$ in $A$ if $\delta_{\mathrm{ws}}(U) > \delta_{\mathrm{ws}}(W)$ for every such subvariety $U$. If $X$ is a union of translates of a given abelian subvariety, we let $X^\circ$ denote this abelian subvariety.
    
    Let $K$ be an arbitrary algebraically closed field of characteristic $0$ and let $A$ be an abelian variety over $K$ of dimension $g \geq 3$. In the sequel, we repeatedly use the fact that for every extension $L$ of $K$, all torsion cosets of $A_L$ are base changes of torsion cosets of $A$ since $K$ is algebraically closed.
    
    We want to prove that Theorem \ref{theorem::main_reduction} holds for irreducible algebraic curves $C_1, C_2 \subseteq A$ such that there is no proper algebraic subgroup of $A$ containing both of them. We set $X = \langle C_1 \times C_2 \rangle$. By Lemma \ref{lemma::nocoset}, we may assume that $C_1$ and $C_2$ are not both contained in a translate of the same proper abelian subvariety of $A$, which implies that
    \begin{equation}\label{equation::generate_everything}
        p_1(X) + \ker((p_1)|_X) = A,
    \end{equation}
    where $p_i: A \times A \to A$ denotes the projection to the $i$-th factor ($i = 1,2$). By Lemma \ref{lemma::nocoset2}, we may furthermore assume that $C_1$ and $C_2$ both have finite stabilizer and that neither is contained in a coset of codimension $\geq 2$. This implies that
    \begin{equation}\label{equation::dim_X}
        \dim X \geq \dim p_i(X) + g-1 \geq 2g-2
    \end{equation}
    for $i = 1,2$.

    We have that $[n](C_1) \not\subseteq C_2$ for all but at most three integers $n$ by Lemma \ref{lem:inters_one_n}. It therefore suffices to show that the set of points $p = (x,y) \in (C_1 \times C_2)(K)$ such that $y = [n](x)$ for some integer $n \in \IZ$ with $[n](C_1) \not\subseteq C_2$ is finite.
    
    There exists an algebraically closed subfield $K^\prime \subseteq K$ of finite transcendence degree over $\IQbar$ together with an abelian variety $A^\prime$ and irreducible curves $C^\prime_1$ and $C^\prime_2$, both defined over $K^\prime$, such that $A= A^\prime_K$, $C_1= (C^\prime_1)_K$, and $C_2= (C^\prime_2)_K$. For each integer $n$ such that $[n](C_1) \not\subseteq C_2$, the set $ C_1 \cap [n]^{-1}(C_2)$ is finite and consists of base changes of $K^\prime$-rational points of $C^\prime_1$. It hence suffices to show that the set of points $p = (x,y) \in (C^\prime_1 \times C^\prime_2)(K^\prime)$ such that $y = [n](x)$ for some integer $n \in \IZ$ with $[n](C^\prime_1) \not\subseteq C^\prime_2$  is finite. Therefore, we may and do assume subsequently that $K$ has finite transcendence degree over $\IQbar$.

    Let $\mathscr{B}_{2g} \rightarrow \mathscr{A}_{2g}$ denote the universal family of principally polarized abelian varieties of dimension $2g$ with symplectic level $3$-structure over $\IQbar$ as constructed in \cite[Theorem 2.3.1]{GenestierNgo} (see also \cite[Theorem 7.9 and the following remark]{Mumford1965}). By \cite[Lemma 2.7]{BarroeroDill2022}, there exists an irreducible subvariety $S \subseteq \mathscr{A}_{2g}$ with the following property: let $\mathcal{A} = \mathscr{B}_{2g} \times_{\mathscr{A}_{2g}} S$ and let $\mathcal{A}_\eta$ be the generic fiber of $\mathcal{A} \to S$. There exist a field embedding $\IQbar(S) \hookrightarrow K$ and an isogeny $\varphi: A \times A \to \mathcal{A}_{\eta,K}$.
    
    Let $K_0$ denote the algebraic closure of $\IQbar(S)$ inside $K$. We obtain that $A = (A_0)_K$ for an abelian variety $A_0$ over $K_0$. We proceed by induction on the transcendence degree $\mathrm{trdeg}(K/K_0)$.
   
    \subsection{Base case $\mathrm{trdeg}(K/K_0)=0$}
    \label{section::base_case}

    Since $K_0$ is algebraically closed in $K$, we have that $K = K_0$. For $K= \IQbar$, we can appeal directly to Theorem \ref{theorem::mainQbar}. We hence assume $K\neq \IQbar$ (i.e., $\dim(S) > 0$) in the following.
    
    Consider a point $p = (x,y) \in (C_1 \times C_2)(K)$ such that $y = [n](x)$ for some $n \in \mathbb{Z}$ with $[n](C_1) \not\subseteq C_2$. We want to show that the set of such points is not Zariski dense in $C_1 \times C_2$. The singleton $\{p\}$ is contained in an optimal subvariety $W$ for $C_1 \times C_2$ such that $\delta(W) \leq \delta(\{p\})$. By \cite[Lemmas 2.5 and 2.6~(1)]{BarroeroDill2022}, either $W$ is contained in some fixed proper Zariski closed subset $Z$ of $C_1 \times C_2$ or the subvariety $\varphi(W)$ is also optimal for $V := \varphi(C_1 \times C_2)$ in $\mathcal{A}_{\eta,K}$.
    
    Moreover, $p\in Y_n(K)$, where $Y_n = X \cap (\id_A,[n])(A)$. We have that
    \[ \codim_X Y_n = \dim X - \dim p_1(X) + \codim_{p_1(X)} p_1(Y_n).\]
    It follows from \eqref{equation::dim_X} that $\codim_X Y_n \geq g-1$. Furthermore, if $\codim_X Y_n = g-1$, then $\dim X - \dim p_1(X) = g-1$ and $p_1(X) = p_1(Y_n)$ and so $X$ must contain the graph of multiplication by $n$ restricted to $p_1(X)$. But if this happens for more than one $n$, then it follows from \eqref{equation::generate_everything} that $\dim X - \dim p_1(X) = g$, a contradiction. Since any fixed $n$ with $[n](C_1) \not\subseteq C_2$ gives rise to at most finitely many possibilities for $x$, we can and will assume in the following that $\codim_X Y_n > g-1 \geq 2$.
    
    We deduce that
    \begin{equation}\label{equation::bound_for_delta_W}
    \delta(W) \leq \delta(\{p\}) \leq \dim Y_n < \dim X - 2 = \delta(C_1 \times C_2).
    \end{equation}
    Hence we have that $W \neq C_1 \times C_2$ and so either $W = \{p\}$ or $\dim W = 1$.

    We denote by $(T,\mathrm{Tr}: T_K \rightarrow A)$ the $K/\IQbar$-trace of $A$ as defined in \cite[Definition 2.3]{BarroeroDill2022}. We apply \cite[Theorem 3.1]{BarroeroDill2022} to the subvariety $V \subseteq \mathcal{A}_{\eta,K}$: After potentially enlarging $Z$, we can pick finitely many pairs $(q_0,H)$ consisting of a torsion point $q_0 \in (A \times A)(K)$ and an abelian subvariety $H \subsetneq A \times A$ such that
    \begin{enumerate}
        \item[(i)] $W \subseteq Z \subsetneq C_1 \times C_2$ or
        \item[(ii)] $W \subseteq t + q_0 + H$ for some $t \in \mathrm{Tr}(T(\IQbar)) \times \mathrm{Tr}(T(\IQbar))$ and $H \subseteq \langle W \rangle^\circ$;
    \end{enumerate}
    this choice can be made independently of $p$. Since we have assumed that $K_0 \neq \IQbar$, we know that $\mathrm{Tr}(T_K) \neq A$. In case (i), we are done. We are therefore left with case (ii), which we subdivide further into two subcases
    \begin{enumerate}
        \item[(ii\textsubscript{a})] $\dim(W)=0$ and \item[(ii\textsubscript{b})] $\dim(W)=1$.
    \end{enumerate}

    In case (ii\textsubscript{a}), we have $W = \{p\}$. We know that
    \[ H \subseteq \langle W \rangle^\circ = \langle \{p\} \rangle^\circ \subseteq (\id_A,[n])(A).\]
    If $\dim H > 0$, then the integer $n$ is uniquely determined by $H$. As there are at most finitely many choices for $H$, we are done in this case since $C_1 \cap [n]^{-1}(C_2)$ is finite for each $n$ such that $[n](C_1) \not\subseteq C_2$. If $\dim H = 0$, then $p \in \mathrm{Tr}(T(\IQbar)) \times \mathrm{Tr}(T(\IQbar)) + q_0$ and $q_0$ varies in a finite set. We can find $m \in \mathbb{Z}$, $m\ge 1$, such that $[m](q_0) = 0_{A \times A}$ for all $q_0$ in this finite set. This implies that $[m](p) \in \mathrm{Tr}(T_K)^2 \subsetneq A^2$. Since there is no proper algebraic subgroup of $A$ containing both $C_1$ and $C_2$, we have either $[m](C_1) \not\subseteq \mathrm{Tr}(T_K)$ or $[m](C_2) \not\subseteq \mathrm{Tr}(T_K)$. In the first case, there are at most finitely many possibilities for $x$ while in the second case, there are at most finitely many possibilities for $y$. We are done with (ii\textsubscript{a}).
    
    Let us assume next (ii\textsubscript{b}). In this case, the abelian subvariety $H$ has positive dimension and we consider the quotient homomorphism
    \[ \psi: A \times A \to B := (A \times A)/H.\]
     As $W \subseteq t + q_0 + H$, the image $\psi(W)$ is a singleton. If $\psi(C_1 \times C_2)$ is zero-dimensional, then $\psi(C_1 \times C_2) = \psi(W)$ and so
    \[  C_1 \times C_2 \subseteq \psi^{-1}(\langle \psi(C_1 \times C_2) \rangle) = \psi^{-1}(\langle \psi(W) \rangle) = \langle W \rangle,\]
    where we have used that $H \subseteq \langle W \rangle^{\circ}$. But this contradicts the fact that $W$ is optimal for $C_1 \times C_2$ and $W \subsetneq C_1 \times C_2$. Therefore $\dim \psi(C_1 \times C_2) > 0$.
    
    We note that 
    \begin{align*}
        \delta(\psi(W)) 
        = \dim \langle \psi (W) \rangle 
        \le\dim \psi (\langle W \rangle)
        &= \dim \langle W \rangle - \dim H \\
        &= \delta(W) - \dim H + 1,
    \end{align*}
    because $H\subseteq\langle W \rangle^\circ$.    

    Consider first the case $\dim \psi(C_1 \times C_2) = 1$. It follows from \eqref{equation::bound_for_delta_W} that
    \begin{align*}
    \delta(\{\psi(p)\}) = \delta(\psi(W)) &\le  \delta(W) - \dim H + 1 \\
    &< \dim X - \dim H - 1 = \delta(\psi(C_1 \times C_2)),
    \end{align*}
    where we have used that $H \subseteq \langle W \rangle^\circ \subseteq X^\circ$.
    
    Hence, ${\psi(p)}$ is optimal for $\psi(C_1 \times C_2)$ in $B$ and then, by the Zilber-Pink conjecture for curves in abelian varieties in the form of \cite[Corollary 1.6]{BarroeroDill2022}, there are at most finitely many possibilities for $\psi(p)$. This again forces $p$ to lie in a proper Zariski closed subset of $C_1 \times C_2$ and we are done in this case.

    Assume next $\dim \psi(C_1 \times C_2) = 2$. By the Fiber Dimension Theorem \cite[Theorem 7.3]{BD2}, the variety $W$ is then contained in a proper Zariski closed subset of $C_1 \times C_2$, which can be chosen independently of $p$.
    
    This completes the proof of the base step of the induction, where $K = K_0$.

    \subsection{Inductive step}

    Consider next an algebraically closed field $K$ with $\mathrm{trdeg}(K/K_0) > 0$. Recall that $A = (A_0)_K$ for an abelian variety $A_0$ over $K_0$. We choose an algebraically closed subfield $L \subseteq K$ containing $K_0$ such that $\mathrm{trdeg}(L/K_0) = \mathrm{trdeg}(K/K_0)-1$. By our induction hypothesis together with the reductions we made, we can assume that the assertion of Theorem \ref{theorem::main_reduction} is already proven for curves in $(A_0)_L$ and it suffices to show that the set of points $p = (x,y) \in (C_1 \times C_2)(K)$ such that $y = [n](x)$ for some integer $n \in \IZ$ with $[n](C_1) \not\subseteq C_2$ is not Zariski dense for irreducible curves $C_1, C_2 \subseteq A$ that are not both contained in the same proper algebraic subgroup of $A$. Recall that we may and do assume that $C_1$ and $C_2$ are not both contained in a translate of the same proper abelian subvariety of $A$ and that neither of them is contained in a coset of codimension $\geq 2$. To simplify our notation, we write $A$ for $(A_0)_L$ and $A_K$ for $(A_0)_K$ in the following.
    
    Assume first that both $C_1$ and $C_2$ are equal to base changes $(C'_i)_{K}$ of some curves $C'_i \subseteq A$ ($i = 1,2$). For every integer $n$ with $[n](C_1) \not\subseteq C_2$, the set $C_1 \cap [n]^{-1}(C_2)$ is finite. As $L$ is algebraically closed, we have that $C_1 \cap [n]^{-1}(C_2) \subseteq A(L)$ and we are done by the inductive hypothesis. (We do not invoke the inductive hypothesis elsewhere in the proof.)
    
    In the sequel, we may and do assume that $C_1$ and $C_2$ are not both base changes $(C'_i)_{K}$ of some curves $C'_i \subseteq A$ ($i = 1,2$).
    This implies that the set of points $p = (x,y) \in (C_1 \times C_2)(K)$ that are base changes of $L$-rational points is not Zariski dense in $C_1 \times C_2$.
    
    We are left with proving that the set of points $p = (x,y) \in (C_1 \times C_2)(K)\setminus (A \times A)(L)$ such that $y = [n](x)$ for some $n \in \mathbb{Z}$ with $[n](C_1) \not\subseteq C_2$ is not Zariski dense in $C_1 \times C_2$. Let $p$ denote such a point. We can apply \cite[Lemma 2.2]{BarroeroDill2022} to find an irreducible subvariety $V \subseteq A \times A$ such that $\dim V \leq 3$ and $C_1 \times C_2 \subseteq V_K$. Since $C_1$ and $C_2$ are not both base changes $(C'_i)_{K}$ of some curves $C'_i \subseteq A$ ($i = 1,2$), we must in fact have $\dim V = 3$ and $V$ is minimal with respect to inclusion such that $C_1 \times C_2 \subseteq V_K$. Recall that $X = \langle C_1 \times C_2 \rangle$ and therefore $X = \langle V \rangle$ by the minimality of $V$.
    
    Applying \cite[Lemma 2.2]{BarroeroDill2022} once more, we find an irreducible subvariety $W \subseteq V$ such that $\dim W \leq 1$ and $p \in W_K$. We must also have $\dim W = 1$ as $p$ would be $L$-rational otherwise. Since $p$ is not $L$-rational, the subvariety $W$ is minimal such that $p \in W_K$; in particular, $W \subseteq \langle \{p\} \rangle$ and therefore $\langle W \rangle = \langle \{p\} \rangle$.

    We set $Y_n = X \cap (\id_A, [n])(A)$ so that $\langle \{p\} \rangle \subseteq Y_n$. As in the proof of the base case, we can and will assume that $\codim_X Y_n > g-1 \geq 2$.
    
    There exists an optimal subvariety $W' \subseteq V$ such that $W \subseteq W'$ and $\delta(W') \leq \delta(W)$. By \cite[Proposition 4.5]{Habegger2016}, the subvariety $W'$ is also weakly optimal for $V$ in $A \times A$; we note that the notion of geodesic optimality used in \cite{Habegger2016} is equivalent to weak optimality. Since
    \begin{align}\label{equation::bound_for_delta_w_prime}
    \delta(W') \leq \delta(W) \leq \dim \langle \{p\} \rangle - 1 &\leq \dim Y_n - 1 \\ \nonumber
    &< \dim X - 3 = \delta(V),
    \end{align}
    we have $W' \neq V$ and thus $\dim W' \in \{1, 2\}$.

    An application of \cite[Lemme 2.6 and Proposition 3.2]{Remond2009} in the form of \cite[Proposition 6.1]{Habegger2016} to $V$ yields a finite set $\mathcal{S}$ of abelian subvarieties of $A \times A$, independent of $p$ or $W$, such that $\langle W' \rangle^\circ_{\mathrm{ws}} \in \mathcal{S}$. For ease of notation, we set $H = \langle W' \rangle^\circ_{\mathrm{ws}} \subseteq X^\circ$. We have $H \neq X^\circ$ since otherwise $\delta(W') \geq \dim X - 2$, contradicting \eqref{equation::bound_for_delta_w_prime}.

    Assume first that $\dim W' = 2$. Then, we consider the quotient homomorphism
    \[ \psi: A \times A \to B := (A \times A)/H.\]
    If $\psi(C_1 \times C_2)$ is zero-dimensional, then $\psi(C_1 \times C_2) = \psi(W')$ and so
    \[ \langle V \rangle = X = \psi^{-1}(\langle \psi(C_1 \times C_2) \rangle) = \psi^{-1}(\langle \psi(W') \rangle) = \langle W' \rangle,\]
    contradicting the optimality of $W' \subsetneq V$ for $V$. Since $X = \langle V \rangle$ and $H \subseteq X^\circ$, we must have $\psi(X) = \langle \psi(V) \rangle$ and
    \[\delta(\psi(V)) = \dim X - \dim(H) - \dim \psi(V).\]
    
    If $\dim \psi(V) \geq 2$, then, by the Fiber Dimension Theorem \cite[Theorem 7.3]{BD2}, $W'$ is contained in a proper Zariski closed subset $Z$ of $V$, which is defined over $L$ and independent of $p$ or $W$. We must have $Z_K \cap (C_1 \times C_2) \subsetneq C_1 \times C_2$ since otherwise $C_1 \times C_2$ would be equal to the base change of an irreducible component of $Z$ for dimension reasons. Therefore, the point $p$ belongs to a proper Zariski closed subset of $C_1 \times C_2$ and we are done.
    
    In case $\dim \psi(V) = 1$, we find that
    \begin{align*} \delta(\{\psi(p)\}) = \delta(\psi(W')) &\leq \dim \psi(\langle W' \rangle) \\
    &= \dim \langle W' \rangle - \dim H \\
    &= \delta(W') - \dim H + 2\\
    &\leq \delta(W) - \dim H + 2 \\
    &< \dim X - \dim H -1 \\ 
    &= \delta(\psi(V)),
    \end{align*}
    where we have used \eqref{equation::bound_for_delta_w_prime} for the last inequality. We infer that $\psi(p)$ is optimal for $\psi(V)$ in $B$. By the Zilber-Pink conjecture for curves in abelian varieties in the form of \cite[Corollary 1.6]{BarroeroDill2022}, there are at most finitely many possibilities for $\psi(p)$. As $\dim(\psi(C_1\times C_2)) = \dim(\psi(V))=1$, this forces $p$ to lie in a proper Zariski closed subset of $C_1 \times C_2$ and we are done.

    Lastly, let us consider the case $\dim(W^\prime)=1$. By irreducibility, we have $W = W^\prime$ and
    \[ H = \langle W'\rangle_{\mathrm{ws}}^\circ = \langle W\rangle_{\mathrm{ws}}^\circ \subseteq \langle W \rangle^\circ = \langle \{p\} \rangle^\circ \subseteq (\id_A,[n])(A). \]
    Since $\dim H \geq \dim W' > 0$, it follows that $n$ is uniquely determined by $H$. This leaves only finitely many candidates for the point $p$, completing the proof.

    \section{Disclosure of AI assistance}
    All the main mathematical ideas in this article were developed by the authors since 2022 without AI assistance. During the preparation of this article, GPT-5.5 and GPT-5.6 were used for literature research as well as for the conceptualization of the technical Lemma \ref{lemma::stabilizer_of_difference}. After completing a first version of the manuscript and uploading it to arXiv, the authors used Claude Fable 5.1 to proofread the article and propose various minor corrections. The authors prepared the final manuscript and take full
responsibility for its content and mathematical correctness.
    
\section{Acknowledgements}
The very first ideas that went into this project were developed in October 2022 in Rome while L.K.~was visiting F.B.~with the support of the ``National Group for Algebraic and Geometric Structures, and their Applications'' (GNSAGA – INDAM). During the Spring 2023 semester, G.D.~and L.K.~were in residence at the Mathematical Sciences Research Institute in Berkeley, California and their work was supported by the National Science Foundation under grant no.\ DMS--1928930.
The authors thank the DFG (grant no.\ EXC-2047/1 - 390685813) and the Mathematical Institute of the University of Bonn for their support to F.B.~and L.K.~during their visits to Bonn in February and June 2024.
F.B.~was supported by the PRIN 2022 project ``2022\-HPSNCR: Semi-abelian varieties, Galois representations and related Diophantine problems'' and by the project SERICS (PE00000014)
under the MUR National Recovery and Resilience Plan funded by the European Union - NextGenerationEU.

\vspace{\baselineskip}
\noindent
\framebox[\textwidth]{
\begin{tabular*}{0.96\textwidth}{@{\extracolsep{\fill} }cp{0.84\textwidth}}
		\raisebox{-0.7\height}{%
			\begin{tikzpicture}[y=0.80pt, x=0.8pt, yscale=-1, inner sep=0pt, outer sep=0pt, 
			scale=0.12]
			\definecolor{c003399}{RGB}{0,51,153}
			\definecolor{cffcc00}{RGB}{255,204,0}
			\begin{scope}[shift={(0,-872.36218)}]
			\path[shift={(0,872.36218)},fill=c003399,nonzero rule] (0.0000,0.0000) rectangle (270.0000,180.0000);
			\foreach \myshift in 
			{(0,812.36218), (0,932.36218), 
				(60.0,872.36218), (-60.0,872.36218), 
				(30.0,820.36218), (-30.0,820.36218),
				(30.0,924.36218), (-30.0,924.36218),
				(-52.0,842.36218), (52.0,842.36218), 
				(52.0,902.36218), (-52.0,902.36218)}
			\path[shift=\myshift,fill=cffcc00,nonzero rule] (135.0000,80.0000) -- (137.2453,86.9096) -- (144.5106,86.9098) -- (138.6330,91.1804) -- (140.8778,98.0902) -- (135.0000,93.8200) -- (129.1222,98.0902) -- (131.3670,91.1804) -- (125.4894,86.9098) -- (132.7547,86.9096) -- cycle;
			\end{scope}
			\end{tikzpicture}%
		}
		&
		G.D.~ has received funding from the European Research Council (ERC) under the European Union's Horizon 2020 research and innovation programme (grant agreement n$^\circ$ 945714).
	\end{tabular*}
}

\bibliographystyle{amsalpha}
\bibliography{references}

@book {Yuan2026,
    AUTHOR = {Yuan, Xinyi and Zhang, Shou-Wu},
     TITLE = {Adelic line bundles on quasi-projective varieties},
    SERIES = {Annals of Mathematics Studies},
    VOLUME = {221},
 PUBLISHER = {Princeton University Press, Princeton, NJ},
      YEAR = {2026},
     PAGES = {x+270},
      ISBN = {978-0-691-27173-6; 978-0-691-27172-9; 978-0-691-27870-4},
   MRCLASS = {14G40 (11G50 14G22)},
  MRNUMBER = {5035732},
}

@book {Hormander2003,
    AUTHOR = {H{\"o}rmander, Lars},
     TITLE = {The analysis of linear partial differential operators. {I}},
    SERIES = {Classics in Mathematics},
      NOTE = {Distribution theory and Fourier analysis,
              Reprint of the second (1990) edition [Springer, Berlin]},
 PUBLISHER = {Springer-Verlag, Berlin},
      YEAR = {2003},
     PAGES = {x+440},
      ISBN = {3-540-00662-1},
   MRCLASS = {35-02},
  MRNUMBER = {1996773},
       DOI = {10.1007/978-3-642-61497-2},
       URL = {https://doi.org/10.1007/978-3-642-61497-2},
}

@incollection{MumfordCIME,
	author = {Mumford, David},
	booktitle = {Questions on {A}lgebraic {V}arieties ({C}.{I}.{M}.{E}., {III} {C}iclo, {V}arenna, 1969)},
	mrclass = {14.01},
	mrnumber = {0282975},
	mrreviewer = {Y. Nakai},
	note = {With an appendix by {G}. {K}empf},
	pages = {29--100},
	publisher = {Edizioni Cremonese, Rome},
	title = {Varieties defined by quadratic equations},
	year = {1970}}

@Article{Bays2015,
  author     = {Bays, Martin and Habegger, Philipp},
  title      = {A note on divisible points of curves},
  journal    = {Trans. Amer. Math. Soc.},
  year       = {2015},
  volume     = {367},
  number     = {2},
  pages      = {1313--1328},
  issn       = {0002-9947},
  doi        = {10.1090/S0002-9947-2014-06494-5},
  fjournal   = {Transactions of the American Mathematical Society},
  mrclass    = {14H25 (03C64 11G50 11J86 11U09)},
  mrnumber   = {3280045},
  mrreviewer = {Yasuhiro Goto},
  url        = {https://doi.org/10.1090/S0002-9947-2014-06494-5},
}

@Article{Bombieri1999,
  author     = {Bombieri, Enrico and Masser, David and Zannier, Umberto},
  journal    = {Internat. Math. Res. Notices},
  title      = {Intersecting a curve with algebraic subgroups of multiplicative groups},
  year       = {1999},
  issn       = {1073-7928},
  number     = {20},
  pages      = {1119--1140},
  doi        = {10.1155/S1073792899000628},
  fjournal   = {International Mathematics Research Notices},
  mrclass    = {11J95 (11G30 11G50)},
  mrnumber   = {1728021 (2001c:11081)},
  mrreviewer = {Dimitrios Poulakis},
  url        = {http://dx.doi.org/10.1155/S1073792899000628},
}

@book{Bombieri2006,
	address = {Cambridge},
	author = {Bombieri, Enrico and Gubler, Walter},
	doi = {10.1017/CBO9780511542879},
	isbn = {978-0-521-84615-8; 0-521-84615-3},
	mrclass = {11G50 (11-02 11G10 11G30 11J68 14G40)},
	mrnumber = {2216774 (2007a:11092)},
	mrreviewer = {Yuri Bilu},
	pages = {xvi+652},
	publisher = {Cambridge University Press},
	series = {New Mathematical Monographs},
	title = {Heights in {D}iophantine geometry},
	url = {http://dx.doi.org/10.1017/CBO9780511542879},
	volume = {4},
	year = {2006}
}

@Article{Bombieri2010,
  author     = {Bombieri, Enrico and Habegger, Philipp and Masser, David and Zannier, Umberto},
  journal    = {Atti Accad. Naz. Lincei Cl. Sci. fis. Mat. Natur. Rend. Lincei (9) Mat. Appl.},
  title      = {A note on {M}aurin's theorem},
  year       = {2010},
  issn       = {1120-6330},
  number     = {3},
  pages      = {251--260},
  volume     = {21},
  doi        = {10.4171/RLM/570},
  fjournal   = {Atti della Accademia Nazionale dei Lincei. Classe di Scienze fisiche, Matematiche e Naturali. Rendiconti Lincei. Serie IX. Matematica e Applicazioni},
  mrclass    = {11G30 (11G35 11G50)},
  mrnumber   = {2677603 (2011f:11078)},
  mrreviewer = {Dimitrios Poulakis},
  url        = {http://dx.doi.org/10.4171/RLM/570},
}

@Article{Bost1994,
  author     = {Bost, Jean-Beno\^{i}t and Gillet, Henri and Soul\'e, Christophe},
  journal    = {J. Amer. Math. Soc.},
  title      = {Heights of projective varieties and positive {G}reen forms},
  year       = {1994},
  issn       = {0894-0347},
  number     = {4},
  pages      = {903--1027},
  volume     = {7},
  doi        = {10.2307/2152736},
  fjournal   = {Journal of the American Mathematical Society},
  mrclass    = {14G40 (11G35 14C17)},
  mrnumber   = {1260106 (95j:14025)},
  mrreviewer = {Yoichi Miyaoka},
  url        = {http://dx.doi.org/10.2307/2152736},
}

@article {David2002,
    AUTHOR = {David, Sinnou and Philippon, Patrice},
     TITLE = {Minorations des hauteurs normalis\'{e}es des sous-vari\'{e}t\'{e}s de
              vari\'{e}t\'{e}s abeliennes. {II}},
   JOURNAL = {Comment. Math. Helv.},
  FJOURNAL = {Commentarii Mathematici Helvetici},
    VOLUME = {77},
      YEAR = {2002},
    NUMBER = {4},
     PAGES = {639--700},
      ISSN = {0010-2571},
   MRCLASS = {11G50 (11G10 11J95 14G40)},
  MRNUMBER = {1949109},
MRREVIEWER = {Timothy D. Browning},
       DOI = {10.1007/PL00012437},
       URL = {https://doi.org/10.1007/PL00012437},
}

@misc{Demailly2012,
	author = {Demailly, Jean-Pierre},
	howpublished = {available from \url{http://www-fourier.ujf-grenoble.fr/~demailly/}},
	title = {Complex Analytic and Differential Geometry},
	year = {2012}
}

@incollection {GenestierNgo,
    AUTHOR = {Genestier, Alain and Ng\^{o}, Bao Ch\^{a}u},
     TITLE = {Lectures on {S}himura varieties},
 BOOKTITLE = {Autour des motifs---\'{E}cole d'\'{e}t\'{e} {F}ranco-{A}siatique de
              {G}\'{e}om\'{e}trie {A}lg\'{e}brique et de {T}h\'{e}orie des
              {N}ombres/{A}sian-{F}rench {S}ummer {S}chool on {A}lgebraic
              {G}eometry and {N}umber {T}heory. {V}olume {I}},
    SERIES = {Panor. Synth\`eses},
    VOLUME = {29},
     PAGES = {187--236},
 PUBLISHER = {Soc. Math. France, Paris},
      YEAR = {2009},
}

@Article{EGA2,
  author   = {Grothendieck, Alexander},
  journal  = {Inst. Hautes Études Sci. Publ. Math.},
  title    = {{É}lements de géométrie algébrique. {II}. \'{E}tude globale élémentaire de quelques classes de morphismes},
  year     = {1961},
  issn     = {0073-8301},
  number   = {8},
  pages    = {5--222},
  fjournal = {Institut des Hautes Études Scientifiques. Publications Math\'ematiques},
  mrclass  = {14.05},
  mrnumber = {0163909},
}

@Article{Faltings1991,
  author     = {Faltings, Gerd},
  journal    = {Ann. of Math. (2)},
  title      = {Diophantine approximation on abelian varieties},
  year       = {1991},
  issn       = {0003-486X},
  number     = {3},
  pages      = {549--576},
  volume     = {133},
  coden      = {ANMAAH},
  doi        = {10.2307/2944319},
  fjournal   = {Annals of Mathematics. Second Series},
  mrclass    = {11G10 (11J99)},
  mrnumber   = {1109353 (93d:11066)},
  mrreviewer = {Joseph H. Silverman},
  url        = {http://dx.doi.org/10.2307/2944319},
}

@InCollection{Faltings1994,
  author     = {Faltings, Gerd},
  booktitle  = {Barsotti {S}ymposium in {A}lgebraic {G}eometry ({A}bano {T}erme, 1991)},
  publisher  = {Academic Press},
  title      = {The general case of {S}. {L}ang's conjecture},
  year       = {1994},
  address    = {San Diego, CA},
  pages      = {175--182},
  series     = {Perspect. Math.},
  volume     = {15},
  mrclass    = {11G10 (14G05 14K99)},
  mrnumber   = {1307396 (95m:11061)},
  mrreviewer = {Juan-Carlos Naranjo},
}

@book{Fulton1998,
	author = {Fulton, William},
	doi = {10.1007/978-1-4612-1700-8},
	edition = {Second},
	isbn = {3-540-62046-X; 0-387-98549-2},
	mrclass = {14C17 (14-02)},
	mrnumber = {1644323},
	pages = {xiv+470},
	publisher = {Springer-Verlag, Berlin},
	series = {Ergebnisse der Mathematik und ihrer Grenzgebiete. 3. Folge. A Series of Modern Surveys in Mathematics [Results in Mathematics and Related Areas. 3rd Series. A Series of Modern Surveys in Mathematics]},
	title = {Intersection theory},
	url = {http://dx.doi.org/10.1007/978-1-4612-1700-8},
	volume = {2},
	year = {1998}
}

@Article{Gao2018,
  author   = {Gao, Ziyang},
  title    = {Mixed {A}x-{S}chanuel for the universal abelian varieties and some applications},
  journal  = {Compos. Math.},
  year     = {2020},
  volume   = {156},
  number   = {11},
  pages    = {2263--2297},
  issn     = {0010-437X},
  doi      = {10.1112/S0010437X20007447},
  fjournal = {Compositio Mathematica},
  mrclass  = {14G35 (11G10 11G18)},
  mrnumber = {4190045},
  url      = {https://doi.org/10.1112/S0010437X20007447},
}

@Article{Hindry1988,
  author     = {Hindry, Marc},
  journal    = {Invent. Math.},
  title      = {Autour d'une conjecture de {S}erge {L}ang},
  year       = {1988},
  issn       = {0020-9910},
  number     = {3},
  pages      = {575--603},
  volume     = {94},
  coden      = {INVMBH},
  doi        = {10.1007/BF01394276},
  fjournal   = {Inventiones Mathematicae},
  mrclass    = {11G10 (11J89 14K15)},
  mrnumber   = {969244 (89k:11046)},
  mrreviewer = {Michel Laurent},
  url        = {http://dx.doi.org/10.1007/BF01394276},
}

@book{Hindry2000,
	author = {Hindry, Marc and Silverman, Joseph H.},
	doi = {10.1007/978-1-4612-1210-2},
	isbn = {0-387-98975-7; 0-387-98981-1},
	mrclass = {11Gxx (11-02 11G10 11G30 11G50 14G25)},
	mrnumber = {1745599},
	mrreviewer = {Dino J. Lorenzini},
	note = {An introduction},
	pages = {xiv+558},
	publisher = {Springer-Verlag, New York},
	series = {Graduate Texts in Mathematics},
	title = {Diophantine geometry},
	url = {http://dx.doi.org/10.1007/978-1-4612-1210-2},
	volume = {201},
	year = {2000}
}

@book{Lazarsfeld2004,
	author = {Lazarsfeld, Robert},
	doi = {10.1007/978-3-642-18808-4},
	isbn = {3-540-22533-1},
	mrclass = {14-02 (14C20)},
	mrnumber = {2095471},
	mrreviewer = {Mihnea Popa},
	note = {Classical setting: line bundles and linear series},
	pages = {xviii+387},
	publisher = {Springer-Verlag, Berlin},
	series = {Ergebnisse der Mathematik und ihrer Grenzgebiete. 3. Folge. A Series of Modern Surveys in Mathematics [Results in Mathematics and Related Areas. 3rd Series. A Series of Modern Surveys in Mathematics]},
	title = {Positivity in algebraic geometry. {I}},
	url = {http://dx.doi.org/10.1007/978-3-642-18808-4},
	volume = {48},
	year = {2004}
}

@misc{Masser,
	author = {Masser, David and Zannier, Umberto},
	howpublished = {preprint},
	title = {Torsion points on families of simple abelian surfaces and {P}ell's equation over polynomial rings}
}

@Article{Maurin2008,
  author     = {Maurin, Guillaume},
  journal    = {Math. Ann.},
  title      = {Courbes alg\'ebriques et \'equations multiplicatives},
  year       = {2008},
  issn       = {0025-5831},
  number     = {4},
  pages      = {789--824},
  volume     = {341},
  coden      = {MAANA},
  doi        = {10.1007/s00208-008-0212-9},
  fjournal   = {Mathematische Annalen},
  mrclass    = {14G25 (14G05 14H25)},
  mrnumber   = {2407327 (2009g:14026)},
  mrreviewer = {Boris √à. Kunyavskiƒ≠},
  url        = {http://dx.doi.org/10.1007/s00208-008-0212-9},
}

@Article{McQuillan1995,
  author     = {McQuillan, Michael},
  journal    = {Invent. Math.},
  title      = {Division points on semi-abelian varieties},
  year       = {1995},
  issn       = {0020-9910},
  number     = {1},
  pages      = {143--159},
  volume     = {120},
  coden      = {INVMBH},
  doi        = {10.1007/BF01241125},
  fjournal   = {Inventiones Mathematicae},
  mrclass    = {14G05 (14K05)},
  mrnumber   = {1323985 (96b:14020)},
  mrreviewer = {Alexandru Buium},
  url        = {http://dx.doi.org/10.1007/BF01241125},
}

@book{Mumford1965,
	author = {Mumford, David},
	mrclass = {14.55},
	mrnumber = {0214602},
	mrreviewer = {F. Oort},
	pages = {vi+145},
	publisher = {Springer-Verlag, Berlin-New York},
	series = {Ergebnisse der Mathematik und ihrer Grenzgebiete, Neue Folge, Band 34},
	title = {Geometric invariant theory},
	year = {1965}
}

@Book{Mumford1970,
  author     = {Mumford, David},
  publisher  = {Published for the Tata Institute of Fundamental Research, Bombay},
  title      = {Abelian varieties},
  year       = {1970},
  series     = {Tata Institute of Fundamental Research studies in Mathematics, No. 5},
  mrclass    = {14.51},
  mrnumber   = {0282985 (44 \#219)},
  mrreviewer = {J. S. Milne},
  pages      = {viii+242},
}

@Misc{Pink2005a,
  author       = {Pink, Richard},
  howpublished = {available from \url{http://www.math.ethz.ch/~pink}},
  month        = {04},
  title        = {A common generalization of the conjectures of {A}ndr\'e-{O}ort, {M}anin-{M}umford, and {M}ordell-{L}ang},
  year         = {2005},
}

@Article{Raynaud1983,
  author     = {Raynaud, Michel},
  journal    = {Invent. Math.},
  title      = {Courbes sur une vari\'et\'e ab\'elienne et points de torsion},
  year       = {1983},
  issn       = {0020-9910},
  number     = {1},
  pages      = {207--233},
  volume     = {71},
  coden      = {INVMBH},
  doi        = {10.1007/BF01393342},
  fjournal   = {Inventiones Mathematicae},
  mrclass    = {14H25 (14G25 14K15)},
  mrnumber   = {688265 (84c:14021)},
  mrreviewer = {F. Oort},
  url        = {http://dx.doi.org/10.1007/BF01393342},
}

@Article{Remond2007,
  author     = {R{\'e}mond, Ga\"{e}l},
  journal    = {J. Inst. Math. Jussieu},
  title      = {Intersection de sous-groupes et de sous-vari\'et\'es. {II}},
  year       = {2007},
  issn       = {1474-7480},
  number     = {2},
  pages      = {317--348},
  volume     = {6},
  doi        = {10.1017/S1474748006000107},
  fjournal   = {Journal of the Institute of Mathematics of Jussieu. JIMJ. Journal de l'Institut de Math\'ematiques de Jussieu},
  mrclass    = {11G50 (11G10 11G35 14K12)},
  mrnumber   = {2311666},
  mrreviewer = {Alexander A. Borisov},
  url        = {http://dx.doi.org/10.1017/S1474748006000107},
}

@Article{Remond2009,
  author     = {R{\'e}mond, Ga\"{e}l},
  journal    = {Comment. Math. Helv.},
  title      = {Intersection de sous-groupes et de sous-vari\'et\'es. {III}},
  year       = {2009},
  issn       = {0010-2571},
  number     = {4},
  pages      = {835--863},
  volume     = {84},
  doi        = {10.4171/CMH/183},
  fjournal   = {Commentarii Mathematici Helvetici. A Journal of the Swiss Mathematical Society},
  mrclass    = {11G50 (11G10 11J95 14K12)},
  mrnumber   = {2534482 (2010i:11098)},
  mrreviewer = {Oscar G. Villareal},
  url        = {http://dx.doi.org/10.4171/CMH/183},
}

@misc{stacksProjectAuthors2015,
	author = {{Stacks Project Authors}, The},
	howpublished = {\url{http://stacks.math.columbia.edu}},
	shorthand = {Stacks},
	title = {{\itshape Stacks Project}}
}

@Article{Ullmo1998,
  author     = {Ullmo, Emmanuel},
  journal    = {Ann. of Math. (2)},
  title      = {Positivit\'e et discr\'etion des points alg\'ebriques des courbes},
  year       = {1998},
  issn       = {0003-486X},
  number     = {1},
  pages      = {167--179},
  volume     = {147},
  coden      = {ANMAAH},
  doi        = {10.2307/120987},
  fjournal   = {Annals of Mathematics. Second Series},
  mrclass    = {14G40 (11G30 14H25 14H40)},
  mrnumber   = {1609514 (99e:14031)},
  mrreviewer = {Yuichiro Takeda},
  url        = {http://dx.doi.org/10.2307/120987},
}

@Article{Vojta1991,
  author     = {Vojta, Paul},
  journal    = {Ann. of Math. (2)},
  title      = {Siegel's theorem in the compact case},
  year       = {1991},
  issn       = {0003-486X},
  number     = {3},
  pages      = {509--548},
  volume     = {133},
  coden      = {ANMAAH},
  doi        = {10.2307/2944318},
  fjournal   = {Annals of Mathematics. Second Series},
  mrclass    = {11G30 (11J68 14G25)},
  mrnumber   = {1109352 (93d:11065)},
  mrreviewer = {Joseph H. Silverman},
  url        = {http://dx.doi.org/10.2307/2944318},
}

@book{Zannier2012,
	address = {Princeton, NJ},
	author = {Zannier, Umberto},
	isbn = {978-0-691-15371-1},
	mrclass = {11G35 (11G15 11G18 11G50 14G05)},
	mrnumber = {2918151},
	mrreviewer = {Yuri Bilu},
	note = {With appendixes by David Masser},
	pages = {xiv+160},
	publisher = {Princeton University Press},
	series = {Annals of Mathematics studies},
	title = {Some problems of unlikely intersections in arithmetic and geometry},
	volume = {181},
	year = {2012}
}

@Article{Zhang1998,
  author     = {Zhang, Shou-Wu},
  journal    = {Ann. of Math. (2)},
  title      = {Equidistribution of small points on abelian varieties},
  year       = {1998},
  issn       = {0003-486X},
  number     = {1},
  pages      = {159--165},
  volume     = {147},
  coden      = {ANMAAH},
  doi        = {10.2307/120986},
  fjournal   = {Annals of Mathematics. Second Series},
  mrclass    = {14G40 (11G10 14G05 14K15)},
  mrnumber   = {1609518 (99e:14032)},
  mrreviewer = {Yuichiro Takeda},
  url        = {http://dx.doi.org/10.2307/120986},
}

@Article{Zilber2002,
  author     = {Zilber, Boris},
  journal    = {J. London Math. Soc. (2)},
  title      = {Exponential sums equations and the {S}chanuel conjecture},
  year       = {2002},
  issn       = {0024-6107},
  number     = {1},
  pages      = {27--44},
  volume     = {65},
  coden      = {JLMSAK},
  doi        = {10.1112/S0024610701002861},
  fjournal   = {Journal of the London Mathematical Society. Second Series},
  mrclass    = {11U09 (03C65 11J99 11L99)},
  mrnumber   = {1875133 (2002m:11104)},
  mrreviewer = {Ricardo Bianconi},
  url        = {http://dx.doi.org/10.1112/S0024610701002861},
}

@Article{EGA4IV,
  author     = {Grothendieck, Alexander},
  journal    = {Inst. Hautes Études Sci. Publ. Math.},
  title      = {{É}lements de géométrie algébrique. {IV}. \'{E}tude locale des schémas et des morphismes de schémas {IV}},
  year       = {1967},
  issn       = {0073-8301},
  number     = {32},
  pages      = {5--361},
  fjournal   = {Institut des Hautes Études Scientifiques. Publications Math\'ematiques},
  mrclass    = {14.55},
  mrnumber   = {0238860},
  mrreviewer = {J. P. Murre},
}

@Article{Kleiman1966,
  author     = {Kleiman, Steven L.},
  journal    = {Ann. of Math. (2)},
  title      = {Toward a numerical theory of ampleness},
  year       = {1966},
  issn       = {0003-486X},
  pages      = {293--344},
  volume     = {84},
  fjournal   = {Annals of Mathematics. Second Series},
  mrclass    = {14.55 (14.10)},
  mrnumber   = {0206009},
  mrreviewer = {J. W. Gray},
}

@Article{EGA4III,
  author     = {Grothendieck, Alexander},
  journal    = {Inst. Hautes Études Sci. Publ. Math.},
  title      = {{\'E}l\'ements de g\'eom\'etrie alg\'ebrique. {IV}. \'{E}tude locale des sch\'emas et des morphismes de sch\'emas. {III}},
  year       = {1966},
  issn       = {0073-8301},
  number     = {28},
  pages      = {5--255},
  fjournal   = {Institut des Hautes Études Scientifiques. Publications Math\'ematiques},
  mrclass    = {14.55},
  mrnumber   = {0217086},
  mrreviewer = {J. P. Murre},
}

@Article{Zhang1995a,
  author     = {Zhang, Shou-Wu},
  journal    = {J. Algebraic Geom.},
  title      = {Small points and adelic metrics},
  year       = {1995},
  issn       = {1056-3911},
  number     = {2},
  pages      = {281--300},
  volume     = {4},
  fjournal   = {Journal of Algebraic Geometry},
  mrclass    = {14G40},
  mrnumber   = {1311351},
  mrreviewer = {Antoine Chambert-Loir},
}

@Article{Chambert-Loir2000,
  author     = {Chambert-Loir, Antoine},
  journal    = {Ann. Sci. École Norm. Sup. (4)},
  title      = {Points de petite hauteur sur les vari\'et\'es semi-ab\'eliennes},
  year       = {2000},
  issn       = {0012-9593},
  number     = {6},
  pages      = {789--821},
  volume     = {33},
  coden      = {ASENAH},
  doi        = {10.1016/S0012-9593(00)01053-3},
  fjournal   = {Annales Scientifiques de l'√âcole Normale Sup\'erieure. Quatri√®me S\'erie},
  mrclass    = {14G40 (11G50 14K15)},
  mrnumber   = {1832991},
  mrreviewer = {Annette Werner},
  url        = {http://dx.doi.org/10.1016/S0012-9593(00)01053-3},
}

@book{Lang1983,
	author = {Lang, Serge},
	doi = {10.1007/978-1-4757-1810-2},
	isbn = {0-387-90837-4},
	mrclass = {11-02 (11Dxx 11Gxx 14G25)},
	mrnumber = {715605},
	mrreviewer = {Gerd Faltings},
	pages = {xviii+370},
	publisher = {Springer-Verlag, New York},
	title = {Fundamentals of {D}iophantine geometry},
	url = {http://dx.doi.org/10.1007/978-1-4757-1810-2},
	year = {1983}
}

@book{Gunning1990,
	author = {Gunning, Robert C.},
	isbn = {0-534-13308-8},
	mrclass = {32-01},
	mrnumber = {1052649},
	mrreviewer = {Eric Bedford},
	note = {Function theory},
	pages = {xx+203},
	publisher = {Wadsworth \& Brooks/Cole Advanced Books \& Software, Pacific Grove, CA},
	series = {The Wadsworth \& Brooks/Cole Mathematics Series},
	title = {Introduction to holomorphic functions of several variables. {V}ol. {I}},
	year = {1990}
}

@Article{Zhang2000,
  author     = {Zhang, Shou-Wu},
  journal    = {Duke Math. J.},
  title      = {Distribution of almost division points},
  year       = {2000},
  issn       = {0012-7094},
  number     = {1},
  pages      = {39--46},
  volume     = {103},
  doi        = {10.1215/S0012-7094-00-10313-4},
  fjournal   = {Duke Mathematical Journal},
  mrclass    = {11G10 (11G50 14K15)},
  mrnumber   = {1758238},
  mrreviewer = {Yuri Bilu},
  url        = {http://dx.doi.org/10.1215/S0012-7094-00-10313-4},
}

@Article{Yuan2008,
  author     = {Yuan, Xinyi},
  journal    = {Invent. Math.},
  title      = {Big line bundles over arithmetic varieties},
  year       = {2008},
  issn       = {0020-9910},
  number     = {3},
  pages      = {603--649},
  volume     = {173},
  doi        = {10.1007/s00222-008-0127-9},
  fjournal   = {Inventiones Mathematicae},
  mrclass    = {14G40 (11G50 14C20)},
  mrnumber   = {2425137},
  mrreviewer = {Antoine Chambert-Loir},
  url        = {http://dx.doi.org/10.1007/s00222-008-0127-9},
}

@Article{Gillet1990,
  author     = {Gillet, Henri and Soul\'e, Christophe},
  journal    = {Inst. Hautes Études Sci. Publ. Math.},
  title      = {Arithmetic intersection theory},
  year       = {1990},
  issn       = {0073-8301},
  number     = {72},
  pages      = {93--174 (1991)},
  fjournal   = {Institut des Hautes Études Scientifiques. Publications Math\'ematiques},
  mrclass    = {14G40 (14C17 32C30)},
  mrnumber   = {1087394},
  mrreviewer = {I. Dolgachev},
  url        = {http://www.numdam.org/item?id=PMIHES_1990__72__93_0},
}

@Article{Remond2005,
  author     = {R{\'e}mond, Ga\"{e}l},
  journal    = {Bull. Soc. Math. France},
  title      = {In\'egalit\'e de {V}ojta g\'en\'eralis\'ee},
  year       = {2005},
  issn       = {0037-9484},
  number     = {4},
  pages      = {459--495},
  volume     = {133},
  fjournal   = {Bulletin de la Soci\'et\'e Math\'ematique de France},
  mrclass    = {11G50 (11G35 14G25)},
  mrnumber   = {2233693},
  mrreviewer = {Bjorn Poonen},
}

@Article{Szpiro1997,
  author     = {Szpiro, Lucien and Ullmo, Emmanuel and Zhang, Shou-Wu},
  journal    = {Invent. Math.},
  title      = {\'Equir\'epartition des petits points},
  year       = {1997},
  issn       = {0020-9910},
  number     = {2},
  pages      = {337--347},
  volume     = {127},
  fjournal   = {Inventiones Mathematicae},
  mrclass    = {14G40 (11G35 14G05)},
  mrnumber   = {1427622},
  mrreviewer = {Antoine Chambert-Loir},
  url        = {https://doi.org/10.1007/s002220050123},
}

@InCollection{Chambert-Loir2011,
  author    = {Chambert-Loir, Antoine},
  booktitle = {Motivic integration and its interactions with model theory and non-{A}rchimedean geometry. {V}olume {II}},
  publisher = {Cambridge Univ. Press, Cambridge},
  title     = {Heights and measures on analytic spaces. {A} survey of recent results, and some remarks},
  year      = {2011},
  pages     = {1--50},
  series    = {London Math. Soc. Lecture Note Ser.},
  volume    = {384},
  mrclass   = {14G22},
  mrnumber  = {2885340},
}

@Article{Gubler1998,
  author     = {Gubler, Walter},
  journal    = {J. Reine Angew. Math.},
  title      = {Local heights of subvarieties over non-{A}rchimedean fields},
  year       = {1998},
  issn       = {0075-4102},
  pages      = {61--113},
  volume     = {498},
  fjournal   = {Journal f√ºr die Reine und Angewandte Mathematik. [Crelle's Journal]},
  mrclass    = {14G20 (11G25 14D15 14G40)},
  mrnumber   = {1629925},
  mrreviewer = {J√∂rg Jahnel},
  url        = {https://doi.org/10.1515/crll.1998.054},
}

@Article{Chambert-Loir2006,
  author   = {Chambert-Loir, Antoine},
  journal  = {J. Reine Angew. Math.},
  title    = {Mesures et \'equidistribution sur les espaces de {B}erkovich},
  year     = {2006},
  issn     = {0075-4102},
  pages    = {215--235},
  volume   = {595},
  fjournal = {Journal f√ºr die Reine und Angewandte Mathematik. [Crelle's Journal]},
  mrclass  = {14G40 (14G25)},
  mrnumber = {2244803},
  url      = {https://doi.org/10.1515/CRELLE.2006.049},
}

@book{Berkovich1990,
	author = {Berkovich, Vladimir G.},
	isbn = {0-8218-1534-2},
	mrclass = {32P05 (32C15 32C37 46S10 47S10)},
	mrnumber = {1070709},
	mrreviewer = {W. Bartenwerfer},
	pages = {x+169},
	publisher = {American Mathematical Society, Providence, RI},
	series = {Mathematical Surveys and Monographs},
	title = {Spectral theory and analytic geometry over non-{A}rchimedean fields},
	volume = {33},
	year = {1990}
}

@book{Kollar1996,
	author = {Koll\'ar, J\'anos},
	isbn = {3-540-60168-6},
	mrclass = {14-02 (14C05 14E05 14F17 14J45)},
	mrnumber = {1440180},
	mrreviewer = {Yuri G. Prokhorov},
	pages = {viii+320},
	publisher = {Springer-Verlag, Berlin},
	series = {Ergebnisse der Mathematik und ihrer Grenzgebiete. 3. Folge. A Series of Modern Surveys in Mathematics [Results in Mathematics and Related Areas. 3rd Series. A Series of Modern Surveys in Mathematics]},
	title = {Rational curves on algebraic varieties},
	url = {https://doi.org/10.1007/978-3-662-03276-3},
	volume = {32},
	year = {1996}
}

@Article{Snapper1960,
  author     = {Snapper, Ernst},
  journal    = {J. Math. Mech.},
  title      = {Polynomials associated with divisors},
  year       = {1960},
  pages      = {123--139},
  volume     = {9},
  mrclass    = {14.00},
  mrnumber   = {0114818},
  mrreviewer = {M. Nagata},
}

@book{Moriwaki2014,
	author = {Moriwaki, Atsushi},
	isbn = {978-1-4704-1074-2},
	mrclass = {14G40 (14-02)},
	mrnumber = {3244206},
	mrreviewer = {Dan Abramovich},
	note = {Translated from the 2008 Japanese original},
	pages = {x+285},
	publisher = {American Mathematical Society, Providence, RI},
	series = {Translations of Mathematical Monographs},
	title = {Arakelov geometry},
	volume = {244},
	year = {2014}
}

@InCollection{Yuan2012,
  author     = {Yuan, Xinyi},
  booktitle  = {fifth {I}nternational {C}ongress of {C}hinese {M}athematicians. {P}art 1, 2},
  publisher  = {Amer. Math. Soc., Providence, RI},
  title      = {Algebraic dynamics, canonical heights and {A}rakelov geometry},
  year       = {2012},
  pages      = {893--929},
  series     = {AMS/IP stud. Adv. Math., 51, pt. 1},
  volume     = {2},
  mrclass    = {37P15 (11G50 14G40 37P30)},
  mrnumber   = {2918034},
  mrreviewer = {J√∂rg Jahnel},
}

@Article{Gubler2007a,
  author     = {Gubler, Walter},
  journal    = {Invent. Math.},
  title      = {Tropical varieties for non-{A}rchimedean analytic spaces},
  year       = {2007},
  issn       = {0020-9910},
  number     = {2},
  pages      = {321--376},
  volume     = {169},
  fjournal   = {Inventiones Mathematicae},
  mrclass    = {14K05 (11G50 14G40)},
  mrnumber   = {2318559},
  mrreviewer = {J√∂rg Jahnel},
  url        = {https://doi.org/10.1007/s00222-007-0048-z},
}

@Article{Lelong1957,
  author     = {Lelong, Pierre},
  journal    = {Bull. Soc. Math. France},
  title      = {Int\'egration sur un ensemble analytique complexe},
  year       = {1957},
  issn       = {0037-9484},
  pages      = {239--262},
  volume     = {85},
  fjournal   = {Bulletin de la Soci\'et\'e Math\'ematique de France},
  mrclass    = {32.00},
  mrnumber   = {0095967},
  mrreviewer = {H. Grauert},
  url        = {http://www.numdam.org/item?id=BSMF_1957__85__239_0},
}

@Article{Maillot2000,
  author     = {Maillot, Vincent},
  journal    = {M\'em. Soc. Math. Fr. (N.S.)},
  title      = {G\'eom\'etrie d'{A}rakelov des vari\'et\'es toriques et fibr\'es en droites int\'egrables},
  year       = {2000},
  issn       = {0249-633X},
  number     = {80},
  pages      = {vi+129},
  fjournal   = {M\'emoires de la Soci\'et\'e Math\'ematique de France. Nouvelle S\'erie},
  mrclass    = {14G40 (11G50 14M25)},
  mrnumber   = {1775582},
  mrreviewer = {Yuichiro Takeda},
}

@Article{Abbes1997,
  author     = {Abbes, Ahmed},
  journal    = {Ast\'erisque},
  title      = {Hauteurs et discr\'etude (d'apr\`es {L}. {S}zpiro, {E}. {U}llmo et {S}. {Z}hang)},
  year       = {1997},
  issn       = {0303-1179},
  note       = {S\'eminaire Bourbaki, Vol. 1996/97},
  number     = {245},
  pages      = {Exp.\ No.\ 825, 4, 141--166},
  fjournal   = {Ast\'erisque},
  mrclass    = {14G40 (11G35)},
  mrnumber   = {1627110},
  mrreviewer = {Antoine Chambert-Loir},
}

@Article{Gubler2003,
  author     = {Gubler, Walter},
  journal    = {Ann. Sc. Norm. Super. Pisa Cl. Sci. (5)},
  title      = {Local and canonical heights of subvarieties},
  year       = {2003},
  issn       = {0391-173X},
  number     = {4},
  pages      = {711--760},
  volume     = {2},
  fjournal   = {Annali della Scuola Normale Superiore di Pisa. Classe di Scienze. Serie V},
  mrclass    = {14G40 (11G50 14G22)},
  mrnumber   = {2040641},
  mrreviewer = {Robin de Jong},
}

@InCollection{Gubler1997,
  author     = {Gubler, Walter},
  booktitle  = {Arithmetic geometry ({C}ortona, 1994)},
  publisher  = {Cambridge Univ. Press, Cambridge},
  title      = {Heights of subvarieties over {$M$}-fields},
  year       = {1997},
  pages      = {190--227},
  series     = {Sympos. Math., XXXVII},
  mrclass    = {14G40 (14G25 32H30)},
  mrnumber   = {1472498},
  mrreviewer = {J√∂rg Jahnel},
}

@unpublished{Smith26,
	Author = {Smith, David J.},
    Note = {\url{https://arxiv.org/abs/2610.10209}},
	Title = {Points and their multiples on curves in powers of simple abelian varieties},
    Year = {2026}}

@unpublished{OpenAI,
	Author = {OpenAI},
	Note = {\url{https://github.com/openai/math/blob/main/preprints/The-Abelian-Zilber-Pink-Conjecture-September-24-2026/paper.pdf}},
	Title = {The abelian {Z}ilber-{P}ink conjecture},
    Year = {2026}}

@unpublished{Boxall,
	Author = {Boxall, Gareth},
	Note = {\url{https://arxiv.org/abs/2401.05024}},
	Title = {Points on a curve with a power on a curve},
	Year = {2024}}

@unpublished{Baldi_Urbanik,
	Author = {Baldi, Gregorio and Urbanik, David},
	Note = {\url{https://arxiv.org/abs/2604.02186}},
	Title = {Intersections and the {B\'e}zout range: abelian varieties},
	Year = {2026}}

@unpublished{Dogra,
	Author = {Dogra, Netan},
	Note = {\url{https://arxiv.org/abs/2403.05481}, version 3},
	Title = {The {Z}ilber--{P}ink conjecture for products of curves with highly degenerate reduction},
	Year = {2026}}

@phdthesis{Ballini,
    Author = {Ballini, Francesco},
    Title = {Diophantine problems
of avoiding unlikely intersections},
    School = {University of Oxford},
    Year = {2024}}

@mastersthesis{Smith,
    Author = {Smith, David J.},
    Title = {An equivalence of categories in algebraic
geometry and some unlikely intersections in
powers of elliptic curves},
    School = {Stellenbosch University},
    Year = {2024}}

@Article{Yuan2017,
  author     = {Yuan, Xinyi and Zhang, Shou-Wu},
  title      = {The arithmetic {H}odge index theorem for adelic line bundles},
  journal    = {Math. Ann.},
  year       = {2017},
  volume     = {367},
  number     = {3-4},
  pages      = {1123--1171},
  issn       = {0025-5831},
  doi        = {10.1007/s00208-016-1414-1},
  fjournal   = {Mathematische Annalen},
  mrclass    = {14G40 (32S35 37B30 37P55)},
  mrnumber   = {3623221},
  mrreviewer = {Shu Kawaguchi},
  url        = {https://doi.org/10.1007/s00208-016-1414-1},
}

@Article{Habegger2012a,
  author     = {Habegger, Philipp and Pila, Jonathan},
  title      = {Some unlikely intersections beyond {A}ndr\'{e}-{O}ort},
  journal    = {Compos. Math.},
  year       = {2012},
  volume     = {148},
  number     = {1},
  pages      = {1--27},
  issn       = {0010-437X},
  doi        = {10.1112/S0010437X11005604},
  fjournal   = {Compositio Mathematica},
  mrclass    = {11G18 (03C64 11G15 11G50)},
  mrnumber   = {2881307},
  mrreviewer = {Alexandra Shlapentokh},
  url        = {https://doi.org/10.1112/S0010437X11005604},
}

@article{D20,
	author = {Dill, Gabriel A.},
	fjournal = {International Journal of Number Theory},
	journal = {Int. J. Number Theory},
	number = {1},
	pages = {107--120},
	title = {Generalized {V}ojta-{R}\'{e}mond inequality},
	volume = {16},
	year = {2020}}

@Article{Habegger2016,
  author     = {Habegger, Philipp and Pila, Jonathan},
  title      = {O-minimality and certain atypical intersections},
  journal    = {Ann. Sci. \'{E}c. Norm. Sup\'{e}r. (4)},
  year       = {2016},
  volume     = {49},
  number     = {4},
  pages      = {813--858},
  issn       = {0012-9593},
  doi        = {10.24033/asens.2296},
  fjournal   = {Annales Scientifiques de l'\'{E}cole Normale Sup\'{e}rieure. Quatri\`eme S\'{e}rie},
  mrclass    = {14G35 (03C64 14K12)},
  mrnumber   = {3552014},
  mrreviewer = {Fabrizio Barroero},
  url        = {https://doi.org/10.24033/asens.2296},
}

@article {BarroeroKuehneSchmidt2023,
    AUTHOR = {Barroero, Fabrizio and K\"uhne, Lars and Schmidt, Harry},
     TITLE = {Unlikely intersections of curves with algebraic subgroups in
              semiabelian varieties},
   JOURNAL = {Selecta Math. (N.S.)},
  FJOURNAL = {Selecta Mathematica. New Series},
    VOLUME = {29},
      YEAR = {2023},
    NUMBER = {2},
     PAGES = {Paper No. 18, 37},
      ISSN = {1022-1824,1420-9020},
   MRCLASS = {11G10 (03C64 11G50 14G40)},
  MRNUMBER = {4537438},
MRREVIEWER = {Alina\ Ostafe},
       DOI = {10.1007/s00029-022-00823-w},
       URL = {https://doi.org/10.1007/s00029-022-00823-w},
}

@article {BarroeroDill2022,
    AUTHOR = {Barroero, Fabrizio and Dill, Gabriel A.},
     TITLE = {On the {Z}ilber-{P}ink conjecture for complex abelian
              varieties},
   JOURNAL = {Ann. Sci. \'Ec. Norm. Sup\'er. (4)},
  FJOURNAL = {Annales Scientifiques de l'\'Ecole Normale Sup\'erieure.
              Quatri\`eme S\'erie},
    VOLUME = {55},
      YEAR = {2022},
    NUMBER = {1},
     PAGES = {261--282},
      ISSN = {0012-9593,1873-2151},
   MRCLASS = {11G10 (14K15)},
  MRNUMBER = {4411863},
MRREVIEWER = {Guy\ Fowler},
       DOI = {10.24033/asens.2496},
       URL = {https://doi.org/10.24033/asens.2496},
}

@article{BD2,
  AUTHOR = {Barroero, Fabrizio and Dill, Gabriel A.},
     TITLE = {Distinguished categories and the {Z}ilber-{P}ink conjecture},
   JOURNAL = {Amer. J. Math.},
  FJOURNAL = {American Journal of Mathematics},
    VOLUME = {147},
      YEAR = {2025},
    NUMBER = {3},
     PAGES = {715--778},
      ISSN = {0002-9327,1080-6377},
   MRCLASS = {14G05 (11G18 14G35 18A99)},
  MRNUMBER = {4914754},
}

@Article{Remond2000,
  author     = {R{\'e}mond, Ga\"{e}l},
  title      = {In\'{e}galit\'{e} de {V}ojta en dimension sup\'{e}rieure},
  journal    = {Ann. Scuola Norm. Sup. Pisa Cl. Sci. (4)},
  year       = {2000},
  volume     = {29},
  number     = {1},
  pages      = {101--151},
  issn       = {0391-173X},
  fjournal   = {Annali della Scuola Normale Superiore di Pisa. Classe di Scienze. Serie IV},
  mrclass    = {11J95 (11G10 11G50 14G05 14K15)},
  mrnumber   = {1765539},
  mrreviewer = {Carlo Gasbarri},
  url        = {http://www.numdam.org/item?id=ASNSP_2000_4_29_1_101_0},
}

@Article{Boucksom2021,
  author   = {Boucksom, S\'{e}bastien and Eriksson, Dennis},
  title    = {Spaces of norms, determinant of cohomology and {F}ekete points in non-{A}rchimedean geometry},
  journal  = {Adv. Math.},
  year     = {2021},
  volume   = {378},
  pages    = {107501, 124},
  issn     = {0001-8708},
  doi      = {10.1016/j.aim.2020.107501},
  fjournal = {Advances in Mathematics},
  mrclass  = {32P05 (14G22 31C15)},
  mrnumber = {4192993},
  url      = {https://doi.org/10.1016/j.aim.2020.107501},
}

@Article{Carrizosa2009,
  author     = {Carrizosa, Mar\'{\i}a},
  title      = {Petits points et multiplication complexe},
  journal    = {Int. Math. Res. Not. IMRN},
  year       = {2009},
  number     = {16},
  pages      = {3016--3097},
  issn       = {1073-7928},
  doi        = {10.1093/imrn/rnp042},
  fjournal   = {International Mathematics Research Notices. IMRN},
  mrclass    = {11G50 (11G10 11G15)},
  mrnumber   = {2533796},
  mrreviewer = {Matthew H. Baker},
  url        = {https://doi.org/10.1093/imrn/rnp042},
}

@Article{Remond2005a,
  author     = {R{\'e}mond, Ga\"{e}l},
  title      = {Intersection de sous-groupes et de sous-vari\'{e}t\'{e}s. {I}},
  journal    = {Math. Ann.},
  year       = {2005},
  volume     = {333},
  number     = {3},
  pages      = {525--548},
  issn       = {0025-5831},
  doi        = {10.1007/s00208-005-0673-z},
  fjournal   = {Mathematische Annalen},
  mrclass    = {11J95 (11G10 11G50 14K12)},
  mrnumber   = {2198798},
  mrreviewer = {St\'{e}phane Fischler},
  url        = {https://doi.org/10.1007/s00208-005-0673-z},
}

@Article{Remond2003a,
  author     = {R\'{e}mond, Ga\"{e}l and Viada, Evelina},
  title      = {Probl\`eme de {M}ordell-{L}ang modulo certaines sous-vari\'{e}t\'{e}s ab\'{e}liennes},
  journal    = {Int. Math. Res. Not.},
  year       = {2003},
  number     = {35},
  pages      = {1915--1931},
  issn       = {1073-7928},
  doi        = {10.1155/S1073792803130164},
  fjournal   = {International Mathematics Research Notices},
  mrclass    = {11G10 (11G50 11J95 14G25 14K15)},
  mrnumber   = {1995142},
  mrreviewer = {Dino J. Lorenzini},
  url        = {https://doi.org/10.1155/S1073792803130164},
}

@Article{Kuehne2022,
  author   = {K{\"u}hne, Lars},
  journal  = {J. Eur. Math. Soc. (JEMS)},
  title    = {Points of small height on semiabelian varieties},
  year     = {2022},
  issn     = {1435-9855},
  number   = {6},
  pages    = {2077--2131},
  volume   = {24},
  doi      = {10.4171/jems/1125},
  fjournal = {Journal of the European Mathematical Society (JEMS)},
  mrclass  = {14G40 (11G10 11G50 14G05 14K15)},
  mrnumber = {4404794},
  url      = {https://doi.org/10.4171/jems/1125},
}

@Book{Edixhoven1993,
  editor     = {Edixhoven, Bas and Evertse, Jan-Hendrik},
  publisher  = {Springer-Verlag, Berlin},
  title      = {Diophantine approximation and abelian varieties},
  year       = {1993},
  isbn       = {3-540-57528-6},
  note       = {Introductory lectures, Papers from the conference held in Soesterberg, April 12--16, 1992},
  series     = {Lecture Notes in Mathematics},
  volume     = {1566},
  doi        = {10.1007/978-3-540-48208-6},
  mrclass    = {11G35 (11-06 11J99 14G05 14G10 14G40)},
  mrnumber   = {1288998},
  mrreviewer = {Michel\ Waldschmidt},
  pages      = {xiv+127},
  url        = {https://doi.org/10.1007/978-3-540-48208-6},
}

@Article{Evertse1995,
  author     = {Evertse, Jan-Hendrik},
  journal    = {Acta Arith.},
  title      = {An explicit version of {F}altings' product theorem and an improvement of {R}oth's lemma},
  year       = {1995},
  issn       = {0065-1036,1730-6264},
  number     = {3},
  pages      = {215--248},
  volume     = {73},
  doi        = {10.4064/aa-73-3-215-248},
  fjournal   = {Acta Arithmetica},
  mrclass    = {11J17 (11G35 11J13 14G05)},
  mrnumber   = {1364461},
  mrreviewer = {Michel\ Waldschmidt},
  url        = {https://doi.org/10.4064/aa-73-3-215-248},
}

@book {Matsumura1986,
    AUTHOR = {Matsumura, Hideyuki},
     TITLE = {Commutative ring theory},
    SERIES = {Cambridge Studies in Advanced Mathematics},
    VOLUME = {8},
      NOTE = {Translated from the Japanese by M. Reid},
 PUBLISHER = {Cambridge University Press, Cambridge},
      YEAR = {1986},
     PAGES = {xiv+320},
      ISBN = {0-521-25916-9},
   MRCLASS = {13-01},
  MRNUMBER = {879273},
MRREVIEWER = {W.\ V.\ Vasconcelos},
}

@book{Bosch1984,
    AUTHOR = {Bosch, S. and G\"untzer, U. and Remmert, R.},
     TITLE = {Non-{A}rchimedean analysis},
    SERIES = {Grundlehren der mathematischen Wissenschaften [Fundamental
              Principles of Mathematical Sciences]},
    VOLUME = {261},
      NOTE = {A systematic approach to rigid analytic geometry},
 PUBLISHER = {Springer-Verlag, Berlin},
      YEAR = {1984},
     PAGES = {xii+436},
      ISBN = {3-540-12546-9},
   MRCLASS = {32K10 (30G05 46P05)},
  MRNUMBER = {746961},
MRREVIEWER = {W.\ Bartenwerfer},
       DOI = {10.1007/978-3-642-52229-1},
       URL = {https://doi.org/10.1007/978-3-642-52229-1},
}

@incollection {Gubler2013,
    AUTHOR = {Gubler, Walter},
     TITLE = {A guide to tropicalizations},
 BOOKTITLE = {Algebraic and combinatorial aspects of tropical geometry},
    SERIES = {Contemp. Math.},
    VOLUME = {589},
     PAGES = {125--189},
 PUBLISHER = {Amer. Math. Soc., Providence, RI},
      YEAR = {2013},
      ISBN = {978-0-8218-9146-9},
   MRCLASS = {14T05 (14M25 32P05)},
  MRNUMBER = {3088913},
MRREVIEWER = {Aljo\v sa\ Peperko},
       DOI = {10.1090/conm/589/11745},
       URL = {https://doi.org/10.1090/conm/589/11745},
}

@book {Noguchi2016,
    AUTHOR = {Noguchi, Junjiro},
     TITLE = {Analytic function theory of several variables},
      NOTE = {Elements of Oka's coherence},
 PUBLISHER = {Springer, Singapore},
      YEAR = {2016},
     PAGES = {xvi+397},
      ISBN = {978-981-10-0289-2; 978-981-10-0291-5},
   MRCLASS = {32-02 (14F05 32Axx 32Cxx 32Exx 32Txx)},
  MRNUMBER = {3526579},
MRREVIEWER = {Eduardo\ S.\ Zeron},
       DOI = {10.1007/978-981-10-0291-5},
       URL = {https://doi.org/10.1007/978-981-10-0291-5},
}

\end{document}